\documentclass[final]{amsart}

\usepackage[dvipsnames,hyperref]{xcolor}
\usepackage[notref,notcite]{showkeys}
\usepackage{accents}
\usepackage{url, hyperref}
\usepackage{amsmath, amssymb, amsfonts, amsthm}
\usepackage{mathrsfs}
\usepackage{euscript}
\usepackage{tikz-cd}
\usepackage{enumitem}   
\usepackage{longtable}
\usepackage{array}

\setlist[enumerate]{font={\normalfont}}
\setlist[enumerate,1]{label={(\roman*)}}

\definecolor{DarkBlue}{RGB}{9,53,122}
\hypersetup{
  colorlinks   = true, 
  urlcolor     = PineGreen, 
  linkcolor    = DarkBlue, 
  citecolor    = NavyBlue, 
}

\newcommand\nc{\newcommand}
\nc\nop\DeclareMathOperator

\nc\JPW[1]{{\color{blue} \sf JW: [#1]}}
\nc\YS[1]{{\color{magenta} \sf
    $\diamondsuit\diamondsuit\diamondsuit$ YS: [#1]}}

\newtheorem{thm}[subsubsection]{Theorem}
\newtheorem{lem}[subsubsection]{Lemma}

\newtheorem{prop}[subsubsection]{Proposition}
\newtheorem{prop-const}[subsubsection]{Proposition-Construction}
\newtheorem{cor}[subsubsection]{Corollary}
\newtheorem{conj}[subsubsection]{Conjecture}

\theoremstyle{remark}
\newtheorem{rem}[subsubsection]{Remark}

\theoremstyle{definition}
\newtheorem{defn}[subsubsection]{Definition}
\newtheorem{eg}[subsubsection]{Example}

\nop\Aut{Aut}
\nop\Hom{Hom}
\nop\Ext{Ext}
\nop\Tor{Tor}
\nop\Maps{Maps}
\nop\Spec{Spec}
\nop\Coh{Coh}
\nop\QCoh{QCoh}
\nop\IndCoh{IndCoh}
\nop\Rep{Rep}
\nop\Pic{Pic}
\nop\Bun{Bun}
\nop\LocSys{LocSys}
\nop\Sat{Sat}
\nop\Hecke{Hecke}
\nop\Sym{Sym}
\nop\Div{Div}
\nop\Loc{Loc}
\nop\Eis{Eis}
\nop\coker{coker}
\nop\Four{Four}
\nop\codim{codim}
\nop\Lie{Lie}
\nop\act{act}
\nop\pr{pr}
\nop\Prim{Prim}
\nop\len{len}

\nc\on{\operatorname}
\nc\PGL{\mathrm{PGL}}
\nc\SL{\mathrm{SL}}
\nc\GL{\mathrm{GL}}
\nc\End{\mathrm{End}}
\nc\uHom{\underline{\mathrm{Hom}}}
\nc\Gr{\mathrm{Gr}}
\nc\IC{\mathrm{IC}}
\nc\pt{\mathsf{pt}}
\nc\diag{\mathrm{diag}}
\nc\Ran{\mathrm{Ran}}
\nc\vareps{\varepsilon}
\nc\ch{\check}
\nc\mbb[1]{\mathbb{#1}}
\nc\Cal[1]{\mathcal{#1}}
\nc\Scr[1]{\EuScript{#1}}
\nc\mf[1]{\mathfrak{#1}}
\nc\msf[1]{\mathsf{#1}}
\nc\into\hookrightarrow
\nc\onto{\twoheadrightarrow}
\nc\bs\backslash
\nc\can{\mathrm{can}}
\nc\eff{\mathrm{eff}}
\nc\aff{\mathrm{aff}}
\nc\disj{\mathrm{disj}}
\nc\pos{\mathrm{pos}}
\nc\gen{\mathrm{gen}}
\nc\red{\mathrm{red}}
\nc\glob{\mathrm{glob}}
\nc\flag{\mathrm{flag}}
\nc\oo[1]{\mathring{#1}}
\nc\abs[1]{\lvert#1\rvert}
\nc\brac[1]{\langle#1\rangle}
\nc\tbrac[1]{[\![#1]\!]}
\nc\lbrac[1]{(\!(#1)\!)}
\nc\smallfrac[2]{{\textstyle \frac{#1}{#2}}}
\nc\ol\overline
\nc\ul\underline
\nc\wh\widehat
\nc\wt\widetilde
\nc\mss{\mathrm{ss}}
\nc\olsf[1]{{\ol{\msf #1}}}
\nc\barY{{\ol{\Scr Y}}}
\nc\barM{{\ol{\Scr M}}}
\nc\sm{-}
\nc\ds\displaystyle

\nc\bbD{\mbb D}
\nc\bbQ{\mbb Q}
\nc\bbN{\mbb N}
\nc\bbZ{\mbb Z}
\nc\bbA{\mbb A}
\nc\bbP{\mbb P}
\nc\rmP{\mathrm P}
\nc\rmH{\mathrm H}
\nc\rmD{\mathrm D}
\nc\mfU{\mathfrak U}
\nc\mfn{\mathfrak n}
\nc\sY{\Scr Y}
\nc\sZ{\Scr Z}
\nc\sM{\Scr M}
\nc\sA{\Scr A}
\nc\sF{\Scr F}
\nc\sG{\Scr G}
\nc\sH{\Scr H}
\nc\sP{\Scr P}
\nc\mB{\mathfrak B}
\nc*{\sslash}{/\!\!/}
\nc\Fr{\mathrm{Fr}}
\nc{\tr}{tr}

\DeclareMathOperator*{\ttimes}{\tilde\times}

\DeclareMathOperator*{\ot}{\otimes}
\DeclareMathOperator*{\xt}{\times}
\DeclareMathOperator*{\bt}{\boxtimes}
\DeclareMathOperator*{\hot}{\wh\otimes}

\DeclareMathOperator*{\colim}{colim}

\numberwithin{equation}{section}

\title{Intersection complexes and unramified $L$-factors}
\author{Yiannis Sakellaridis}
\author{Jonathan Wang}

\begin{document} 

\begin{abstract}
%
Let $X$ be an affine spherical variety, possibly singular, and $\msf L^+X$ its arc space. The intersection complex of $\msf L^+X$, or rather of its finite-dimensional formal models, is conjectured to be related to special values of local unramified $L$-functions. Such relationships were previously established in \cite{BFGM} for the affine closure of the quotient of a reductive group by the unipotent radical of a parabolic, and in \cite{BNS} for toric varieties and $L$-monoids. In this paper, we compute this intersection complex for the large class of those spherical $G$-varieties whose dual group is equal to $\check G$, and the stalks of its nearby cycles on the horospherical degeneration of $X$. We formulate the answer in terms of a Kashiwara crystal, which conjecturally corresponds to a finite-dimensional $\check G$-representation determined by the set of $B$-invariant valuations on $X$. 
We prove the latter conjecture in many cases. Under the sheaf--function dictionary,
our calculations give a formula for the Plancherel density of the IC function
of $\msf L^+X$ as a ratio of local $L$-values for a large class of spherical varieties.
\end{abstract}

\maketitle

\tableofcontents

\section{Introduction}

\subsection{Arc spaces and their IC functions}

Let $\mathbb F$ be a finite field, $G$ a connected reductive group over $\mathbb F$, and $X$ an affine spherical $G$-variety. The formal arc space $\msf L^+X$ is the infinite-dimensional scheme that represents the functor $R\mapsto X(R\tbrac t)$, and it is singular, if $X$ is. However, its singularities at ``generic'' $\mathbb F$-points (namely, arcs $\Spec \mathbb F\tbrac t \to X$ which generically lie in the smooth locus $X^{\text{sm}}$) are of finite type, according to the theorem of Grinberg--Kazhdan and Drinfeld, and this allows one to define an \emph{IC function} \cite{BNS}, that is, the function that should correspond under Frobenius trace to the ``intersection complex'' of $\msf L^+X$. This is a function $\Phi_0$ on $X(\mathfrak o)\cap X^{\text{sm}}(F)$ (where $\mathfrak o=\mathbb F\tbrac t$ and $F = \mathbb F\lbrac t$), and it was conjectured in \cite{SaRS} (before the rigorous definition of this function was available) that it is related to special values of $L$-functions. Such a relation was established in \cite{BNS} for toric varieties and certain group embeddings, termed $L$-monoids, generalizing the local unramified Godement--Jacquet theory. One goal of the present paper is to prove such relations, for a different and broad class of spherical varieties. 

In order to compute the IC function, we need to work with finite-dimensional global models of the arc space, or rather, the arc space over the algebraic closure $k$ of $\mathbb F$; from now on, we  base change to $k$,  without changing notation. However, for motivational purposes, let us here pretend that $\msf L^+X$ were already finite-dimensional, and the intersection complex on it were defined as a $\ol\bbQ_\ell$-valued constructible (derived) sheaf, for some prime $\ell$ different from the characteristic of $\mathbb F$. We would normalize such a sheaf to be constant in degree zero over the arc space of the smooth locus $\msf L^+X^{\text{sm}}$, and we would normalize the intersection complex of a substratum $S\subset \msf L^+X$ of codimension $d$ to be $\ol\bbQ_\ell(\frac{d}{2})[d]$ on its smooth locus, where $\ol\bbQ_\ell(\frac{1}{2})$ denotes a chosen square root of the cyclotomic Tate twist. 

The IC function is computed via some version of a ``Satake transform'' for the spherical variety --- we will recover the original function from its Satake transform in Corollary \ref{cor:asymptotics}. Let $B\supset N$ be a Borel subgroup of $G$, and its unipotent radical, and consider the invariant-theoretic quotient $X\sslash N = \Spec k[X]^N$, which is an affine embedding of a quotient $T_X$ of the Cartan $T=B/N$. In this paper, we restrict ourselves to varieties where $B$ acts freely on an open subset $X^\circ$ of $X$, so let us already make this assumption for notational simplicity. Then $T_X=T$, and we fix a base point to identify $T$ as the open orbit in $X\sslash N$. In fact, our assumption on $X$ is stronger, requiring that the (Gaitsgory--Nadler) \emph{dual group} of $X$, $\check G_X$, is equal to the Langlands dual group of $G$ --- this condition is equivalent\footnote{To be precise, we are referring to the modification of the Gaitsgory--Nadler dual group described in \cite{SV}, because the Gaitsgory--Nadler dual group would be $\check G$ even if the stabilizers were \emph{normalizers} of tori. This small distinction is important, and such cases (for example, $\mathrm{O}_n\backslash \GL_n$) are \emph{not} expected to be directly related to $L$-functions, and not included in the present paper.} to the following:
\begin{equation}
\label{equation-typeT}
\parbox{0.8\linewidth}{$B$ acts freely on $X^\circ$, and for every simple root $\alpha$, if $P_\alpha$ is the parabolic generated by $B$ and the root space of $-\alpha$, the stabilizer of a point in the open $P_\alpha$-orbit $X^\circ P_\alpha$ is a torus (necessarily one-dimensional).}
\end{equation}

The pushforward map $\pi: X\to X\sslash N$ induces a map between arc spaces. Under our assumptions, the ``generic'' $\msf L^+T$-orbits on the arc space of $X\sslash N$, that is, those corresponding to arcs whose generic fiber lands in the open $T$-orbit, are naturally parametrized by a strictly convex (i.e., not containing non-trivial subgroups) submonoid $\mathfrak c_X\subset \check\Lambda$ of the cocharacter group of $T$, with $\check\lambda\in\mathfrak c_X$ corresponding to the image $t^{\ch\lambda}:=\check\lambda(t)$ of a uniformizer (acting on a fixed base point of $X\sslash N$). The pushforward $\pi_! \IC_{\msf L^+X}$ of the IC sheaf is $\msf L^+T$-equivariant, and under Frobenius trace translates to a $T(\mathfrak o)$-invariant function on $X\sslash N (\mathfrak o) \cap T(F)$, which will be denoted by $\pi_! \Phi_0$. Explicitly, 
\[ \pi_! \Phi_0 (a) = \int_{N(F)} \Phi_0(an) dn,\] 
is a Radon transform on the spherical variety, i.e., the integral of the IC function $\Phi_0$ over \emph{generic horocycles}, that is, over the fibers of the map $X(\mathfrak o)\to X\sslash N(\mathfrak o)$, where the Haar measure on $N(F)$ is so that $dn(N(\mathfrak o))=1$. 

\smallskip

This integral is really a finite sum, hence makes sense over $\ol\bbQ_\ell$, but let us for simplicity choose an isomorphism $\ol\bbQ_\ell \simeq \mathbb C$, such that the geometric Frobenius morphism $\Fr$ acts on the chosen half-Tate twist $\ol\bbQ_\ell(\frac{1}{2})$ by $q^\frac{1}{2}$. 
Our results and conjectures are best expressed under the assumption that $X$ carries a $G$-eigen-volume form;  we will assume this for the rest of the introduction. The absolute value of the volume form is a $G(F)$-eigenmeasure on the open orbit $X^\bullet(F)$, whose  eigencharacter we will denote by $\eta$. (We will not impose this assumption in the rest of the paper, but see Remark \ref{rem:length}.) Then, one should consider the following normalized form of the integral above, analogous to the standard normalization of the Satake isomorphism:
\begin{equation}
(\eta\delta)^\frac{1}{2}(a) \pi_! \Phi_0 (a) = (\eta\delta)^\frac{1}{2}(a)\int_{N(F)} \Phi_0(an) dn,
\end{equation}
where $\delta = |e^{2\rho_G}|$ is the modular character of the Borel subgroup.\footnote{We use additive notation for the character group $\ch\Lambda$ of $\ch T$, so $e^{\ch\nu}$ will denote the actual morphism $\ch T\to \mathbb G_m$ corresponding to $\ch\nu\in \ch\Lambda$.}

Ideally, we would like to prove a conjecture such as the following. To formulate it, recall that $T(\mathfrak o)$-orbits on $T(F)$ are parametrized by Galois-fixed (that is, $\Fr$-fixed) elements of $\ch\Lambda$.

\begin{conj}
\label{conjecture-pushforward-IC}
There is a symplectic representation $\rho_X$ of the $L$-group, $\rho_X: {^LG}=\check G \rtimes \left<\Fr\right> \to \GL(V_X)$, and a ${^LT}$-stable polarization $V_X=V_X^+\oplus V_X^-$, such that the multiset $\mB^+$ of $\check T$-weights of $V_X^+$ belongs to $\mathfrak c_X$, and the pushforward of the IC function satisfies:
 \begin{equation}
 \label{equation-pushforward-IC}
  (\eta\delta)^\frac{1}{2} \cdot \pi_! \Phi_0 = \left(\tr_{\check T}(\Fr, \Sym^\bullet (\check{\mathfrak n}(1)))\right)^{-1} \cdot \tr_{\check T}(\Fr, \Sym^\bullet (V_X^+)).
 \end{equation}
 Here, for a representation $V$ of ${^LT}=\check T\rtimes\left<\Fr\right>$, the expression $\tr_{\check T}(\Fr,V)$ denotes the function on $\ch\Lambda^{\Fr}$ whose value on $\ch\lambda$ is equal to the trace of geometric Frobenius $\Fr$ on the $(\ch T, \ch\lambda)$-eigenspace of $V$. 
\end{conj}

From the point of view of number theory, the spherical varieties satisfying our assumption $\check G_X=\check G$ give, in some sense, the most interesting periods, because they are associated to $L$-values at the center of the critical strip. Indeed, one should always be able to choose the eigencharacter $\eta$ such that the Frobenius morphism acts on $V_X^+$ by permuting elements of a basis and scaling by $q^\frac{1}{2}$. 
For example, when $G$ is split and the colors (see below) of $X$ are all defined over $\mathbb F$, this permutation action should be trivial, and \eqref{equation-pushforward-IC} should read:
\begin{equation}\label{equation-pushforward-IC-expl}
(\eta\delta)^\frac{1}{2} \pi_! \Phi_0 = \frac{\prod_{\check\alpha\in\check\Phi^+} (1-q^{-1} e^{\check\alpha})}{\prod_{\check\lambda \in \mB^+} (1-q^{-\frac{1}{2}} e^{\check\lambda})},  
\end{equation}
where $\mB^+$ is the multiset of weights, as in the conjecture, and $e^{\ch\lambda}$ denotes the characteristic function of the $T(\mf o)$-orbit of $t^{\ch\lambda}$.

We will explain the relationship of this conjecture to various conjectures of arithmetic and geometric origin below. We do not quite prove the conjecture in all cases, but we determine the weights of $\rho_X$ 
(in terms of $X$) 
and in the cases where $\rho_X$ is minuscule, we prove the conjecture (Corollary~\ref{cor:minuscule}). 
It is helpful to distinguish between the special case when $X = \overline{H\backslash G}^\aff$ is the affine closure of a homogeneous quasiaffine variety, and the general case. In the special case, let us also assume, at first, that 
the monoid $\mf c_X$ is freely generated with a basis $\ch\nu_1,\dotsc, \ch\nu_r$, so 
$X\sslash N$ may be identified with $\mbb A^r$. This condition can always be achieved by passing to an abelian cover of $H\backslash G$ and taking its affine closure, see \S \ref{sect:freemonoid}. In that case, the generators $\ch\nu_i$ are the valuations associated to the \emph{colors}, that is, the prime $B$-stable divisors of $H\backslash G$. 

\begin{thm}[See \S \ref{sect:finitefields}]
\label{theorem-affine-closure}
 Assume that $X = \overline{H\backslash G}^\aff$ satisfies the conditions above ($T_X=T$, $\check G_X=\check G$ and $\mf c_X \cong \mbb N^r$ is free). Then there is a $(\check T\rtimes\left<\Fr\right>)$-representation $V^+_X$ satisfying:
 \begin{enumerate}
  \item the $\check T$-weights of $V^+_X$ belong to $\mathfrak c_X\sm 0$;
  \item the set of weights of $V^+_X \oplus (V^+_X)^*$ (without multiplicities) equals the set of weights of a $\ch G$-representation $\rho_X$;
  \item the dimensions of the weight spaces of $V^+_X \oplus (V^+_X)^*$ are invariant under the Weyl group $W$ of $G$;
  \item the weight spaces in $V^+_X$ of the basis elements $\check\nu_1,\dots,\check\nu_r$ of $\mathfrak c_X$ have multiplicity one,
  \item under the Frobenius action, we have 
  \begin{equation}\label{e:Frobaction}
   V^+_X= \bigoplus_{\mB^+} \ol\bbQ_\ell (\smallfrac{1}{2}),
\end{equation}
 for some permutation action of $\Fr$ on the multiset $\mB^+$ of weights, compatible with its action on $\ch T$,    
\end{enumerate}
 such that the pushforward $\pi_! \Phi_0$ of the IC function\footnote{Under the assumptions of the theorem, the restriction of the eigencharacter $\eta$ to the colors $\ch\nu_i$ is uniquely determined, see Remark \ref{rem:length}.} satisfies the formula \eqref{equation-pushforward-IC} above.
\end{thm}

In fact, we show more: we endow the multiset $\mB := \mB^+ \sqcup (-\mB^+)$ of weights of $V_X$ with the structure of a \emph{Kashiwara crystal}, see Theorem \ref{theorem-crystal} below, and show that, if Conjecture~\ref{conjecture-pushforward-IC} holds (equivalently, if the crystal is the one corresponding to the crystal basis of a finite-dimensional $\ch G$-module),
then $\rho_X$ must be the direct sum of the irreducible $\ch G$-modules with highest weights 
in $\ch\Lambda^+ \cap W \{ \ch\nu_1,\dotsc,\ch\nu_r \}$, each with multiplicity one 
(see Remark~\ref{rem:conj}).
In other words, the highest weights of $\rho_X$ are the dominant coweights that are Weyl translates
of the basis elements $\ch\nu_1,\dotsc,\ch\nu_r$. 

\smallskip

As we already mentioned, Theorem~\ref{theorem-affine-closure} implies Conjecture~\ref{conjecture-pushforward-IC} when $\rho_X$ is minuscule. 
In particular, when $H\bs G$ is itself affine (equivalently, $H$ is reductive: see \cite{Lun73,Richardson}),
we observe that $\rho_X$ is always minuscule (Corollary~\ref{cor:Hreductive}). In this case
Conjecture~\ref{conjecture-pushforward-IC} was previously proved by \cite[Theorem 7.2.1]{SaSph}, under some additional assumptions,
and Theorem~\ref{theorem-affine-closure} gives a geometric interpretation of this result.

For an example when $H\bs G$ is not affine:

\begin{eg}\label{example-nfold}
Let $X^\bullet=$ the quotient of $\SL_2^n$ by the unipotent subgroup 
\[ H_0= \left\{ \left.\left(\begin{array}{cc} 1 \\ x_1 & 1\end{array}\right) \times \left(\begin{array}{cc} 1 \\ x_2 & 1\end{array}\right) \times \cdots \times\left(\begin{array}{cc} 1 \\ x_n & 1\end{array}\right) \right| x_1+ x_2 + \dots+ x_n = 0\right\},\]
under the action of $G = $ the quotient of $\mathbb G_m \times \SL_2^n$ by the diagonal copy of $\mu_2$, where $a\in \mathbb G_m$ acts as left multiplication by $\begin{pmatrix} a^{-1} \\ & a\end{pmatrix}$. 
Let $X$ be the affine closure of $X^\bullet$. Denoting by $\check m$ the identity cocharacter of $\mathbb G_m$, the monoid $\mathfrak c_X$ is freely generated by the coweights
\[ \frac{\check\alpha_1+\check\alpha_2+\dots+\check\alpha_n - \check m}{2}\]
and 
\[ \frac{-\check\alpha_1- \dots - \check\alpha_{i-1} + \check\alpha_i - \dots -\check\alpha_n + \check m}{2}, \,\, i =1,\dots, n,\]
see Remark \ref{rem:valuation}. These are minuscule weights of $\check G= \GL_2\times_{\det} \GL_2 \times_{\det} \cdots \times_{\det}\GL_2$, and in that case Conjecture \ref{conjecture-pushforward-IC} holds, confirming an expectation of \cite[\S 4.5]{SaRS}.  
\end{eg}

For the general case, $X$ contains an open $G$-orbit $X^\bullet = H\backslash G$, which we will assume to satisfy the conditions of Theorem \ref{theorem-affine-closure}, except perhaps for the freeness of $\mathfrak c_{X^\bullet}$. In that case, the free monoid $\mathbb N^{\mathcal D}$, where $\mathcal D$ denotes the set of colors, maps through the valuation map to $\mathfrak c_{X^\bullet}$, and $\mathfrak c_X$ is generated by its image and a minimal set $\Cal D^G_{\mathrm{sat}}(X)=\{\check \theta_1, \dots,\check\theta_d\}$ of distinct \emph{antidominant} elements of the cocharacter group $\check\Lambda$ of $T$. For each $\check\theta_i$, we let $V^{\check\theta_i}$ be the irreducible module of $\check G$ with lowest weight $\check\theta_i$, and assume that the eigencharacter $\eta$ of the $G$-eigenmeasure on $X(F)$ is of the form $|e^{\mathfrak h}|$ for some algebraic character $\mathfrak h \in \Lambda^W$. 

\begin{thm}[See \S \ref{sect:finitefields}]
\label{theorem-general}
In the setting above, if $V_{X^\bullet}^{+}$ denotes the $\check T$-representation for $\overline{X^\bullet}^\aff$ described in Theorem \ref{theorem-affine-closure},\footnote{See \S \ref{sect:freemonoid} for a reduction to the case where $\mathfrak c_{X^\bullet}$ is free.} then the pushforward $  (\eta\delta)^\frac{1}{2}\cdot \pi_! \Phi_0 $ of the IC function for $\msf L^+X$ is given by \eqref{equation-pushforward-IC} with 
\begin{equation}
 V^+_X = V_{X^\bullet}^{+} \oplus \bigoplus_i V^{\check\theta_i}\left(\frac{\left<\mathfrak h + 2\rho_G,\ch\theta_i\right>}{2}\right).
\end{equation}
\end{thm}

Notice that the set $\Cal D^G_{\mathrm{sat}}(X)$ is stable under the Frobenius morphism. The action of Frobenius on the sum of $V^{\check\theta_i}$'s is the one obtained by identifying the crystal basis of this space with a set of subvarieties of the affine Grassmannian (see Section \ref{sect:crystal}), and considering the Frobenius action on those.

The reader should compare the passage from Theorem \ref{theorem-affine-closure} to Theorem \ref{theorem-general} to the passage from a reductive group $G$ to an $L$-monoid $X$ in \cite[Theorem 4.1]{BNS, BNS-erratum}. While the result is similar, however, the straightforward proof of \cite{BNS} uses the monoid structure on $X$ in a crucial way, and cannot be used here. 

\medskip

In \S \ref{subsection-introduction-global} below we will describe the sheaf-theoretic statements of these theorems, in the setting of appropriate finite type models, the \emph{Zastava spaces} for $X$ and $X\sslash N$. Before we do that, let us relate the results above to conjectures in number theory and geometry.

\subsection{Asymptotics and $L$-functions}

The Radon transform $\pi_! \Phi_0$ of the IC function (also known as ``basic function'') under the map $X\to X\sslash N$ admits various interpretations in terms of harmonic analysis, and, in particular,  allows us to compute the Plancherel density of the basic function,
\begin{equation}\label{Plancherel} \Vert \Phi_0 \Vert^2 = \int_{\ch T^1/W} \Omega(\chi) d\chi,
\end{equation}
that is, the decomposition of its norm in the space $L^2(X^\bullet(F))$ (with respect to the fixed eigenmeasure) in terms of seminorms $\Vert \bullet \Vert_\chi$ that factor through eigenquotients for the action of the unramified Hecke algebra. The variable $\chi\in \ch T^1$ here denotes an unramified unitary character of $T(F)$ (an element of the maximal compact subgroup of the complex dual Cartan), which modulo the action of $W$ represents the Satake parameter of such an eigenquotient, and $\Omega(\chi) = \Vert \Phi_0\Vert^2_\chi$. 

The passage from $\pi_! \Phi_0$ to the Plancherel formula \eqref{Plancherel} is achieved through the theory of \emph{asymptotics}, which is of geometric interest because of its relation to \emph{nearby cycles}. For an affine spherical variety $X$ over a field $k$ in characteristic zero, one can define its \emph{horospherical degeneration} (or asymptotic cone) $X_\emptyset$, by passing to the associated graded of the coordinate ring $k[X]$ as a $G$-module. A similar degeneration exists under some assumptions in positive characteristic, see \S\ref{sect:degeneration}. Under our current assumptions, its open $G$-orbit $X_\emptyset^\bullet$ is isomorphic to $N^-\backslash G$, where we use $N^-$ to denote the unipotent radical of a Borel $B^-$ opposite to $B$ --- an expository choice 
without mathematical significance, which we will use to identify the abstract Cartan $T=B/N$ as an automorphism group of $X_\emptyset$ by identifying it with the torus $B^-\cap B$ (the latter acting ``on the left'' on $N^-\backslash G$). 

The theory of asymptotics states that there is a canonical morphism 
\[ e_\emptyset^*: C^\infty(X^\bullet(F)) \to C^\infty(X_\emptyset^\bullet(F))\] 
which describes the behavior of any function ``at infinity'', see \cite[\S 5]{SV}. Of interest to us is that the spaces $X\sslash N$ and $X_\emptyset\sslash N$ are \emph{canonically} identified, and the corresponding pushforwards $\pi_!$ and $\pi_{\emptyset !}$ satisfy
\begin{equation}\label{Radon-asymptotics} \pi_! = \pi_{\emptyset !} \circ e_\emptyset^*
\end{equation}
 (for appropriate functions, in order to ensure convergence), see 
\cite[Proposition 5.4.6]{SV}. 

Inverting the Radon transform, \eqref{Radon-asymptotics} leads to a calculation of the asymptotics of the basic function. This, in turn, leads to a calculation of the basic function itself, as a function on $X^\bullet(F)/G(\mf o)$. To express both, we note that there is a natural parametrization of the coset space $X_\emptyset^\bullet(F)/G(\mf o)$ by the Frobenius-stable elements of the coweight lattice $\ch\Lambda$ (simply assigning $\ch\theta \in \ch\Lambda^\Fr$ to the orbit of $N^-t^{\ch\theta}$), and of the coset space $X^\bullet(F)/G(\mf o)$ by the Frobenius-stable elements of its antidominant submonoid $\ch\Lambda^-$ --- this is explained in Theorem \ref{thm:Gorbits}, in the geometric setting, from which we will borrow the notation $x_0 t^{\ch\theta}$ to represent the orbit corresponding to $\ch\theta\in \ch\Lambda^{-,\Fr}$. Then we have:

\begin{cor}[See \S \ref{sect:asymptotics}] \label{cor:asymptotics}
In the setting of Theorem \ref{theorem-general}, we have 
 \begin{equation}
 \label{equation-asymptotics-IC}
(\eta\delta)^\frac{1}{2}(a)  e_\emptyset^* \Phi_0 (N^- a G(\mf o)) = \tr_{\check T}(\Fr, \Sym^\bullet (\check{\mathfrak n}))^{-1} \cdot \tr_{\check T}(\Fr, \Sym^\bullet (V_X^+)) ,
\end{equation}
and 
\begin{equation}
 \label{equation-basicfunction}
(\eta\delta)^\frac{1}{2}(a)  \Phi_0 (x_0 a G(\mf o)) = \left. \tr_{\check T}(\Fr, \Sym^\bullet (\check{\mathfrak n}))^{-1} \cdot \tr_{\check T}(\Fr, \Sym^\bullet (V_X^+))\right|_{\ch\Lambda^{-,\Fr}},
\end{equation}
where we use the notation $\tr_{\ch T}(\dotsb)$, as in Conjecture \ref{conjecture-pushforward-IC}, to represent a $T(\mf o)$-invariant function on $T(F)$. 
\end{cor}

For example, in the setting of  \eqref{equation-pushforward-IC-expl}, the right hand side reads
\[\frac{\prod_{\check\alpha\in\check\Phi^+} (1- e^{\check\alpha})}{\prod_{\check\lambda \in \mB^+} (1-q^{-\frac{1}{2}} e^{\check\lambda})}.\]

Finally, the Plancherel decomposition of $\Phi_0$ coincides\footnote{See the proof of Proposition \ref{prop:localzeta}.} up to the factor $|W|^{-1}$ with that of $e_\emptyset^*\Phi_0$, which by Mellin transform on $X_\emptyset^\bullet$ (with respect to the action of $T$) gives:
\begin{equation}
 \label{equation-Plancherel-IC}
\Vert \Phi_0\Vert^2 = \int_{\hat T/W} \frac{\prod_{\check\alpha\in\check\Phi} (1- e^{\check\alpha}(\chi))}{\prod_{\check\lambda\in S} (1-q^{-\frac{1}{2}} e^{\check\lambda}(\chi))} d\chi = \int_{\hat T/W} L(\chi, V_X, \frac{1}{2}) \frac{d\chi}{L(\chi, \check{\mathfrak g}/\check{\mathfrak t},0)}.
\end{equation}

Here $\mB = \mB^+ \sqcup (-\mB^+)$, and  $L(\chi, V_X, 0)$ denotes a local unramified $L$-factor, while the density $\frac{d\chi}{L(\chi, \check{\mathfrak g}/\check{\mathfrak t},0)}$ is the unramified Plancherel measure for $G$. 

The importance of \eqref{equation-Plancherel-IC} for arithmetic is that, according to the generalized Ichino--Ikeda conjecture of \cite{SV}, the quotient of the Plancherel density of $\Phi_0$ by the Plancherel measure of $G$ is related to the local Euler factor of the ``$X$-period integral'' of automorphic forms. A provable case of such an Euler factorization is related to Example \ref{example-nfold}:

\begin{eg}\label{example-nfold-global}
 Let $G$, $X$ be as in Example \ref{example-nfold}, over the function field $\Bbbk = \mathbb F(C)$ of a curve $C$, and let $\Phi = \prod_v \Phi_v$ be a smooth, factorizable function of moderate growth on the adelic points of the open $G$-orbit $X^\bullet$, such that the support of $\Phi_v$ has compact closure in $X(F_v)$, and for almost every $v$ the function $\Phi_v$ is equal to the IC function of $X$. Let $\Theta_\Phi(g) = \sum_{\gamma\in X^\bullet(\Bbbk)} \Phi(\gamma g)$ be the corresponding theta series, a function on the adelic points of $G$. 
 
 Let $\phi$ be a cusp form, belonging to a cuspidal automorphic representation $\pi =  \bigotimes_v' \pi_v$, and let $W_\phi(g) = \prod_v W_{\phi, v}(g_v) = \int_{N^-(\Bbbk)\backslash N^-(\mbb A)} \phi(ng) \psi(n) dn$ be the Whittaker function of $\phi$ (where $N^-$ is the lower triangular subgroup), with a chosen Euler factorization with $W_{\phi,v}(1)=1$ for almost all $v$. It can be proven by the usual unfolding argument that the pairing
 \[ \int_{G(\Bbbk)\backslash G(\mbb A)} \phi(g) \Theta_\Phi(g) dg\]
 is convergent if the central character of $\phi$ is ``large'' enough (i.e., its restriction to $\mathbb G_m$ has absolute value $|\bullet|^s$ for some $s \gg 0$), and equal to the Euler product of zeta integrals 
 \[ \int_{H_0\backslash G (\Bbbk_v)} W_{\phi,v}(g_v) \Phi_v(g_v) dg_v.\]
 
 Our Theorem \ref{theorem-affine-closure} implies that almost every Euler factor is equal to the local unramified $L$-factor $L(\pi_v, \otimes, 1 - \frac{n}{2})$, where $\otimes$ is the tensor product representation of $\check G$.  We explain this in \S \ref{sect:nfold-explanation}. 
\end{eg}

\medskip

Such global applications are beyond the main focus of this paper. Of more immediate interest here is the relation of the asymptotics map $e_\emptyset^*$ to nearby cycles: Since $X_\emptyset$ is obtained by degenerating the coordinate ring of $X$, there is an associated $\mathbb G_m\times G$-equivariant Rees family $\mf X\to \mathbb A^1$ (depending, really, on the choice of a strictly dominant cocharacter into $T$), whose general fiber is isomorphic to $X$, and whose special fiber is isomorphic to $X_\emptyset$. This also induces a family of arc spaces, or loop spaces $\msf L_{\mbb A^1}\mf X\to \mathbb A^1$ where $\msf L_{\mbb A^1}\mf X$ denotes the family of \emph{fiberwise} loop spaces, not the loop space of $\mf X$. In the context of an appropriate sheaf theory, to be denoted by $D$, this would give rise to a nearby cycles map:
\[ \Psi: D(\msf LX) \to D(\msf LX_\emptyset),\]
whose Frobenius trace is expected to recover the asymptotics map $e_\emptyset^*$. 

After replacing $\msf L^+X, \msf L^+ X_\emptyset, \msf L^+(X\sslash N)$ by finite type models $\sY, \sY_\emptyset, \sA$, respectively (see \S\ref{subsection-introduction-global} below), we prove: 

\begin{thm}[see Theorem~\ref{thm:PsiRadon}]
\label{thm:nearby-asymptotics}
Let $X$ be an affine spherical variety such that $B$ acts freely on $X^\circ$. 
Then the following triangle of functors commutes up to natural isomorphism:
\[ \begin{tikzcd} 
\rmD^b_c(\sY) \ar[r, "\Psi"] \ar[rd,swap, "\pi_!"] & \rmD^b_c(\sY_\emptyset) \ar[d, "\pi_{\emptyset !}"]\\
& \rmD^b_c(\sA) 
\end{tikzcd}
\]
\end{thm}

The function-theoretic asymptotics map $e_\emptyset^*$ satisfies the same 
commutative triangle \eqref{Radon-asymptotics}, which suggests that nearby cycles $\Psi$
is in a suitable sense the geometrization of the asymptotics map.
This resembles an analogous result \cite[Corollary 6.2]{BFO} in the setting 
of character sheaves.
In the case where $X$ is a reductive group, the nearby cycles of the IC complex 
of (finite type models of) $\msf L^+X$ 
has been computed by S.~Schieder \cite{Sch-SL2,Sch-G,Sch}. 

In Theorem~\ref{thm:nearby}, we compute the intersection complex of the global model $\sM_X$ of the arc space, for the spherical varieties in \eqref{equation-typeT}, at the level of Grothendieck groups. We do this by relating the nearby cycles complex to $\pi_!\IC_\sY$, in a way that corresponds to the known relation \eqref{Radon-asymptotics} between asymptotics and Radon transforms. 
\medskip

Finally, we explain how a formula like \eqref{equation-Plancherel-IC} relates to recent conjectures of Ben-Zvi--Sakellaridis--Venkatesh \cite{BZSV}: According to those conjectures, the formula should follow by applying Frobenius traces on the \emph{endomorphism ring} $\End(\IC_{\msf L^+X})$, where the endomorphism ring is taken in the derived sense, in the dg-category of derived constructible sheaves on $\msf LX/\msf L^+G$. (The proper definition of the intersection complex on the arc space remains conjectural, in the singular setting.) This conjecture identifies (ignoring cohomological grading) 
\begin{equation}\label{equation-BZSV}
  \End( \IC_{\msf L^+X_k} ) \simeq \ol\bbQ_\ell [V_X]^{\check G}
\end{equation}
for some symplectic representation $V_X$ as above. We hope that our results can be related to this conjecture in both directions: Namely, that by relating our results to such endomorphism rings, one can upgrade the $\ch T$-structure of Theorem \ref{theorem-affine-closure} to a $\ch G$-structure, proving Conjecture \ref{conjecture-pushforward-IC}; and, vice versa, one can make progress towards the conjecture of \cite{BZSV} by utilizing our results.

The conjectures of \cite{BZSV}, and the problems addressed by the present paper, can be formulated for more general affine spherical varieties, without the assumption that $\ch G_X = \ch G$. As already mentioned, from the point of view of number theory, the latter case is perhaps the most interesting one, as it corresponds to central values of $L$-functions. In the general case, the vector space $V_X$ appearing in the conjectural relation \eqref{equation-BZSV} needs to be replaced by a Hamiltonian manifold living over the quotient $\ch G_X\backslash \ch G$.

\subsection{Zastava spaces and the main theorems in terms of sheaves}
\label{subsection-introduction-global}

From now on, we work over the algebraic closure $k$ of the finite field $\mathbb F$, or over an algebraically closed field $k$ in characteristic zero. When $X$ is defined over a finite field $\mathbb F$, we will keep track of Weil structures on our sheaves, which will always have the form of half-integral Tate twists, where, as mentioned, $(\frac{1}{2})$ denotes a fixed square root of the cyclotomic twist. The intersection complex of a $d$-dimensional scheme over $k$ will be understood to have stalks $\ol\bbQ_\ell (\frac{d}{2})[d]$ over the smooth locus. 

In order to replace the arc space by a model of finite type, we fix a smooth projective curve $C$ over $k$ (or $\mathbb F$). For an algebraic stack $\Scr X$ and an open substack $\Scr X^\circ \subset \Scr X$,
we will use 
\[ \Maps_\gen(C, \Scr X \supset \Scr X^\circ) \]
to denote the prestack that assigns to a test scheme $S$
the groupoid of maps $C\times S \to \Scr X$
such that the open locus of points sent to $\Scr X^\circ$ 
maps surjectively to $S$. 
Equivalently, these are the maps such that for every geometric point $\bar s \to S$, 
the restricted map $C\times \bar s \to \Scr X$ generically lands in $\Scr X^\circ$.
Since $C$ is smooth, $\Maps_\gen(C, \Scr X \supset \Scr X^\circ)$ is 
an open substack of the prestack $\Maps(C,\Scr X)$. 

\smallskip

Given an affine spherical $G$-variety $X$ with open $G$-orbit $X^\bullet$ and open $B$-orbit $X^\circ$, we consider the following two models for the arc space of $X$:
\begin{itemize}
 \item the Artin stack 
\[ \sM= \sM_{X} = \Maps_\gen(C, X/ G \supset X^\bullet /G), \]
that we will simply refer to as ``the global model'';
 \item the stack
\[ \Scr Y=   \Scr Y_{X} = \Maps_\gen(C, X/B \supset X^\circ/B)  \]
that we will refer to as ``the Zastava model''. In our setting ($X^\circ\cong B$), this turns out to be a scheme. Such a model is often referred to as the ``local model'' for reasons that have to do with factorization structures, but since this can create confusion with the genuinely local arc space, we will avoid such terminology. 
\end{itemize}
For a discussion of why these are indeed formal models of the arc space (in the formal neighborhoods of suitable points), see Theorem~\ref{thm:DGK}, Lemma~\ref{lem:localglobalyoga}.
Note that the choice of a Borel subgroup is immaterial, since $X/B$ can also be written as $(X\times \mathcal B)/G$, where $\mathcal B$ is the flag variety, with $X^\circ/B = (X\times \mathcal B)^\bullet/G$, where $(X\times \mathcal B)^\bullet$ is the open 
orbit under the diagonal $G$-action.  

We will also let $\Scr A$ denote the analog of these models for the toric variety $X\sslash N$, that is,
\[    \Scr A = \Maps_\gen(C, (X\sslash N)/T \supset (X^\circ\sslash N)/T),  \]
and notice that under our assumptions $X^\circ\sslash N\simeq T$. Fixing such an identification, for every $\chi\in \mathfrak c_X^\vee$ 
(the subset of $\Lambda_X$, the character group of $T$, of those elements that are $\ge 0$ on $\mathfrak c_X$), the corresponding map $X\sslash N\to \mathbb G_a$ gives rise to a morphism $\Scr A \to \Maps_\gen(C, \mathbb G_a/\mathbb G_m \supset \mathbb G_m/\mathbb G_m) = \Sym C$, the scheme of effective divisors on the curve. Thus, $\Scr A$ can be thought of as the scheme of $\mathfrak c_X$-valued divisors. This scheme is well understood \cite{BNS}: its normalization is a disjoint union of partially symmetrized powers $C^{\mf P}$ of the curve, indexed by formal $\mathbb N$-linear combinations $\mf P\in \Sym^\infty(\Prim(\mathfrak c_X))$ of the primitive elements of $\mathfrak c_X$ (see \S\ref{def:partition}).

\smallskip

The sheaf-theoretic analog of Theorems \ref{theorem-affine-closure} and \ref{theorem-general} is a statement about the pushforward of the IC sheaf under 
\[ \pi: \Scr Y \to \Scr A.\]
We will only compute $\pi_! \IC_{\Scr Y}$ in the Grothendieck group of sheaves on $\Scr A$ 
(see Corollary~\ref{cor:pibarY}), which is enough to determine the trace of Frobenius on stalks. 
The reason we do not compute the pushforward in the DG category (although this can be done in principle)
is related to the fact that the map $\pi$ is not proper. 
However, one can compactify $\pi$ by considering\footnote{One makes the usual modification to the definition when $[G,G]$ is not simply connected.} the \emph{compactified Zastava space} 
\[ \barY =   \Maps_\gen(C, (\overline{X/N})/T \supset X^\circ/B),  \]
where $\overline{X/N}$ stands for the stack $(X\times \overline{N\backslash G})/G$, where $\overline{N\backslash G}= \Spec k[N\backslash G]$ is the affine closure of $N\backslash G$. 

The difference between $\barY$ and $\Scr Y$ will account for the factor of $\prod_{\check\alpha\in\check\Phi^+} (1-q^{-1} e^{\check\alpha})$ in \eqref{equation-pushforward-IC}. The number theory-minded reader will recognize in this factor, in the case of $G=\SL_2$, the Euler factor of the quotient between Eisenstein series obtained by summing over integral points of $N\backslash \SL_2$, versus integral points of $\overline{N\backslash \SL_2}=\mathbb A^2$. More generally, this is the factor that relates the ``naive'' and ``compactified'' Eisenstein series of \cite{BG}, \cite{BFGM}.

\smallskip

In Proposition \ref{prop:piproper} and Theorem \ref{thm:semismall} we prove:
\begin{thm}
\label{theorem-semismall}
 The map $\bar\pi: \overline{\Scr Y}\to \Scr A$ is proper and stratified semi-small.
\end{thm}
This is one of the key technical results of this paper, because it allows us to get our hands on the pushforward of the intersection complex, without having a description of the complex itself. 
The assumption that $\ch G_X=\ch G$ is critical for the theorem: the analogous statement 
for the usual Finkelberg--Mirkovi\'c Zastava space (\cite{FM, BFGM}) is far from true. 

The condition of being stratified semi-small is a condition on ``smallness'' of fibers, relative to a fixed stratification which, in this case, is the natural stratification of $\Scr A$ by strata of the form
\[ \iota^{\mf P}: \oo C^{\mf P} \into \Scr A, \] 
where $\mathfrak c_X$-valued divisors take a fixed set of values. (Here, $\oo C^{\mf P}$ denotes the open ``disjoint'' locus in a certain product of symmetric powers of the curve, corresponding to divisors of the form $\sum_{\check\lambda \in\mathfrak c_X} \sum_{i=1}^{N_{\check\lambda}} (v_i) \check\mu$ with all $v_i \in \abs C$ distinct; we will denote by $\bar\iota^{\mf P}$ the natural compactification.)

\smallskip

By the decomposition theorem, stratified semi-smallness ensures that $\bar\pi_! \IC_{\barY}$ is a direct sum of irreducible \emph{perverse} sheaves. By a factorization property of $\barY$, this easily implies an expression for $\bar\pi_! \IC_{\barY}$ of the form
\begin{equation} \label{e:format1-intro}
 \bar\pi_!(\IC_{\barY}) \cong 
\bigoplus_{\mf P} 
\Bigl( \bigotimes_{\ch\lambda} 
\Sym^{N_{\ch\lambda}}(V_{X,\ch\lambda}) \Bigr) \ot 
\bar\iota^{\mf P}_!( \IC_{C^{\mf P}} ),
\end{equation}
(see Proposition \ref{prop:format}), where $V_{X,\ch\lambda}$ are the contributions of \emph{diagonal strata} $\iota^{[\check\lambda]}: C \into \Scr A$, corresponding to divisors supported at one point. Moreover, perversity and an estimate of dimensions imply that these contributions all come from the \emph{top degree} cohomology of those fibers of the map $\Scr Y\to \Scr A$ which (over diagonal strata) are of ``maximal possible dimension'' in terms of the semi-smallness  inequality. (We will discuss this dimension calculus in detail below.) These fibers, over points on the diagonal stratum $\iota^{[\check\lambda]}(C)$ are called the \emph{central fibers}, and denoted by $\msf Y^{\check\lambda}$, with the point not appearing explicitly in the notation.

\medskip

Therefore, to complete the calculation and prove Conjecture \ref{conjecture-pushforward-IC}, we would need to count, for every $\check\lambda \in \mf c_X$, the irreducible components of the central fiber $\msf Y^{\check\lambda}$ which achieve the maximal dimension, and show that their cardinality is equal to the dimension of the $\check\lambda$-weight space in $V_X^+$, which is ``half'' of a symplectic $\ch G$-representation $\rho_X$. 
Note that $\mf c_X$ is strictly convex, and the weights of $\rho_X$ will have the property that 
they belong to either $\mf c_X\sm 0$ or $-\mf c_X \sm 0$.  Thus, the monoid $\mf c_X$ determines
which ``half'' of $\rho_X$ to consider.

\smallskip

There is a reduction from the general case of Theorem \ref{theorem-general} to the special case of Theorem \ref{theorem-affine-closure}, that is, when $X=\overline{X^\bullet}^\aff$ is the affine closure of its open orbit (Theorem \ref{thm:heckeIC}), and the monoid $\mathfrak c_X$ is free. This reduction uses the action of the Hecke algebra, and resembles the calculation of the IC sheaf of reductive $L$-monoids in \cite{BNS}, but is much harder. This reduction will be discussed in \S \ref{subsection-reduction-affineclosure} below.

Hence for now let us focus on the case $X=\overline{X^\bullet}^\aff$, assuming that $\mathfrak c_X$ is free. In that case, the maximal possible dimension for the central fiber $\msf Y^{\check\lambda}$ will be called the \emph{critical dimension}, and is equal to 
\[ \frac{1}{2} (\len(\check\lambda)-1),\]
where $\len(\check\lambda) := \sum_i m_i$, for $\check\lambda\in \mathfrak c_X$ written uniquely
as $\sum_i m_i \ch\nu_i$ in terms of our basis of $\mathfrak c_X$. 
Conjecturally, the set of irreducible components of $\msf Y^{\ch\lambda}$ of critical dimension,
ranging over all $\ch\lambda \in \mf c_X$, should correspond to the subset
of the crystal basis 
of $\rho_X$ with weights in $\mf c_X$.

We do not go quite as far in general, but we show that these irreducible components give rise to the aforementioned \emph{crystal}, in the sense of Kashiwara \cite{Kas93}, over the Langlands dual Lie algebra $\ch{\mathfrak g}$.
Namely, let $\mB_{X^\bullet}^+$ denote the set of irreducible components of 
the central fibers of critical dimension, so $\mB_{X^\bullet}^+$ corresponds to a basis 
of $V_{X^\bullet}^+ := \bigoplus_{\mf c_X} V_{X^\bullet,\ch\lambda}$. 
Formally define $\mB_{X^\bullet}^-$ to be the ``negatives'' of $\mB_{X^\bullet}^+$, so 
$\mB_{X^\bullet}^-$ corresponds to a basis of the dual space $(V_{X^\bullet}^+)^*$. 
Let $\mf B_{X^\bullet} = \mf B_{X^\bullet}^+ \sqcup \mf B_{X^\bullet}^-$. We prove in 
Theorems~\ref{thm:crystal} and \ref{thm:crystal-properties}:

\begin{thm} 
\label{theorem-crystal}
Let $X=\overline{X^\bullet}^\aff$ satisfy the assumptions of Theorem \ref{theorem-affine-closure}. The set $\mB_{X^\bullet}$ has the structure of a semi-normal, self-dual crystal over 
$\ch{\mf g}$ such that the weights have the properties described in Theorem~\ref{theorem-affine-closure}. 
\end{thm}

We conjecture:

\begin{conj} 
\label{conjecture-crystal}
The crystal $\mB_{X^\bullet}$ is isomorphic to the unique crystal basis 
of a finite-dimensional $\ch{G}$-module $V_X$. 
\end{conj} 
This would imply Conjecture \ref{conjecture-pushforward-IC}. 

\medskip

Theorem~\ref{theorem-crystal} endows $V_X^+ \oplus (V_X^+)^*$, as we have defined it, 
with an action of an $\SL_2$-triple corresponding to every simple root of $\ch G$. 
These actions imply that the dimensions of the weight spaces 
are invariant under the Weyl group of $G$, which provides a kind of 
``functional equation'' for $\pi_!\Phi_0$. This functional equation can be seen as a geometric analog of the functional equation of the Casselman--Shalika method \cite{C, CS, SaSph}.
The content of 
Conjecture~\ref{conjecture-crystal} is to show that these $\SL_2$-triples satisfy 
the Weyl relations: $[e_\alpha, f_\beta]=0$ for simple roots $\alpha\ne \beta$.
(The Weyl relations imply the Serre relations by \cite[Corollary 4.3.2]{Chriss-Ginzburg}.)

\smallskip

The construction of the action of the $\SL_2$ corresponding to a simple root $\alpha$ of $G$ 
goes as follows: 
we factor $X \to X \sslash N$ through $X \to X\sslash N_{P_\alpha} \to X\sslash N$, where $P_\alpha$ is the 
sub-minimal parabolic corresponding to $\alpha$ and $N_{P_\alpha}$ is its unipotent radical. 
Then the GIT quotient $X_\alpha:= X\sslash N_{P_\alpha}$ is a spherical variety for the Levi factor $M_\alpha$. 
But now $X_\alpha$ is (usually) larger than the affine closure of its homogeneous part $X_\alpha^\bullet$. 
The irreducible components of $\msf Y_X$ of critical dimension (i.e., elements of $\mB_{X^\bullet}^+$) 
will either go to irreducible components 
\begin{enumerate}
\item of $\msf Y_{X_\alpha^\bullet}$ of critical dimension or 
\item of $\msf Y_{X_\alpha} \sm \msf Y_{X_\alpha^\bullet}$ (not necessarily of critical dimension). 
\end{enumerate}
While the fibers of $\msf Y_X \to \msf Y_{X_\alpha}$ are not necessarily irreducible, we
show that the irreducible components of different relevant fibers can be canonically identified. 
Then we define the $\SL_2$-action by analyzing the two cases above in the base $\msf Y_{X_\alpha}$. 

Under our assumptions, $X_\alpha^\bullet$ is a torus torsor over $\mbb G_m \bs \PGL_2$ and case (i) is an easy
calculation. Meanwhile our study of non-canonical affine embeddings using Hecke actions 
shows that in case (ii) we always get a Mirkovi\'c--Vilonen cycle (i.e., irreducible
component of the intersection of a semi-infinite orbit with a $\msf L^+G$-orbit in the affine Grassmannian). 
The crystal structure on these cycles was constructed by \cite{BGcrys}.

To check the Weyl/Serre relations, one can similarly reduce to a spherical variety $X^\bullet_{\alpha,\beta}$ for a Levi of semisimple rank two. There are only a handful such varieties (up to center) satisfying our assumptions --- a small subset of the spherical (wonderful) varieties of rank two classified by Wasserman \cite{Wasserman}.

However, checking the Weyl/Serre relations, even in a few cases, ``by hand'' does not seem to be easy, and we do not have a conceptual proof of them; therefore, we refrain from attempting such a verification. 

The remainder of the introduction will be devoted to describing the two most important elements in the proofs of the theorems above.

\subsection{Reduction to canonical affine closure}
\label{subsection-reduction-affineclosure}

We give an overview of how to reduce the case of an arbitrary affine $X$ with 
$X^\bullet=H\bs G$ to the 
canonical affine closure $X^\can = \ol{H\bs G}^\aff$.  

There is a canonical map $X^\can \to X$, which induces an inclusion 
$X^\can(\mf o) \cap X^\bullet(F) \subset X(\mf o)\cap X^\bullet(F)$ of 
$G(\mf o)$-stable spaces.
Of course, all points of $X^\bullet(F)$ are $G(F)$-translates of points in $X^\can(\mf o)\cap X^\bullet(F)$. 
It is a fact that if $\ch\theta \in \ch\Lambda^-$ is \emph{antidominant} 
and belongs to the monoid $\mf c_X$, then 
the action of the double coset $G(\mf o) t^{\ch\theta} G(\mf o)$ preserves $X(\mf o)\cap X^\bullet(F)$. 
The idea for what follows is that we can obtain 
$X(\mf o)\cap X^\bullet(F)$ by acting on $X^\can(\mf o)\cap X^\bullet(F)$ by $G(\mf o) t^{\ch\theta} G(\mf o)$
for $\ch\theta \in \mf c_X^- := \ch\Lambda^- \cap \mf c_X$.  

The Zastava model $\sY$ lives over $\Bun_B$ and does not carry a Hecke action. 
Thus, to model the $G(F)$-action on $X(\mf o)\cap X^\bullet(F)$ we must use the global model
$\sM=\sM_X$, which lives over $\Bun_G$. 
The canonical map $\sM_{X^\can} \to \sM_X$ is a closed embedding. 
For $\ch\theta \in \mf c_X^-\sm 0$, 
 let $\sH^{\ch\theta}_{G,C}$ denote the Hecke stack 
over $\Bun_G \times C$ with fibers isomorphic to $\ol\Gr_G^{\ch\theta}$, the closure
of the $\msf L^+G$-orbit in the affine Grassmannian corresponding to $\ch\theta$.  
In reality, we need a symmetrized (multi-point) version of the Hecke stack, but we only describe
the case where there is one point on the curve in this introduction for simplicity. 
There is a well-defined map 
\begin{equation}
 \label{e:Heckeaction}
\sM_{X^\can} \xt_{\Bun_G} \sH_{G,C}^{\ch\theta} \to \sM_X 
\end{equation}
modeling the action of Hecke operators, and we show (Theorem~\ref{thm:comp}) that this map is birational onto its image. 
If we allow multiple points above, then the images of the corresponding Hecke actions, with $\ch\theta$ varying, 
stratify $\sM_X$. 

Under the assumptions of the previous subsection, we show that 
$\sM_{X^\can}$ is irreducible (Corollary~\ref{cor:can-comp}), so the study 
of $\IC_{\sM_X}$ reduces to the study of the Hecke action on $\IC_{\sM_{X^\can}}$ and 
the determination of which of the strata above form irreducible components of $\sM_X$. 
For the latter, we need to understand the closure relations among the different strata (Proposition~\ref{prop:closure-rel}). When $\mf c_{X^\bullet}=\mbb N^r$, the stratum corresponding to $\ch\theta$ 
is contained in the closure of the stratum corresponding to $\ch\theta'$ if and only if $\ch\theta-\ch\theta' \in \mf c_{X^\bullet}$ (more generally, the closure relations are determined by the colors of $X$).

\subsection{Semi-infinite orbits and dimension estimates}
\label{subsection-dimensions}

There is another way to understand the central fibers $\msf Y^{\check\lambda}$: as subsets of the affine Grassmannian of $G$. 
Let us fix the point $v \in C$ that we take central fibers with respect to. 
Then a $k$-point of $\msf Y^{\ch\lambda}$ is a map $C \to X/B$ such that $C\sm v$ is sent to $X^\circ/B=\pt$. 
Restricting to the completed local ring $\mf o_v$ at $v$ gives a map 
$\msf Y^{\ch\lambda} \to \msf L X^\circ / \msf L^+B$. 
If we fix a base point $x_0\in X^\circ(k)$ to identify $X^\circ \cong B$, we get a map 
$\msf Y^{\ch\lambda} \to \Gr_B$ and this turns out to be a closed embedding. 
The reduced image of the components of $\Gr_B$ in $\Gr_G$ are the
semi-infinite orbits $\msf S^{\ch\lambda}(k) = N(F) t^{\ch\lambda} G(\mf o)/G(\mf o)$. 
After passing to reduced schemes we get identifications 
$\msf Y^{\ch\lambda}_\red = (\msf S^{\ch\lambda} \xt_{\msf L X/ \msf L^+G} \msf L^+ X / \msf L^+G)_\red$ 
and $\olsf Y^{\ch\lambda}_\red = (\olsf S^{\ch\lambda} \xt_{\msf L X /\msf L^+G} \msf L^+ X / \msf L^+ G)_\red$ (see Lemma \ref{lem:Y=ScapGr}). 

\smallskip

Semi-infinite orbits have an important meaning for the geometric Satake equivalence \cite{MV}: the fundamental classes of the irreducible components of the intersection $\msf S^{\check\lambda} \cap \overline\Gr_G^{\check\theta}$, the \emph{Mirkovi\'c--Vilonen cycles}, are in bijection with the ``canonical basis'' for the $\check\lambda$-eigenspace of the irreducible $\ch G$-module of lowest weight $\check\theta$. 

\smallskip

Our analysis of the central fibers $\msf Y^{\ch\lambda}$ is founded upon the following argument from
\cite[\S 3]{MV}. 
The boundary $\olsf S^{\check\lambda}\sm \msf S^{\check\lambda} = \cup_{\check\nu<\check\lambda} \msf S^{\check\nu}$ is a hyperplane section for some projective embedding of $\Gr_G$.
Hence any closed subscheme of $\Gr_G$ which intersects $\olsf S^{\check\lambda}$, also intersects its boundary in codimension one (unless already contained in the boundary). By inductively ``cutting'' by these hyperplanes,
we prove:

\begin{thm}
\label{theorem-slicing}
Let $X = X^\can$ be as in Theorem \ref{theorem-affine-closure}. 
 Let $\msf b$ be an irreducible component of the central fiber $\msf Y^{\check\lambda}$. Then
 \begin{itemize}
  \item $\dim \msf b \le \frac{1}{2} (\len(\check\lambda)-1)$, 
  \item for a basis element $\ch\nu_i$ of $\mf c_X$ (corresponding to a color), $\msf Y^{\ch\nu_i} = \pt$,
  \item the inequality is an equality only if there is a sequence $\alpha_1,\dots,\alpha_d$ of simple roots (with repetitions) such that $\olsf b \cap \msf S^{\ch\lambda-\check\alpha_1-\cdots-\check\alpha_j}$ is of dimension $\dim \msf b - j$ (hence, also of critical dimension), and $\check\lambda - \sum_{i=1}^d \check\alpha_i = \check\nu$ for a color $\check\nu$.
 \end{itemize}
\end{thm}

The operation of hyperplane ``cutting'' can almost be thought of as the lowering operator 
for some $\SL_2$-triple; unfortunately it is not quite precise enough, see Proposition~\ref{prop:f-hyper}. 

\medskip

If $X \ne X^\can$, then we also show that if $\msf b$ is an irreducible component of 
$\msf Y^{\ch\lambda}_X$ of critical dimension that is not contained in $\msf Y^{\ch\lambda}_{X^\can}$, then
$\ch\lambda$ must be a weight of $V^{\ch\theta_i}$ for one of the $\ch\theta_i$ appearing in 
Theorem~\ref{theorem-general}, and $\msf b$ is birational to a Mirkovi\'c--Vilonen cycle in 
$\msf S^{\ch\lambda} \cap \Gr_G^{\ch\theta_i}$. The latter correspondence comes from the Hecke action \eqref{e:Heckeaction}.

\medskip

Let us comment on how the above relates to Theorems \ref{theorem-semismall} and \ref{theorem-general}. 
Under our assumption that $\mf c_{X^\bullet} = \mbb N^r$, the space $\sY^{\ch\lambda}_{X^\can}$ 
is irreducible. 
Then the dimension estimate in Theorem~\ref{theorem-slicing}, together with a factorization property 
of $\sY_X$, implies the ``stratified semi-smallness'' condition. 

The irreducible components of $\msf Y_{X^\can}$ of critical dimension 
correspond to a basis of $V^+_{X^\bullet}$. The other irreducible components of $\msf Y_X \sm \msf Y_{X^\can}$
of critical dimension correspond to Mirkovi\'c--Vilonen cycles in $\Gr_G^{\ch\theta_i}$, and
these provide a basis for $V^{\ch\theta_i}$ in Theorem~\ref{theorem-general} by geometric Satake. 

\subsection{Organization of the paper}
In \S\ref{section-spherical} we briskly review the salient combinatorics
of spherical varieties and the classification of $G(\mf o)$-orbits of the loop
space of $X^\bullet$. In \S\ref{sect:models} we introduce the
global and Zastava models for the arc space of $X$ and their stratifications, explain why they are indeed models,
and prove some foundational properties. In \S\ref{sect:compactification}, we
introduce the compactification of the Zastava model and define the
central fibers of (compactified and non-compactified) Zastava models. Then we 
perform the comparison between $\pi_!\IC_{\sY}$ and $\bar\pi_! \IC_{\barY}$ 
that accounts for the ``numerator'' in the Euler factor \eqref{equation-pushforward-IC-expl}. 

Sections \ref{sect:Heckeact} and \ref{sect:centralfiber} are the 
technical heart of this paper. In \S\ref{sect:Heckeact} we establish 
the closure relations for the global model $\sM_X$ and determine its irreducible
components. This involves a study of the $G$-Hecke action on the global model,
which also reduces the problem to the canonical affine closure, as explained 
earlier. In \S\ref{sect:centralfiber}, we analyze the geometry of the
central fiber and prove the crucial dimension estimates using the
Mirkovi\'c--Vilonen boundary hyperplanes of semi-infinite orbits. 
This allows us to prove Theorem~\ref{theorem-semismall}. 

In \S\ref{sect:crystal}, we prove the aforementioned results on crystals. 
In \S\ref{sect:nearby}, we combine the results of the preceding sections to 
compute the nearby cycles of the IC complex on the global model using a 
well-known contraction principle. Here we establish that 
the nearby cycles functor does indeed correspond to the asymptotics map under 
the sheaf--function dictionary. 

In Appendix \ref{appendix:A}, we collect various technical results 
concerning the stratification of the global model, some of which use the
notion of generic-Hecke modification from \cite{GN}.

\subsection{Index of notation}\  

In general, we will use calligraphic 
letters $\Cal D, \Cal V$ to denote standard combinatorial objects associated
to spherical varieties in the literature, script letters $\Scr M, \Scr Y, \Scr F$ for
algebraic stacks and sheaves, sans serif letters $\msf Y, \msf S$ for (ind-)schemes 
that are subspaces of certain loop spaces with respect to a fixed point $v\in \abs C$. The following table contains most of the notation used in this paper, except for notation defined and used locally.

\begin{longtable}{>{\raggedleft}p{.18\textwidth} p{.78\textwidth} }
$k$ & an algebraically closed field. The characteristic of $k$ can be zero or positive, but in the latter case we will impose some restrictions on our spherical varieties (see \S \ref{assumptions-finitefield}), to ensure that their geometry is similar to that in characteristic zero. \\
{$\mathbb F$, $\Fr$} & {At some points in this paper, $k$ is the algebraic closure of a finite field $\mathbb F$, and then $\Fr$ denotes the geometric Frobenius morphism}. \\
$\pt$ & $\Spec k$. \\ 
{$C$} & {a connected smooth projective curve over $k$.} \\ 
{$\Sym C,\, C^{(n)}$} & {the scheme of effective divisors on (=symmetric powers of) $C$, and the component of divisors of degree $n$.} \\
{$\oo C^n, \oo C^{(n)}$} & {the open subsets of distinct $n$-tuples of points, resp.\ multiplicity free divisors of degree $n$, on the curve.} \\
{$\oo \prod$, $\oo\times$} & {for schemes living over any partially symmetrized powers of the curve, the restriction of their Cartesian product over the multiplicity-free locus.} \\
{$\Bbbk = k(C)$} & {the field of rational functions on $C$.}  \\
{$\abs C = C(k)$} & the set of closed points of $C$. \\
{$\mf o_v$} & {for $v\in \abs C$, it denotes the completion of the local ring at $v$.} \\
{$F_v$} & {the fraction field of $\mf o_v$.
By choosing a local coordinate $t$ we have a non-canonical isomorphism 
$\mf o_v \cong k\tbrac t =: \mf o$ and $F_v \cong k\lbrac t =: F$. We sometimes implicitly
make this identification when the choice of local coordinate is irrelevant.} \\
{$\mbb N$} & {the monoid of non-negative integers.} \\
& \\ 
{$G$} & {a connected reductive group over $k$.} \\
{$T$} & {the (abstract) Cartan of $G$, i.e., the reductive quotient of any Borel subgroup. We sometimes fix a splitting $T\hookrightarrow B\hookrightarrow G$ of the abstract Cartan into a Borel subgroup.} \\
{$\Cal B$} & {the flag variety of Borel subgroups of $G$.} \\
{$W$} & {the (abstract) Weyl group of $G$.} \\
{$s_\alpha \in W$} & {for a simple root $\alpha$, the corresponding reflection.} \\
{$\check \Lambda_G$ ($\Lambda_G$)} & {the coweight (resp.~weight) lattice of $T$. The index $G$ will often be omitted.} \\
{$\check\Lambda^+_G$ ($\Lambda_G^+,\, \ch\Lambda^-_G$)} & 
The monoid of dominant coweights (resp., dominant weights, antidominant coweights). \\
{$\ch\Lambda^\pos_G$ ($\Lambda^\pos_G \subset \Lambda_G$)} & 
The monoid generated by the non-negative \emph{integral} span of the positive coroots
(resp.~roots) in $\ch\Lambda_G$.  \\
{$\check\Delta_G$ ($\Delta_G$)} & {the set of simple coroots (resp.~roots)
of $G$.} \\
{$2\ch\rho_G \in \ch\Lambda_G$ ($2\rho_G \in \Lambda_G$)} & {the sum of 
the positive coroots (roots) of $G$} \\
{$\ch\lambda \ge \ch\mu$} & {For $\ch\lambda,\ch\mu \in \ch\Lambda_G$, this means that $\ch\lambda - \ch\mu \in \ch\Lambda^\pos_G$.} \\
{$\ch G$} & {the Langlands dual group of $G$ over $\ol\bbQ_\ell$, i.e., $\ch G$ is
the connected reductive group where the weights, roots of $\ch G$ equal the 
coweights, coroots of $G$, etc.} \\
{$V^{\ch\lambda}$} & {for $\ch\lambda \in \ch\Lambda_G$ either dominant or antidominant, this denotes
the irreducible $\ch G$-module over $\ol\bbQ_\ell$
with highest (resp.~lowest) weight $\ch\lambda$.
Similarly, $V^\lambda$ denotes the irreducible $G$-module over $k$ for $\lambda\in \Lambda_G^+$.}
\\
& \\
{variety} & {will mean a reduced, finite type $k$-scheme (not necessarily irreducible).}  \\
{$\Scr Z/H$} & {for a stack $\Scr Z$ with an action of an algebraic group $H$, this will denote the quotient stack.} \\
{$Z\sslash H$} & {for an affine variety $Z$ over $k$, and a group $H$ acting on it, the invariant-theoretic quotient $\Spec k[Z]^H$.} \\ 
{$\Scr X \xt^G \Scr Y$} & 
if $\Scr X, \Scr Y$ are stacks with (right) $G$-actions, we will use  this 
to denote the stack quotient $(\Scr X \xt \Scr Y)/G$ by the diagonal action. \\
{$\msf L^+ X,\, \msf L X$} & {the formal arc and loop spaces of a scheme $X$ (see \S \ref{sect:formalloop}).} 
\\
& \\
{$\rmD^b_c(\sZ)$} & {for an algebraic stack $\sZ$, this is the derived category of 
bounded constructible $\ol\bbQ_\ell$-complexes on $\sZ$.} \\
{`sheaf'} & {means a complex of sheaves. 
All functors between sheaves are derived functors.} \\
{$\rmP(\sZ) \subset \rmD^b_c(\sZ)$} & {when $\sZ$ is locally of finite type over $k$, this is the abelian category of perverse sheaves.}\\
{$^p\rmD^{\le 0}, {^p}\rmD^{\ge 0}$} & 
the subcategories with respect
to the perverse $t$-structure. \\
{$\IC_{\Scr Z}$} & {the direct sum of the intersection cohomology complexes
of all irreducible components of $\Scr Z$. When working over a finite field, we will normalize this sheaf to be pure of weight zero.} 
\\
&\\
{$X$} & {an affine $G$-spherical variety over $k$. (See \S \ref{sect:spherical} for notions pertaining to spherical varieties.)} \\
{$X^\circ$} & {the open $B$-orbit, for a fixed choice of Borel subgroup $B$.} \\
{$x_0 \in X^\circ(k)$} & {a fixed base point.} \\
{$H$} & {the stabilizer of $x_0$.} \\
{$X^\bullet$} & {the open $G$-orbit $H\bs G $.} \\
{$X^\can$} & {the ``canonical'' affine embedding $\Spec k[X^\bullet]$ of $X^\bullet$.} \\
{$T_X$} & {the (abstract) Cartan of $X$, that is, the quotient by which the abstact Cartan of $G$ acts on $X^\circ\sslash N$, where $N$ is the unipotent radical of $B$.} \\
{$\Lambda_X, \ch \Lambda_X$} & {the character and cocharacter groups of $T_X$. Our assumptions on $X$ will identify $\ch\Lambda_X$ with $\ch\Lambda_G$, so it will often just be denoted by $\ch\Lambda$.} \\
{$\check G_X$} & {the dual group of $X$; it has a canonical maximal torus isomorphic to the dual of $T_X$.} \\
{$\Cal V$} & {the cone of invariant valuations of $X$; equivalently, the antidominant chamber of the dual group of $X$.} \\
{$\mathfrak c_X^\vee \subset \Lambda_X$ ($\mathfrak c_X\subset \ch\Lambda_X$)} & 
{the monoid of weights of $T_X$ on $k[X\sslash N]$ (resp., its dual monoid).} \\
{$\Cal C_0=\Cal C_0(X)\subset \mathfrak t_X$} & 
{the cone spanned by $\mathfrak c_X$, inside of the vector space spanned by $\ch\Lambda_X$.} \\
{$\mathfrak c_X^-$} & {the intersection of $\mathfrak c_X$ with the cone $\Cal V$ of invariant valuations.} \\
{$\mathcal D(X)$} & {the set of irreducible $B$-stable divisors in $X$.} \\
{$\mathcal D$} & {the set of colors, i.e., irreducible $B$-stable divisors which are not $G$-stable; equivalently, this can be identified with $\mathcal D(X^\bullet)$.} \\
{$\mathcal D(\alpha)$} & {for a simple root $\alpha$, the set of colors $D$ of $X^\bullet$ such that $D P_\alpha \supset X^\circ$, 
where $P_\alpha$ is the parabolic generated by $B$ and the root space $\mathfrak g_{-\alpha}$.}\\
{$\varrho_X(D)=\ch\nu_D$}&{for $D\in \mathcal D(X)$, the associated $B$-invariant valuation, restricted to the group of nonzero $B$-eigenfunctions: $\check \nu_D: k(X)^{(B)} \to \mathbb Z$, and understood as  a functional on the character group $\Lambda_X = k(X)^{(B)}/k^\times$.}\\
{$\mathfrak c_X^{\Cal D} \subset\ch\Lambda_X$} & 
the monoid generated by the $\ch\nu_D,\, D\in \Cal D$.  \\
{$\ch\lambda\preceq \ch\mu$} & {for $\ch\lambda,\ch\mu\in \ch\Lambda_X$, this means that 
$\ch\mu-\ch\lambda \in \mathfrak c_X^{\Cal D}$.} \\
{$\Cal D^G_{\mathrm{sat}}(X)$} & {the set of those primitive (=indecomposable) elements in $\mathfrak c_X^-$ that
are minimal with respect to the $\preceq$ partial order.}
\\
&\\
{$\Bun_G, \Bun_B$} & {the moduli stack of $G$-bundles, resp.\ $B$-bundles, on $C$.} \\
{$\sM_{X}$} & the ``global model'' of generic maps from a curve to $X/G$; it lives over $\Bun_G$. The restriction of such a map, defined over $k$, to the formal neighborhood of a point $v\in |C|$ gives rise to a well-defined ``valuation'', that is, an element of $(X^\bullet(F_v)\cap X(\mf o_v))/G(\mf o_v)$. 
\\
& See Section \ref{sect:models} for the various models of the arc space. For any model, when $X$ is understood, the index will be omitted. \\
{$\sY_X$} & {the ``Zastava model'' of generic maps from a curve to $X/B$; it lives over $\Bun_B$. The restriction of such a map, defined over $k$, to the formal neighborhood of a point $v\in |C|$ gives rise to a well-defined element of $(X^\circ(F_v)\cap X(\mf o_v))/B(\mf o_v)$} \\
{$\ol\Bun_B$} & {Drinfeld's compactification of $\Bun_B$, see \S\ref{sect:basicproperties-comp}.}
\\
{$\barY_X$}& {the compactified Zastava model (Section \ref{sect:compactification}).}
\\
{$\Scr A$} & {the global/Zastava model for the $T_X$-space $X\sslash N$.}
\\
{$\pi,\bar\pi$} & the natural maps $\pi:\sY_X\to\Scr A$, $\bar\pi:\barY_X\to\Scr A$ (extending $\pi$). 
\\
$X^\bullet(F)_{G:\ch\theta}$, $\msf L^{\ch\theta}X$ & 
{the $G(\mf o)$-orbit on $X^\bullet(F)$ parametrized by $\ch\theta\in \Cal V \cap \ch\Lambda_X$ (Theorem \ref{thm:Gorbits}), and the corresponding stratum of the loop space. When $\ch\theta \in \mathfrak c_X^-$, these belong to $X(\mf o)$, resp.\ the arc space $\msf L^+X$.}
\\
{$\Sym^\infty(\msf S)$} & 
{the set of multisets in elements of a set $\msf S$; equivalently, the free monoid $\bigoplus_{\msf S}\mathbb N$ in the elements of $\msf S$.}
\\
{$C^{\mf P}, \oo C^{\mf P}$} & 
{for a multiset $\mf P = \sum_{s\in \msf S} N_s [s]$, the partially symmetrized power $\prod_{\msf S} C^{(N_s)}$ of the curve, and its disjoint locus $\oo\prod_{\msf S} \oo C^{(N_s)}$.}
\\
{$\sM^{\ch\Theta}$} & 
{for $\ch\Theta \in \Sym^\infty(\mathfrak c_X^-\sm \{0\})$, the stratum of $\sM$ containing those maps whose multiset of nontrivial valuations (as elements of $\mathfrak c_X^-\sm \{0\}$) is equal to $\ch\Theta$. When $\ch\Theta=\{\ch\theta\}$ is a singleton, we will write $\sM^{\ch\theta}$.}
\\
{$\Scr A^{\ch\lambda}$} & {the connected component of $\Scr A$ of maps with total valuation $\ch\lambda \in \mathfrak c_X \subset \ch\Lambda_X= T_X(F)/T_X(\mf o)$.} 
\\
{$\sY^{\ch\lambda}, \barY^{\ch\lambda}$} & {the preimage of $\Scr A^{\ch\lambda}$ in $\sY$, resp.\ in $\barY$. They live over strata $\Bun_B^{-\ch\lambda}$, $\ol\Bun_B^{-\ch\lambda}$ of $\Bun_B$, resp.\ $\ol\Bun_B$.}
\\
{$\sY^{\ch\lambda,\ch\Theta},\, \barY^{\ch\lambda,\ch\Theta}$} & {the fiber products of $\sY^{\ch\lambda}, \barY^{\ch\lambda}$ with $\sM^{\ch\Theta}$ over $\sM$.}
\\
{$\sY^D$} & {for $D \in \mathbb N^{\mathcal D}$, a certain connected/irreducible component of the ``open Zastava'' space $\sY_{X^\bullet} = \sY^{?,0}$, defined in \S \ref{sect:Y0cc}. The question mark $?$ corresponds to the valuation $\varrho_X(D)$.}
\\
{$\barM^{\ch\Theta}$} & {denotes the closure of the stratum $\sM^{\ch\Theta}$. Note that $\barY^{\ch\lambda}$, $\barY^{\ch\lambda,\ch\Theta}$, in contrast, are \emph{not} closures of strata, but strata of the compactified Zastava space. In the case of the global model, there is no room for confusion, so we allow ourselves this notation, for typographical reasons.}
\\
$C_{\ch\nu}$, ${_{\ch\nu}}\ol\Bun_B^{\ch\lambda}$, ${}_{\ch\nu} \barY^{\ch\lambda}$ & 
for $\ch\nu\in \ch\Lambda_G^{\pos}$, the partially symmetrized power $C^{\mf P}$ when $\ch\nu$ is thought of as a multiset $\mf P$ in the simple coroots, a stratum of $\ol\Bun_B^{\ch\lambda}$, and a stratum of $\barY^{\ch\lambda}$, isomorphic, respectively, to $C_{\ch\nu} \times \Bun^{\ch\lambda+\ch\nu}_B \into \ol\Bun^{\ch\lambda}_B$ and  $C_{\ch\nu} \times \Scr Y^{\ch\lambda-\ch\nu}$ (see \S \ref{sect:strat-compactified}).
\\
& \\
{$\Gr_G, \Gr_B$} & the affine Grassmannian of $G$, resp.\ $B$. 
\\
$\Gr_{G,\Sym C}$, $\Gr_{B,\Sym C}$ & {the Beilinson--Drinfeld affine Grassmannians, living over $\Sym C$ (see \S \ref{sect:YGr}).}
\\
{$\Gr_G^{\ch\theta}$, $\ol\Gr_G^{\ch\theta}$} & {for $\ch\theta\in \ch\Lambda_G^-$, the $\msf L^+G$-orbit in the affine Grassmannian containing the class of $t^{\ch\theta}$, and its closure.} \\
{$\ol\Gr^{\ch\Theta}_{G,C^{\ch\Theta}}$} & {for $\ch\Theta \in \Sym^\infty(\ch\Lambda_G^- \sm 0)$, the multi-point version of $\ol\Gr^{\ch\theta}$, see \S \ref{sect:strata-Grassmannian}.} \\
{$\msf S^{\ch\lambda}$, $\olsf S^{\ch\lambda}$} & {the ``semi-infinite'' $\msf LN$-orbit of $t^{\ch\lambda}$ in $\Gr_G$, and its closure.}
\\
{$\msf Y^{\ch\lambda},\, \olsf Y^{\ch\lambda}$} & {the central fibers of $\sY^{\ch\lambda}, \barY^{\ch\lambda}$, living over a point of the diagonal stratum $C\hookrightarrow \Scr A^{\ch\lambda}$. They can be identified as subspaces of $\msf S^{\ch\lambda}$, $\olsf S^{\ch\lambda}$ (see \S \ref{def:centralfiber}). From \S\ref{def:centralfiber} onwards, we will implicitly consider
only the underlying reduced structure on all central fibers.
}
\\
& \\
{$\mB_{X,\ch\lambda}$} & {for $\ch\lambda\in \mathfrak c_X$, the set of all irreducible components of critical dimension of the central fiber $\msf Y^{\ch\lambda}$, see Proposition \ref{prop:Vbasis}.} \\
{$\mB_X^+, \mB_X$} & {the union of all $\mB_{X,\ch\lambda}$, $\ch\lambda\in \mf c_X$, and the ``crystal of $X$'' (\S \ref{sect:defcrys}).}
\end{longtable}

\subsection{Acknowledgments}

We thank R.~Bezrukavnikov, V.~Drinfeld, T.~Feng, M.~Finkelberg, V.~Ginzburg, F.~Knop and D.~Gaitsgory for helpful comments and conversations. We thank the anonymous referee for helpful suggestions. We thank the Institute for Advanced Study for its hospitality during the academic year 2017--2018, during which part of our work was conducted. Y.S.~was supported by NSF grants DMS-1801429 and DMS-1939672, and by a stipend to the IAS from the Charles Simonyi Endowment.
J.W.~was supported by NSF grant DMS-1803173. 

\section{Spherical varieties and their arc spaces}
\label{section-spherical}

\subsection{Spherical varieties}
\label{sect:spherical}

A spherical variety over $k$ is a normal variety with an open $B$-orbit. 
Let $X$ be an affine $G$-spherical variety over $k$. Let $X^\circ$ denote the open $B$-orbit. We choose and fix a point $x_0 \in X^\circ(k)$ and let $H$ denote its stabilizer.
Let $X^\bullet= H\bs G $ denote the open $G$-orbit. 

The quotient $X^\circ\sslash N$ has an action of the universal Cartan $T=B/N$, and is a torsor for a quotient torus $T\twoheadrightarrow T_X$. By our choice of base point, we can identify this torsor with $T_X$. In the rest of this paper, we will assume that our spherical variety satisfies $T_X=T$; however, for now we proceed with general definitions.

All important combinatorial invariants of the spherical variety live in the rational vector space $\mathfrak t_X^*$ spanned by the character group $\Lambda_X$ of this torus, or in the dual vector space $\mathfrak t_X$, containing the dual lattice $\ch\Lambda_X$. By ``lattice points'', below, we will mean points belonging to one of these lattices. The spaces $\mathfrak t_X, \mathfrak t_X^*$ are the root and coroot space for the dual group $\check G_X$ of $X$.\footnote{The \emph{Gaitsgory--Nadler dual group} was defined in a Tannakian way in \cite{GN}, but not completely identified in all cases. A combinatorial description  of a dual group (presumably the same) was consequently afforded by Knop and Schalke \cite{Knop-Schalke}. The invariants that we present here are those of Knop and Schalke, which match standard invariants of the theory of spherical varieties. For this paper, however, this distinction between constructions of the dual group is immaterial, as we impose the condition that $T_X= T$ and ``all simple roots of $G$ are spherical roots of type $T$'', which implies that in both versions of the dual group, $\check G_X= \check G$.} The antidominant Weyl chamber for $\check G_X$ in $\mathfrak t_X$ is denoted by $\Cal V$ in the theory of spherical varieties, because it coincides with the so-called cone of $G$-invariant valuations, see \cite{KnLV}. Up to this point, all data depend only on the open $G$-orbit $X^\bullet$, not on its affine embedding $X$.

The affine embedding $X$ of $X^\circ$ defines an affine toric embedding $X\sslash N$ of $T_X$, described by the cone $\Cal C_0(X) \subset \mathfrak t_X$ whose lattice points are all cocharacters $\check\lambda$ into $T_X$ such that $\lim_{t\to 0} t^{\ch\lambda}\in X\sslash N$. We will denote by $\mathfrak c_X$ the monoid of lattice points $\check\Lambda_X \cap \Cal C_0(X) \subset \mathfrak t_X$, and by $\mathfrak c_X^-$ its intersection with the cone $\Cal V$ of invariant valuations. 
The cone $\Cal C_0(X)\subset \mathfrak t_X$ has a canonical set of generators $\ch\nu_D$, the valuations associated to all $B$-stable divisors $D\subset X$. The set of all irreducible $B$-stable divisors in $X$ will be denoted by $\mathcal D(X)$, and by ``valuation associated'' we mean the restriction of the corresponding valuation to the group of nonzero $B$-eigenfunctions: $\check \nu_D: k(X)^{(B)} \to \mathbb Z$, which factors through the character group $\Lambda_X$ and hence can be identified with an element of $\check\Lambda_X\subset \mathfrak t_X$. The map $\mathcal D(X)\ni D\mapsto \check\nu_D\in \check\Lambda_X$ will be denoted by $\varrho_X$. Inside of $\mathcal D(X)$ there is a distinguished subset $\mathcal D$, depending only on the open $G$-orbit $X^\bullet$, which consists of the closures of $B$-stable divisors in $X^\bullet$; those are called \emph{colors}. We will often abuse language and write ``colors'' for the images of $\mathcal D$ in $\ch\Lambda_X$.

We remark that the valuation map $\varrho_X$ may fail to be injective, but this can only happen when two colors have the same image. If this is the case for $X^\bullet = H\backslash G$, there is always a torus covering of it such that all colors have distinct valuations (see \S\ref{sect:freemonoid}); for example, $\GL_1\backslash \PGL_2$ has two colors with valuation $\frac{\check\alpha}{2}$, but their preimages in $\GL_1\backslash\GL_2$ (where $\GL_1$ is embedded as the general linear group of a one-dimensional subspace) induce different valuations. 
In any case, colors give rise to a map $\mathbb N^{\mathcal D}\to \mathfrak c_X$, whose image we denote by $\mathfrak c_X^{\Cal D}$. 

We define an ordering $\preceq$ on $\ch\Lambda_X$ by postulating that $\ch\lambda \preceq \ch\lambda'$ if 
$\ch\lambda'-\ch\lambda$ can be written as a non-negative \emph{integral} combination of the valuations $\ch\nu_D$, with $D\in \mathcal D$, i.e., if $\ch\lambda'-\ch\lambda \in \mathfrak c_X^{\Cal D}$. We use the symbol $\le $ for the ordering on $\ch\Lambda_G$ defined by the positive coroots of $G$, that is, $\ch\lambda \le \ch\lambda'$ iff $\ch\lambda'-\ch\lambda$ is a sum of positive coroots. Notice that the ordering $\preceq$ is not defined simply in terms of the cone spanned by the $\ch\nu_D$'s: there can be non-comparable lattice points in this cone (and similarly for the ordering $\le$). 
As we will see later, this ordering describes the closure relations on the global model 
of the arc space of $X$.

\subsubsection{Spherical roots of type $T$} 
\label{subsection:typeT}

From Section \ref{sect:models} onwards we assume that $T_X=T$, equivalently, $B$ acts simply transitively on $X^\circ$. From Section \ref{sect:Heckeact} on we assume, further, that \emph{all simple roots of $G$ are spherical roots of type $T$}. Let us explain what this means: 
For a simple root $\alpha$ of $G$, let $P_\alpha \supset B$ denote 
the corresponding parabolic of semisimple rank one. The quotient $P_\alpha/\mathfrak R(P_\alpha)$ by its radical is isomorphic to $\PGL_2$, and the invariant-theoretic (or geometric) quotient $X^\circ P_\alpha/\mathfrak R(P_\alpha)$ is a spherical variety for $\PGL_2$. In characteristic zero, over an algebraically closed field, those belong to one of the following types, see \cite[Lemma 3.2]{KnBorbit}:
\begin{itemize}
 \item a point: $\PGL_2\backslash\PGL_2$;
 \item type $T$: $\mathbb G_m\backslash \PGL_2$;
 \item type $N$: $\mathcal N(\mathbb G_m)\backslash\PGL_2$, where $\mathcal N$ denotes normalizer;
 \item type $U$: $S\backslash \PGL_2$, where $N\subset S\subset B$. 
\end{itemize}

In positive characteristic there are some more cases, investigated by Knop in \cite{Knop-sphericalroots}. 

Our assumption from Section~\ref{sect:Heckeact} onwards is that for every simple root $\alpha$, this $\PGL_2$-spherical variety is isomorphic to $\mathbb G_m\backslash \PGL_2$. Our assumptions imply that the stabilizer in $P_\alpha$ of a point on the open orbit is isomorphic to $\mathbb G_m$, and that there are precisely two colors $D_\alpha^+, D_\alpha^-$ contained in $X^\circ P_\alpha$ (the $\pm $ labeling is arbitrary). We will denote the set of these two elements by $\mathcal D(\alpha)\subset\mathcal D$; notice that these sets are not disjoint as $\alpha$ varies. Moreover, the associated valuations satisfy (see \cite[\S 3.4]{Lun97}):
\[ \ch\nu_{D^+_\alpha} = -s_\alpha \ch\nu_{D^-_\alpha},\]
\[ \ch\nu_{D^+_\alpha} + \ch\nu_{D^-_\alpha} = \ch\alpha, \]
and finally an element $D\in \mathcal D$ belongs to $\mathcal D(\alpha)$ iff $\left<\alpha,\ch\nu_{D}\right> >0$ (in which case $\left<\alpha,\ch\nu_{D}\right> =1$, by the above). 

\begin{rem}
 Our assumptions above are over the algebraically closed field $k$. Over a finite field $\mathbb F$, the Galois group acts on the set of colors, compatibly with its action on the set of simple roots. If, for example, $G$ is split, each set $\mathcal D(\alpha)$ is preserved by Frobenius, and the stabilizer of a point on the open $P_\alpha$-orbit is a form of $\mathbb G_m$. If that form is split, Frobenius acts trivially on $\mathcal D(\alpha)$; if not, it permutes the two colors.
\end{rem}

\begin{rem}\label{rem:valuation}
A very straightforward way to compute the valuations $\ch\nu_{D^{\pm}_\alpha}$ in any example is the following: If all simple roots are spherical roots of type $T$ and $T_X=T$, the stabilizer $S$ of a point in the open $P_\alpha$-orbit on $X$ is a subgroup isomorphic (over the algebraic closure) to $\mathbb G_m$. Choose a Borel subgroup $B\subset P_\alpha$ containing $S$. An isomorphism $\mathbb G_m\simeq S$ gives rise to a cocharacter $\ch\nu:\mathbb G_m \to B\to T$, and we can choose this isomorphism so that $\left<\alpha, \ch\nu \right> = 1$. There are two choices for $B$, and they correspond to the valuations $\ch\nu_{D^{\pm}_\alpha}$.
\end{rem}

\subsection{Affine degeneration, and assumptions in positive characteristic} \label{sect:affine-degeneration}  \label{assumptions-finitefield}

In characteristic zero, there is a well-known affine family degenerating $X$ to a horospherical variety \cite{Pop86}, \cite[\S 5.1]{GN}, which in turn is related to its degeneration to the normal bundle of orbits in a smooth toroidal (e.g., a ``wonderful'') compactification of $X$, 
\cite{Brion07}, \cite[\S 2.5]{SV}. These constructions sometimes fail in positive characteristic, and therefore \textbf{the statements of this subsection should be considered as assumptions in positive characteristic}. These assumptions will be imposed on $X$ for the remainder of the paper. Notice that the Luna--Vust theory of spherical embeddings holds in arbitrary characteristic over an algebraically closed field by \cite{KnLV}.

\subsubsection{}
We define a filtration $\Cal F_{\lambda}$ on $k[X]$ for $\lambda \in \Lambda_X$ 
by letting $\Cal F_{\lambda}$ consist of all $f \in k[X]$
such that every highest weight $\mu \in \Lambda_X$ of the rational $G$-module 
generated by $f$ satisfies $\brac{\lambda-\mu, \Cal V} \le 0$.
The affine degeneration $\mathscr X$ is the affine variety defined to be the
spectrum of the Rees algebra associated to the filtration above: 
\[ k[\mathscr X] = \bigoplus_{\lambda\in \Lambda_X} \Cal F_\lambda \ot e^\lambda \subset k[X \times T_X] 
\] 
where $e^\lambda \in k[T_X]$ denotes the character corresponding to $\lambda \in \Lambda_X$. This family is naturally equipped with an action of the product $G\times T_X$.

Note that $k[\mathscr X]$ contains $\bigoplus_{\brac{\lambda,\Cal V}\le 0} k e^\lambda$. 
Define $\ol{T_{X,\mss}}$ to be the spectrum of $\bigoplus_{\brac{\lambda,\Cal V}\le 0} k e^\lambda$. 
This is an affine toric variety with open orbit isomorphic to the
quotient $T_{X,\mss}$ of $T_X$ by the subtorus $\mathcal Z(X)^0$ (the ``connected center of $X$'') generated by cocharacters in $\Cal V \cap (-\Cal V)$.

In summary, we get a $G \times T_X$-equivariant map 
\[ \mathscr X \to \ol{T_{X,\mss}}, \] 
which we consider as our affine family of degenerations.

\subsubsection{} 
For $a \in \ol{T_{X,\mss}}(k)$, let $X_a$ denote the fiber of 
$\mathscr X \to \ol{T_{X,\mss}}$ over $a$. 
From the definition of the filtration $\Cal F_\lambda$, it follows that 
$k[X_a]^N = k[X]^N$. 
The following ``multiplicity-free'' characterization of spherical varieties (in arbitrary
characteristic) implies that each fiber $X_a$ is spherical.

\begin{thm}[{\cite{VinbergKimelfeld}, \cite[Theorem 25.1]{Timashev}}]
A normal quasi-affine variety $X$ is spherical if and only if 
the non-zero weight spaces of $k[X]^N$ are all $1$-dimensional.
\end{thm}

Since the torus $T_X$ is determined by $k[X]^N$, we have a family of spherical varieties
$X_a$ with associated torus $T_{X_a} \cong T_X$. 

We also have $k[\mathscr X]^N = \bigoplus_{\brac{\lambda-\mu,\Cal V}\le 0} k[X]^{(B)}_\mu \ot e^\lambda$ where $k[X]^{(B)}_\mu$ is the $1$-dimensional $B$-eigenspace of weight $\mu$. 
This gives a canonical isomorphism 
\begin{equation} \label{e:scrXsslashN} 
    \mathscr X/\!\!/N \cong X/\!\!/N \times \ol{T_{X,\mss}}. 
\end{equation}
Define $\mathscr X^\circ$ to be the preimage of 
$T_X \times \ol{T_{X,\mss}} \subset \mathscr X/\!\!/N$ under the projection
$\mathscr X \to \mathscr X/\!\!/N$. 
Equivalently, $\mathscr X^\circ$ is the union of the open $B$-orbits of the fibers $X_a$.

It will be convenient to lift the isomorphism \eqref{e:scrXsslashN} to $\mathscr X^\circ$.
\begin{prop} \label{prop:noncanon-section}
There is a (non-canonical) $B$-equivariant isomorphism $\mathscr X^\circ \cong X^\circ \times \ol{T_{X,\mss}}$ over $\ol{T_{X,\mss}}$, compatible with \eqref{e:scrXsslashN}.
\end{prop}

We will comment on the proof of this proposition in conjunction with Theorem \ref{thm:localstructure} below. 

The affine degeneration is closely related to compactifications of the spherical variety. 
Namely, let $\mathscr X^\bullet \subset \mathscr X$ denote the union of the $G$-translates of $\mathscr X^\circ$. (It is the open subvariety which
specializes over each fiber $X_a$ to the open $G$-orbit.)
The quotient $\mathscr X^\bullet/ T_X$ turns out to be a proper embedding of $X^\bullet/\mathcal Z(X)^0$. For applications, one is interested in proper embeddings of $X^\bullet$ itself, and preferably smooth ones. Without getting into the details of the construction (we point the reader to \cite{KnLV}), we formulate the following result on the existence and local structure of ``smooth toroidal compactifications'':

Define $P(X) \supset B$ to be the parabolic subgroup of $G$ equal to 
$\{ g \in G \mid X^\circ \cdot g = X^\circ \}$, and let $N_{P(X)}$ denote its unipotent radical. Recall that we have fixed a base point $x_0\in X^\circ$, for convenience.

\begin{thm}\label{thm:localstructure}
 There is a proper, smooth, $G$-equivariant embedding $X^\bullet \hookrightarrow \ol X$, a Levi\footnote{Of course, there is no distinguished Levi in $P(X)$ abstractly, but the choice of base point $x_0$ sometimes imposes restrictions on the Levi subgroups that work; for example, the derived subgroup of $L$ must stabilize $x_0$. In any case, the assumptions on $X$ in the main body of this paper, that $B$ acts simply transitively on $X^\circ$ and $\Cal V=$ the antidominant cone, imply that any $L\subset P(X)= B$ satisfies the Local Structure Theorem.} $L \subset P(X) \subset G$,  and a smooth toric embedding $\overline{T_X}$ of the quotient $T_X$ of $L$, such that the action map 
 \[ L \times N_{P(X)} \ni (l, u) \mapsto x_0 lu\]
 descends to $T_X$ and extends to an open embedding
 \[ \overline{T_X} \times N_{P(X)} \hookrightarrow \ol X,\]
 whose image is the union of all open $B$-orbits on $\ol X$. Moreover, the support of the fan describing the toric variety $\overline{T_X}$ is the cone $\Cal V$ of invariant valuations.
\end{thm}

\begin{rem}
 Note that the embedding that we denoted above by $\ol{T_{X,\mss}}$ is associated to the image of the \emph{opposite} cone $-\Cal V$ modulo $\Cal V\cap (-\Cal V)$. Thus, there is a map from $\overline{T_X}$ to $\ol{T_{X,\mss}}$ only after inversion in $T_X$. We allow ourselves this notational flaw, since $\overline{T_X}$ will not be used beyond this subsection and the next.
\end{rem}

Theorem \ref{thm:localstructure} is the local structure theorem of Brion--Luna--Vust \cite[Th\'eor\`eme 3.5]{BLV}, applied to smooth toroidal compactifications \cite{KnLV}. We outline its proof due to Knop \cite[Theorem 2.3]{KnMotion}, which also applies to Proposition \ref{prop:noncanon-section}. The main issue is how to choose the pair $(x_0,L)$ appropriately. 
Let $\Cal B$ denote the flag variety of $G$, and assume for the moment that no choice of $B$, $x_0$ has been made.
One considers triples $(x,B,\chi) \in X\times\mathcal B\times \mf t_X^*$ such that 
$x$ lies in the open $B$-orbit. Out of these data one constructs elements in the cotangent bundle $T^*X^\bullet$ as follows: If $\chi$ is the differential of a character (also to be denoted by $\chi\in \Lambda_X$), the corresponding cotangent vector is the logarithmic differential $\left.\frac{d f_\chi}{f_\chi}\right|_x$ at $x$ of a rational $B$-eigenfunction with eigencharacter $\chi$, and for general $\chi$ we extend this construction by linearity. This gives rise to a map 
\[ (X\times\mathcal B)^\bullet \times \mf t_X^* \to T^*X^\bullet \]
(where the bullet denotes ``open $G$-orbit'') with dense image. We can also apply this construction to the family $\mathscr X^\bullet$, obtaining vectors in the \emph{relative} cotangent bundle over $\ol{T_{X,\mss}}$. Composing with the moment map $T^*X^\bullet \to \mathfrak g^*$, Knop shows that if we choose a sufficiently generic, semisimple vector $\xi$ in the image, and take $L=\on{Stab}_G(\xi)$ and $x_0$ a point of a triple $(x_0, B, \chi)$ in its fiber, Theorem \ref{thm:localstructure} is satisfied. The same argument applies to prove Proposition \ref{prop:noncanon-section}, as follows: Instead of fixing $x_0$, fix first just the Borel subgroup $B$, and consider all triples $(x,B,\chi)$ over $\xi$, now with $x \in \mathscr X^\circ$ (defined with respect to this Borel). The set of these triples is a 
$T_X$-equivariant section of the map 
$\mathscr X \to \mathscr X\sslash N \cong X\sslash N \times \ol{T_{X,\mss}}$ 
over $T_X\times \ol{T_{X,\mss}}$, where $T_X$ acts diagonally on the latter.

\subsection{The formal loop space}
\label{sect:formalloop}
For any $k$-scheme $X$, define the space of formal arcs by
\[ \msf L^+ X(R) = X(R\tbrac t) \] 
for a test ring $R$. 
It is well-known (cf.~\cite[Proposition 1.2.1]{KV}) that $\msf L^+ X$ is representable by a scheme (of infinite type), which is
the projective limit of the schemes $\msf L^+_n X,\, n \in \mbb N$, 
representing the spaces of $n$-arcs 
$\msf L^+_n X(R) = X(R[t]/t^n)$.
If $X$ is of finite type over $k$, then so is each $\msf L^+_n X$. If $X$ is smooth over $k$, then each $\msf L^+_n X$ is smooth over $k$
and $\msf L^+ X$ is formally smooth over $k$. If $X$ is affine, then so are
$\msf L^+_n X$ and $\msf L^+ X$. 

Define the formal loop space 
$\msf L X(R) = X(R\lbrac t)$. 
If $X$ is affine, then $\msf L X$ is representable by an ind-affine ind-scheme,
and we have a closed embedding $\msf L^+ X \into \msf L X$.

\subsubsection{}
Let $X$ be an affine spherical $G$-variety.
Define 
\[ \msf L^\bullet X := \msf L X \sm \msf L(X\sm X^\bullet), \]
which admits an open embedding into $\msf L X$. 
The $G$-action on $X$ induces a natural action of $\msf L^+ G$ on 
$\msf L X$ and $\msf L^+X, \msf L^\bullet X$ are stable under this action.

\begin{eg} \label{eg:A1}
Let $X = \mbb A^1$ with the scaling $\mbb G_m$-action. Then  
$\msf L^+ X = \mbb A^\infty$, where we consider $\mbb A^\infty = \Spec k[a_0,a_1,\dotsc]$
as the coefficients of infinite Taylor series, and the ind-scheme
$\msf L X = ``\underset{\longrightarrow m}\lim" \Spec k[a_{-m},a_{-m+1},\dotsc]$ considered
as the coefficients of Laurent series.
Let $X^\bullet = \mbb A^1 \sm \{0\}$. Then $\msf L^\bullet X = \msf L X \sm \{0\}$ so $\msf L^+ X \cap \msf L^\bullet X = \mbb A^\infty \sm \{0\}$, whereas $\msf L^+(X^\bullet) = \mbb G_m \times \mbb A^\infty$. 
\end{eg}

\begin{rem} 
The $k$-points of $\msf L^\bullet X$ are in bijection with $X^\bullet(k\lbrac t)$. 
Since $X^\bullet$ is in general not affine, however, $X^\bullet(k\lbrac t)$ does not
always have an ind-scheme structure. Even when $X^\bullet$ is affine, 
$\msf L(X^\bullet)$ may not be isomorphic to $\msf L^\bullet X$, despite having the same sets of $k$-points. 
\end{rem}

\subsubsection{Orbits on the formal loop space} \label{sect:stratLXpoint}

For ease of notation, 
let $\mf o = k\tbrac t$ and $F = k\lbrac t$, so 
$\msf L^+G(k) = G(\mf o),\, \msf L^\bullet X(k) = X^\bullet(F)$.
We review the decomposition of $G(\mf o)$-orbits on $X^\bullet(F)$ due to \cite{LV}.
We present the reformulation of this result found in \cite[Theorem 3.3.1]{GN}. 

\smallskip

Let $X$ be an affine spherical variety, and pick a pair $(x_0, L)$ as in Theorem \ref{thm:localstructure} to write $X^\circ= T_X \times N_{P(X)}$ for its open Borel orbit. 
For a cocharacter $\ch\theta\in \ch\Lambda_X$ we let $t^{\ch\theta} \in T_X(F)$ 
denote the image of the uniformizer $t\in \mf o$ under the map $\ch\theta: \mbb G_m \to T_X$. 
Let $X^\bullet (F)/ G(\mf o)$ denote the \emph{set} of equivalence
classes of $G(\mf o)$-orbits in $X^\bullet (F)$.

\begin{thm}[{\cite{LV}, \cite[Theorem 3.3.1]{GN}}] \label{thm:Gorbits}
\begin{enumerate} 
\item The map 
\begin{equation}
 \ch\theta \mapsto 
X^\bullet(F)_{G:\ch\theta} := x_0 \cdot t^{\ch\theta} G(\mf o) 
\end{equation}
is a bijection of sets $\Cal V \cap \ch\Lambda_X \cong X^\bullet(F)/ G(\mf o)$ (which is independent of the choice of $(x_0,L)$). 
\item This bijection restricts to a bijection $\mf c_X^- 
\cong (X(\mf o)\cap X^\bullet(F))/ G(\mf o)$.
\end{enumerate}
\end{thm}

The theorem can be viewed as a generalization of the Cartan decomposition.
Define $\msf L^{\ch\theta} X$ to be the $\msf L^+G$-orbit of $x_0 t^{\ch\theta} \in \msf L^\bullet X$.

\begin{rem} \label{rem:C0V}
Observe that if $\ch\theta\in \mf c_X^-$, then 
in particular $\ch\theta\in \Lambda^-_X$ is antidominant (as a weight for the dual group $\ch G_X$ of $X$). 
Therefore, $w\ch\theta- \ch\theta\in \Lambda^\pos_X \subset \mf c_X$ for any $w\in W_X$, and hence $w\ch\theta\in \mf c_X$.
We suggest that the monoid $\mf c_X^- = \Cal C_0(X) \cap \Cal V \cap \ch\Lambda_X$ should be
thought of as the set of $W_X$-orbits in $\ch\Lambda_X$ that are entirely
contained in the cone $\Cal C_0(X)$.
\end{rem}

\begin{prop} \label{prop:smoothfiberstrata}
If $k$ has positive characteristic, 
assume that the stabilizer of $B$ acting on $x_0 \in X^\circ$ is a smooth subgroup. 
For any $\ch\theta \in \Cal V \cap \ch\Lambda_X$, 
the $\msf L^+G$-orbit $\msf L^{\ch\theta} X$ is a formally smooth $k$-scheme
with $k$-points equal to $X^\bullet(F)_{G:\ch\theta}$, and 
$\msf L^{\ch\theta} X$ is open in its closure in $\msf L X$. 
\end{prop}

Therefore, the collection of  
$\msf L^{\ch\theta}X,\, \ch\theta \in \Cal V \cap \ch\Lambda_X$, 
form a stratification of $\msf L^\bullet X$, in the sense that the strata are disjoint
and contain all the $k$-points.

\begin{proof}
Let $S \subset \msf L^+G$ denote the stabilizer subgroup of $x_0 t^{\ch\theta}
\in \msf L X(k)$. 
\emph{A priori}, the $\msf L^+G$-orbit $\msf L^{\ch\theta} X$ is 
defined as the fpqc sheaf quotient $\msf L^+G/S$. 
For $n \in \mbb N$, let $S_n$ denote the scheme-theoretic image 
of $S$ under the projection $\msf L^+G \to \msf L^+_n G$. 
For $m > n$, the transition map 
\[ \msf L^+_m G / S_m \to \msf L^+_n G / S_n = \msf L^+_m G / (S_n \cdot \ker(\msf L^+_m G \to \msf L^+_n G)) \]
is smooth affine since $\ker(\msf L^+_m G \to \msf L^+_n G)$ is smooth and unipotent. 
The schemes $\msf L^+_n G/S_n$ are also smooth since $\msf L^+_n G$ is smooth. 
Therefore $\msf L^{\ch\theta} X \cong \msf L^+ G / S \cong \varprojlim 
\msf L^+_n G / S_n$ is representable by a formally smooth scheme.

The fact that $\msf L^{\ch\theta} X \to \msf L X$ is open in its closure
will follow from Lemma~\ref{lem:theta-orbit-identification} below.
\end{proof}

\subsubsection{}
Recall the family $\mathscr X$ introduced in \S \ref{sect:affine-degeneration}. We define a map of formal loop spaces 
\begin{equation}\label{e:ithetaLX} 
    i_{\ch\theta} : \msf L X \to \msf L \mathscr X 
\end{equation}
as follows:
By definition, we have a map $p: X \times T_X \to \mathscr X$. 
For a test ring $R$ and $\gamma \in X( R \lbrac t)$, 
define $i_{\ch\theta}(\gamma) = p( \gamma, t^{-\ch\theta} ) \in \mathscr X(R \lbrac t)$,
where we are considering $t^{-\ch\theta} \in T_X(R \lbrac t)$.

\begin{lem} \label{lem:itheta-lands-open}
The point $i_{\ch\theta}(x_0 t^{\ch\theta}) \in \msf L \mathscr X(k)$ is contained in 
$\msf L^+( \mathscr X^\circ )$. 
\end{lem}
\begin{proof}
By Proposition~\ref{prop:noncanon-section} (and the ensuing discussion), 
there exists a non-canonical section $s: T_X \times \ol{T_{X,\mss}} \to \mathscr X^\circ$
such that $s(1,1) = x_0$. 
If we choose a splitting $T \into G$ contained in the Levi $L$ from Theorem~\ref{thm:localstructure}, then the section $s$ is $T \times T_X$-equivariant, where 
$T_X$ acts diagonally on $T_X \times \ol{T_{X,\mss}}$. 
Thus on $k\lbrac t$-points, 
we have $i_{\ch\theta}(x_0 t^{\ch\theta}) = p(x_0 t^{\ch\theta}, t^{-\ch\theta})
= s(t^{\ch\theta}\cdot t^{-\ch\theta}, t^{-\ch\theta}) = s(1, t^{-\ch\theta})$,
which lies in $\mathscr X^\circ (k \tbrac t)$
since $-\ch\theta\in -\Cal V$. 
\end{proof}

The map $i_{\ch\theta}$ is $\msf L G$-equivariant. 
As a consequence of the lemma, we see that the $\msf L^+G$-orbit of
$i_{\ch\theta}(x_0 t^{\ch\theta})$ is contained in $\msf L^+(\mathscr X^\bullet)$. 
Thus
$i_{\ch\theta}$ restricts to a map $\msf L^{\ch\theta}X \to \msf L^+(\mathscr X^\bullet)$. 

We have a closed embedding $\msf L^+ \mathscr X \into \msf L \mathscr X$
and an open embedding $\msf L^+(\mathscr X^\bullet) \into \msf L^+\mathscr X$. 

\begin{lem}  \label{lem:theta-orbit-identification}
Assume that the stabilizer of $B$ acting on $x_0 \in X^\circ$ is a smooth subgroup. 
Then the map $i_{\ch\theta}$ induces an isomorphism 
\[ \msf L^{\ch\theta} X \overset\sim\to \msf L X \xt_{i_{\ch\theta}, \msf L \mathscr X} \msf L^+(\mathscr X^\bullet).
\]
\end{lem}
\begin{proof}
Since $i_{\ch\theta} : \msf L X \to \msf L \mathscr X$ is injective on $S$-points,
it suffices to show that 
$\msf L X \xt_{i_{\ch\theta}, \msf L \mathscr X} \msf L^+(\mathscr X^\bullet)$
is the $\msf L^+G$-orbit of $i_{\ch\theta}(x_0 t^{\ch\theta}) \in \msf L^+(\mathscr X^\bullet)(k)$.

Let $\Cal B = G/B$ denote the flag variety, and let
$(\mathscr X \times \Cal B)^\bullet \subset \mathscr X \times \Cal B$ denote 
the $G$-stable open subvariety 
consisting of all points $(x,\tilde B)$ where $x \in \mathscr X$ in the open $\tilde B$-orbit.
We have a smooth surjection $(\mathscr X\times \Cal B)^\bullet \to \mathscr X^\bullet$ by
first projection. 
On the other hand, Proposition~\ref{prop:noncanon-section} 
gives a non-canonical $G$-equivariant isomorphism 
$\mathscr X^\circ \xt^B G \cong (X^\circ \xt^B G) \times \ol{T_{X,\mss}}$. 
The action map and second projection induce an isomorphism 
$\mathscr X^\circ \xt^B G \cong (\mathscr X \times \Cal B)^\bullet$ and 
similarly for $X^\circ \xt^B G$. 
Thus there is a $G$-equivariant
isomorphism 
\[ (\mathscr X\times \Cal B)^\bullet \cong (X \times \Cal B)^\bullet \times \ol{T_{X,\mss}} \] 
over the base $\ol{T_{X,\mss}}$. 
The choice of base point $(x_0,B) \in (X \times \Cal B)^\bullet$ gives
an isomorphism $G/\textnormal{Stab}_B(x_0) \cong (X \times \Cal B)^\bullet$. 
Thus we have a $G$-equivariant smooth surjection 
$G \times \ol{T_{X,\mss}} \to \mathscr X^\bullet$ which induces
a surjection 
\begin{equation} \label{e:arc-family-orbit}
    \msf L^+( G \times \ol{T_{X,\mss}} ) \to \msf L^+ (\mathscr X^\bullet)
\end{equation}
because $\on{Stab}_B(x_0)$ is assumed to be smooth. 
Since the composition 
$\msf L X \overset{i_{\ch\theta}}\to \msf L \mathscr X \to \msf L \ol{T_{X,\mss}}$ 
sends everything to the image of $t^{-\ch\theta}$ in $\msf L\ol{T_{X,\mss}}(k)$, 
we deduce that \eqref{e:arc-family-orbit} induces a 
surjection 
$\msf L^+G \to \msf L X \xt_{i_{\ch\theta}, \msf L \mathscr X} \msf L^+(\mathscr X^\bullet)$.
Thus the latter fiber product is a single $\msf L^+G$-orbit.
\end{proof}

\subsubsection{}
In this subsection we will use properties of toroidal compactifications of 
$X$. We refer the reader to \cite[\S 2.3-2.5]{SV}, \cite{KnLV}, \cite[\S 8]{GN} for an overview. 

Let $\ol X$ denote a complete, smooth toroidal embedding of $X^\bullet=H\bs G$ 
(the embedding is not 
simple if $N_G(H)/H$ is not finite). 
For any $\ch\theta \in \Cal V \cap \ch\Lambda_X$, the point 
$x_0 \cdot t^{\ch\theta} \in X^\bullet(F) \subset \ol X(F)$ defines
an $\mf o$-point of $\ol X$ by the valuative criterion of properness. 
In particular, we can take the limit as $t\to 0$ to get a point $x_{\ch\theta}
:= \lim_{t\to 0} x_0 \cdot t^{\ch\theta} \in\ol X(k)$.
Let $Z \subset \ol X$ denote the $G$-orbit of $x_{\ch\theta}$. 

Let $\Delta_X$ denote the set of normalized spherical roots of $X$.
Equivalently, $\Delta_X$ is the set of simple coroots of $\ch G_X$. 
Let $I \subset \Delta_X$ denote the set of spherical roots $\sigma$ 
such that $\brac{\sigma, \ch\theta}=0$. In the language of \cite[\S 2.3.6]{SV}, 
the orbit $Z$ ``belongs to $I$-infinity''. Let $X_I^\bullet = H_I\bs G$ denote 
the open $G$-orbit on the normal bundle $N_Z \ol X$ of $Z$ in $\ol X$; 
this is called a \emph{boundary degeneration} of $X$. 
As the notation suggests, $X_I^\bullet$ and the conjugacy class of $H_I$ depend only on the subset $I$
and not $\ch\theta$ (see \cite[Proposition 2.5.3]{SV}). We call $H_I$ a ``satellite'' of $H$,
following the terminology of \cite{BM}.

\begin{lem} \label{lem:H_I-conn}
Assume that $B$ acts simply transitively on $X^\circ$. 
Then the satellite subgroup $H_I \subset G$ is connected and smooth, for any subset $I\subset \Delta_X$.
\end{lem}
\begin{proof}
We have $X^\circ \cong X_I^\circ \cong B$. 
Let $H^0_I$ denote the reduced identity component of $H_J$.
Then $H^0_I\bs G \to H_I\bs G$ is a finite covering of spherical varieties sending the open $B$-orbit 
of $H^0_I\bs G$ to $X^\circ_I$. Since $B$ acts on $X^\circ_I\cong X^\circ$ with 
trivial stabilizer, the covering must be an isomorphism. 
\end{proof}

\begin{cor} \label{cor:stab-conn}
Assume that $B$ acts simply transitively on $X^\circ$. 
Let $\ch\theta \in \Cal V \cap \ch\Lambda_X$. 
Then the stabilizer of the group ind-scheme $\msf L H$ acting on $t^{\ch\theta} \in \Gr_G$
is a connected scheme.
\end{cor}
\begin{proof}
The cocharacter $-\ch\theta \in (-\Cal V)$ extends to a map 
$\mbb A^1 \to \ol{T_{X,\mss}}$. 
Let $\mf X := \mathscr X \xt_{\ol{T_{X,\mss}},-\ch\theta} \mbb A^1$. 
Then $\mf X \to \mbb A^1$ is an affine family with preimage over $\mbb G_m$ isomorphic
to $X \times \mbb G_m$. By \cite[Proposition 2.5.3]{SV}, the 
open $G$-orbit of the fiber over $0$ is isomorphic to $X_I^\bullet$. 
Let $\mf X^\bullet$ denote the union of the open $G$-orbits on each fiber,
so $\mf X^\bullet$ is a family over $\mbb A^1$ degenerating $X^\bullet$ to $X^\bullet_I$. 

The map $i_{\ch\theta}$ from \eqref{e:ithetaLX} 
factors through an embedding 
\[ \tilde i_{\ch\theta}: \msf L X \into \msf L \mf X \xt_{\msf L \mbb A^1} \pt \]
where $\pt \to \msf L \mbb A^1$ corresponds to $t \in k\lbrac t$. 
Lemma~\ref{lem:itheta-lands-open} implies that  
$\gamma := \tilde i_{\ch\theta}(x_0 t^{\ch\theta}) \in \msf L^+(\mf X^\bullet)$. 

Observe that $\on{Stab}_{\msf L H}(t^{\ch\theta}, \Gr_G) = 
\msf L H \cap t^{\ch\theta} (\msf L^+ G) t^{-\ch\theta}$ conjugates to 
\[ t^{-\ch\theta} (\msf L H) t^{\ch\theta} \cap \msf L^+G = 
\on{Stab}_{\msf L^+G}( x_0 \cdot t^{\ch\theta}, \msf L X ) = \on{Stab}_{\msf L^+G}( \gamma, \msf L^+(\mf X^\bullet) ). 
\]
Let $\mf J \subset G \times \mf X^\bullet$ denote the inertia group scheme
consisting of all $(g,x) \in G \times \mf X^\bullet$ such that $g x = x$. 
The fibers of $\mf J \to \mf X^\bullet$ are all conjugate to either $H$ or $H_I$. 
Since $\mf X^\bullet$ is smooth over $\mbb A^1$ (cf.~\cite[Corollary 5.2.3]{GN}), 
we have $H$ and $H_I$ are of the same dimension,
and they are smooth by Lemma~\ref{lem:H_I-conn}. 
Thus $\mf J \to \mf X^\bullet$ is a smooth morphism. 

Consider $\gamma$ as an arc $\Spec k\tbrac t \to \mf X^\bullet$. 
For $n\in \mbb N$, let $\gamma_n : \msf D_n := \Spec (k[t]/t^n) \to \mf X^\bullet$ denote the
corresponding $n$-jet and $\mf J_n := \mf J \xt_{\mf X^\bullet,\gamma_n} \msf D_n$ the
fiber product, which is a smooth group scheme over $\mf D_n$. 
Then 
\[ \on{Stab}_{\msf L^+G}( \gamma, \msf L^+(\mf X^\bullet) ) \cong 
\varprojlim_n \on{Sect}( \msf D_n, \mf J_n) 
\] 
where $\on{Sect}( \msf D_n, \mf J_n)$ 
is the scheme of maps $\msf D_n \to \mf J_n$ over $\msf D_n$. 
By smoothness of $\mf J \to \mf X^\bullet$, the
transition maps $\on{Sect}( \msf D_m, \mf J_m) \to \on{Sect}( \msf D_n, \mf J_n)$
are surjective for $m > n$. 
The transition maps are also group homomorphisms with unipotent kernels, so they
have contractible fibers. Therefore 
$\on{Stab}_{\msf L^+G}(\gamma, \msf L^+(\mf X^\bullet))$ contracts to 
$\on{Sect}(\msf D_0, \mf J_0) = \on{Stab}_G(\gamma(0), X_I^\bullet)$, which is conjugate
to the subgroup $H_I$. Since $H_I$ is connected by Lemma~\ref{lem:H_I-conn}, we are done.
\end{proof}

\section{Models for the arc space} \label{sect:models}

In this section we define two models (in the sense of Grinberg--Kazhdan) 
for the arc space of $X$, both of which
were already introduced in \cite{GN} and go back to ideas of Drinfeld. 
We call these the global and Zastava models (the term `global' refers to the fact
that the model depends on the curve $C$). 
The global model $\sM_X$ is crucial because it allows us to model the $G(\mf o)$-orbits of $X(\mf o)$,
something which cannot be done directly via the Zastava model. 
On the other hand, the Zastava model $\Scr Y_X$ is more suitable for finite type calculations. 

\smallskip

Various incarnations of these constructions have been used in 
\cite{FM, FFKM, BG, BFGM, BFG, ABB}.
To place our work in this context, we remark that when $X = \ol{N\bs G}^\aff$ 
we have $\sM_X = \ol\Bun_N$ is Drinfeld's compactification of $\Bun_N$ 
and $\sY_X$ is ``the'' Zastava space of \cite{FM}. 

\smallskip

While Gaitsgory--Nadler define the global and Zastava models for any affine $X$, 
in order to avoid various technical difficulties they faced (such as the existence
of \emph{twisted strata}, which are related to the existence of 
disconnected stabilizer subgroups) we make the following simplifying assumption:

\medskip

Starting from \S\ref{sect:localmodel}, 
\emph{we assume for the rest of the paper that $B$ acts simply transitively on $X^\circ$.} 
If $\ch G_X = \ch G$ and $X$ has no spherical roots of type $N$, then the assumption above always holds.

\begin{rem} \label{rem:Hconnected}
The assumption that $B$ acts simply transitively on $X^\circ$ implies that $H$ must
be connected, by Lemma~\ref{lem:H_I-conn}. 
\end{rem}

\subsection{Global model}

Gaitsgory--Nadler \cite{GN} define certain stacks of meromorphic quasimaps from $C \dashrightarrow X/G$ 
to model $X^\bullet(F)$, the loop space of $X^\bullet$. 
Our global model $\sM_X$ is the substack consisting of those quasimaps that 
extend to regular maps\footnote{In the literature, when $X=\ol{N\bs G}$ regular maps 
$C \to \ol{N\bs G}$ are still referred to as quasimaps to $N\bs G$.} $C \to X/G$. 

\subsubsection{Definition}
Define the stack 
\[ \sM_{X} := \Maps_\gen(C, X/ G \supset X^\bullet /G) \]
to be the open substack of $\Maps(C, X/G)$ representing maps generically
landing in $X^\bullet/G = (H\bs G)/G = H\bs\pt$. 

\smallskip

An $S$-point of $\Scr M_{X}$ is a map $f : C\times S \to X/G$, which is
equivalent to the datum $(\Scr P_G, \sigma)$ where 
\begin{itemize}
\item $\Scr P_G$ is a $G$-bundle on $C\times S$ and 
\item $\sigma : C\times S \to X\times^G \Scr P_G$ is a section over $C\times S$ 
such that
\item $\sigma|_{\Spec k(C) \xt S}$ lands in $X^\bullet \xt^G \Scr P_G = H\bs \Scr P_G$
and gives $\Scr P_G |_{\Spec k(C)\xt S}$ the structure of an $H$-bundle.
\end{itemize}
We call the preimage $\sigma^{-1}(X^\bullet \xt^G \sP_G)\subset C \xt S$ the 
locus of $G$-nondegenerate points. 

\begin{prop}
The natural map $\Scr M_{X} \to \Bun_G$ is schematic locally of finite type. 
In particular, $\Scr M_{X}$ is an algebraic stack locally of finite type over $k$.
\end{prop}
\begin{proof}
Since $\Scr M_{X}$ is an open substack of $\Maps(C, X/G)$,
it suffices to show the latter is schematic locally of finite type
over $\Bun_G$. 
Let $S \to \Bun_G$ correspond to a $G$-bundle $\Scr P_G$ on $C\xt S$.
Then the fiber product $S \underset{\Bun_G}\times \Maps(C, X/G)$ 
is isomorphic to the space of sections $C\xt S \to X \times^G \Scr P_G$ over $C \xt S$.
This space is representable by a $k$-scheme locally of finite type 
(\cite[Theorem 5.23]{FGA-explained}). 
\end{proof}

If $X = H\bs G$ is homogeneous, then $\Scr M_X = \Maps(C,H\bs \pt) = \Bun_H$. 

\subsubsection{Adelic description} \label{sect:adelic}
For motivational purposes, we give an ``adelic'' description of the $k$-points of 
$\sM_X$. 
Let $\mbb A$ denote the restricted product $\prod'_{v\in \abs C} F_v$ and 
$\mbb O = \prod_{v\in \abs C} \mf o_v$. 
(If $k=\mbb F_q$ then $\mbb A$ is the ring of adeles of the function field $\Bbbk = k(C)$.)

The underlying set of meromorphic quasimaps $C \dashrightarrow X/G$ 
can be identified with the set
\[ X^\bullet(\Bbbk) \xt^{G(\Bbbk)} G(\mbb A) / G(\mbb O) 
= H(\Bbbk) \bs G(\mbb A) / G(\mbb O). 
\]
Note that the $G$-action on $X$ induces a map 
\[ X^\bullet(\Bbbk) \xt^{G(\Bbbk)} G(\mbb A) / G(\mbb O) \to X^\bullet(\mbb A)/G(\mbb O) \subset X(\mbb A)/G(\mbb O). 
\]
The underlying set of $\sM_X(k)$ identifies with 
the preimage of $(X^\bullet(\mbb A) \cap X(\mbb O)) / G(\mbb O)$ under the map above.

\medskip

Note that the topologies on $X^\bullet(\mbb A)$ vs.~$X(\mbb A)$ are different: the
fact that the geometric constructions above depend on $X$ can be expressed
by saying that we are always using the topology of $X(\mbb A)$, not of $X^\bullet(\mbb A)$.

\subsubsection{} \label{def:partition}
For any set $\msf S$, define the set of unordered \emph{multisets} in
$\msf S$ to be the formal direct sum
\[ \Sym^\infty(\msf S) := \bigoplus_{\ch\lambda\in \msf S} \mbb N[\ch\lambda]. \] 
An element of $\Sym^\infty(\msf S)$ is a 
formal sum $\mf P = \sum_{\ch\lambda \in \msf S} N_{\ch\lambda}[\ch\lambda]$
where $N_{\ch\lambda}\ge 0$ are integers and only finitely many 
are nonzero. 
For a multiset $\mf P$, define 
\[ \oo C^{\mf P} = \oo \prod_{\ch\lambda \in \msf S} \oo C^{(N_{\ch\lambda})}
\]
to be the open subscheme of $C^{\mf P}:=\prod C^{(N_{\ch\lambda})}$ 
with all diagonals
removed, i.e., the subscheme of multiplicity free divisors.
We write $\mf P=0$ for the zero element, and we will use the 
convention $\oo C^{0}=C^{0}=\pt$. 
If we define $\abs{\mf P} = \sum N_{\ch\lambda}$, then there are 
natural maps $C^{\abs{\mf P}} \to C^{\mf P} \to C^{(\abs{\mf P})} \subset \Sym C$.

Now consider the case when $\msf S = \msf M\sm 0$ for a commutative monoid
$\msf M$. Then $\Sym^\infty(\msf M\sm 0)$ identifies with the set of 
\emph{partitions} of arbitrary elements in $\msf M$.
Given a partition $\mf P\in \Sym^\infty(\msf M\sm 0)$ as above, 
we define $\deg : \Sym^\infty(\msf M\sm 0) \to \msf M$ by 
\[
    \deg(\mf P) := \sum_{\ch\lambda\in \msf M\sm 0} N_{\ch\lambda}\ch\lambda 
\]
with addition taking place in $\msf M$. 
There is a natural order on the set $\Sym^\infty(\msf M\sm 0)$: 
we say that $\mf P$ refines $\mf P'$ if the difference $\mf P-\mf P'$
viewed as an element of $\bigoplus_{\ch\lambda\in \msf M\sm 0}
 \mbb Z [\ch\lambda]$ can be written as a sum of elements of the
form $[\ch\lambda']+[\ch\lambda'']-[\ch\lambda]$ with $\ch\lambda'+\ch\lambda''=\ch\lambda$ in $\msf M$. 

Let $\Prim(\msf M)$ be the set of primitive elements of $\msf M$, i.e.,
the elements $\ch\lambda\in \msf M\sm 0$ that cannot be decomposed
as a sum $\ch\lambda=\ch\lambda_1+\ch\lambda_2$ where $\ch\lambda_1,\ch\lambda_2 \in \msf M\sm 0$. Then any partition in $\Sym^\infty(\msf M\sm 0)$
can be refined to an element in $\Sym^\infty(\Prim(\msf M))$.

\subsubsection{Stratification of $\sM_X$}

\label{sect:globstrat}

We would like to stratify $\sM_X$ according to $G(\mf o_v)$-orbits of 
$X(\mf o_v)\cap X^\bullet(F_v)$ at each $v\in \abs C$, described in Theorem \ref{thm:Gorbits}.

Consider the set $\Sym^\infty(\mathfrak c_X^-\sm 0)$
of partitions in $\mathfrak c_X^-$, as defined
in \S\ref{def:partition}. 
Let  
\[ \ch\Theta = \sum_{\ch\theta\in \mathfrak c_X^- \sm 0} N_{\ch\theta} [\ch\theta] \] 
denote such a partition. 
We will write $\ch\Theta=0$ for the empty partition and 
$\ch\Theta=[\ch\theta]$ for the singleton partition corresponding to a
single element $\ch\theta$. 

In \S\ref{proof:globstrata}, we define locally closed substacks $\Scr M^{\ch\Theta}_X$ of $\Scr M_X$ ranging over all partitions 
$\ch\Theta$ in $\mathfrak c_X^-$.  
For simplicity we only describe $\Scr M^{\ch\Theta}_X$ on $k$-points below:
Such a point consists of a map $f : C\to X/G$ and a formal sum
$\sum_{v\in \abs C} \ch\theta_v \cdot v$ satisfying the following
conditions: 
\begin{itemize}
\item $\ch\theta_v \ne 0$ for finitely many $v\in \abs C$, 
\item for a fixed $\ch\theta \ne 0$, the cardinality of $\{ v\in \abs C \mid \ch\theta_v =\ch\theta\}$ equals $N_{\ch\theta}$, 
\item for each $v\in \abs C$
the restriction $f|_{\Spec \mf o_v} : \Spec \mf o_v \to X/G$ defines a point in $\msf L^{\ch\theta_v} X / \msf L^+ G$.  
\end{itemize} 

\begin{lem} \label{lem:globstrata}
The substack $\Scr M^{\ch\Theta}_X$ is
smooth and locally closed in $\Scr M_X$. 
\end{lem}

We defer the proof of Lemma~\ref{lem:globstrata} to \S\ref{proof:globstrata} of the appendix. 

\medskip
\begin{prop} \label{prop:globwhitney}
Assume $k=\mbb C$. Let $\Scr S$ denote the collection of connected components of 
$\Scr M^{\ch\Theta}_X$, ranging over all 
$\ch\Theta \in \Sym^\infty(\mathfrak c_X^- \sm 0)$. Then $\Scr S$ 
is a Whitney\footnote{We say that a stratification on an algebraic stack $\Scr M$ locally of
finite type is Whitney if the stratification is Whitney after pullback
along any (equivalently all) smooth cover of $\Scr M$ by a scheme.} 
stratification of $\Scr M_X$. 
\end{prop}

By a stratification we mean a collection of locally closed substacks that form a
disjoint union on $k$-points and such that the closure of any stratum is a union of
strata. The stratification is Whitney if the strata are smooth and
every pair of strata satisfies Whitney's condition B. (This only makes sense
if the characteristic of $k$ is zero; in positive characteristic, see 
\S \ref{sect:whitneyperverse} below.)

\smallskip

The proof of Proposition~\ref{prop:globwhitney} is given in \S\ref{proof:whitney}. 
We call this the \emph{fine stratification} of $\Scr M_X$. 
The open stratum $\sM_X^0 = \Maps(C, H\bs G/G)$ identifies
with $\Bun_H$. We abbreviate $\sM_X^{\ch\theta} := \sM_X^{[\ch\theta]}$. 

\begin{rem}
If $X$ is affine and homogeneous, then $\Cal C_0(X) \cap \Cal V = 0$ 
so the stratification is trivial, consisting of
the single smooth stratum $\Scr M_X$ itself.
\end{rem}

\subsubsection{} \label{sect:whitneyperverse}
Let $\rmD^b_{\Scr S}(\sM_X,\ol\bbQ_\ell)$  
denote the subcategory of bounded $\Scr S$-constructible complexes, i.e., 
the usual cohomology sheaf $\rmH^i(\Scr F)|_S$ is a local system
of finite rank for all $i\in\bbZ$ and $S\in\Scr S$. 
As explained in \cite{BBDG}, the category $\rmD^b_{\Scr S}(\sM_X)$ has a 
perverse $t$-structure. Let $\rmP_{\Scr S}(\sM_X)$ denote the heart
of this $t$-structure, i.e., this is the abelian subcategory of all perverse sheaves
that are $\Scr S$-constructible. 

In particular, the IC complex of the closure of any stratum $\sM_X^{\ch\Theta}$ 
is an object of $\rmP_{\Scr S}(\sM_X)$.
When $k$ has positive characteristic, this is the condition on $\Scr S$ 
that we need (in place of the Whitney condition). 
Proposition~\ref{prop:Whitney-poschar} shows that this condition 
indeed holds in positive characteristic.

\subsection{Toric case} \label{sect:baseA}
If we apply the definitions above to the special case where $G$ is replaced by the torus $T_X$ and $X$ is replaced by the toric variety $X\sslash N$, we obtain the space \begin{equation}  \label{e:defA}
    \Scr A = \Maps_\gen(C, (X/\!\!/N)/T_X \supset \pt) 
\end{equation}
of maps generically landing in $T_X/T_X=\pt$.

The stack $\Scr A$ has been previously studied in \cite[\S 3]{BNS} as 
a model for the formal arc space of the toric variety $X/\!\!/N$ 
(in particular, $\Scr A$ turns out to be representable by a scheme). 
We review the relevant properties below.

For any $N \in \mbb N$, we have the $N$th symmetric product
$C^{(N)}$ of $C$, which identifies with the Hilbert scheme $\on{Hilb}^N(C)$
parametrizing relative effective divisors in $C$ of degree $N$.
Let $\Sym C$ denote the disjoint union 
$\bigsqcup_{N\in \mathbb N} C^{(N)}$ (where $C^{(0)}=\pt$).

\begin{eg}
Observe that $\Maps_\gen(C, \mbb A^1/\mbb G_m\supset \pt)$ sends a test scheme $S$ to 
the set of relative effective Cartier divisors on $C \times S$, i.e., 
$\Maps_\gen(C, \mbb A^1/\mbb G_m\supset \pt) \cong \Sym C$. 
Addition of divisors gives $\Sym C$ the structure of a monoid. 
\end{eg}

Let $\mf c_X^\vee = \Hom(\mf c_X, \mbb N)$ denote the monoid dual 
to $\mf c_X$, so $k[X/\!\!/N]$ is the semigroup algebra of $\mf c_X^\vee$. 
Then there is an isomorphism 
\[ \Scr A \cong \Hom(\mf c_X^\vee, \Sym C) \] 
where the right hand side represents homomorphisms of monoid objects in the category of schemes
(with $\mf c_X^\vee$ viewed as a discrete scheme). 
A $k$-point of $\Scr A$ is a formal finite sum 
\[ \sum_{v\in \abs C} \ch\lambda_v \cdot v \]
where $\ch\lambda_v$ is an element of the dual monoid $\mathfrak c_X$ and $\ch\lambda_v=0$ for all but finitely many $v$. 

\subsubsection{} 
The stratification described in \S \ref{sect:globstrat} takes here the following form:
For any $\mf P \in \Sym^\infty(\mathfrak c_X\sm 0)$ there is a natural map 
\[ \oo C^{\mf P} \into \Scr A, \] 
where the image consists of the $k$-points $\sum_{v\in \abs C} \ch\lambda_v\cdot v$ such that the unordered multiset of nonzero $\ch\lambda_v$, counted
with multiplicities, coincides with $\mf P$.

\begin{prop}[{\cite[Proposition 3.5]{BNS}}] 
\label{prop:Astrata}
\ 
\begin{enumerate}
\item The maps $\oo C^{\mf P}\into \Scr A$ are locally closed embeddings,
and the collection of such embeddings over all $\mf P\in \Sym^\infty(\mathfrak c_X \sm 0)$ forms a stratification of $\Scr A$.
\item $\oo C^{\mf P'}$ lies in the closure of $\oo C^{\mf P}$
if and only if $\mf P$ refines $\mf P'$. 
\item The irreducible components of $\Scr A$ are in bijection with 
the closures of $\oo C^{\mf P}$ for $\mf P \in \Sym^\infty(\Prim(\mf c_X))$.
\end{enumerate}
\end{prop}

\begin{cor}[{\cite[Corollary 3.6]{BNS}}] 
For $\ch\lambda \in \mathfrak c_X$, let 
\[ \Scr A^{\ch\lambda} \subset \Scr A \]
denote the subscheme whose $k$-points consist of all 
$\sum_{v\in \abs C} \ch\lambda_v \cdot v$
such that $\sum_v \ch\lambda_v = \ch\lambda$. 
Then $\Scr A^{\ch\lambda}$ is a connected component of $\Scr A$, and
this gives a bijection 
\[ \pi_0(\Scr A) \cong \mathfrak c_X. \] 
\end{cor}

\begin{rem}\label{rem:Afree}
If $\mathfrak c_X$ is a free monoid with basis $\{ \ch\nu_i \}$, 
then for $\ch\lambda = \sum_i N_i \ch\nu_i$ we have 
$\Scr A^{\ch\lambda} = \prod_i C^{(N_i)}$. The bases for the Zastava spaces
in \cite[\S 2.1]{BFGM} take this form.
\end{rem}

\subsection{Zastava model} \label{sect:localmodel}

\emph{For the rest of this paper we assume that $B$ acts simply transitively on $X^\circ$.} 
This implies that $T_X = T$ and $\Lambda_X = \Lambda_G$, so we will use the notation
interchangeably. 

\medskip

We introduce a special case of the model used in \cite[Part III]{GN}, which is based 
on a general pattern pointed out by Drinfeld (see \cite[\S 4.2--4.4]{Drinfeld}). 
These are a generalization of the Zastava\footnote{``Zastava'' is Croatian for ``flag''.} spaces introduced by Finkelberg--Mirkovi\'c in \cite{FM, FFKM, BFGM}, and we will henceforth call them the Zastava model for $X$. 

\medskip

The Zastava model for $X$ is defined as 
\begin{equation}
    \sY = \Scr Y_{X} = \Maps_\gen(C, X/B \supset \pt),
\end{equation}
the stack of maps $C \times S \to X/B$ generically landing 
in $X^\circ/B=\pt$.

Applying $\Maps(C,?/T)$ to the natural map $X/N \to X/\!\!/N $ induces a map 
\begin{equation}\label{mappi} \pi: \Scr Y \to \Scr A.
\end{equation}
For $\ch\lambda \in \mf c_X$, let $\sY^{\ch\lambda}$ denote the preimage of 
$\sA^{\ch\lambda}$ under $\pi$. 

We show in Proposition~\ref{prop:YGr} below that $\sY$ is representable by a
scheme locally of finite type over $k$. This was predicted by 
Drinfeld \cite[Conjecture 4.2.3]{Drinfeld} in 
a more general setting. 

\begin{eg} \label{eg:YHecke}
Let $X = \mbb G_m \bs \GL_2$ where $\mbb G_m$ is embedded as $\left( \begin{smallmatrix} 1 & \\ & * \end{smallmatrix} \right)$. 
Then $\ch G_X = \ch G = \GL_2$ and $\ch\Lambda_X = \ch\Lambda_G = \mbb Z^2$ with 
standard basis $\ch\vareps_1,\ch\vareps_2$.
The $B$-orbits on $X$ are the same as $\mbb G_m$-orbits on $G/B=\mbb P^1$, so
there are three orbits: $\mbb G_m, \{0\}, \{\infty\}$. 
These correspond to $X^\circ$ and two colors $D^+, D^- \subset X$, respectively.
We have $\ch\nu_{D^+} = \ch\vareps_1,\, \ch\nu_{D^-} = -\ch\vareps_2$
and $\mf c_X = \mbb N^2$ is the free monoid generated by $\ch\nu_{D^+}, \ch\nu_{D^-}$. 
Note that $\ch\nu_{D^+} + \ch\nu_{D^-} = \ch\alpha$ is the simple coroot. 

Since $X$ is affine homogeneous, $\sM_X = \Bun_H = \Bun_1$ is the moduli stack
of line bundles. The Zastava model is $\sY=\Maps_\gen(C, \mbb G_m \bs \GL_2/B \supset \pt) 
= \Maps_\gen(C, \mbb G_m \bs \mbb P^1 \supset \pt)$. 
This is the stack parametrizing two line bundles $\Scr L, \Scr L'$ on $C$
and a fiberwise injective map of vector bundles $\sigma :\Scr L \into \Scr O_C \oplus \Scr L'$ 
such that \emph{both} coordinates $\sigma_1 : \Scr L \into \Scr O_C$ and $\sigma_2 : \Scr L \into \Scr L'$ are generically nonzero. 
Thus, $\sigma_1,\sigma_2$ are equivalent to two effective Cartier divisors 
$D_1, D_2$ which give $\Scr L = \Scr O(-D_1)$ and $\Scr L' = \Scr O(D_2-D_1)$. 
The condition that $\sigma=(\sigma_1,\sigma_2)$ is fiberwise injective is equivalent
to saying that the supports of $D_1,D_2$ are disjoint. 
Therefore, we have an identification
\[ \sY_{\mbb G_m \bs \GL_2} = \Sym C \oo\xt \Sym C = \bigsqcup_{(n_1,n_2)\in \mbb N^2} C^{(n_1)} \oo\xt C^{(n_2)}. \]
Meanwhile $\Scr A_{\mbb G_m \bs \GL_2} = \Sym C \xt \Sym C = \bigsqcup_{\mbb N^2} C^{(n_1)} \times C^{(n_2)}$. 
In this case $\sY \to \sA$ is the natural open embedding, and 
$\sY^{n_1\ch\nu_{D^+}+n_2\ch\nu_{D^-}} = C^{(n_1)} \oo\xt C^{(n_2)}$ is connected.
The map $\sY \to \sM_X$ forgetting the $B$-structure corresponds to 
$(D_1,D_2) \mapsto \Scr O(D_2-D_1)$. 

Note that for the embedding above of $\mbb G_m \into \GL_2$, 
the open $B$-orbit is the orbit of $\left( \begin{smallmatrix} 1 & 0 \\ 1 & 1 \end{smallmatrix}
\right) \in \mbb G_m \bs \GL_2$. In particular $X^\circ$ does not contain the identity coset. 
If we conjugate $\mbb G_m$ to the embedding $\left( \begin{smallmatrix} 1 & 0 \\ 1-a & a \end{smallmatrix}\right)$, then we may take the base point $x_0=1$. 
\end{eg}

\begin{eg} \label{eg:YHeckePGL}
In the example above we could instead replace $\GL_2$ by $G = \PGL_2$ and 
$H = \mbb G_m$ becomes the split torus inside $\PGL_2$. 
Then we still have $X = \mbb G_m\bs \PGL_2$ affine spherical with $\ch G_X = \ch G = \mathrm{SL}_2$ and two colors $D^+,D^-$. 
However now $\ch\Lambda_X = \ch\Lambda_G = \mbb Z \frac {\ch\alpha}2$ 
and $\ch\nu_{D^+} = \ch\nu_{D^-} = \frac {\ch\alpha}2$. 
The space $\Scr Y_{\mbb G_m \bs \PGL_2}$ still identifies with 
$\Sym C \oo\xt \Sym C$. However now $\Scr A_{\mbb G_m \bs \PGL_2} = \Sym C$
with $\Scr A^{n \frac{\ch\alpha}2} = C^{(n)}$. 
Then $\sY_{\mbb G_m \bs \PGL_2}^{n \frac{\ch\alpha}2} = \bigsqcup_{n_1+n_2=n} C^{(n_1)}\oo\xt C^{(n_2)}$ and the map $\sY \to \sA$ corresponds to addition of divisors.

Note that $\sM_X = \Bun_1$ and $C^{(n_1)}\oo\xt C^{(n_2)}$ maps to the component 
$\Bun_1^{n_2-n_1}$ of degree $n_2-n_1$ line bundles. 
Thus, while $\sY^{n \frac{\ch\alpha}2}$ is not connected, 
fixing a connected component of $\Bun_1$ determines the connected
component of $\sY^{n \frac{\ch\alpha}2}$.
\end{eg}

\subsection{Graded factorization property}
Note that our assumption on $X^\circ$ implies that $X^\circ/B=\pt$ is a dense open
substack of $X/B$. In the language of \cite[\S 4.2.1]{Drinfeld}, the stack $X/B$ is \emph{pointy}. Drinfeld observed that maps from a curve to a pointy stack will have local
behavior with respect to $C$ (compared to \cite{Drinfeld}, we are in the special
setting where we have a group $B$ acting on $X$, not just a groupoid).

Let $\ch\lambda = \ch\lambda_1 + \ch\lambda_2$ with $\ch\lambda_i \in \mathfrak c_X$ and let us denote by $\Scr A^{\ch\lambda_1} \oo\times \Scr A^{\ch\lambda_2}$ the open subset of the direct product $\Scr A^{\ch\lambda_1} \times \Scr A^{\ch\lambda_2}$ consisting of $\alpha_1,\alpha_2 : C \to (X/\!\!/N)/T$ such that
the supports of $C\sm \alpha_1^{-1}(\pt)$ and $C\sm \alpha_2^{-1}(\pt)$ are 
disjoint. 
We have a natural \'etale map $\Scr A^{\ch\lambda_1} \oo\times \Scr A^{\ch\lambda_2} \to \Scr A^{\ch\lambda}$.

\begin{prop} \label{prop:factorization}
The scheme $\Scr Y$ has the graded factorization property, in the sense
that there is a natural isomorphism 
\[ 
\Scr Y^{\ch\lambda}  \underset{\Scr A^{\ch\lambda}}\times 
(\Scr A^{\ch\lambda_1} \oo\xt \Scr A^{\ch\lambda_2})
\cong 
(\Scr Y^{\ch\lambda_1} \times \Scr Y^{\ch\lambda_2})|_{\Scr A^{\ch\lambda_1}\oo\times \Scr A^{\ch\lambda_2}}. 
\]
\end{prop}
\begin{proof}
Let $y_1, y_2 : C \times S \to X/B$ be $S$-points of $\Scr Y^{\ch\lambda_1}, \Scr Y^{\ch\lambda_2}$, respectively. 
Let $U_i = y_i^{-1}(\pt)\subset C\times S$; the condition that $(\pi(y_1),\pi(y_2))\in
\Scr A^{\ch\lambda_1} \oo\xt \Scr A^{\ch\lambda_2}$ is equivalent to 
requiring that $U_1 \cup U_2 = C\times S$. 
Then 
$y_1|_{U_1\cap U_2} \cong y_2|_{U_1\cap U_2} : U_1\cap U_2 \to \pt$
provides a gluing 
data for $y_1,y_2$ on the covering of $C\times S$ by $U_1$ and $U_2$.
Since $X/B$ is a stack, the gluing data descends to a map
$y : C\times S \to X/B$ that sends $U_1\cap U_2$ to $\pt$. 
This defines $y\in \Scr Y^{\ch\lambda}(S)$. 
The map in the opposite direction
is constructed in the same way and they are mutually inverse.
\end{proof}

We will henceforth use the notation 
$\Scr Y^{\ch\lambda_1} \oo\xt \Scr Y^{\ch\lambda_2}$ to denote
$(\Scr Y^{\ch\lambda_1} \times \Scr Y^{\ch\lambda_2})|_{\Scr A^{\ch\lambda_1}\oo\times \Scr A^{\ch\lambda_2}}$. 

\subsection{Global-to-Zastava yoga} 

We have a map $\Scr Y \to \Scr M_{X}$ by forgetting the $B$-structure. 
More precisely, we have an open embedding 
\begin{equation}
    \Scr Y \into \Scr M_{X} \underset{\Bun_G}\times \Bun_B.
\end{equation}

Although the natural map $\Bun_B \to \Bun_G$ is in general not smooth, 
it is smooth over a large enough open substack: 
consider $T$ as the Levi quotient of $B^-$. Let $\mf n^-$ denote the Lie algebra of $N^-$
viewed as a $T$-module. Define the open substack $\Bun^r_T \subset \Bun_T$ to consist
of those $T$-bundles $\Scr P_T$ for which 
$H^1(C, V \times^T \Scr P_T)=0$, for all $T$-modules $V$ which appear as
subquotients of $\mf n^-$. Let 
$\Bun_B^r$ be the preimage of $\Bun^r_T$ under the natural projection 
$\Bun_B \to \Bun_T$.

For $\ch\mu\in \ch\Lambda_T$, let $\Bun^{\ch\mu}_T$ denote the corresponding
connected component of $\Bun_T$ of degree $\ch\mu$, and 
let $\Bun^{\ch\mu}_B$ (resp.~$\Bun^{\ch\mu,r}_B$) be 
its preimage in $\Bun_B$ (resp.~$\Bun^r_B$). 
Note that by Riemann--Roch, $\Bun_B^{-\ch\mu,r} = \Bun_B^{-\ch\mu}$ 
if $\brac{\alpha_i,\ch\mu} > 2g-2$ for all simple roots $\alpha_i$, where $g$ is the genus of $C$. We say that $\ch\mu$ is ``large enough'' if $\brac{\alpha_i,\ch\mu}>N$
for all simple roots $\alpha_i$ and some $N \gg 0$.

\begin{lem}[{\cite{DS}, \cite[Lemma 3.7]{BFGM}, \cite[Lemma 14.2.1]{GN}}] 
\label{lem:BunBtoGsm}
The restriction of the map $\Bun_B \to \Bun_G$ to $\Bun^r_B$ is smooth. 
Any open substack $U\subset \Bun_G$ of finite type is contained in the image of 
$\Bun^{-\ch\mu}_B$ for $\ch\mu$ large enough, and the fibers of $\Bun^{-\ch\mu}_B\to \Bun_G$ over $U$ are geometrically connected. 
\end{lem}

Under our conventions, the composition $\Scr Y^{\ch\lambda} \to \Bun_B \to \Bun_T$ lands in the connected component $\Bun_T^{-\ch\lambda}$.

\begin{cor}\label{cor:Mcover}
(i) The map $\Scr Y^{\ch\mu} \to \Scr M_{X}$
is smooth with geometrically connected fibers (when nonempty) for $\ch\mu$ large enough.

(ii)
Any $k$-point of $\Scr M_{X}$ lies in the image of $\Scr Y^{\ch\mu}$ for all $\ch\mu$ in a translate of $\ch\Lambda^\pos_G$. 
\end{cor}
\begin{proof}
We have an open embedding $\Scr Y^{\ch\mu} \into \Scr M_{X} \xt_{\Bun_G} \Bun_B^{-\ch\mu}$, so (i) follows from Lemma~\ref{lem:BunBtoGsm} 
by base change. To show (ii), we consider the fiber of 
$\Scr Y \to \Scr M_{X}$ on $k$-points:

A point in $\Scr M_X(k)$ is equivalent to a datum 
$(\Scr P_G, \sigma : C\to X\xt^G \Scr P_G)$. 
First we show that there exists some point in $\sY(k)$ that maps to $f$. 
By \cite{Steinberg}, there exists an open subscheme $U\subset C$
on which $\Scr P_G|_U$ can be trivialized. 
If we fix such a trivialization, then $\sigma|_U$ 
identifies with a section $U \to H\bs G$. Since $H\bs G$ is
spherical, $\sigma(U)$ intersects the open orbit of a 
Borel subgroup $g B g^{-1} \subset G$ for some $g\in G(k)$. 
Then $g$ defines a point in  
$(G/B)(k) \subset (G/B)(\Bbbk)$, where $\Bbbk=k(C)$. 
Using our fixed trivialization
of $\Scr P_G|_U$, we get a section 
\[ \Spec \Bbbk \overset{g}\to \Spec \Bbbk \times G/B \cong \Scr P_G^0|_{\Spec \Bbbk} \xt^G (G/B) 
\cong \sP_G|_{\Spec \Bbbk} \xt^G G/B, 
\]
which extends to a section $C \to \Scr P_G \xt^G G/B$ since 
$G/B$ is proper. The latter is equivalent to giving a 
$B$-structure $\Scr P_B$ on $\Scr P_G$. 
By construction $(\Scr P_B,\sigma)$ satisfies the generic condition 
for it to lie in $\Scr Y(k)$, and $(\Scr P_B,\sigma)$ maps to $(\sP_G,\sigma) \in \sM_X(k)$. 

Next let us fix an arbitrary lift of $(\Scr P_G,\sigma)$ to $(\Scr P_B^1,\sigma)\in \Scr Y(k)$.
Fix a trivialization of $\Scr P_B^1|_{\Spec \Bbbk}$, which also specifies a 
trivialization $\Scr P_G|_{\Spec \Bbbk} \cong \Scr P_G^0|_{\Spec \Bbbk}$.
With respect to this trivialization, we have a bijection of sets
\begin{equation}\label{e:liftBstructure}
    H(\Bbbk)B(\Bbbk)/B(\Bbbk) \overset\sim\to \{\text{lift of } (\sP_G,\sigma) \text{ to a point in } \sY(k) \} 
\end{equation}
where the map is given by sending $h \in H(\Bbbk)B(\Bbbk)/B(\Bbbk)\subset (G/B)(\Bbbk)$ to 
the section
\[ \Spec \Bbbk \overset{h}\to \Spec \Bbbk \times G/B \cong \sP^0_G|_{\Spec \Bbbk} \xt^G(G/B)
\cong \sP_G|_{\Spec \Bbbk} \xt^G G/B \]
and uniquely extending to a section $C \to \sP_G \xt^G G/B$.
The point $B \in (G/B)(\Bbbk)$ is sent under \eqref{e:liftBstructure} to the lift $(\sP_B^1,\sigma)$.

We are concerned with the possible degrees of $\Scr P_B$ for 
lifts $(\Scr P_B,\sigma)\in \Scr Y(k)$ above $(\Scr P_G,\sigma)$.
For a simple root $\alpha$ of $G$, let $P_\alpha$ denote the corresponding minimal
parabolic in $G$. 
The image of $H\cap P_\alpha$ in $P_\alpha/\mf R(P_\alpha) = \PGL_2$ must contain 
a subgroup conjugate to $\left(\begin{smallmatrix} * & 0 \\ 0 & * \end{smallmatrix}\right)$
or $\left(\begin{smallmatrix} 1 & * \\ 0 & 1 \end{smallmatrix}\right)$.
Thus we can choose an identification $P_\alpha/B \cong \mbb P^1$ such that
the image of $(H\cap P_\alpha)(\Bbbk)$ in $\mbb P^1(\Bbbk)$ contains 
$\Bbbk^\times = k(C)^\times$. 
For $f \in (H\cap P_\alpha)(\Bbbk)$, let $\bar f$ denote its image in $(P_\alpha/B)(\Bbbk) \cong \mbb P^1(\Bbbk)$. 
For such a function $f$, let $(\sP_B,\sigma)$ denote the corresponding lift
of $(\sP_G,\sigma)$ under the bijection \eqref{e:liftBstructure}.
If $\bar f \in k(C)^\times$, then $\sP_B$ is of degree $-(\ch\mu_1 + N\ch\alpha)$, 
where $-\ch\mu_1$ is the degree of $\Scr P_B^1$ and 
$N$ is the degree of the divisor of zeros of $\bar f$.

By the Riemann--Roch theorem, we have a rational function $\bar f\in k(C)^\times$ 
with divisor of zeros of degree $N$ for any $N \gg 0$. 
Therefore for any $N\gg 0$, there is a lift $\Scr P_B$ of degree 
$-(\ch\mu_1+N\ch\alpha)$. 
We conclude that there exists a lift 
$(\Scr P_B,\sigma)\in \Scr Y^{\ch\mu}(k)$ for any $\ch\mu$ 
in a certain translate of $\ch\Lambda^\pos_G$.
\end{proof}

\subsubsection{} \label{sect:locglobyoga}
It is not in general true that the natural map 
$\Scr Y^{\ch\lambda} \to \Scr M_{X}$ is smooth for 
arbitrary $\ch\lambda$. However, the following 
well-known argument (cf.~\cite[Theorem 16.2.1]{GN}) 
shows that any neighborhood of a point in $\Scr Y^{\ch\lambda}$
is smooth locally isomorphic to a neighborhood of a point in 
$\Scr M_{X}$:

Let $\ch\lambda\in \mathfrak c_X$ be arbitrary. 
Since $\ch\Lambda^\pos_G \subset \mf c_X$, we can always
find $\ch\mu \in \ch\Lambda^\pos_G$ large enough such that 
$\ch\lambda+\ch\mu$ is also large enough. 
Let $\Scr Y^{\ch\mu,0}$ denote the preimage of $\Scr M_X^0=\Bun_H$ under the smooth map $\Scr Y^{\ch\mu} \to \Scr M_{X}$.
Then $\Scr Y^{\ch\mu,0}$ is smooth and   
the first projection
\[  \Scr Y^{\ch\lambda} \times \Scr Y^{\ch\mu,0} \to \Scr Y^{\ch\lambda} \]
is smooth.
On the other hand, by the graded factorization property (Proposition~\ref{prop:factorization}), there is a natural \'etale map 
\[ \Scr Y^{\ch\lambda} \oo\xt \Scr Y^{\ch\mu,0} 
\to \Scr Y^{\ch\lambda+\ch\mu}. 
\]
We can compose this with the smooth map $\Scr Y^{\ch\lambda+\ch\mu}\to \Scr M_X$ to get a smooth map $\Scr Y^{\ch\lambda} \oo\xt \Scr Y^{\ch\mu,0} \to \Scr M_X$.
To summarize, we have constructed:

\begin{lem} \label{lem:localglobalyoga}
For any $\ch\lambda\in \mathfrak c_X$ and any $\ch\mu \in \ch\Lambda_G^\pos$ large enough, there is a correspondence 
\[ 
\Scr Y^{\ch\lambda} \leftarrow 
\Scr Y^{\ch\lambda} \oo\times \Scr Y^{\ch\mu,0} \rightarrow 
\Scr M_{X} 
\]
where the left arrow is smooth surjective, and the right arrow is smooth.
\end{lem}

\subsection{Stratification of the Zastava model}
We stratify $\sY^{\ch\lambda}$ according to the fine stratification of $\sM_X$:
for a partition $\ch\Theta \in \Sym^\infty(\mf c_X^- \sm 0)$
and $\ch\lambda\in \mf c_X$,
define the stratum 
\[ \Scr Y^{\ch\lambda,\ch\Theta} := \Scr Y^{\ch\lambda}
\underset{\Scr M_X}\times \Scr M_X^{\ch\Theta}. \]
Abbreviate $\sY^{\ch\lambda,\ch\theta} := \sY^{\ch\lambda,[\ch\theta]}$.

\begin{prop} \label{prop:localstrata}
The stratum $\Scr Y^{\ch\lambda,\ch\Theta}$ is a smooth locally closed
subscheme of $\Scr Y^{\ch\lambda}$.
\end{prop}
\begin{proof}
By Lemma~\ref{lem:localglobalyoga}, there exists $\ch\mu\in\ch\Lambda_X$
such that there is a smooth correspondence 
\[ \Scr Y^{\ch\lambda} \leftarrow \Scr Y^{\ch\lambda}\oo\times \Scr Y^{\ch\mu,0} \to \Scr M_X \]
where the left arrow is surjective.
Note that by definition of $\Scr M_X^{\ch\Theta}$, the preimage of 
$\Scr M_X^{\ch\Theta}$ in $\Scr Y^{\ch\lambda}\oo\times\Scr Y^{\ch\mu,0}$
is isomorphic to $\Scr Y^{\ch\lambda,\ch\Theta} \oo\times \Scr Y^{\ch\mu,0}$
since $\Scr Y^{\ch\mu,0}$ consists
of maps $C\to X^\bullet/B$ which can only define points in $\msf L^0 X/\msf L^+ G$ 
upon restriction to $\mf o_v$ for any $v\in \abs C$.
Therefore, we get a smooth correspondence
\begin{equation} \label{e:localglobalcorresp}
    \Scr Y^{\ch\lambda,\ch\Theta} \leftarrow \Scr Y^{\ch\lambda,\ch\Theta} \oo\times \Scr Y^{\ch\mu,0} \to \Scr M_X^{\ch\Theta} 
\end{equation}
where the left arrow is still surjective. 
Now smoothness of $\Scr Y^{\ch\lambda,\ch\Theta}$ follows from smoothness
of $\Scr M_X^{\ch\Theta}$ (Lemma~\ref{lem:globstrata}).
\end{proof}

We call the collection of connected components of $\Scr Y^{\ch\lambda,\ch\Theta}$ the
\emph{fine stratification} of $\Scr Y^{\ch\lambda}$. By the smooth
correspondence \eqref{e:localglobalcorresp} above 
and Proposition~\ref{prop:globwhitney}, this is a Whitney stratification (in fact
the Zastava model is used in the proof of \emph{loc~cit.}).

Note that for a fixed $\ch\lambda$, many of the 
 $\Scr Y^{\ch\lambda,\ch\Theta}$ are empty.

\subsection{Relation to the affine Grassmannian} \label{sect:YGr} 

Let $\Gr_{G,\Sym C} \to \Sym C$ denote the following version of the Beilinson--Drinfeld affine Grassmannian: an $S$-point consists of a relative effective Cartier divisor $D \subset C\times S$ and a $G$-bundle $\Scr P_G$ on $C\times S$ together with a trivialization
$\Scr P_G|_{C\times S \sm D} \cong \Scr P_G^0|_{C\times S\sm D}$ where 
$\Scr P_G^0$ is the trivial $G$-bundle. 
For any linear algebraic group $G$, the functor $\Gr_{G,\Sym C}$ is representable by an ind-scheme, ind-of finite type
over $\Sym C$ (cf.~\cite{BD}, \cite[Theorem 3.1.3]{Xinwen}).

In particular, we can consider the ind-scheme 
$\Gr_{B,\Sym C}$. Let 
$\Gr_{B,C^{(N)}}$ denote the preimage over $C^{(N)}$. 

\subsubsection{} \label{sss:YGr}
Choose some $\delta \in \Lambda_X$ that lies on the interior of the
cone dual to $\Cal C_0(X)$, so $\brac{\delta,\ch\lambda}>0$
for any nonzero $\ch\lambda \in \Cal C_0(X)$.
Let $f_\delta \in k[X]^{(B)}$ denote the corresponding $\delta$-eigenfunction. 
Then $f_\delta$ induces a map $X/\!\!/N \to \mbb A^1$, which 
in turn induces a map $\Scr A \to \Sym C$ sending 
$\Scr A^{\ch\lambda} \to C^{(\brac{\delta,\ch\lambda})}$. 
We can map\footnote{
The effective Cartier divisor cut out by $f_\delta$ 
has the property that the complement of its support in $X$ equals $X^\circ$. 
In this guise, the map $\sY \to \Sym C$ we have constructed coincides with
the one described in \cite[Remark 4.2.6]{Drinfeld}.}

\begin{equation} \label{e:YGr}
    \Scr Y^{\ch\lambda} \to \Gr_{B,C^{(\brac{\delta,\ch\lambda})}} 
\end{equation}
as follows:
let $y : C\times S \to X/B$ be an $S$-point of $\Scr Y$. 
Let $D \subset C \times S$ be the relative effective divisor corresponding
to the image of $\pi(y)\in \Scr A(S)$ under the map $\Scr A \to \Sym C$.
Since $\delta$ was chosen in the interior of the dual cone of $\Cal C_0(X)$, we have
$C\times S \sm D = y^{-1}(\pt) =: U$. 
Thus, $y$ defines a $B$-bundle $\sP_B$ on $C \xt S$ together with 
a section $U \to X^\circ \xt^B \sP_B \cong \sP_B$, i.e., a trivialization of 
$\sP_B|_U$. 
The datum $(D, \Scr P_B, \Scr P_B|_U \cong \Scr P_B^0|_U)$ defines an $S$-point 
of $\Gr_{B,\Sym C}$.

\begin{prop} \label{prop:YGr}
The map \eqref{e:YGr} is a closed embedding. 
Moreover, the stack $\Scr Y$ is representable by a scheme locally of finite type over $k$, and for fixed $\ch\lambda\in \mathfrak c_X$, the scheme 
$\Scr Y^{\ch\lambda}$ is of finite type over $k$. 
Consequently, the map $\pi : \Scr Y \to \Scr A$ is schematic of finite type. 
\end{prop}
\begin{proof}
First we show that \eqref{e:YGr} is a closed embedding.
Fix a map $S \to \Gr_{B,\Sym C}$ corresponding to a pair 
$(D,\Scr P_B)$ and a trivialization of $\Scr P_B$ away from $D$. 
A trivialization of $\Scr P_B$ on $C\times S\sm D$ is equivalent to a section 
$\sigma_0:C \times S\sm D \to \Scr P_B \cong X^\circ \xt^B \sP_B$. 
The fiber of $\Scr Y \to \Gr_{B,\Sym C}$ over $S$
 parametrizes commutative diagrams 
\[ 
\begin{tikzcd}
C \times S' \sm D' \ar[d,hook] \ar[r, "{\sigma_0}"] & X^\circ \overset B\times\Scr P_B \ar[d, hook]  \\ 
C \times S' \ar[r, "\sigma"] & X \overset B\times \Scr P_B  
\end{tikzcd}
\]
where $S'$ is an $S$-scheme, $D':= D\times_S S'$, and $\sigma$ is a 
section over $C\times S$. 
Observe that $\sigma$ is uniquely determined by $\sigma_0$. 
By Lemma~\ref{lem:extendsection} below, the condition that 
$\sigma_0$ extends to $\sigma$ is closed in $S$. 

We have shown that $\Scr Y$ is representable by an ind-scheme
ind-closed in $\Gr_{B,\Sym C}$. On the other hand, we have a map
$\Scr Y \to \Bun_B$ whose fiber over an $S$-point $\Scr P_B \in \Bun_B(S)$
is open in the space of sections $C\times S \to X \times^B \Scr P_B$ over $C\times S$.
This space is representable by a scheme locally of finite type over $k$ (\cite[Theorem 5.23]{FGA-explained}). Since $\Bun_B$ is an algebraic stack, we conclude that 
$\Scr Y$ is an algebraic stack representable by an ind-scheme. Hence 
$\Scr Y$ is representable by a scheme. 
For fixed $\ch\lambda \in \mathfrak c_X$, we now know that $\Scr Y^{\ch\lambda}$ is a closed subscheme of $\Gr_{B,C^{(\brac{\delta,\ch\lambda})}}$, which is ind-of finite type. It follows that $\Scr Y^{\ch\lambda}$ is of finite type over $k$. The other assertions all follow.  
\end{proof}

\begin{lem} \label{lem:extendsection} 
Let $S$ be a test scheme and $D \subset C\xt S$ a relative effective divisor.
Let $\Scr X$ be a scheme affine of finite presentation over $C\xt S$.
Suppose that there exists a section 
$\sigma : C\xt S \sm D \to \Scr X$ over $C\xt S$. 
Then the functor sending $S'$ to the set of 
maps $S'\to S$ such that $\sigma$ extends to a regular map on $C\xt S'$
is representable by a closed subscheme of $S$. 
\end{lem}
\begin{proof} The map $\sigma$ is equivalent to a map of 
$\Scr O_{C\xt S}$-algebras
$\Scr O_{\Scr X} \to \Scr O_{C\xt S\sm D}$. 
Given $S' \to S$, the condition that 
$\sigma$ extends to $C\xt S'$ is equivalent to 
requiring the image of 
$\Scr O_{\Scr X} \to \Scr O_{C\times S\sm D}$ to land in $\Scr O_{C\times S'} \subset \Scr O_{C\times S' \sm D'}$
after base change to $S'$, where $D':= D\xt_S S'$. 
The claim is local in $S$, so we may assume that
$\Scr O_{\Scr X}$ is surjected onto by 
$\Sym_{\Scr O_{C\xt S}}(\Scr E^\vee)$ for some
vector bundle $\Scr E$ on $C\xt S$. 
Then we just need the composed $\Scr O_{C\times S}$-linear map $\Scr E^\vee \to \Scr O_{C\times S \sm D} \to \Scr O_{C\times S\sm D}/\Scr O_{C\times S}$ to vanish after base change to $S'$.
Since $\Scr E^\vee$ is coherent, the image of this map 
is contained in a submodule $\Scr F = \Scr O(m\cdot D)/\Scr O_{C\times S}$
for some integer $m \ge 0$.  
The projections $p: C\times S \to S,\, p':C\times S'\to S'$ are proper, and we are considering
when an element in $H^0 p'_*( \Scr E \otimes_{\Scr O_{C\times S}} \Scr F \otimes_{\Scr O_S} \Scr O_{S'})$ vanishes. 
Note that $p_*(\Scr E \otimes_{\Scr O_{C\times S}} \Scr F)$ is finite locally free
as an $\Scr O_S$-module.
By cohomology and base change 
we are reduced to asking when an element of $p_*(\Scr E \otimes_{\Scr O} \Scr F)$
vanishes after base change to $S'$. 
This is a closed condition on $S$.
\end{proof}

\begin{rem}[Open curves] \label{rem:opencurve}
The graded factorization property of $\Scr Y^{\ch\lambda}$ implies that 
the geometry of $\Scr Y^{\ch\lambda}$ is purely local
with respect to the curve $C$. Therefore, we could define 
\[ \Scr Y(C) = \Maps_\gen(C, X/B \supset \pt) \]
for any smooth curve $C$ (not necessarily projective), and all 
the same properties would still hold.
For example, in \cite{FM}, \cite{Drinfeld}, the affine curve $C = \mbb A^1$ is used.
\end{rem}

\subsubsection{Beauville--Laszlo's theorem} \label{sect:BLthm}
Let $S$ be an affine scheme and $D$ a closed affine subscheme of $C\times S$.
Denote by $\wh C_D$ the formal completion of $C\times S$ along $D$ and by
$\wh C'_D$ the spectrum of the ring of regular functions on $\wh C_D$ (so $\wh C_D$ is an ind-affine formal scheme and $\wh C'_D$ is the corresponding true scheme). 
Let $\wh C^\circ_D := \wh C'_D \sm D$ denote the open subscheme.

There are maps $\hat p: \wh C_D \to C\times S$ and $i: \wh C_D \to \wh C'_D$. 
We will implicitly use the following fact in what follows:

\begin{prop}[{\cite[Proposition 2.12.6]{BD}}] \label{prop:formalext}
There exists a unique map 
$p : \wh C'_D \to C\times S$ such that $\hat p = p\circ i$. 
\end{prop}

To justify that $\Scr Y$ is of local nature, we record 
the following consequence of the globalized version of Beauville--Laszlo's theorem (cf.~\cite[Theorem 2.12.1]{BD}, \cite{BLa}). 
Let $S$ be an affine scheme and $D\subset C\times S$ a relative effective divisor.  
Proposition~\ref{prop:formalext} implies that there exists a map $p:\wh C'_D\to C\times S$. 

\begin{lem} \label{lem:B-L}
Let $X$ be any affine scheme with an action of an algebraic group $B$
such that $X/B$ is pointy, i.e., $X$ has an open $B$-orbit $X^\circ$ with $X^\circ/B=\pt$.
Let $C,S$ and $D$ be as above. 

Then there is a natural equivalence between the following categories:
\begin{enumerate}
\item the groupoid of $(\Scr P_B, \sigma)$ where $\Scr P_B$ is a $B$-bundle on $C\times S$ and a section $\sigma : C\times S \to X\times^B \Scr P_B$ 
that sends $C\times S\sm D$ to $X^\circ \times^B \Scr P_B$,
\item the groupoid of $(\hat{\Scr P}_B, \hat \sigma)$ where $\hat{\Scr P}_B$ is a $B$-bundle on $\wh C'_D$ and a section $\hat\sigma: \wh C'_D \to X\times^B \hat{\Scr P}_B$ that sends $\wh C^\circ_D$ to $X^\circ \times^B \hat{\Scr P}_B$.
\end{enumerate}
\end{lem}

\begin{proof}
The functor from (i) to (ii) is just pullback along $p$. 
To define the functor from (ii) to (i) we descend along the covering 
$\wh C'_D \sqcup (C\times S \sm D) \to C\times S$. The justification
for this is Beauville--Laszlo's theorem (cf.~\cite[Theorem 2.12.1]{BD}). 
First, $\hat \sigma$ induces 
a section $\wh C^\circ_D \to X^\circ \times^B \hat{\Scr P}_B \cong \hat{\Scr P}_B$.
Then we can ``descend'' $\hat{\Scr P}_B$ to a $B$-bundle $\Scr P_B$ on $C\times S$ with a section $C\times S\sm D \to \Scr P_B$ (which
is equivalent to a trivialization of $\Scr P_B|_{C\times S\sm D}$).
The section $\hat\sigma$ is equivalent to a map of quasicoherent
$\Scr O_{\wh C'_D}$-algebras $\Scr O_{X\times^B \hat{\Scr P}_B} \to \Scr O_{\wh C'_D}$ since $X$ is affine. 
Again by \cite[Theorem 2.12.1]{BD}, this descends 
to a map of quasicoherent $\Scr O_{C\times S}$-algebras $\Scr O_{X\times^B \Scr P_B} \to \Scr O_{C\times S}$ such that the restriction to 
$C\times S\sm D$ factors through the trivialization 
$\Scr O_{X^\circ \times^B \Scr P_B}\cong \Scr O_{\Scr P_B} \to \Scr O_{C\times S\sm D}$.
By construction, the two functors are mutually inverse.
\end{proof}

\subsection{Theorem of Grinberg--Kazhdan, Drinfeld} 

We now justify why $\sM_X$ and $\sY$ are indeed ``models'' for the formal arc space
$\msf L^+ X$. Since $\sM_X$ and $\sY$ are smooth-locally isomorphic,
it suffices to explain the latter. 

\begin{defn}\label{def:formalmodel} 
A finite type formal model of $\msf L^+ X$ at $\gamma_0\in \msf L^+ X(k)$ is the formal completion $\wh{Y}_y$ of a $k$-scheme of finite type $Y$ at a point $y\in Y$ equipped with an isomorphism of formal schemes
\begin{equation} \label{finite model}
\wh{\msf L^+ X}_{\gamma_0} \simeq \wh{Y}_y \times \wh{\mbb A}^\infty,
\end{equation}
where $\wh{\mbb A}^\infty$ is the product of countably many copies of the formal disk $\on{Spf} k\tbrac t$.
\end{defn}

Since $X/B \supset \pt$ is a pointy stack, 
Drinfeld's proof \cite[\S 4.2-4.3]{Drinfeld} of the Grinberg--Kazhdan theorem 
essentially shows that the scheme $\sY$ (which we have shown is a disjoint union
of finite type schemes) explicitly satisfies the following:

\begin{thm}[Grinberg--Kazhdan, Drinfeld]  \label{thm:DGK}
Fix an arc $\gamma_0 : \Spec k\tbrac t \to X$ in $\msf L^+ X(k)$
such that $\gamma_0(\Spec k\lbrac t) \subset X^\circ$. 
Then there exists a point $y \in \Scr Y(k)$ 
such that the formal completion of $\Scr Y$ at $y$ is a finite type formal model of $\msf L^+ X$ at $\gamma_0$. 

More precisely, if $\gamma_0$ belongs to the stratum $\msf L^{\ch\theta} X$, for $\ch\theta\in \mf c_X^-$, we can take $y$ to be the point $t^{\ch\theta}$ in the central fiber $\msf Y^{\ch\theta,\ch\theta}$ over any point $v\in \abs C$.
\end{thm}

The central fiber $\msf Y^{\ch\theta}$ 
over a point $v\in \abs C$ is the fiber of $\sY^{\ch\theta}\to \sA^{\ch\theta}$ over the 
``diagonal divisor'' $\ch\theta \cdot v \in \sA^{\ch\theta}(k)$.
We define $\msf Y^{\ch\theta,\ch\theta} := \sY^{\ch\theta} \xt_{\sM_X} \sM_X^{\ch\theta}$. 
Then $\msf Y^{\ch\theta}$ is naturally a subvariety of $\Gr_B$ (see \S\ref{def:centralfiber}), 
and $t^{\ch\theta}$ denotes the corresponding point in $\Gr_B$. The significance of
this point will become evident in Corollary~\ref{cor:Ytheta}(iii).

\smallskip

The statements all follow from the proof of \cite[\S 4]{Drinfeld}. 
We also give the same argument, with some notational changes, in the proof of Theorem~\ref{thm:DGKfamily}.

\begin{rem}
The point $y : C \to X/B$ can be chosen so that $y^{-1}(\pt) = C\sm v$ for a single
point $v\in \abs C$. However it is essential, for the theorem to hold, that $\sY$ contains
maps with multiple points of $C$ mapping to $(X\sm X^\circ)/B$. 
\end{rem}

\section{Compactification of the Zastava model} \label{sect:compactification}

The map $\pi : \Scr Y \to \Scr A$ defined in \eqref{mappi}
is in general not proper, so for example
we cannot apply the decomposition theorem. To rectify this, we introduce
a compactification.

\subsection{Basic properties} \label{sect:basicproperties-comp}

Let $\ol{G/N} = \Spec k[G/N]$ denote the canonical affine closure of the 
quasiaffine variety $G/N$. 
For an arbitrary connected reductive group $G$, Drinfeld's compactification 
$\ol\Bun_B$ is defined\footnote{The definition as a closure is only true in characteristic $0$. In positive characteristic, see \cite[\S 4.1]{ABB}, \cite[\S 7.2]{Sch}.}
 as the closure of $\Bun_B$ inside 
\[ \Maps_\gen(C, G \bs \ol{G/N}/T \supset \pt/B), \] 
the stack parametrizing maps $C \to G\bs \ol{G/N}/T$ that generically land in 
the open substack $G \bs (G/N)/T = \pt/B$.
(When $[G,G]$ is simply connected, 
\cite[Proposition 1.2.3]{BG} show that $\Bun_B$ is dense in 
$\Maps_\gen(C, G\bs \ol{G/N}/T \supset \pt/B)$.)

Consider the \emph{stack} quotient $X \xt^G \ol{G/N} := (X\times \ol{G/N})/G$, where $G$ acts 
anti-diagonally. Then $X/N = X\times^G G/N$ is an open substack of 
$X \xt^G \ol{G/N}$, so we also have the open substack 
$\pt = X^\circ/B \subset X/B \subset X \xt^G \ol{G/N}/T$. 
Define 
\begin{equation} \label{e:defbarY}
    \barY = (\sM_X \xt_{\Bun_G} \ol \Bun_B)^\circ \subset 
    \Maps_\gen( C, X \xt^G \ol{G/N}/T \supset \pt)  
\end{equation}
where the superscript $^\circ$ denotes the open substack of 
\[ \sM_X \xt_{\Bun_G} \ol \Bun_B \subset \Maps_\gen( C, X \xt^G \ol{G/N}/T \supset 
X^\bullet / B) \] 
parametrizing maps generically landing in $\pt = X^\circ/B$. 
In particular, $\barY$ is an algebraic stack locally of finite type. 
We can identify
$\Scr Y \cong \barY \xt_{\ol\Bun_B} \Bun_B$ as an open substack of $\barY$.
(If $[G,G]$ is simply connected, the containment in \eqref{e:defbarY} is an 
equality.)

There is a natural map from $X\xt^G \ol{G/N}$ to 
\[ (X \times \ol{G/N})/\!\!/G = \Spec k[X \times G]^{G \times N}. \]
Since $k[X \times G]^G = k[X]$, we deduce that $(X \times \ol{G/N})/\!\!/G = X/\!\!/N$.
Therefore, we have a map 
$X \xt^G \ol{G/N} \to X/\!\!/N$ extending the natural map $X/N \to X/\!\!/N$.
Applying $\Maps(C,?/T)$ to the former, we have constructed a map 
\begin{equation}
    \bar\pi : \barY \to \Scr A 
\end{equation}
extending $\pi : \Scr Y \to \Scr A$.
Let $\barY^{\ch\lambda}$ denote the preimage of the subscheme $\Scr A^{\ch\lambda}$. 
The same proof as in Proposition~\ref{prop:factorization} 
shows that $\barY$ has the graded factorization property. 
We will see from Lemma~\ref{lem:barYGr} below that $\barY^{\ch\lambda}$ is 
representable by a scheme of finite type over $k$.

\medskip

First, we show that $\bar\pi$ is indeed a compactification:

\begin{prop} \label{prop:piproper}
The map $\bar \pi :\barY\to \Scr A$ is proper.
\end{prop}

We proceed as in \S\ref{sss:YGr}. 
Choose $\delta\in \Lambda_X$ lying on the interior of the cone 
dual to $\Cal C_0(X)$.
Then $\delta$ defines a map $X/\!\!/N \to \mbb A^1$, which induces a map 
$\Scr A \to \Sym C$. 
This allows us to consider $\Scr Y$ as a scheme over $\Sym C$.
Proposition~\ref{prop:YGr} gives
a closed embedding $\Scr Y \into \Gr_{B,\Sym C}$ over $\Sym C$. 
We compose this with the natural map $\Gr_{B,\Sym C} \to \Gr_{G,\Sym C}$
to get a map $\Scr Y \to \Gr_{G,\Sym C}$. 
We can extend this to a map 
\begin{equation} \label{e:barYGr}
\barY \to \Gr_{G,\Sym C} 
\end{equation}
using the same idea as in the definition of \eqref{e:YGr}. Namely, 
let $y: C\times S \to X \times^G \ol{G/N} / T$ be an $S$-point of $\barY$
and let $D\subset C\times S$ denote the divisor it maps to. 
In particular, $y$ defines a $G$-bundle $\Scr P_G$ on $C\times S$ 
with a $B$-reduction on $C\times S\sm D$. Since $y(C\times S\sm D) = X^\circ/B =\pt$,
the $G$-bundle in fact admits a trivialization on $C\times S\sm D$. 
The data of $D, \Scr P_G$, and the trivialization defines an $S$-point of $\Gr_{G,\Sym C}$.

\begin{lem} \label{lem:barYGr}
The map $\barY \to \Gr_{G,\Sym C} \underset{\Sym C}\times \Scr A$
is a closed embedding.
\end{lem}

Since $\Gr_{G,\Sym C}$ is ind-proper over $\Sym C$ (cf.~\cite[Remark 3.1.4]{Xinwen}), 
Proposition~\ref{prop:piproper} follows from Lemma~\ref{lem:barYGr}.
We also deduce from the lemma that $\barY$ is representable by a scheme, and 
each $\barY^{\ch\lambda}$ is of finite type.

\begin{proof}
The discussion above really defines a map 
\begin{equation} \label{e:biggeremb}
    \Maps_\gen(C, X \xt^G \ol{G/N}/T \supset \pt) \to \Gr_{G,\Sym C} \xt_{\Sym C} \sA, 
\end{equation}
and $\barY \into \Maps_\gen(C, X\xt^G \ol{G/N}/T \supset \pt)$ is a closed
embedding. Thus, it suffices to prove that \eqref{e:biggeremb} is a closed
embedding.  
Fix a test scheme $S$. 
Let $\Scr P^0_G,\Scr P^0_B,\Scr P^0_T$ denote the respective trivial bundles on $C\times S$. 
An $S$-point of $\Gr_{G,\Sym C} \times_{\Sym C} \Scr A$
consists of the data 
\[ (\Scr P_G,\Scr P_T,D,\tau,\alpha)  \]
with $\Scr P_G\in \Bun_G(S),\, \Scr P_T \in \Bun_T(S),\, D\in \Sym C(S)$,
a trivialization $\tau : \Scr P^0_G|_{C\times S\sm D}\cong \Scr P_G|_{C\times S\sm D}$, 
and a section
$\alpha: C\times S \to (X/\!\!/N)\times^T \Scr P_T$ such that 
$(\Scr P_T,\alpha) \in \Scr A(S)$ maps to $D$. 
In particular, this means that $\alpha$ induces a trivialization 
$\Scr P_T^0|_{C\times S\sm D} \cong \Scr P_T|_{C\times S\sm D}$. 
Using the identification $X^\circ \cong B$, we get a section
\[ \sigma_0 : C\times S\sm D \to \Scr P^0_B \cong X^\circ \xt^B \Scr P_B^0 
\into X \xt^B \Scr P_B^0= X \xt^G \Scr P_G^0. \]
The composition $\sigma := \tau \circ \sigma_0$ then defines a section
$C\times S \sm D \to X\xt^G \Scr P_G$.
On the other hand, the trivial $B$-bundle also corresponds to a section
$\kappa_0 : C\times S \sm D \to \Scr P_G^0 \xt^G (G/N)$.
Composing $\kappa_0$ with the trivialization $\alpha:\Scr P_T^0|_{C\times S\sm D} \cong \Scr P_T|_{C\times S\sm D}$, we get a section
\[ \kappa : C\times S \sm D \to \Scr P_G \xt^G \ol{G/N} \xt^T \Scr P_T. \]
The datum $(\Scr P_G, \Scr P_T, \sigma, \kappa)$ defines an $S$-point 
of $\Maps(C, X\xt^G \ol{G/N}/T)$ if and only if $\sigma,\kappa$ both extend to regular
maps on $C\times S$. 
Therefore, the fiber of our chosen $S$-point over 
the map \eqref{e:biggeremb}
parametrizes maps $S' \to S$
such that the base change of $\sigma,\kappa$ to $S'$ both extend to 
$C\times S'$. By Lemma~\ref{lem:extendsection}, 
this fiber is represented by a closed subscheme of $S$ (here the key point 
is that both $X$ and $\ol{G/N}$ are affine). 
\end{proof}

\begin{rem}
We need the extra factor of $\Scr A$ in Lemma~\ref{lem:barYGr} which was
not present in Proposition~\ref{prop:YGr} because the map $\Gr_B \to \Gr_G$
is a bijection on $k$-points but far from an isomorphism of ind-schemes. 
Since $\Scr A$ embeds into $\Gr_{T,\Sym C}$, the lemma is really
embedding $\barY$ into $\Gr_{G,\Sym C}\times_{\Sym C} \Gr_{T,\Sym C}$. 
The ind-scheme $\Gr_T$ is highly non-reduced, while $(\Gr_T)_\red$ is a disjoint
union of points.
\end{rem}

\begin{eg} \label{eg:Heckecompact}
Let $X = \mbb G_m \bs \GL_2$ as in Example~\ref{eg:YHecke}. 
Then $\barY = \Sym C \xt \Sym C$ and $\barY \to \sA$ is the identity morphism. 
\end{eg}

\subsection{Stratification} 
\label{sect:strat-compactified}
We can uniquely write any $\ch\nu \in \ch\Lambda^\pos_G$ 
as a sum $\ch\nu = \sum_{\alpha\in \Delta_G} n_\alpha \ch\alpha$ where $\Delta_G$ is the set of 
simple coroots and $n_\alpha$ are positive integers. 
Let $C_{\ch\nu} := \prod_{\Delta_G} C^{(n_\alpha)}$ denote the corresponding partially
symmetrized power of $C$. 
Recall that Drinfeld's 
compactification $\ol\Bun_B$ of $\Bun_B$ has a 
stratification by \emph{defect}, where the strata are given by 
locally closed embeddings
\[ \mf i_{\ch\nu}: C_{\ch\nu} \times \Bun^{\ch\mu+\ch\nu}_B \into \ol\Bun^{\ch\mu}_B, \]
for $\ch\nu\in \ch\Lambda^\pos_G,\, \ch\mu \in \ch\Lambda_G$, cf.~\cite[\S 1.5, p.~7]{BFGM}. 
Define the substack ${_{\ch\nu}}\ol\Bun_B^{\ch\mu}$ to be the image 
of the corresponding embedding. 
We obtain an open substack $_{\le \ch\nu}\ol\Bun_B^{\ch\mu} \subset \ol\Bun^{\ch\mu}_B$ by taking the union of the strata $_{\ch\nu'}\ol\Bun_B^{\ch\mu}$ 
for all $\ch\nu' \le \ch\nu$. 

Since $\barY^{\ch\lambda}$ maps to $\ol\Bun_B^{-\ch\lambda}$ for $\ch\lambda \in \mathfrak c_X$, by base change we have locally
closed subschemes 
\[ {}_{\ch\nu} \barY^{\ch\lambda} := \barY^{\ch\lambda} \xt_{\ol\Bun_B^{-\ch\lambda}} {}_{\ch\nu} \ol\Bun_B^{-\ch\lambda} \into \barY^{\ch\lambda} \]
and open subschemes ${}_{\le \ch\nu} \barY^{\ch\lambda} \into \barY^{\ch\lambda}$ defined analogously.
Observe that the identification ${}_{\ch\nu} \ol\Bun_B^{-\ch\lambda} \cong 
C_{\ch\nu} \times \Bun_B^{\ch\nu-\ch\lambda}$ induces an isomorphism
\begin{equation} \label{e:barYstrata1}
    {}_{\ch\nu} \barY^{\ch\lambda}  \cong C_{\ch\nu} \times \Scr Y^{\ch\lambda-\ch\nu}. 
\end{equation}
On the other hand, $\barY^{\ch\lambda}$ also maps to $\Scr M_X$, so 
we get a locally closed subscheme 
\[ {}_{\ch\nu}\barY^{\ch\lambda,\ch\Theta} := {}_{\ch\nu}\barY^{\ch\lambda} \xt_{\Scr M_X} \Scr M_X^{\ch\Theta} \into {}_{\ch\nu}\barY^{\ch\lambda} \into \barY^{\ch\lambda} \] 
for $\ch\Theta$ any partition in $\mathfrak c_X^-$ 
(by Lemma~\ref{lem:globstrata}). 
We deduce from \eqref{e:barYstrata1} that there is an isomorphism
\[ {}_{\ch\nu} \barY^{\ch\lambda,\ch\Theta} \cong C_{\ch\nu} \times \Scr Y^{\ch\lambda-\ch\nu, \ch\Theta}. \] 
In particular, Proposition~\ref{prop:localstrata} implies that 
${}_{\ch\nu} \barY^{\ch\lambda,\ch\Theta}$ is smooth. 
In summary: 

\begin{prop}  \label{prop:barYstrata}
The collection of locally closed subschemes 
${}_{\ch\nu} \barY^{\ch\lambda,\ch\Theta}$, ranging over all $\ch\nu\in \ch\Lambda^\pos_G$ and partitions $\ch\Theta \in \Sym^\infty(\mathfrak c_X^-\sm 0)$, 
forms a smooth stratification of $\barY^{\ch\lambda}$. 
\end{prop}

Note that for fixed $\ch\lambda$, many of these strata may be empty.

\subsubsection{Changing the curve}

In this subsection we let $C$ be a smooth curve which is not necessarily proper 
and we define 
\[ \sA(C) = \Hom( \mf c_X^\vee, \Sym C), \quad 
\sY(C)=\Maps_\gen(C, X/B\supset \pt) \]
and $\barY(C)$ to be the closure
of $\sY(C)$ in $\Maps_\gen(C, X \xt^G \ol{G/N}/T \supset \pt)$ to emphasize
the curve being used. 
Here $\Hom$ denotes homomorphisms of monoid objects in the category of schemes.
Similarly we have $\sA^{\ch\lambda}(C), \sY^{\ch\lambda}(C),\barY^{\ch\lambda}(C)$. 
The local nature of $\barY(C)$ 
(in particular Lemma~\ref{lem:barYGr}) ensures that $\barY^{\ch\lambda}(C)$ is still 
a finite type $k$-scheme. 

\smallskip

Let $p: \wt C \to C$ be an \'etale map of smooth curves. 
Let $(\Sym \wt C)_\disj \subset \Sym \wt C$
be the open subset that consists of divisors $\wt D$ on $\wt C$ such that
any fiber of $p$ contains at most one point of the support of $\wt D$.
Let $\sA(\wt C)_\disj$ denote the open subset of 
$\Hom(\mf c_X^\vee, \Sym \wt C)$ consisting of homomorphisms landing in $(\Sym \wt C)_\disj$.
Then pushforward of divisors defines a map 
$\sA^{\ch\lambda}(\wt C) \to \sA^{\ch\lambda}(C)$.

\begin{prop}[{\cite[Proposition 2.19]{BFG}}] \label{prop:changecurve}
For an \'etale map $\wt C \to C$ we have a canonical isomorphism 
\[ \barY^{\ch\lambda}(C) \xt_{\sA^{\ch\lambda}(C)} \sA^{\ch\lambda}(\wt C)_\disj 
\cong \barY^{\ch\lambda}(\wt C) \xt_{\sA^{\ch\lambda}(\wt C)} \sA^{\ch\lambda}(\wt C)_\disj 
\] 
which preserves the fine stratification. 
\end{prop}
Note that setting $\wt C = C\sqcup C$ recovers the graded factorization property. 

\begin{proof}
Choose $\delta \in \mf c_X^\vee$ as in \S\ref{sss:YGr}. 
For $(y,\tilde a) \in \barY^{\ch\lambda}(C) \xt_{\sA^{\ch\lambda}(C)} \sA^{\ch\lambda}(\wt C)_\disj$, let $\wt D$ (resp.~$D$) denote the divisor corresponding to $\delta$ paired with 
$\tilde a$ (resp.~$\pi(y)\in \sA^{\ch\lambda}(C)$). 
Then $p_*\wt D = D$ and we deduce that there is an isomorphism 
$\wh C'_D \cong \wh{\wt C}{}'_{\wt D}$.
Lemma~\ref{lem:B-L}, applied to
the affine $G\xt T$-scheme $X \xt \ol{G/N}$, implies that the point 
$y\in \barY^{\ch\lambda}(C)$ is 
equivalent to its restriction $y|_{\wh C'_D}$. 
Applying the same lemma again shows that $y|_{\wh C'_D}$ is equivalent to 
a point $\tilde y \in \barY^{\ch\lambda}(\wt C)$ such that $\tilde y(\wt C \sm \wt D)=\pt$.
This defines mutually inverse maps in both directions and the compatibility 
with strata is clear.
\end{proof}

Since the diagonal $\delta^{\ch\lambda} : \wt C \into \sA^{\ch\lambda}(\wt C)$ 
sending $\tilde v \mapsto \ch\lambda\cdot \tilde v$ is contained in $\sA^{\ch\lambda}(\wt C)_\disj$, we deduce from the graded factorization property
and Proposition~\ref{prop:changecurve} that $\barY^{\ch\lambda}(C)$ 
is \'etale-locally isomorphic to $\barY^{\ch\lambda}(\bbA^1)$ for any smooth curve $C$.

\subsection{The central fiber} \label{def:centralfiber}
Let $\ch\lambda \in \mathfrak c_X$. There is a 
diagonal map $\delta^{\ch\lambda}:C \to \Scr A^{\ch\lambda}$
sending $v \mapsto \ch\lambda \cdot v$. For a fixed point
$v\in \abs C$, let us consider 
$\delta^{\ch\lambda}_v : v \to \Scr A^{\ch\lambda}$. 

\medskip

Define the central fiber $\msf Y^{\ch\lambda}$ of $\Scr Y^{\ch\lambda}$ 
to be the
preimage of $\delta^{\ch\lambda}_v$ under 
$\pi : \Scr Y^{\ch\lambda} \to \Scr A^{\ch\lambda}$. 
If we take central fibers of the map \eqref{e:YGr}, Proposition~\ref{prop:YGr} implies that we have a closed embedding 
$\msf Y^{\ch\lambda} \into \Gr_B$.
In fact, $\ch\lambda \cdot v$ can be considered as a point in 
$\Gr_T$. If we let $\msf S^{\ch\lambda}$ denote the preimage of this point
under the projection $\Gr_B \to \Gr_T$, then we get a closed embedding 
\begin{equation} \label{e:YembedGr}
    \msf Y^{\ch\lambda} \into \msf S^{\ch\lambda}.
\end{equation}

Note that $\msf S^{\ch\lambda}$ identifies with 
the $\msf L N$-orbit of $t^{\ch\lambda}$ in $\Gr_G$. 
The orbits $\msf S^{\ch\lambda}$ are commonly known
as the semi-infinite orbits of $\Gr_G$, and their geometric properties
were extensively studied by Mirkovi\'c--Vilonen in \cite{MV}. 
Let $\olsf S^{\ch\lambda}$ denote the scheme-theoretic closure 
of $\msf S^{\ch\lambda}$ in $\Gr_G$.

We analogously define $\olsf Y^{\ch\lambda}$ as the central fiber\footnote{The open subscheme $\msf Y^{\ch\lambda}$ does not need to be dense in $\olsf Y^{\ch\lambda}$.}
of $\bar\pi : \barY^{\ch\lambda} \to \sA^{\ch\lambda}$, and 
we deduce from Lemma~\ref{lem:barYGr} that there is a 
closed embedding $\olsf Y^{\ch\lambda} \into \Gr_G$. 
From the moduli-theoretic description of $\olsf S^{\ch\lambda}$ 
(see \cite[proof of Proposition 5.3.6]{Xinwen}), we see that
this factors through a closed embedding 
\begin{equation} \label{e:olYembedGr}
    \olsf Y^{\ch\lambda} \into \olsf S^{\ch\lambda}.
\end{equation}

\subsubsection{Local description of $\msf Y^{\ch\lambda}$} 
Note that the central fiber $\msf Y^{\ch\lambda}$
intersects the stratum $\Scr Y^{\ch\lambda,\ch\Theta}$ 
only if $\ch\Theta =[\ch\theta]$ is the singleton partition corresponding to 
a single $\ch\theta \in \mathfrak c_X^-$. 
In this case, let $\msf Y^{\ch\lambda,\ch\theta}$ denote
the intersection $\msf Y^{\ch\lambda} \cap \Scr Y^{\ch\lambda,\ch\theta} = \msf Y^{\ch\lambda}
\xt_{\sM_X} \sM_X^{\ch\theta}$.
Also let $\olsf Y^{\ch\lambda,\ch\theta} = \olsf Y^{\ch\lambda} \xt_{\sM_X} \sM_X^{\ch\theta}$.

The $\msf L G$-action on the base point $x_0 \in \msf L X(k)$
defines a map $\msf L G \to \msf L X$, which induces a map 
$\Gr_G \to \msf L X/ \msf L^+G$.  

\begin{lem} \label{lem:Y=ScapGr}
There are natural isomorphisms
\begin{align} 
    \msf Y^{\ch\lambda} & 
        \cong \msf S^{\ch\lambda} \xt_{\msf L X/\msf L^+G} \msf L^+ X/\msf L^+G 
\label{e:YScapGr} \\
    \olsf Y^{\ch\lambda} & 
        \cong \olsf S^{\ch\lambda} \xt_{\msf L X/\msf L^+G} \msf L^+ X / \msf L^+G 
\label{e:olYScapGr}
\end{align}
which induce 
$\displaystyle (\msf Y^{\ch\lambda,\ch\theta})_\red \cong (\msf S^{\ch\lambda} \xt_{\msf L X/\msf L^+G} \msf L^{\ch\theta}X/\msf L^+G)_\red$ 
and 
$\displaystyle (\ol{\msf Y}^{\ch\lambda,\ch\theta})_\red \cong (\ol{\msf S}^{\ch\lambda}
 \xt_{\msf L X/\msf L^+G} \msf L^{\ch\theta}X/\msf L^+G)_\red$.
\end{lem}

\begin{proof}
Consider the composition $\Gr_B^{\ch\lambda} \to  \Gr_G \to \msf L X/\msf L^+G$. 
It follows from the definitions that we have an embedding
\[  \msf Y^{\ch\lambda} \into \Gr^{\ch\lambda}_B \xt_{\msf L X/\msf L^+G} \msf L^+X/\msf L^+G
\]
The map in the reverse direction is defined using Beauville--Laszlo's theorem: for a $k$-algebra $R$, an $R$-point of $\Gr^{\ch\lambda}_B$ consists
of a $B$-bundle $\hat{\Scr P}_B$ on $\Spec R\tbrac t$ and a section 
$\hat\sigma_0 : \Spec R\lbrac t \to \hat{\Scr P}_B$. Using the
identification $B \cong X^\circ$, we can identify $\hat\sigma_0$
with a section $\hat\sigma: \Spec R\lbrac t \to X^\circ \times^B \hat{\Scr P}_B$. 
If the image of $(\hat{\Scr P}_B, \hat \sigma)$ in $\msf L X/\msf L^+G$ belongs to $\msf L^+X/\msf L^+G$, then 
$\hat\sigma$ extends to a section $\Spec R\tbrac t \to X\times^B \hat{\Scr P}_B$. 
By Lemma~\ref{lem:B-L}, the pair $(\hat{\Scr P}_B, \hat \sigma)$ is equivalent to an $R$-point $y\in \Scr Y$. 
By construction, $y\in \msf Y^{\ch\lambda}$.
The two maps are mutually inverse, so we get \eqref{e:YScapGr}.

The same argument proves \eqref{e:olYScapGr},
and the other identifications are by definition.
\end{proof}

We included the subscript $_\red$ above for clarity, but from now on we 
consider only the underlying reduced structure on all central fibers 
and omit the subscript since the \'etale site is insensitive
to reduced structures.

\begin{eg} \label{eg:Yfiber}
Resume the setting of Example~\ref{eg:YHecke}.  
Then $\msf Y^{n_1 \ch\nu_{D^+} + n_2 \ch\nu_{D^-}}$ is empty if both $n_1$ and $n_2$ are nonzero.
Otherwise it consists of a single point. 
 If we use the embedding $\mbb G_m \into \GL_2$
via $a\mapsto \left( \begin{smallmatrix} 1 & 0 \\ 1-a& a \end{smallmatrix}\right)$
and fix the base point $x_0=1$, then $\msf Y^{\ch\nu_{D^+}}$ 
corresponds to the point 
$\left( \begin{smallmatrix} 1 & 1 \\ 0 & 1 \end{smallmatrix} \right) t^{\ch\vareps_1} \in \Gr_B$ 
while $\msf Y^{\ch\nu_{D^-}}$ corresponds to $t^{-\ch\vareps_2} \in \Gr_B$.
\end{eg}

\subsubsection{} \label{sect:YttimesC}
Lemma~\ref{lem:Y=ScapGr} implies that 
the central fibers $\msf Y^{\ch\lambda}, \olsf Y^{\ch\lambda}$ 
can be defined purely locally (in particular, independently of 
the point $v\in \abs C$).
We also deduce that $\sY^{\ch\lambda} \xt_{\sA^{\ch\lambda},\delta^{\ch\lambda}} C 
\cong \msf Y^{\ch\lambda} \ttimes C$ and 
$\barY^{\ch\lambda} \xt_{\sA^{\ch\lambda},\delta^{\ch\lambda}} C \cong \olsf Y^{\ch\lambda} \ttimes C$,
where $? \ttimes C := {?}\times^{\Aut k\tbrac t} C^\wedge $
and $C^\wedge \to C$ is the $\Aut k\tbrac t$-torsor classifying $v\in C$
together with an isomorphism $\mf o_v \cong k\tbrac t$.

\subsection{Dimension of central fibers}

We will now discuss an argument that is critical when estimating dimensions of Zastava spaces and their central fibers. This argument appeared in the proof 
of \cite[Theorem 3.2]{MV}. The second author thanks M.~Finkelberg for
explaining this proof to him.

Semi-infinite orbits have a very simple geometric structure, summarized by the following:
\begin{prop}[{\cite[Proposition 3.1]{MV}, \cite[Proposition 5.3.6, Corollary 5.3.8]{Xinwen}}]  \label{prop:Shyperplane}
\ 

\begin{enumerate}
\item We have a stratification
$\olsf S^{\ch\lambda} = \bigcup_{\ch\lambda' \le \ch\lambda} \msf S^{\ch\lambda'}$.

\item
Inside $\olsf S^{\ch\lambda}$, the boundary of $\msf S^{\ch\lambda}$
is given by a hyperplane section under an embedding of $\Gr_G$ in projective space. 
\end{enumerate}
\end{prop}

In particular, if a projective subvariety meets the semi-infinite orbit $\msf S^{\ch\lambda}$, it also meets its boundary $\bigcup_{\lambda' < \ch\lambda}  \msf S^{\ch\lambda'}$. We will use this simple fact several times, in order to estimate the dimensions of central fibers.

\medskip

For $\ch\lambda \in \mf c_X$, consider
the central fiber $\olsf Y^{\ch\lambda} \subset \olsf S^{\ch\lambda} \subset \Gr_G$. 
Observe that the defect stratification of $\barY^{\ch\lambda}$ (\S \ref{sect:strat-compactified}) gives a stratification
of $\olsf Y^{\ch\lambda} = \bigcup_{\ch\lambda' \le \ch\lambda} \msf Y^{\ch\lambda'}$,
where $\msf Y^{\ch\lambda'} = \olsf Y^{\ch\lambda}\cap ({_{\ch\lambda-\ch\lambda'}} \barY^{\ch\lambda})$. 
This is compatible with the stratification of $\olsf S^{\ch\lambda} = \bigcup_{\ch\lambda'\le\ch\lambda}
\msf S^{\ch\lambda'}$ under the closed embedding $\olsf Y^{\ch\lambda}\subset \olsf S^{\ch\lambda}$.
From Lemma~\ref{lem:Y=ScapGr} we have $\olsf Y^{\ch\lambda} \cap \msf S^{\ch\lambda'} = \msf Y^{\ch\lambda'}$
for $\ch\lambda' \le \ch\lambda$.

\begin{prop}
\label{prop:intersections}
 Given an irreducible component $\olsf Y \subset \olsf S^{\ch\lambda}$ of the central fiber $\olsf Y^{\ch\lambda}$ of $\barY^{\ch\lambda}$, there is a $\ch\lambda'\le \ch\lambda$ with $\olsf Y \cap \msf S^{\ch\lambda'}$ nonempty of dimension zero, and $d:=\dim \olsf Y \le \brac{ \rho_G, \ch\lambda-\ch\lambda'}$. 

Moreover, if $d= \brac{ \rho_G, \ch\lambda-\ch\lambda'}$, then there exists a sequence of simple roots $\alpha_1, \dots,\alpha_d$ (possibly with repetitions) 
and subvarieties $\msf Y_j \subset \msf Y^{\ch\lambda-\ch\alpha_1-\dotsb-\ch\alpha_j}$ for $j=0,\dotsc,d$  such that 
\begin{itemize}
\item $\ol{\msf Y_0} = \olsf Y$,
\item $\msf Y_j$ is an irreducible component of $\ol{\msf Y_{j-1}} \cap \msf S^{\ch\lambda - \ch\alpha_1 - \cdots -\ch\alpha_j}$ of dimension $d-j$ for $j=1,\dotsc, d$.
\end{itemize}
\end{prop}

\begin{proof}
By the stratification $\olsf S^{\ch\lambda} = \bigcup_{\ch\lambda'\le\ch\lambda} \msf S^{\ch\lambda}$,
there must exist $\ch\lambda_0 \le \ch\lambda$ such that 
$\msf Y_0 := \olsf Y \cap \msf S^{\ch\lambda_0}$ is dense in $\olsf Y$. 
Then $\olsf Y \subset \olsf S^{\ch\lambda_0}$ so we may assume $\ch\lambda=\ch\lambda_0$. 
Write $\partial \olsf S^{\ch\lambda} = \olsf S^{\ch\lambda} \sm \msf S^{\ch\lambda}$ for the hyperplane of Proposition~\ref{prop:Shyperplane}(ii). 
Since $\olsf Y$ is a projective subvariety of $\olsf S^{\ch\lambda}$, it must
 meet $\partial \olsf S^{\ch\lambda}$. 
Hence there exists $\ch\lambda_1 < \ch\lambda$ such that
$\dim(\olsf Y \cap \msf S^{\ch\lambda_1}) = d-1$. 
Let $\msf Y_1$ be any irreducible component of $\olsf Y\cap \msf S^{\ch\lambda_1}$ and 
$\ol{\msf Y_1}$ its closure in $\olsf S^{\ch\lambda_1}$. 
Continuing in this fashion we produce a sequence of coweights $\ch\lambda_0=\ch\lambda, \ch\lambda_1,\dotsc,\ch\lambda_d$ and irreducible components $\msf Y_j \subset \ol{\msf Y_{j-1}}\cap \msf S^{\ch\lambda_j}$ for $j=1,\dotsc,d$, such that $\ch\lambda_j < \ch\lambda_{j-1}$ and 
$\dim \msf Y_j = d-j$. 
Since $\ch\lambda_j < \ch\lambda_{j-1}$ implies $\brac{\rho_G,\ch\lambda_{j-1}-\ch\lambda_j}\ge 1$, 
we get $\brac{\rho_G,\ch\lambda-\ch\lambda_d} \ge d$ and $\ch\lambda'=\ch\lambda_d$ is the claimed
coweight in the proposition statement. 

In order for the last inequality to be an equality, we must have $\ch\lambda=\ch\lambda_0$
and $\ch\lambda_j = \ch\lambda_{j-1} - \ch\alpha_j$ for some simple coroot $\ch\alpha_j,\, j=1,\dotsc,d$. 
\end{proof}

\subsection{Comparison of IC complexes} \label{sect:compareIC}
In order to describe the restriction of $\IC_{\barY}$ to strata 
we will need to introduce several (graded) factorization algebras on the
collection of $C_{\ch\nu},\, \ch\nu\in\ch\Lambda^\pos_G$. 
We only state their definitions; we refer the reader to \cite{BG2, G:whatacts} for
further context.

\smallskip

Let $\ch{\mf n}_C = \ch{\mf n} \ot_{\ol\bbQ_\ell} {\ol\bbQ_\ell}_C$ be the constant sheaf of 
Lie algebras over $C$. 
Let $\mf U(\ch{\mf n}_C)^{\ch\nu}$ denote the following sheaf on $C_{\ch\nu}$
for $\ch\nu \in \ch\Lambda^\pos_G$: its $*$-stalk at a point $\sum \ch\nu_i \cdot v_i$
with $v_i \in \abs C$ distinct is the tensor product $\bigotimes_i U(\ch{\mf n})^{\ch\nu_i}$
where the superscript $\ch\nu_i$ refers to the corresponding weight space of 
the universal enveloping algebra $U(\ch{\mf n})$. 
These stalks glue to a sheaf by means of the co-multiplication map on $U(\ch{\mf n})$. 
Let $\mf U^\vee(\ch{\mf n}_C)^{-\ch\nu} = \mbb D (\mf U(\ch{\mf n}_C)^{\ch\nu}) \in \rmD^b_c(C_{\ch\nu})$ denote the Verdier dual. 

We will also need the Chevalley--Cousin complex $\Upsilon(\ch{\mf n}_C)$ 
of $\ch{\mf n}_C$. 
Consider the (homological) Chevalley complex $C_\bullet(\ch{\mf n}_C)$
as a sheaf of co-commutative DG co-algebras on $C$, endowed with a grading by elements of
$\ch\Lambda^\pos_G$. 
By the general procedure of \cite[\S 3.4]{BD:chiral}, a sheaf of $\ch\Lambda^\pos_G$-graded
co-commutative DG co-algebras on $C$ is equivalent to a commutative factorization algebra 
on $\bigsqcup C_{\ch\nu}$. We let $\Upsilon(\ch{\mf n}_C)^{\ch\nu}$ denote
the corresponding complex on $C_{\ch\nu}$, which can be explicitly constructed
via a Cousin complex. 
Its $*$-stalk at $\sum \ch\nu_i \cdot v_i$ with $v_i$'s disinct is 
the tensor product $\bigotimes_i C_\bullet(\ch{\mf n})^{\ch\nu_i}$. 
We remark that $\Upsilon(\ch{\mf n}_C)^{\ch\nu}$ is actually
a perverse sheaf, and $\Upsilon(\ch{\mf n}_C)$ and $\mf U^\vee(\ch{\mf n}_C)$ are
related by a certain Koszul duality.

\subsubsection{} Let 
\[ \mf i_{\barY,\ch\nu} : C_{\ch\nu} \times \Scr Y^{\ch\lambda-\ch\nu}
\into \barY^{\ch\lambda} \]
denote the locally closed embedding corresponding to \eqref{e:barYstrata1}. 

\begin{prop} \label{prop:ICbarY}
For any $\ch\lambda\in \mathfrak c_X,\, \ch\nu\in \ch\Lambda^\pos_G$, 
there is an equality 
\[ 
[\mf i^*_{\barY,\ch\nu}(\IC_{\barY^{\ch\lambda}}) ]
= [\mfU^\vee(\ch\mfn_C)^{-\ch\nu} \bt \IC_{\Scr Y^{\ch\lambda-\ch\nu}}] 
\]
in the Grothendieck group of perverse sheaves on $C_{\ch\nu} \xt \sY^{\ch\lambda-\ch\nu}$.
\end{prop}

We will prove the proposition in the course of this subsection. It is based on the following result of {\cite[Proposition 4.4]{BG2}, \cite[Theorem 1.12]{BFGM}}. 
\begin{thm}
\label{thm:BFGM}
There exists a canonical isomorphism 
\[ \mf i_{\ch\nu}^* (\IC_{\ol\Bun_B^{\ch\mu}}) \cong 
\mfU^\vee(\ch{\mfn}_C)^{-\ch\nu} \bt \IC_{\Bun_B^{\ch\mu+\ch\nu}} 
\]
for any $\ch\nu\in \ch\Lambda^\pos_G,\, \ch\mu\in \ch\Lambda_G$. 
\end{thm}

Given a map $f : Y \to S$ between finite type algebraic $k$-stacks, 
we use the notion of a complex on $Y$ which is 
\emph{universally locally acyclic} (ULA) with 
respect to $f$, as in \cite[D\'efinition 2.12]{SGA4h}.

The general lemma we will use is the following:

\begin{lem}[{\cite[Lemma 7.1.3]{BG}}] \label{lem:ULAbox} 
Consider a Cartesian diagram of finite type algebraic stacks
\[
\begin{tikzcd} 
Y' \ar[r, "{f'}"] \ar[d, "{g'}", swap] & S' \ar[d, "g"] \\ 
Y \ar[r, "f"] & S
\end{tikzcd}
\]
where $S$ is smooth.
Let $j : Y_0 \into Y$ be an open dense substack such that
the map $f \circ j : Y_0 \to S$ is smooth. 
In addition, assume that the complexes $\IC_Y$ and $j_!(\IC_{Y_0})$
are ULA with respect to the map $f$. 

Denote the closure\footnote{In general it is possible for $Y'$ to have more irreducible
components than $\ol{Y_0\xt_{S} S'}$.} 
of $Y_0 \xt_S S'$ in $Y'$ by $\ol{Y_0 \xt_S S'}$. 
Then there is a natural isomorphism 
\[ 
\IC_{\ol{Y_0 \xt_S S'}} \cong \IC_{S'}\bt_S \IC_{Y}  := 
f'^*(\IC_{S'})\ot g'^*(\IC_Y) [-\dim S],
\]
where the left hand side is implicitly extended by zero to $Y'$.
\end{lem}
For a proof of Lemma~\ref{lem:ULAbox}, see \cite{Wang:ULA}.

\medskip

We would like to apply this lemma with $Y = \ol\Bun_B, \, Y_0 = \Bun_B,\,
S = \Bun_G$, and $S' = \sM_X$. Unfortunately the stack $\ol\Bun_B$ is
``too big'' for all of $\IC_{\ol\Bun_B}$ to be ULA over $\Bun_G$. 
However, we do get the ULA property if we restrict to open substacks
where the defect is not ``too big'':
Let $\mf j : \Bun_B \into \ol\Bun_B$ denote the open embedding.

\begin{prop}[{\cite{Campbell, Campbell2}}]
\label{prop:ULA}
Fix $\ch\nu \in \ch\Lambda^\pos_G$. Then for any $\ch\mu \in \ch\Lambda_G$
large enough:
\begin{enumerate}
\item The complexes $\IC_{_{\le \ch\nu}\ol\Bun^{-\ch\mu}_B}$ and 
$\mf j_! \IC_{\Bun_B}|_{_{\le \ch\nu}\ol\Bun^{-\ch\mu}_B}$ are 
ULA over $\Bun_G$.
\item The fiber of $\Bun^{-\ch\mu}_B \to \Bun_G$ is dense in the
fiber of $_{\le\ch\nu}\ol\Bun^{-\ch\mu}_B \to \Bun_G$ over any $k$-point of $\Bun_G$.
\end{enumerate}
\end{prop}

Here ``large enough'' is in the same sense as in Lemma~\ref{lem:BunBtoGsm}, i.e.,
deep enough in the dominant chamber $\ch\Lambda^+_G$.

\begin{proof}
The ULA property for $\IC_{_{\le \ch\nu}\ol\Bun^{-\ch\mu}_B}$ 
is \cite[Corollary 4.1.1.1]{Campbell}. 
The ULA property for $\mf j_! \IC_{\Bun_B}|_{_{\le \ch\nu}\ol\Bun^{-\ch\mu}_B}$ can be deduced from the former, cf.~\cite[\S 4.3]{Campbell2}. 
We review the salient features of the proof in \cite{Campbell} to deduce (ii). 

The key object introduced in \cite{Campbell} is Kontsevich's compactification
$\ol\Bun_B^K$ of $\Bun_B$, which is a resolution of singularities $\ol\Bun_B^K \to \ol\Bun_B$
of Drinfeld's compactification. For an affine test scheme $S$, 
an $S$-point of $\ol\Bun_B^K$ is a commutative square
\[ 
\begin{tikzcd}
\wt{\Scr C} \ar[d] \ar[r] & \pt/B \ar[d] \\ C\xt S \ar[r, "\sP_G"] & \pt/G 
\end{tikzcd} 
\]
where $\wt{\Scr C} \to S$ is a flat family of connected nodal projective curves of the
same arithmetic genus as $C$, the map $\wt{\Scr C} \to C\xt S$ has degree $1$, 
and the induced section $\wt{\Scr C} \to \sP_G \xt^G G/B$ is \emph{stable} 
in the sense of \cite{Kontsevich} over every geometric point of $S$.
Let $\mathscr M_C$ denote the moduli stack of proper connected nodal curves $\wt{\Scr C}$
equipped with a degree one map to $C$. 
Then the natural map $\ol\Bun_B^K\to \mathscr M_C$ is smooth, and $\mathscr M_C$ is
smooth (\cite[Proposition 2.4.1]{Campbell}) and there is a proper map
$\ol\Bun_B^K \to \ol\Bun_B$ over $\Bun_G$ (\cite[Proposition 3.2.2]{Campbell}). 
Inside $\mathscr M_C$ we have the open point corresponding to $\mathrm{id} : C\to C$
and the complement $\partial\mathscr M_C$ is a normal crossings divisor (\cite[Proposition 4.4.1]{Campbell}). Evidently the preimage of the open point in $\ol\Bun_B^K$ identifies
with $\Bun_B$. 
Let $_{\le\ch\nu}\ol\Bun_B^{K,-\ch\mu}$ denote the preimage of $_{\le\ch\nu}\ol\Bun_B^{-\ch\mu}$. Then \cite[Proposition 4.1.1]{Campbell} shows that
for $\ch\nu$ fixed and $\ch\mu$ large enough, the map 
$_{\le\ch\nu}\ol\Bun_B^{K,-\ch\mu} \to \mathscr M_C \xt \Bun_G$ is smooth. 
Now for any $\sP_G \in \Bun_G(k)$, the fiber of $\wt{\mf p}^K: {_{\le\ch\nu}}\ol\Bun_B^{K,-\ch\mu} \to \Bun_G$ is smooth over $\mathscr M_C$.
The preimage over the open point equals $(\wt{\mf p}^K)^{-1}(\sP_G) \cap \Bun_B$, 
and since $\partial \mathscr M_C$ is a normal crossings divisor, 
$(\wt{\mf p}^K)^{-1}(\sP_G) \cap \Bun_B$ must be dense in $(\wt{\mf p}^K)^{-1}(\sP_G)$.
Since $(\wt{\mf p}^K)^{-1}(\sP_G)$ is a resolution of
the fiber of $_{\le\ch\nu}\ol\Bun_B^{-\ch\mu}\to\Bun_G$, we have proved (ii).
\end{proof}

\begin{cor} \label{cor:barYdense}
For any $\ch\lambda\in \mathfrak c_X$,
the subscheme $\sY^{\ch\lambda}$ is dense in $\barY^{\ch\lambda}$. 
\end{cor}

\begin{proof}
The subschemes $_{\le \ch\eta}\barY^{\ch\lambda} = \barY^{\ch\lambda} \xt_{\ol\Bun_B} {_{\le\ch\eta}}\ol\Bun_B$
for $\ch\eta\in \ch\Lambda^\pos_G$ form an open covering of $\barY^{\ch\lambda}$. 
Since $\barY^{\ch\lambda}$ is quasicompact, there must exist some 
$\ch\eta$ such that $_{\le \ch\eta}\barY^{\ch\lambda} = \barY^{\ch\lambda}$.
Fix $\ch\lambda,\ch\eta$ as above. Then we can choose 
$\ch\mu \in \ch\Lambda^\pos_G$ such that $\ch\lambda+\ch\mu-\ch\nu$ is 
large enough, for all $0 \le \ch\nu\le \ch\eta$, for the purposes of
Proposition~\ref{prop:ULA} and Lemma~\ref{lem:BunBtoGsm}. 
Now we may consider 
$_{\le\ch\eta}\barY^{\ch\lambda+\ch\mu}$ as an open subscheme of 
$Y':=\Scr M_X \underset{\Bun_G}\xt {_{\le \ch\eta}\ol\Bun^{-\ch\lambda-\ch\mu}_B}$.
Proposition~\ref{prop:ULA}(ii) implies that 
$\sM_X \underset{\Bun_G}\xt \Bun_B^{-\ch\lambda-\ch\mu}$
is dense in $Y'$,
which implies that $\sY^{\ch\lambda+\ch\mu}$ is dense in $\barY^{\ch\lambda+\ch\mu}$.
The graded factorization property of $\barY$ 
gives a natural \'etale map 
\begin{equation} \label{e:barYfactorlarge}
 \barY^{\ch\lambda} \oo\xt \Scr Y^{\ch\mu} := (\barY^{\ch\lambda} \xt \Scr Y^{\ch\mu})|_{\Scr A^{\ch\lambda} \oo\xt \Scr A^{\ch\mu}} \to 
\barY^{\ch\lambda+\ch\mu}, 
\end{equation}
where $\Scr Y^{\ch\mu} \into \barY^{\ch\mu}$
is the open embedding. We deduce that $\sY^{\ch\lambda} \oo\xt \sY^{\ch\mu}$ is
dense in $\barY^{\ch\lambda} \oo\xt \sY^{\ch\mu}$, which implies 
$\sY^{\ch\lambda}$ is dense in $\barY^{\ch\lambda}$. 

\end{proof}

\begin{proof}[Proof of Proposition~\ref{prop:ICbarY}]

Applying Lemma~\ref{lem:ULAbox} to the Cartesian square
\[
\begin{tikzcd} 
Y' \ar[r] \ar[d] & \Scr M_X \ar[d] \\ 
{}_{\le \ch\eta}\ol\Bun_B^{-\ch\lambda-\ch\mu} \ar[r] & \Bun_G,
\end{tikzcd}
\]
we can identify
\begin{equation} \label{e:ICbar=bt}
    \IC_{Y'} \cong \IC_{\Scr M_X} \underset{\Bun_G}\boxtimes 
\IC_{_{\le\ch\eta} \ol\Bun^{-\ch\lambda-\ch\mu}_B}. 
\end{equation}
In particular, this gives us a description of $\IC_{_{\le\ch\eta}\barY^{\ch\lambda+\ch\mu}} = \IC_{Y'}|_{_{\le\ch\eta}\barY^{\ch\lambda+\ch\mu}}$. 
For $0 \le \ch\nu \le \ch\eta$, we have a Cartesian square
\[ 
\begin{tikzcd}
C_{\ch\nu} \times \Scr Y^{\ch\lambda+\ch\mu-\ch\nu} \ar[r,hook,"\mf i_{\barY,\ch\nu}"] \ar[d] & 
\barY^{\ch\lambda+\ch\mu} \ar[d] \\ 
C_{\ch\nu} \times \Bun_B^{-\ch\lambda-\ch\mu+\ch\nu} \ar[r,hook,"\mf i_{\ch\nu}"] &
\ol\Bun_B^{-\ch\lambda-\ch\mu} 
\end{tikzcd}
\]

Theorem \ref{thm:BFGM}, together with the Cartesian square above and 
\eqref{e:ICbar=bt}, allow us to deduce that there exists an isomorphism 
\[ \mf i_{\barY,\ch\nu}^*( \IC_{\barY^{\ch\lambda+\ch\mu}} ) 
\cong \mfU^\vee(\ch{\mfn}_C)^{-\ch\nu} \bt (\IC_{\Scr M_X} \bt_{\Bun_G} \IC_{\Bun_B^{-\ch\lambda-\ch\mu+\ch\nu}})|_{\Scr Y^{\ch\lambda+\ch\mu-\ch\nu}}
= \mfU^\vee(\ch{\mfn}_C)^{-\ch\nu} \bt (\IC_{\Scr M_X}|^{!*}_{\Scr Y^{\ch\lambda+\ch\mu-\ch\nu}}) \]
on $C_{\ch\nu} \xt \Scr Y^{\ch\lambda+\ch\mu-\ch\nu}$. 
In the last equality we have used the fact that $\Bun_B$ is smooth. 
Since we chose $\ch\lambda+\ch\mu-\ch\nu$ large enough to satisfy
Lemma~\ref{lem:BunBtoGsm}, the map $\Scr Y^{\ch\lambda+\ch\mu-\ch\nu} \to \Scr M_X$ is smooth. Therefore, $\IC_{\Scr M_X}|^{!*}_{\Scr Y^{\ch\lambda+\ch\mu-\ch\nu}} \cong \IC_{\Scr Y^{\ch\lambda+\ch\mu-\ch\nu}}$ and 
we get a canonical isomorphism 
\begin{equation} \label{e:ICbarYrestrict}
    \mf i_{\barY,\ch\nu}^*(\IC_{\barY^{\ch\lambda+\ch\mu}}) 
    \cong \mfU^\vee(\ch\mfn_C)^{\ch\nu} \bt \IC_{\Scr Y^{\ch\lambda+\ch\mu-\ch\nu}}. 
\end{equation}

Observe that the following diagram is Cartesian:
\[
\begin{tikzcd}[column sep=large]
    (C_{\ch\nu} \xt \Scr Y^{\ch\lambda-\ch\nu}) \oo\xt \Scr Y^{\ch\mu} 
\ar[r,hook,"\mf i_{\barY,\ch\nu} \times \on{id}"] \ar[d] & 
    \barY^{\ch\lambda} \oo\xt \Scr Y^{\ch\mu} \ar[d, "\text{\eqref{e:barYfactorlarge}}"] \\ 
    C_{\ch\nu} \xt \Scr Y^{\ch\lambda+\ch\mu-\ch\nu} \ar[r, hook, "\mf i_{\barY,\ch\nu}"] & 
    \barY^{\ch\lambda+\ch\mu} 
\end{tikzcd}
\]
where the left vertical arrow is identity on $C_{\ch\nu}$ times 
the map $\Scr Y^{\ch\lambda-\ch\nu} \oo\xt \Scr Y^{\ch\mu} \to \Scr Y^{\ch\lambda+\ch\mu-\ch\nu}$  coming from graded factorization. 
Since the restriction of $\IC_{\barY^{\ch\lambda+\ch\mu}}$ to 
$\barY^{\ch\lambda} \oo\xt \Scr Y^{\ch\mu}$ is 
$\IC_{\barY^{\ch\lambda}} \mathbin{\oo\bt} \IC_{\Scr Y^{\ch\mu}}$, 
we deduce from the Cartesian square and \eqref{e:ICbarYrestrict} 
that there is a canonical isomorphism 
\[ \mf i^*_{\barY,\ch\nu}(\IC_{\barY^{\ch\lambda}}) \mathbin{\oo\bt} \IC_{\Scr Y^{\ch\mu}}
\cong (\mfU^\vee(\ch\mfn_C)^{-\ch\nu} \bt \IC_{\Scr Y^{\ch\lambda-\ch\nu}}) 
\mathbin{\oo\bt} \IC_{\Scr Y^{\ch\mu}} 
\]
when restricted to $(C_{\ch\nu} \xt \Scr Y^{\ch\lambda-\ch\nu})\oo\xt \Scr Y^{\ch\mu}$.
Lastly, let $\mathscr Y$ be a connected component of $\sY^{\ch\mu,0}$ so 
that the projection $p : (C_{\ch\nu} \xt \sY^{\ch\lambda-\ch\nu}) \oo\xt \mathscr Y \to 
C_{\ch\nu} \xt \sY^{\ch\lambda-\ch\nu}$ 
is smooth surjective with irreducible fibers. 
We have constructed an isomorphism 
\[ p^*( \mf i^*_{\barY,\ch\nu}(\IC_{\barY^{\ch\lambda}}) )
\cong p^*(  \mfU^\vee(\ch\mfn_C)^{-\ch\nu} \bt \IC_{\Scr Y^{\ch\lambda-\ch\nu}} ), \]
which implies Proposition~\ref{prop:ICbarY} because 
the functor $p^*[\dim \mathscr Y]$ is fully faithful on the category
of perverse sheaves (\cite[Proposition 4.2.5]{BBDG}).
\end{proof}

\subsubsection{Convolution product}  \label{sect:convolution-product}

The toric variety $X/\!\!/N$ has the natural structure of a commutative 
algebraic monoid. The multiplication operator on $X/\!\!/N$ induces
a finite map 
\[ m_{\Scr A} : \Scr A^{\ch\lambda_1} \xt \Scr A^{\ch\lambda_2} \to \Scr A^{\ch\lambda_1+\ch\lambda_2}. \]
If we have sheaves $\Scr F_i \in \rmD^b_c(\Scr A^{\ch\lambda_i}),\, i=1,2$,
we define their convolution by 
\[ \Scr F_1 \star \Scr F_2 := m_{\Scr A,!}(\Scr F_1 \bt \Scr F_2) \in 
\rmD^b_c(\Scr A^{\ch\lambda_1+\ch\lambda_2}). \]

\subsubsection{}
We have a closed embedding $\mf i_{\Scr A,\ch\nu} : C_{\ch\nu} \into \Scr A^{\ch\nu}$ corresponding to the partition $\sum_i n_i [\ch\alpha_i]$
of degree $\ch\nu = \sum_i n_i\ch\alpha_i$.

\begin{cor} \label{cor:pibarY}
There is an equality 
\[ [ \pi_!(\IC_{\Scr Y^{\ch\lambda}}) ] = \sum_{\ch\nu\in \ch\Lambda^\pos_G} 
[ \mf i_{\Scr A,\ch\nu,!}(\Upsilon(\ch\mfn_C)^{\ch\nu}) \star \bar\pi_!(\IC_{\barY^{\ch\lambda-\ch\nu}})] \]
in the Grothendieck group of perverse sheaves on $\Scr A^{\ch\lambda}$. 
\end{cor}

\begin{proof}
Taking the Grothendieck--Cousin complex associated to the stratifications
$\mf i_{\barY,\ch\nu}$ and applying Proposition~\ref{prop:ICbarY} 
gives an equality 
\[ [\IC_{\barY^{\ch\lambda}}] = \sum_{\ch\nu \in \ch\Lambda^\pos_G} 
[\mf i_{\barY,\ch\nu,!} ( \mf U^\vee(\ch\mfn_C)^{\ch\nu} \bt \IC_{\Scr Y^{\ch\lambda-\ch\nu}} ) ] \]
in the Grothendieck group of perverse sheaves on $\barY^{\ch\lambda}$.
Note that $\Scr Y^{\ch\lambda-\ch\nu}$ is nonempty for finitely many
values of $\ch\nu$. 
The composition $\bar\pi \circ \mf i_{\barY,\ch\nu} : C_{\ch\nu} \xt \Scr Y^{\ch\lambda-\ch\nu} \to \Scr A^{\ch\lambda}$ coincides with
the composition $C_{\ch\nu} \xt \Scr Y^{\ch\lambda-\ch\nu} \overset{\mf i_{\Scr A,\ch\nu} \xt \pi}\longrightarrow \Scr A^{\ch\nu} \xt \Scr A^{\ch\lambda-\ch\nu} \overset{m_{\Scr A}}\to \Scr A^{\ch\lambda}$.
Therefore, applying $\bar\pi_!$ to the equality above, we get
the equality 
\[ [\bar\pi_!(\IC_{\barY^{\ch\lambda}})] = \sum_{\ch\nu\in \ch\Lambda^\pos_G} [\mf i_{\Scr A,\ch\nu,!}(\mf U^\vee(\ch\mfn_C)^{\ch\nu}) \star \pi_!(\IC_{\Scr Y^{\ch\lambda-\ch\nu}})]  \]
in the Grothendieck group of perverse sheaves on $\Scr A^{\ch\lambda}$.
Applying a further convolution by any $\Scr T \in \rmD^b_c(C_{\ch\nu'}), 
\ch\nu'\in \ch\Lambda^\pos_G$
gives 
\[ [ \mf i_{\Scr A, \ch\nu',!}( \Scr T ) \star \bar\pi_! (\IC_{\barY^{\ch\lambda}}) ] = \sum_{\ch\nu} [ \mf i_{\Scr A,\ch\nu+\ch\nu',!}( \Scr T \star \mf U^\vee(\ch\mfn_C)^{\ch\nu}) \star \pi_!(\IC_{\Scr Y^{\ch\lambda-\ch\nu}}) ] . \]
It is known (\cite[\S 6.4]{BG2}) that for a fixed nonzero $\ch\nu \in \ch\Lambda^\pos_G$, we have an equality 
\[ \sum_{\substack{\ch\nu_1,\ch\nu_2 \in \ch\Lambda^\pos_G\\ 
\ch\nu_1+\ch\nu_2=\ch\nu}} 
[ \Upsilon(\ch\mfn_C)^{\ch\nu_1} \star \mf U^\vee(\ch\mfn_C)^{-\ch\nu_2}  ] = 0 \]
in the Grothendieck group of perverse sheaves on $C_{\ch\nu}$. 
The two preceeding equalities and induction prove the claim.
\end{proof}

\section{Global Hecke action and closure relations} \label{sect:Heckeact}

\emph{For the rest of this paper, assume that $\ch G_X = \ch G$
and all simple roots of $G$ are spherical roots of type $T$.} Equivalently, 
we are assuming that $B$ acts simply transitively on $X^\circ$ and 
for every simple root $\alpha$ of $G$, the $\PGL_2$-variety $X^\circ P_\alpha/\mathfrak R(P_\alpha)$ is isomorphic to $\mathbb G_m\backslash\PGL_2$ (over the algebraically closed field $k$).

\smallskip

As a consequence, $\Cal V \cap \ch\Lambda_X = \ch\Lambda^-_G$, the monoid of antidominant coweights of $G$. 
Recall from \S\ref{subsection:typeT} that 
the type $T$ assumption also implies that for every simple root $\alpha$, 
the open $P_\alpha$-orbit 
$X^\circ P_\alpha$ is the union of $X^\circ$ and the open $B$-orbits of 
two colors $\Cal D(\alpha) = \{ D_\alpha^+,D_\alpha^- \}$. 
We will let $\ch\nu_\alpha^\pm$ denote the valuation of $D_\alpha^\pm$, respectively.
Then $\ch\nu_\alpha^+ + \ch\nu_\alpha^- = \ch\alpha$ 
and $\brac{\alpha, \ch\nu_\alpha^\pm} = 1$. 
We encourage the reader to refer to Examples~\ref{eg:YHecke}, \ref{eg:YHeckePGL}, \ref{eg:Heckecompact} and \ref{eg:Yfiber}.

\subsection{Main results of this section}

This section is quite technical, and we advise the reader to read the main results listed here, and skip the rest of the section at first reading. Before we introduce the results, let us observe that, so far, we have uniformly treated all affine spherical varieties. 
However, the classification of spherical varieties is divided into two parts: 
(i) the classification of homogeneous spherical varieties $H\bs G$, and (ii) the
classification of spherical embeddings $H\bs G \into X$ (by a spherical embedding we mean a $G$-equivariant open, dense embedding $H\bs G \into X$, where $X$ is a normal, and hence spherical, $G$-variety). 

Since $X$ is affine, $X^\bullet=H\bs G$ is quasiaffine and there is a 
\emph{canonical affine embedding} 
\[ X^\can := \Spec k[H\bs G]. \] 
(The fact that the coordinate ring of $k[H\bs G]$ is finitely generated follows from the fact that $B$-eigenspaces are one-dimensional, and the $B$-character group is finitely generated.) By normality, $X^\can$ has no divisors that do not meet $H\bs G$, 
i.e., $\Cal D(X)$ has no $G$-stable divisors, and coincides with the set $\Cal D$ of colors. 
Therefore, the cone $\Cal C_0(H\bs G):= \Cal C_0(X^\can)$ 
is generated by the valuations $\varrho_X(\Cal D)$ of colors.

For any other affine spherical embedding $H\bs G\into X$, there is  
a natural map $X^\can \to X$, so $X^\can$ is universal among affine embeddings 
of $H\bs G$. 

It turns out that this distinction between the minimal and the general embeddings is important when we consider arc spaces and their global models. 
As we will recall in Lemma~\ref{lem:Membedding}, the map $X^\can \to X$ induces a closed embedding of mapping stacks 
\[ \Scr M_{X^\can} \hookrightarrow \Scr M_X,\]
whose image is a union of irreducible components.

We are more interested in the closure of $\Scr M^0_{X^\bullet} = \Bun_H$ in the former, which we will denote by $\barM_X^0$ (it may or may not be the same as $\Scr M_{X^\can}$ depending on whether the monoid $\mathfrak c_{X^\can}$ is generated by colors, see Corollary \ref{cor:can-comp}):
\begin{equation}
 \label{cantogeneral}
\barM_X^0 \hookrightarrow \Scr M_{X^\can} \hookrightarrow \Scr M_X.
\end{equation}

The main theme of this section is, in some sense, a reconstruction of $\Scr M_X$ out of suitable Hecke operators acting on $\barM_X^0$: 
For any $\ch\Theta \in \Sym^\infty(\mathfrak c_X^-\sm 0)$, a multiset of nonzero elements in $\mathfrak c_X^-$,
there is a natural proper map 
\begin{equation} \label{equation-action}
\act_\sM : \barM_X^0 \ttimes \ol\Gr^{\ch\Theta}_{G, C^{\ch\Theta}} \to 
\Scr M_X,
\end{equation}
which corresponds to the action on the ``basic stratum'' $\barM_X^0$ by the closed stratum of the affine Grassmannian parametrized by $\ch\Theta$.
(See Proposition--Construction \ref{prop-const:action}.) Recall that such a multiset $\check\Theta$ also parametrizes a stratum $\Scr M^{\ch\Theta}_X$ of the global mapping space (\S \ref{sect:globstrat}). The main technical result of this section is the following:

\begin{thm}
\label{thm:comp}
For every $\ch\Theta \in \Sym^\infty(\mathfrak c_X^- \sm 0)$, the action map \eqref{equation-action} 
\begin{enumerate}
\item has image equal to the closure of $\Scr M^{\ch\Theta}_X$, and  
\item is birational onto its image. 
\item The restriction of $\act_\sM$ to $\act_\sM^{-1}(\sM^{\ch\Theta}_X) \cap 
(\Bun_H \ttimes \Gr^{\ch\Theta}_{G, \oo C^{\ch\Theta}}) \to \sM^{\ch\Theta}_X$ 
is an isomorphism. 
\end{enumerate}
\end{thm}

The proof of Theorem \ref{thm:comp} will be given in \S \ref{proof:comp}.

We use this theorem to achieve two goals in this section:

The first goal is to understand irreducible components of $\Scr M_X$ and closure relations among the strata $\Scr M^{\ch\Theta}_X$. 
We will introduce the natural generalization to multisets of the order $\succeq$ among elements of the lattice $\ch\Lambda_X$ (we remind that this order is determined by the monoid of colors, see \S \ref{sect:spherical}), and prove:
\begin{prop}[See Proposition \ref{prop:closure-rel}.]
\label{proposition-closures}
Let $\ch\Theta,\ch\Theta' \in \Sym^\infty(\mathfrak c_X^-\sm 0)$. 
The stratum $\sM^{\ch\Theta'}_X$ lies in the closure of $\sM^{\ch\Theta}_X$
if and only if there exists $\ch\Theta''$  such that $\ch\Theta$ refines $\ch\Theta''$ and $\ch\Theta' \succeq \ch\Theta''$. 
\end{prop}

Moreover, observe that the irreducible components of $\barM_X^0$ are in bijection with connected components of $\Scr M_{X^\bullet} = \Bun_H$, i.e., parametrized by $\pi_0(\Bun_H) = \pi_1(H)$. Using the action \eqref{equation-action} on those components gives us a parametrization of the irreducible components of the closure of each stratum:
\begin{prop}[See Corollaries \ref{cor:strata-comp} and \ref{cor:Ycomp}]
\label{proposition-irred-strata}
For every $\ch\Theta \in \Sym^\infty(\mathfrak c_X^-\sm 0)$, the irreducible components in the closure $\barM_X^{\ch\Theta}$ of the corresponding stratum are naturally parametrized by $\pi_1(H)$. 
For any $\ch\lambda \in \mf c_X$, the base change of 
an irreducible component to $\sY^{\ch\lambda} \xt_{\sM_X} 
\barM_X^{\ch\Theta}$ is still irreducible (when nonempty).
\end{prop}

The combination of Propositions \ref{proposition-closures} and \ref{proposition-irred-strata} implies:

\begin{cor}[See Corollary \ref{cor:comp}] 
There is a natural bijection between the set of irreducible
components of $\Scr M_X$ and 
\[ \pi_1(H) \xt \Sym^\infty(\Cal D^G_{\mathrm{sat}}(X)),\]
where $\Cal D^G_{\mathrm{sat}}(X)$ denotes the set of primitive elements in $\mathfrak c_X^-$ that
cannot be decomposed as a sum $\ch\theta+\ch\nu_D$ where 
$\ch\theta\in  \mathfrak c_X^-\sm 0$ and $\ch\nu_D$ is the valuation attached to a color. 
\end{cor}

The second goal achieved by Theorem \ref{thm:comp} is to reduce the study of the IC complex of an arbitrary mapping space $\Scr M_X$ to that of the minimal affine embedding:

\begin{thm} \label{thm:heckeIC}
For any $\ch\Theta \in \Sym^\infty(\mathfrak c_X^- \sm 0)$, there
is a natural isomorphism 
\[ \IC_{\barM_X^0} \star \IC_{\ol\Gr^{\ch\Theta}_{G,C^{\ch\Theta}}} \cong 
\IC_{\barM^{\ch\Theta}_X}. \] 
\end{thm}

We point to \S\ref{sect:pf:heckeIC} for the notation and the proof.  
Theorem~\ref{thm:heckeIC} and its proof are independent from the rest of this paper; we
include it only for conceptual completeness. 
The proof involves a passage to Zastava model (the extra structure of a flag)
and uses the results of the sequel \S\ref{sect:centralfiber}.
The argument of proof can also be extended to meromorphic quasimaps to verify
\cite[Conjecture 7.3.2]{GN} in the case $\ch G_X = \ch G$. 

\subsection{Hecke action on global model} \label{sect:Heckeglob}

Now we introduce the action map \eqref{equation-action}, as an analog of the 
action of $G(F)$ on $X^\bullet(F)$ for the global model.

The notation is cumbersome because we give a multi-point version of the action, but
the idea is simple: in the notation of \S\ref{sect:adelic}, 
the $k$-points of $\sM_X$ correspond to a subset of $H(\Bbbk) \bs G(\mbb A) / G(\mbb O)$, where $\mbb O = \prod_{v\in \abs C} \mf o_v$. 
We have a Hecke correspondence 
\[ H(\Bbbk) \bs G(\mbb A) \xt^{G(\mbb O)} G(\mbb A) / G(\mbb O) \to H(\Bbbk) \bs G(\mbb A) / G(\mbb O) \]
induced by multiplication in $G$. We think of $G(\mbb A)/G(\mbb O)$ as a
factorizable version of the affine Grassmannian. Then the ``action map'' 
is simply the restriction of the above to a positively graded subset of $G(\mbb A)/G(\mbb O)$ 
such that everything maps to $(X^\bullet(\mbb A) \cap X(\mbb O))/G(\mbb O)$, i.e., 
the locus of regular maps. 

\subsubsection{Positively graded subscheme of factorizable affine Grassmannian}\label{sect:strata-Grassmannian}
Let $\ch\Theta = \sum_{\ch\theta} N_{\ch\theta} 
[\ch\theta] \in \Sym^\infty(\mathfrak c_X^- \sm 0)$ 
be a partition. We have a closed 
subscheme
\[ \ol\Gr^{\ch\theta}_{G,C} = \ol\Gr_G^{\ch\theta} \ttimes C := \ol\Gr_G^{\ch\theta} \xt^{\on{Aut}^0( k\tbrac t )} \on{Coord}^0(C) \subset \Gr_{G,C} \] 
where $\Aut^0(k\tbrac t) = \Spec k[a_1^{\pm 1}, a_2,\dotsc]$ is the group scheme of algebra automorphisms of $k\tbrac t$ that preserve the maximal ideal, and $\on{Coord}^0(C)\to C$ is the $\Aut^0(k\tbrac t)$-torsor classifying $v\in C$ together with an isomorphism $k\tbrac t \cong \mf o_v$ sending $t$ to a uniformizer 
(see \S\ref{sect:Llocalization} or \cite[(3.1.11)]{Xinwen}).
Consider the $N_{\ch\theta}$-fold product $(\ol\Gr_{G,C}^{\ch\theta})^{N_{\ch\theta}} \xt_{C^{N_{\ch\theta}}} {\oo C^{N_{\ch\theta}}}$ restricted to the disjoint locus with all diagonals removed. 
This descends to a subscheme $\ol\Gr^{N_{\ch\theta}\ch\theta}_{G, \oo C^{(N_{\ch\theta})}} \subset \Gr_{G, \oo C^{(N_{\ch\theta})}}$. 
In the notation of \S\ref{def:partition}, let 
\[ \ol\Gr^{\ch\Theta}_{G, \oo C^{\ch\Theta}} := \oo\prod_{\ch\theta} \ol\Gr^{N_{\ch\theta}\ch\theta}_{G, \oo C^{(N_{\ch\theta})}} 
\subset \Gr_{G, C^{(\abs{\ch\Theta})}} \xt_{C^{(\abs{\ch\Theta})}} 
\oo C^{\ch\Theta} \]
and let $\ol\Gr^{\ch\Theta}_{G,C^{\ch\Theta}}$ denote its closure 
in $\Gr_{G,C^{(\abs{\ch\Theta})}} \xt_{C^{(\abs{\ch\Theta})}} C^{\ch\Theta}$.
We consider $\ol\Gr^{\ch\Theta}_{G, C^{\ch\Theta}}$ as a scheme 
over $C^{(\abs{\ch\Theta})}$ with an action of the group scheme 
$(\Scr L^+ G)_{C^{(\abs{\ch\Theta})}}$, the multi-point version of the arc space 
defined in \S\ref{sect:multijet}. 

\medskip

The closure relations of the Beilinson--Drinfeld affine Grassmannian 
are known (cf.~\cite[Proposition 3.1.14]{Xinwen}), so we can describe
the reduced fiber of $\ol\Gr^{\ch\Theta}_{G,C^{\ch\Theta}}$ 
over a point of $C^{\ch\Theta}$ as follows: 
a point of $C^{\ch\Theta}$ is the collection $(D^{\ch\theta})_{\ch\theta}$, 
for each $\ch\theta$,
of a degree $N_{\ch\theta}$ divisor 
$D^{\ch\theta} = \sum_{v\in \abs C} N_{\ch\theta,v} v$. 
The reduced fiber of $\ol\Gr^{\ch\Theta}_{G,C^{\ch\Theta}}$ 
over this point is the scheme
\begin{equation} \label{e:GrTheta-fiber}
\prod_{v\in \abs C} \ol\Gr^{\sum_{\ch\theta} N_{\ch\theta,v}\ch\theta}_{G,v}. 
\end{equation}

\subsubsection{}  \label{sect:Glevel}
Let $\widehat\sM_X$ denote the stack 
representing the data of 
\[ (\sigma,\Scr P_G)\in \Scr M_X,\, D\in \Sym C,
\text{ and a trivialization }\Scr P_G|_{\wh C'_D} \cong \Scr P^0_G|_{\wh C'_D}
\] 
(see \S\ref{sect:BLthm} for the definition of $\wh C'_D$), i.e., 
a point of $\sM_X$ together with infinite $G$-level structure at points in the support of $D$. 
This admits a natural action by $\Scr L^+ G$, and 
the forgetful map $\widehat{\Scr M}_X
 \to \Scr M_X \xt \Sym C$ is a $\Scr L^+G$-torsor.
Let 
\[ \Scr M_X \ttimes \ol\Gr^{\ch\Theta}_{G, C^{\ch\Theta}} := 
\widehat{\Scr M}_X \xt^{\Scr L^+ G}_{\Sym C}
\ol\Gr^{\ch\Theta}_{G,C^{\ch\Theta}} \]
denote the twisted product over $C^{(\abs{\ch\Theta})} \subset \Sym C$.

For an affine test scheme $S$, 
an $S$-point of $\Scr M_X \ttimes \ol\Gr^{\ch\Theta}_{G,C^{\ch\Theta}}$ 
is the data 
$(\sigma, \Scr P_G,\Scr P'_G, (D^{\ch\theta})_{\ch\theta}, \tau)$
where 
\begin{itemize}
\item $\Scr P_G,\Scr P_G'$ are $G$-bundles on $C\xt S$,
\item $\sigma : C\xt S \to X\xt^G \Scr P_G$ is a section such that $(\sigma,\Scr P_G)\in \Scr M_X(S)$,
\item for each $\ch\theta \in \mathfrak c_X^- \sm 0$,  
we have a degree $N_{\ch\theta}$ effective Cartier divisor 
$D^{\ch\theta} \subset C\xt S$, 
\item $\tau : \Scr P'_G|_{C\xt S\sm D} \cong \Scr P_G|_{C\xt S\sm D}$
is a trivialization for $D:=\sum_{\ch\theta} D^{\ch\theta}$
\end{itemize}
such that after fixing an isomorphism 
$\Scr P_G|_{\wh C'_D} \cong \Scr P^0_G|_{\wh C'_D}$ (which always exists
after flat base change over $S$), the datum
$((D^{\ch\theta}),\Scr P'_G|_{\wh C'_D},\tau)$ defines
a point in $\ol\Gr^{\ch\Theta}_{G,C^{\ch\Theta}}$.
Here we are implicitly using Beauville--Lazlo's theorem to pass between
the global and local descriptions of $\Gr_{G,\Sym C}$, 
cf.~\S\ref{sect:localGr}.

\begin{prop-const}
\label{prop-const:action}
For any $\ch\Theta \in \Sym^\infty(\mathfrak c_X^-\sm 0)$
there is a natural proper map 
\[ \act_\sM : \Scr M_X \ttimes \ol\Gr^{\ch\Theta}_{G, C^{\ch\Theta}} \to 
\Scr M_X. \]
\end{prop-const}

\begin{proof}
Fix an affine test scheme $S$ and $(\sigma,\Scr P_G,\Scr P'_G, (D^{\ch\theta})_{\ch\theta},\tau) \in \Scr M_X \ttimes \ol\Gr^{\ch\Theta}_{G,C^{\ch\Theta}}(S)$. 
The composition $\tau^{-1} \circ \sigma|_{C\xt S\sm D}$ defines
a section $\sigma' : C\xt S \sm D \to X\xt^G \Scr P'_G$. 
We claim that $\sigma'$ extends to a regular map on $C\xt S$. 
Given this claim, we can define $(\sigma',\Scr P'_G)\in \Scr M_X(S)$ 
to be the image of $\act_\sM$. Properness of $\act_\sM$
follows from properness of $\ol\Gr^{\ch\Theta}_{G,C^{\ch\Theta}}$
and Lemma~\ref{lem:extendsection}.

We now prove the claim that $\sigma'$ extends to all of $C\xt S$.
Since $k[X]$ is a locally finite $G$-module, it is generated as an
algebra by some finite dimensional $G$-submodule $V \subset k[X]$. 
This induces a $G$-equivariant embedding of varieties $\varphi: X\into V^*$,
where $V^{*}$ is considered as a right $G$-module. 
It suffices to show that the composition $\varphi(\sigma'):
C\xt S\sm D \to V^{*} \xt^G \Scr P'_G$ extends. 
For any weight $\mu$ of $V \subset k[X]$, we have $\mu \le \lambda$ for $\lambda\in \mf c_X^\vee$. 
Therefore given $\ch\theta\in \mathfrak c_X^-$, we have $\brac{\mu,\ch\theta}\ge 0$
for all weights $\mu$ of $V$. 
We deduce that the group homomorphism $G \to \GL(V)$ induces a natural map
\[ \ol\Gr^{\ch\Theta}_{G,C^{\ch\Theta}} \to 
\Scr L^+\mathrm{End}(V)/\Scr L^+\GL(V). \]
Hence $\tau^{-1}$ induces a regular map 
$V^{*} \xt^G\Scr P_G|_{\wh C'_D} \to V^{*} \xt^G\Scr P'_G|_{\wh C'_D}$
and $\varphi(\tau^{-1}\circ \sigma|_{\wh C'_D})$ defines a section 
$\wh C'_D \to V^{*} \xt^G \Scr P'_G$. 
By Beauville--Laszlo's theorem (cf.~\cite[Theorem 2.12.1]{BD}), 
this implies that $\varphi(\sigma')$ is defined on all of $C\xt S$.
\end{proof}

We describe more precisely what $\act_\sM$ is doing on $k$-points: 
at a single $v \in \abs C$, a $k$-point of $\Scr M_X$ gives an element
of $(X(\mf o_v)\cap X^\bullet(F_v)) / G(\mf o_v)$ and the map $\act_\sM$ corresponds
to the natural $G(F_v)$-action on $X^\bullet(F_v)$.
For $\ch\mu \in \ch\Lambda^-_G$, 
define the set  
\[ X^\bullet(F_v)_{G:\succeq \ch\mu} = \bigcup_{\substack{\ch\mu' \in \ch\Lambda^-_G, \ch\mu' \succeq \ch\mu}} X^\bullet(F_v)_{G:\ch\mu'}, \]
in the notation of \S \ref{sect:stratLXpoint}.
In the next subsection we prove a slightly more precise\footnote{In \cite{SV} 
the ordering $\succeq$ is defined with respect to the \emph{rational} cone generated by
the valuations of colors, whereas we define $\succeq$ with respect
to the monoid generated by non-negative \emph{integral} combinations of valuations of colors.} 
version of \cite[Lemma 5.5.2]{SV}:

\begin{lem} \label{lem:act-kpoints}
Let $\ch\mu,\ch\theta \in \ch\Lambda^-_G$. 
The action map sends 
\[
    X^\bullet(F_v)_{G:\succeq \ch\mu} \xt^{G(\mf o_v)} \ol{\msf L^+G \cdot t^{\ch\theta} \cdot \msf L^+G}(k) \to X^\bullet(F_v)_{G:\succeq \ch\mu+ \ch\theta}. 
\]
\end{lem}

\subsection{Reduction to $\mf c_{X^\bullet} = \mbb N^{\Cal D}$} \label{sect:freemonoid}
We are interested in applying Hecke actions to $\barM_X^0$, the closure of $\sM_{X^\bullet}=\sM_X^0=\Bun_H$ in $\sM_X$, since it is the most basic closure of a stratum. 
On the other hand, in order to determine the stratification of $\barM_X^0$ we
need a moduli description of this stack. A first guess would be 
that $\barM_X^0 = (\sM_{X^\can})_\red$, but this may not be true if $\mf c_{X^\bullet} := \mf c_{X^\can}$ 
is not equal to $\mathfrak c_X^{\Cal D} = \{\ch\lambda\succeq 0\}$.
In this subsection we explain how to get around this technical issue: 
we can always replace $G$ by a central extension $G'\to G$ such that if
$H\bs G = H'\bs G'$ and $X' := \Spec k[H'\bs G']$, then 
$\barM_{X'}^0 = \sM_{X'}$ and $\mf c_{X'} = \mbb N^{\Cal D}$ (see Lemma~\ref{lem:Xcover}).

\smallskip 

This is a generalization of the need to replace $\ol{N\bs G}^\aff$ by 
$\ol{N\bs \wt G}^\aff$ for $\wt G$ a simply connected cover of $G$ 
in \cite[\S 4.1]{ABB}, \cite[\S 7.2]{Sch} 
to correctly define Drinfeld's compactification of $\Bun_N$ 
for an arbitrary reductive group $G$.

\subsubsection{} 

Let us first assume that $k[G]$ is a UFD. Further, assume that $H$ is connected, as is the case under our assumptions (Remark~\ref{rem:Hconnected}).
Then the preimage of any color $D$ in $G$ is irreducible, and this defines 
a bijection between $H\times B$-stable prime divisors in $G$ and colors; moreover, under our UFD assumption, the former are all principal. 

Let $k(G)^{(H\xt B)}$ denote the $H\xt B$-eigenvectors 
of $k(G)$. 
Then, since $H B$ is open dense in $G$, we have
\[ k(G)^{(H\times B)}/k^\times = \Cal X(H) \xt_{\Cal X(B\cap H)} \Cal X(B), 
\] 
where $\Cal X(H)$ is the character group of $H$ 
(so $\Cal X(B)=\Lambda_{G}$ by definition). The valuation map gives rise to a short exact sequence
\begin{equation} \label{e:color-basis}
 0 \to \Cal X(G) \to  
    \Cal X(H) \xt_{\Cal X(B \cap H)} \Cal X(B) \to   \mbb Z^{\Cal D} \to 0,
\end{equation}
which sends an element of $k(G)^{(H\times B)}/k^\times$ to its divisor, identified with a $\mathbb Z$-linear combination of the colors. (We have used here both the UFD property, and the fact that invertible regular functions on $G$ are multiples of characters.)

The preimage in $G$ of each color $D$ defines a valuation $v_{HD}$ on $k(G)^\times$; its restriction to $k(G)^{(H\xt B)}$ defines a map
\[ 
    \wt\varrho_{H\bs G} : \Cal D \to (\Cal X(H) \xt_{\Cal X(B\cap H)} \Cal X(B))^\vee. 
\]
Composing with the second projection to $\ch\Lambda_X$ gives the usual
valuation map $\varrho_{H\bs G}$.
By construction we have $v_{HD}(f_{D'}) = \delta_{D,D'}$ 
for $D,D'\in \Cal D$, and the sequence dual to \eqref{e:color-basis} is 
\[ 
    0\to \mbb Z^{\Cal D} \overset{\wt\varrho_{H\bs G}}\longrightarrow 
(\Cal X(H) \xt_{\Cal X(B\cap H)} \Cal X(B))^\vee \to 
\Cal X(G)^\vee \to 0.
\]

\subsubsection{} \label{sect:Xcover}
Now we return to the case of arbitrary $G$. Let 
\[ 1 \to Z \to \wt G \to G \to 1 \] 
be a central extension with connected 
kernel $Z$ such that the derived group $[\wt G, \wt G]$ is simply
connected (such a $\wt G$ always exists). 
Then $k[\wt G]$ is a UFD (\cite[Proposition 4.6]{KKLV}, \cite{Iv}).
We consider $H\bs G$ as a $\wt G$-spherical variety, so 
$H\bs G = \wt H \bs \wt G$ where $\wt H = HZ$ is the preimage of $H$
in $\wt G$. Let $\wt B$ denote the Borel subgroup of $\wt G$.
If $H$ is connected, then $\wt H$ is also connected.
The colors stay the same, so 
$\wt\varrho_{\wt H\bs \wt G}$ induces a short exact sequence
\begin{equation} \label{e:color-ses} 
0 \to \mbb Z^{\Cal D} \overset{\wt\varrho_{\wt H\bs \wt G}}\longrightarrow 
(\Cal X(\wt H) \xt_{\Cal X(\wt B\cap \wt H)} \Cal X(\wt B))^\vee \to 
\Cal X(\wt G)^\vee \to 0.
\end{equation}

\smallskip

Let $\wt H_{\mathrm{ab}} = \wt H/[\wt H,\wt H]$, so 
$\Cal X(\wt H) = \Cal X(\wt H_{\mathrm{ab}})$. 
We consider $[\wt H,\wt H] \bs \wt G$ 
as a spherical variety for the group $G' := \wt H_{\mathrm{ab}} \overset Z\times \wt G$,
where $\wt H_{\mathrm{ab}}$ acts by \emph{left} translation and 
$\wt G$ acts by right translation. 
Then 
$[\wt H,\wt H] \bs \wt G = H' \bs G'$ 
where $H'=(\wt H/Z)^\diag=H^\diag$ is diagonally embedded in $G'$. 
Observe that $G' \to G$ is a central extension
with kernel $\wt H_{\mathrm{ab}}$
and 
\[ H'\bs G' = [\wt H,\wt H]\bs \wt G \to \wt H \bs \wt G = H\bs G \]
is a $\wt H_{\mathrm{ab}}$-torsor. 
The Borel subgroup $B'$ of $G'$ is $\wt H_{\mathrm{ab}} \xt^Z \wt B$
and $B'\cap H' = ((\wt B \cap \wt H)/Z)^\diag$, so 
 \eqref{e:color-ses} is equivalent to a short exact sequence
\[ 
    0 \to \mbb Z^{\Cal D} \overset{\varrho_{H'\bs G'}}\longrightarrow
    \ch\Lambda_{H'\bs G'} = \ker(\Cal X(B')\to \Cal X(B'\cap H'))^\vee 
    \to \Cal X(\wt G)^\vee \to 0. 
\]

To summarize:
\begin{lem} \label{lem:Xcover}
Let $H\bs G$ be a homogeneous spherical variety with $H$ connected. 
Then there exists a central extension $G' \to G$ 
with connected kernel $Z'$ and a spherical
subgroup $H'\subset G'$ such that 
\begin{enumerate}
\item $[G',G']$ is simply connected. 
\item The covering $G'\to G$ restricts to an isomorphism $H'\cong H$. 
In particular, 
the projection $H' \bs G' \to H \bs G$ is a $Z'$-torsor.
\item 
The valuation map $\varrho_{H'\bs G'}$ embeds 
$\mbb Z^{\Cal D} \into \ch\Lambda_{H'\bs G'}$
as a direct summand. 
\end{enumerate}
\end{lem}

\noindent By (iii), we have that $\mf c_{H'\bs G'} = \mathfrak c^{\Cal D}_{H'\bs G'}= \mbb N^{\Cal D}$ so there is no question of integrality. For a similar result, see \cite[Lemma 2.1.1]{Brion07}.

\begin{eg}
If we start with $H = \mbb G_m$ the torus inside $G = \PGL_2$, then
$G' = \mbb G_m \xt \GL_2$ and $H'= \mbb G_m$ where $\mbb G_m$ 
maps to $\GL_2$ by $\left( \begin{smallmatrix} * & 0 \\ 0 & 1 \end{smallmatrix} \right)$. So $H'\bs G' = \GL_2$, where $\mbb G_m$ acts by 
left translations and $\GL_2$ by right translations.
This is a $\mbb G_m$-torsor over $\mbb G_m \bs \GL_2$ and a 
$\mbb G_m \xt Z(\GL_2)$-torsor over $\mbb G_m \bs \PGL_2$.
\end{eg}

\begin{proof}[{Proof of Lemma~\ref{lem:act-kpoints}}]
Let $1\to Z'\to G'\to G\to 1$ be a central extension as in Lemma~\ref{lem:Xcover}.  
Since $\mf c_{H'\bs G'} = \mathfrak c^{\Cal D}_{H'\bs G'}$, 
there is no question of integrality and \cite[Lemma 5.5.2]{SV} 
applied to the $G'$-variety $H'\bs G'$ says that 
the action map sends 
\[ (H'\bs G')(F_v)_{\succeq \tilde\mu} \xt^{G'(\mf o_v)} G'(\mf o_v) t^{\tilde\theta} G'(\mf o_v) \to (H'\bs G')(F_v)_{\succeq \tilde\mu+\tilde\theta}, \]
where $\tilde\mu, \tilde\theta \in \ch\Lambda_{H'\bs G'}^-$. 
Since $H'\bs G' \to H\bs G = X^\bullet$ is a $Z'$-torsor, 
it is surjective on $F_v$-points (and corresponding orbits). 
Since $\mathfrak c^{\Cal D}_{H'\bs G'}$ maps to $\mathfrak c_X^{\Cal D}$, 
we deduce that the action map sends 
\[ X^\bullet(F_v)_{\succeq \ch\mu} \xt^{G(\mf o_v)} G(\mf o_v) t^{\ch\theta} G(\mf o_v)
\to X^\bullet(F_v)_{\succeq \ch\mu+\ch\theta} \]
for $\ch\mu,\ch\theta \in \ch\Lambda^-_G$.
The preimage of $\ol\Gr^{\ch\theta}_G(k)$ in $G(F_v)$ consists of the union of 
$G(\mf o_v) t^{\ch\theta'} G(\mf o_v)$ for $\ch\theta'\in \ch\Lambda^-_G,\,
\ch\theta' \ge \ch\theta$. 
Since $\ch\theta' \ge \ch\theta$ implies $\ch\theta' \succeq \ch\theta$, 
we deduce the claim. 
\end{proof}

\subsection{Open Zastava} \label{sect:Y0cc}
Consider $\Scr Y^{?,0} = \sY_{X^\bullet}=\sY \xt_{\sM_X} \sM_X^0 = \Maps_\gen(C, H\bs G/B \supset \pt)$, which 
is an open subscheme of $\sY$. 
Since $\sM_X^0 \cong \Bun_H$ is smooth and $\sY_{X^\bullet}$ is smooth locally isomorphic to 
$\sM_X^0$ by Lemma~\ref{lem:localglobalyoga}, the scheme $\sY_{X^\bullet}$ is also smooth. 

The preimage of the connected component $\sA^{\ch\lambda}$ in $\sY_{X^\bullet}$ 
is by definition $\sY^{\ch\lambda,0}$. 
Let $G' \to G$ be a central extension as in Lemma~\ref{lem:Xcover}, so 
we have a torsor $H'\bs G' \to X^\bullet$. 
Let $X' = \Spec k[H'\bs G']$. 
Then $\mathfrak c_{X'} = \mathbb N^{\Cal D}$ and we have an isomorphism of stacks
$X'^\bullet / B' \cong  X^\bullet/ B$.
Hence, 
\begin{equation}\label{e:YXprime}
    \sY_{X^\bullet} \cong \sY_{X'^\bullet} = \Maps_\gen( X'^\bullet/B' \supset \pt ),
\end{equation}
and the map $\sY_{X^\bullet} \to \sA$ factors through $\pi_{X'} : \sY_{X'} \to \sA_{X'}$.
The base $\sA_{X'}$ is a disjoint union of smooth partially symmetrized powers of the curve indexed by $\mathbb N^{\Cal D}$. 
 For $D = \sum_{\Cal D} n_{D'}\cdot D'\in \mathbb N^{\Cal D}$, we will denote by 
$\sY_{X^\bullet}^D$ the preimage of the corresponding component of $\sA_{X'}$. 
We define $\varrho_X(D) = \sum n_{D'}\ch\nu_{D'}$ and $\len(D) = \sum n_{D'}$. Hence, if $\ch\lambda = \varrho_X(D)$, then $\sY_{X^\bullet}^D \subset \sY_{X^\bullet}^{\ch\lambda}$. 

This implies: 

\begin{lem} \label{lem:Ynonempty}
The stratum $\Scr Y^{\ch\lambda,0}$ is nonempty only if $\ch\lambda \succeq 0$.
\end{lem}

Under our assumption that all simple roots of $G$ are spherical roots of type $T$, 
every color $D\in \Cal D$ belongs to $\Cal D(\alpha)$ for some simple root $\alpha$ of $G$
and $X^\circ P_\alpha/\mf R(P_\alpha) = \mbb G_m \bs \PGL_2$. 

\begin{lem} \label{lem:Ycolor}
If $\Cal D(\alpha) = \{ D_\alpha^+, D_\alpha^- \}$ for a simple root $\alpha$,
then for any $n^\pm \in \mbb N$, 
there is an isomorphism 
\[ \sY^{n^+ D_\alpha^+ + n^- D_\alpha^-}_{X^\bullet} = C^{(n^+)}\oo\xt C^{(n^-)}. \]
\end{lem}
\begin{proof} 
We may assume $\mf c_{X^\bullet} = \mbb N^{\Cal D}$. Then, 
$\sY^{n^+ D_\alpha^+ + n^- D_\alpha^-}_{X^\bullet}$
classifies maps from the curve to $X^\bullet/B$ which have zero valuation on every color other than those in $\Cal D(\alpha)$, and therefore is a subscheme of $\Maps_\gen(C, (X^\bullet-\bigcup_{D\in \Cal D \sm \Cal D(\alpha)} D)/B \supset \pt)$. 
Since $X^\circ P_\alpha = \mathbb G_m\backslash P_\alpha$ is affine, its complement is the union of the colors that it does not intersect. Hence 
\[ (X^\bullet-\bigcup_{D\in \Cal D \sm \Cal D(\alpha)}D)/B = (X^\circ P_\alpha)/B = \mathbb G_m\bs\PGL_2/B_{\PGL_2} = \mbb G_m \bs \mbb P^1.\]
By Example~\ref{eg:YHecke}, we see that $\Maps_\gen(C, \mbb G_m \bs \mbb P^1 \supset \pt) = \Sym C \oo\xt \Sym C$,
and it follows that the components correspond to $\mbb N^{\Cal D(\alpha)}$ in
the natural way.
\end{proof}

\begin{rem} \label{rem:Ycomp}
For a general element $D = \sum_{\Cal D} n_{D'}\cdot D' \in \mbb N^{\Cal D}$, 
the graded factorization property together with Lemma~\ref{lem:Ycolor} imply
that there is an open embedding 
$\oo\prod_{\Cal D} C^{(n_{D'})} \into \sY^D_{X^\bullet}$,
where the product is over the disjoint locus. 
We will show in Lemma~\ref{lem:oYconnected} that 
$\sY^D_{X^\bullet},\, D\in \mbb N^{\Cal D}$ are precisely the connected components
of $\sY_{X^\bullet}$, so the open subscheme above is dense. 
(This also implies that the $\sY^D_{X^\bullet}$ are defined intrinsically and independently
of the choice of $G'\to G$.)
\end{rem}

\begin{rem}\label{rem:length}
If $D$ is as above, and $\ch\lambda = \varrho_X(D)$, we can read off the length $\len (D)$ directly from $\ch\lambda$, if $X^\bullet$ happens to admit a $G$-eigen-volume form. Namely, under this assumption there is a character $\gamma\in \Lambda$ such that $\left< \gamma, \ch \nu_{D'}\right> = 1$ for every color $D'$, and therefore $\len (D) = \left< \gamma, \ch\lambda\right>$. Indeed, for such a color $D'$, there is a simple root $\alpha$ such that $D'P_\alpha$ contains the open Borel orbit, and then $D'P_\alpha \simeq \ch\nu_{D'}(\mathbb G_m)\backslash P_\alpha$. For such a homogeneous space to have a $P_\alpha$-eigen-volume form with eigencharacter $\mathfrak h$,  we must have that $\left< \mathfrak h+ 2\rho_{N_\alpha}, \ch\nu_{D'}\right> = 0$, where  $2\rho_{N_\alpha}$ is the sum of roots in the unipotent radical of $P_\alpha$. Equivalently, since $\left<\alpha,\ch\nu_{D'}\right>=1$, this reads $\left<\mathfrak h + 2\rho_G, \ch\nu_{D'}\right> = 1$. Therefore, 
\begin{equation}\label{e:length}
 \len (D) = \left< \mathfrak h + 2\rho_G, \ch\lambda\right>,
\end{equation}
which we can unambiguously denote by $\len (\ch\lambda)$.
\end{rem}

\subsection{Proof of Theorem \ref{thm:comp}} 
\label{proof:comp}

\subsubsection{Idea of the proof}

Let us give a set-theoretic idea for the proof of Theorem \ref{thm:comp}, as well as a guide to new notation we are about to introduce. The discussion here is not rigorous, using sets as avatars for geometric objects.

We have been using the notation $X^\bullet(F)_{G:\check\theta}$, or $(X^\bullet(F)/K)_{G:\check\theta}$ to denote the set of points in the $K$-orbit parametrized by $\ch\theta \in \mathcal V \cap \check\Lambda_X$, where $K = G(\mf o)$. The stratum  $\sM_X^{\check\theta}$ of the global model corresponds to $(X(\mathfrak o)/K)_{G:\check\theta}$, in the sense that it is determined by the condition that the ``$G$-valuation'' of maps in this stratum, at points where they fail to land in $X^\bullet$, is given by $\ch\theta$. 

Similarly, $T(\mathfrak o)N(F)$-orbits on $X^\circ(F)$ are parametrized by $\check \Lambda_X$, and let us denote by $X(F)_{B:\check\lambda}$ the subset of points which: (1) belong to $X^\circ (F)$, and (2) belong to the $T(\mathfrak o)N(F)$-orbit parametrized by $\check\lambda\in\check\Lambda_X$. As with the global model, the \emph{central fibers} of the stratum $\sY_X^{\check\lambda,\check\theta}$ of the Zastava space correspond to $(X(\mathfrak o)/B(\mathfrak o))_{\substack{G:\check\theta \\ B:\check\lambda}}$.

The following facts will be proven about the space $\sY_X^{\check\lambda,\check\theta}$:
\begin{itemize}
 \item It is nonempty only if $\check\lambda \succeq \check\theta$ (see Lemma \ref{lem:Ynonempty} and Corollary \ref{cor:Ytheta}). This is essentially a statement about the image of $X(\mathfrak o)$ in $X\sslash N(F)$. 
 \item It contains $\sY_X^{\check\lambda-\check\theta,0} \mathring\times \sY_X^{\check\theta,\check\theta}$ as an open dense (see Lemma \ref{lem:comp-cover}). This is a geometric statement, and it follows from the constructions that we describe below. 
\end{itemize}

Theorem \ref{thm:comp} says that the geometric analog of the action map 
\begin{equation}
 \label{eq:action-sets}
\act: \overline{X^\bullet(\mathfrak o)} \times^{K}  \overline{K t^{\check\theta} K} \to \overline{X(\mathfrak o)}_{G:\check\theta}
\end{equation}
is birational and proper. (The closures are also understood here in the ``geometric'' topology, see Lemma \ref{lem:act-kpoints}.) 

To prove this, we fix $\check\theta \in \mathfrak c_X^-$, which will not appear in the notation, and work with $B(\mathfrak o)$-orbits, defining spaces whose central fibers satisfy the following set-theoretic analogies:
\begin{center}
\begin{tabular}{p{20mm}p{10mm}p{7cm}}
 $Z^{\check\lambda - \check\nu, \check\lambda}$ && $\overline{X^\bullet(\mathfrak o)}_{B:\check\lambda-\check\nu} \overset{B(\mathfrak o)}\times (T(\mathfrak o) t^{\check\nu} N(F) \cap \overline{K t^{\check\theta} K})/B(\mathfrak o)$ \\
 $Z^{\check\lambda - \check\nu, \check\lambda, \check\theta',\check\eta}$ && $\overline{X^\bullet(\mathfrak o)}_{\substack{G:\check\theta' \\ B:\check\lambda-\check\nu}} \overset{B(\mathfrak o)}\times (T(\mathfrak o) t^{\check\nu} N(F) \cap K t^{\check\eta} K)/B(\mathfrak o)$ \\ 
\end{tabular}
\end{center}

Here, $\ch\eta \ge \ch\theta$ (and antidominant), so that the Cartan double coset $K t^{\check\eta} K$ belongs (in the affine Grassmannian) to the closure of $K t^{\check\theta} K$. Both of the spaces above are subspaces of the preimage of $X(\mathfrak o)_{B:\check\lambda}$ in $\overline{X^\bullet(\mathfrak o)} \times^{K} \overline{K t^{\check\theta} K}$ --- in fact, substacks, in the appropriate setting. Indeed, stratifying $\overline{K t^{\check\theta} K}$ by the ``Mirkovi\'c--Vilonen (MV) cycles'' corresponding to its intersection with the horocycles $K t^{\check\nu} N(F)$, we obtain a stratification of this space by 
\[\left(\overline{X^\bullet(\mathfrak o)}  \times^K (K t^{\check\nu} N(F) \cap \overline{K t^{\check\theta} K})\right)/B(\mathfrak o) \cap \act^{-1} (X(\mathfrak o)/B(\mathfrak o))_{B:\check\lambda}\]
(where $\act$ denotes the action map). Note that $K t^{\check\nu} N(F) /B(\mathfrak o) = K\times^{B(\mathfrak o)} T(\mathfrak o) t^{\check\nu} N(F) /B(\mathfrak o)$, hence the above can also be written 
\[\left(\overline{X^\bullet(\mathfrak o)}  \times^{B(\mathfrak o)} (T(\mathfrak o) t^{\check\nu} N(F) \cap \overline{K t^{\check\theta} K})\right)/B(\mathfrak o) \cap \act^{-1} (X(\mathfrak o)/B(\mathfrak o))_{B:\check\lambda}.\]
An element of $\overline{X^\bullet(\mathfrak o)}$, multiplied by $t^{\check\nu} N(F)$, lands in $(X(\mathfrak o)/B(\mathfrak o))_{B:\check\lambda}$ if and only if it belongs to $\overline{X^\bullet(\mathfrak o)}_{B:\check\lambda-\check\nu}$, showing that the sets corresponding to $Z^{\check\lambda - \check\nu, \check\lambda}$ (and also to $Z^{\check\lambda - \check\nu, \check\lambda, \check\theta',\check\eta}$, after a further stratification by $G$-valuations) are indeed subspaces of $\overline{X^\bullet(\mathfrak o)} \times^{K} \overline{K t^{\check\theta} K}$.

Here comes the important geometric input: the topology on the affine Grassmannian, resp.\ on the arc space of $X$ (or rather, its global models), implies that $Z^{\check\lambda - \check\theta, \check\lambda, 0,\check\theta}$ is open and dense. This is a combination of the statements that $X^\bullet(\mathfrak o)$ is (tautologically) open dense in $\overline{X^\bullet(\mathfrak o)}$ (allowing us to take $\check\theta'=0$), that the Cartan double coset $K t^{\check\theta} K$ is (again tautologically) open dense in its closure, and that the MV cycle $K t^{\ch\theta} K \cap K t^{\ch\theta} N(F)$ is open in it. 

The open MV stratum on the affine Grassmannian $K\backslash G(F)$ meets the coset $K\backslash K t^{\check\theta} K$ along the $N(\mathfrak o)$-orbit $K \backslash K t^{\ch\theta} N(\mf o)$. Therefore, the quotient
\[(T(\mathfrak o)t^{\check\theta} N(F) \cap K t^{\check\theta} K)/B(\mathfrak o)\]
is just a point ($\ch\theta$ is anti-dominant).
Thus, 
\[X^\bullet(\mathfrak o)_{B:\check\lambda-\ch\theta} \times^{B(\mathfrak o)} (T(\mathfrak o) t^{\check\theta} N(F) \cap K t^{\check\theta} K)/B(\mathfrak o) = (X^\bullet(\mathfrak o)/B(\mathfrak o))_{B:\check\lambda-\ch\theta},\]
which, on the other hand, can be identified as an open dense subset of $(X(\mathfrak o)/B(\mathfrak o))_{\substack{G:\ch\theta \\ B:\check\lambda}}$. Thus, in the geometric setting, the action map \eqref{eq:action-sets}, restricted to the subsets with $B$-valuation equal to some (any) $\check\lambda$, is indeed birational. On the other hand, the map $\sY^{\check\lambda}\to \sM_X$ is smooth for large $\check\lambda$ (Corollary \ref{cor:Mcover}), proving the birationality of (the geometric version of) \eqref{eq:action-sets}.

\subsubsection{}

We now turn to the actual proof of Theorem \ref{thm:comp}.

Since $\ol\Gr^{\ch\Theta}_{G,C^{\ch\Theta}}$ lives over $C^{\ch\Theta}$, 
we can factor $\act_\sM$ into 
\begin{equation} \label{e:actC}
    \Scr M_X \ttimes \ol\Gr^{\ch\Theta}_{G,C^{\ch\Theta}} \overset{\act_{C^{\ch\Theta}}}\longrightarrow \Scr M_X \xt C^{\ch\Theta}
\overset{\pr_1}\longrightarrow \sM_X
\end{equation}
where $\act_{C^{\ch\Theta}}$ is the naturally induced map over $C^{\ch\Theta}$.

\begin{lem} \label{lem:actimage} 
Let $\ch\Theta \in \Sym^\infty(\mathfrak c_X^- \sm 0)$. 
\begin{enumerate}
\item
The preimage of $\Scr M^{\ch\Theta}_X$ in 
$\Scr M_{X} \ttimes \ol\Gr^{\ch\Theta}_{G,C^{\ch\Theta}}$ 
under $\act_\sM$ 
is contained in the open substack
$\Bun_H \ttimes \Gr^{\ch\Theta}_{G,\oo C^{\ch\Theta}}$. 

\item
The image $\act_\sM(\Bun_H \ttimes \ol\Gr^{\ch\Theta}_{G,\oo C^{\ch\Theta}}) \subset \Scr M_X$ contains $\Scr M_X^{\ch\Theta}$ as an open dense substack. 
\end{enumerate}
\end{lem}
\begin{proof}
Statement (i) follows from the description of
$\act_\sM$ on $k$-points, Lemma~\ref{lem:act-kpoints}, and the fact that no element of $\mathfrak c_X^- \sm 0$ is $ \preceq 0 $. 

To simplify notation, we give the proof of (ii) in the case when
 $\ch\Theta = [\ch\theta], \ch\theta\ne 0$ is a singleton partition. 
The general case is entirely analogous.

Now we are considering $\act_C : \sM_X \xt \ol\Gr^{\ch\theta}_{G,C} \to \sM_X \times C$. 
Recall that we have a 
locally closed embedding $\sM_X^{\ch\theta} \into \sM_X \times C$. 
First we show that 
$M:=\act_C(\Bun_H \ttimes \ol\Gr^{\ch\theta}_{G,C}) \subset \sM_X\times C$ contains 
$\Scr M^{\ch\theta}_X$.
Fix $v\in \abs C$. 
Take an arbitrary $H$-bundle $\sP_H \in \Bun_H(k)$, 
and choose a trivialization of $\sP_H|_{\Spec \mf o_v}$, which 
induces a trivialization $\tau_0 : (\sP_H \xt^H G)|_{\Spec \mf o_v} \cong \sP_G^0|_{\Spec \mf o_v}$. 
Then $(\sP_H,v,\tau_0)$ defines a point of $\wh\sM_X$. 
We also have the point $t^{\ch\theta} \in \Gr^{\ch\theta}_{G,v}$. 
The image of $(\sP_H,v,\tau_0, t^{\ch\theta})$ gives a point in 
$\sM_X \ttimes \Gr^{\ch\theta}_{G,v}$,
and by construction $\act_C$ will send this point to the stratum $\sM^{\ch\theta}_X$.

Thus we have shown that  
$M$ contains a point of $\sM^{\ch\theta}_X$ over every point $v\in \abs C$. 
It follows from Lemma~\ref{lem:act-equivariant} that 
$M \subset \Scr M_X \xt C$ 
is stable under
generic-Hecke modifications away from the marked point in $C$. 
Since these generic-Hecke modifications act transitively on 
the stratum $\Scr M^{\ch\theta}_X$ (Proposition~\ref{prop:genHecke-transitive}), we 
deduce that $M$ contains all of $\sM^{\ch\theta}_X$. 

The previous paragraph and Lemma~\ref{lem:act-kpoints} imply that 
\[ \sM_X^{\ch\theta} = M \sm \bigcup_{\ch\theta' \succ \ch\theta} \act_C(\sM_X \ttimes \ol\Gr^{\ch\theta'}_G), \] 
which shows that $\sM_X^{\ch\theta}$ is open in $M$.

Since $\ol\Gr^{\ch\theta}_G$ is irreducible, 
for every connected component $U$ of $\Bun_H$, 
the image $M_U := \act_C(U \ttimes \ol\Gr^{\ch\theta}_{G,C})$ is 
irreducible. 
It is easy to see that its intersection with 
$\sM_X^{\ch\theta}$ is nonempty, hence dense. 
\end{proof}

As a corollary, we can now prove Theorem \ref{thm:comp}(i):

\begin{proof}[Proof of Theorem \ref{thm:comp}(i)]
Let $M$ denote an irreducible component of $\barM^0_X$. 
Then $M':=\act_\sM(M \ttimes \ol\Gr^{\ch\Theta}_{G,C})$
is an irreducible closed substack of $\Scr M_X$, 
and since $M$ contains a connected component of $\Bun_H$,
Lemma~\ref{lem:actimage}(ii) implies that 
 $M' \cap \Scr M^{\ch\Theta}_X$ is dense in $M'$.
It follows that $\Scr M^{\ch\Theta}_X$ is dense in $\act_\sM(\barM^0_X \ttimes \ol\Gr^{\ch\theta}_{G,C})$.
\end{proof}

\subsubsection{Base change to Zastava model}  
Fix $\ch\theta \in \mathfrak c_X^- \sm 0$ 
and a point $v \in \abs C$. 
Consider the restriction of $\act_\sM$ to 
$\sM_{X} \ttimes \ol\Gr^{\ch\theta}_{G,v} \to \sM_X$,
which is the fiber of $\act_C$ over 
$v \to C = C^{[\ch\theta]}$. 

Let us consider for $\ch\lambda\in \mathfrak c_X$ 
the fiber product diagram
\begin{equation}  \label{e:Zact-diagram}
\begin{tikzcd}
Z_v^{?,\ch\lambda} \ar[r, "\act_{\Scr Y}"] \ar[d] & \Scr Y^{\ch\lambda} \ar[d] \\ 
\barM_X^0 \ttimes \ol\Gr^{\ch\theta}_{G,v} \ar[r, "\act_\sM"] & \Scr M_X 
\end{tikzcd} 
\end{equation}
An $S$-point of $Z_v^{?,\ch\lambda}$ consists of the data $(\sigma,\Scr P_G, \Scr P'_B, 
v, \tau)$ where 
\begin{itemize}
\item $(\sigma,\Scr P_G) \in \barM_X^0$,
\item $\Scr P'_B\in \Bun_B^{-\ch\lambda}$ is a $B$-structure on a $G$-bundle $\Scr P'_G := G\xt^B \Scr P'_B$, 
\item $\tau : \Scr P'_G|_{(C\sm v)\xt S} \cong \Scr P_G|_{(C\sm v)\xt S}$ is 
a modification inducing a point in $\ol\Gr^{\ch\theta}_G$ such that 
$\tau^{-1}\circ \sigma$ generically lands in $X^\circ \xt^B \Scr P'_B$. 
\end{itemize}

\noindent
Let $\widehat{\Scr Y} \to \Scr Y$
denote the Zariski locally trivial $\msf  L^+B$-torsor parametrizing 
$(\sigma,\Scr P_B)\in \Scr Y$ and a trivialization 
$\Scr P_B|_{\wh C'_v} \cong \Scr P_B^0|_{\wh C'_v}$, and let $\widehat{\Scr Y}_0 \to \Scr Y_0$ be the restriction to $\barM_X^0$. (We will use the index $0$ for the same purpose on the strata of $\Scr Y$.)

\begin{prop-const} \label{prop:Yact-strat}
\ 
\begin{enumerate}
\item The fiber product $Z_v^{?,\ch\lambda}$ admits a stratification by 
\[
    Z_v^{\ch\lambda-\ch\nu,\ch\lambda} := 
    \widehat{\Scr Y}_0 \xt^{\msf L^+B} 
    (\msf L^+ T \cdot t^{\ch\nu}\cdot \msf L N \cap \ol{ \msf L^+G\cdot t^{\ch\theta} \cdot \msf L^+G}) / \msf L^+B, 
\]
where $\ch\nu$ ranges over the weights of the irreducible $\ch G$-module $V^{\ch\theta}$. 

\item
We have an isomorphism at the level of reduced schemes 
\[ Z_v^{\ch\lambda-\ch\nu,\ch\lambda} \cong 
\Scr Y_0^{\ch\lambda-\ch\nu} \ttimes 
(\msf S^{\ch\nu} \cap \ol\Gr^{\ch\theta}_G). \] 
where $\Scr Y_0^{\ch\lambda-\ch\nu} \ttimes -$ 
means 
$\widehat{\Scr Y}_0
\xt^{\msf L^+B} -$. 

\item
The open stratum corresponds to $\ch\nu=\ch\theta$. More precisely, we have
\[ Z_v^{\ch\lambda-\ch\theta,\ch\lambda} \cong \Scr Y_0^{\ch\lambda-\ch\theta} \xt v, \] 
where $v$ corresponds to the embedding $\{ t^{\ch\theta} \} 
\into \Gr^{\ch\theta}_{G,v}$.
\end{enumerate}
\end{prop-const}

\begin{proof}
Recall that $\ol{\msf L^+G\cdot t^{\ch\theta}\cdot \msf L^+G}$ has a stratification 
by intersecting with $\msf L^+G\cdot t^{\ch\nu}\cdot \msf L N$. 
If we take the quotient on the left by $\msf L^+ G$, then 
under the identification 
$\msf L^+G \bs \msf L G \cong \Gr_G : g \mapsto g^{-1}$, we have
described above the stratification of $\ol\Gr^{-\ch\theta}_G$ 
by MV cycles $\ol\Gr^{-\ch\theta}_G \cap \msf S^{-\ch\nu}$, where
$-\ch\theta$ is dominant. 
It is known by \cite[Theorem 3.2]{MV} that 
$\ol\Gr^{-\ch\theta}_G \cap \msf S^{-\ch\nu}$
is non-empty precisely when $-\ch\nu$ is a weight of 
$V^{-\ch\theta}$, and the open stratum corresponds to 
when $-\ch\nu$ equals $-\ch\theta$.

Now let $(\sigma,\Scr P_G, \Scr P'_B,\tau)$ be an $S$-point of 
$Z_v^{?,\ch\lambda}$. The restriction of $(\Scr P_G,\Scr P'_B,\tau)$
to $\wh C'_v$ gives a point in $\msf L^+G \bs \ol{\msf L^+G \cdot t^{\ch\theta} \cdot \msf L^+ G} / \msf L^+ B$. 
Therefore, by the previous paragraph, we can stratify $Z_v^{?,\ch\lambda}$
by the preimages of 
\[  \msf L^+G \bs (\ol{\msf L^+G\cdot t^{\ch\theta}\cdot \msf L^+G} \cap \msf L^+G\cdot t^{\ch\nu}\cdot \msf L N) / \msf L^+B . \] 
Suppose our $S$-point lies in such a stratum corresponding to $\ch\nu$. 
In particular, $\tau$ 
corresponds to a point in $\msf L^+B \bs (\Gr^{\ch\nu}_B)_\red$,
which means that there exists a $B$-structure $\Scr P_B|_{\wh C'_v}$
on $\Scr P_G|_{\wh C'_v}$ such that $\tau$ gives an isomorphism
of generic $B$-bundles
$\tau : \Scr P'_B|_{\wh C^\circ_v} \cong \Scr P_B|_{\wh C^\circ_v}$.  
By Beauville--Laszlo's theorem, the datum $(\Scr P'_B|_{C\sm v}, \Scr P_B|_{\wh C'_v},\tau)$ descends to a $B$-structure $\Scr P_B$ on $\Scr P_G$
such that 
$\Scr P'_B|_{C\sm v} \cong \Scr P_B|_{C\sm v}$. 
Then $(\sigma, \Scr P_B) \in \sY_0^{\ch\lambda-\ch\nu}(S)$ and 
$(\sigma,\Scr P_B, \Scr P'_B, v, \tau) \in Z_v^{\ch\lambda-\ch\nu,\ch\lambda}(S)$. 
The procedure above can be reversed to see that $Z_v^{\ch\lambda-\ch\nu,\ch\lambda}$ is equal to the entire stratum.
This shows (i). 

\smallskip

Observe that $\msf L^+ T \cdot t^{\ch\nu} \cdot \msf L N = \msf L N \cdot t^{\ch\nu} \cdot \msf L^+B$ and 
there is a map from 
$(\msf L N \cdot t^{\ch\nu} \cdot \msf L^+B \cap \ol{\msf L^+G \cdot t^{\ch\theta} \cdot \msf L^+G})/\msf L^+B$ to $\msf S^{\ch\nu}\cap \ol\Gr^{\ch\theta}_G$ which is an isomorphism at the level of reduced schemes. 
Now (ii) follows from (i).

The open stratum $Z_v^{\ch\lambda-\ch\theta,\ch\lambda}$ 
corresponds to the open stratum $\ol\Gr^{-\ch\theta}_G \cap \msf S^{-\ch\theta}$, 
when $\ch\nu=\ch\theta$. 
By \cite[(3.6)]{MV}, we have $\msf S^{\ch\theta} \cap \ol\Gr^{\ch\theta}_G = \{t^{\ch\theta}\}$
so (iii) is a special case of (ii).
\end{proof}

Observe that $\barM_X^0 \ttimes \ol\Gr^{\ch\theta}_{G,v}$ has 
a stratification by $\sM_{X}^{\ch\Theta'} \ttimes \Gr^{\ch\eta}_{G,v}$ for those $\ch\Theta' \in \Sym^\infty(\mathfrak c_X^- \sm 0)$ such that the stratum  $\sM_{X}^{\ch\Theta'}$ belongs to $\barM_X^0$ --- still to be determined, see Lemma \ref{lem:M0closure} --- 
and those $\ch\eta \in \ch\Lambda^-_G$ such that $\ch\eta \ge \ch\theta$ (equivalently, $\ch\eta$ is a weight of $V^{\ch\theta}$). 
Let $Z^{?,\ch\lambda,\ch\Theta',\ch\eta}_v$ denote the
preimage of the corresponding stratum in $Z^{?,\ch\lambda}_v$, 
so we have a Cartesian square 
\begin{equation} \label{e:Z0act-diagram}
\begin{tikzcd} 
Z_v^{?,\ch\lambda,\ch\Theta',\ch\eta} \ar[r, "\act_\sY"] \ar[d] & \Scr Y^{\ch\lambda}_X \ar[d] \\ 
\sM^{\ch\Theta'}_{X} \ttimes \Gr^{\ch\eta}_{G,v} \ar[r, "\act_\sM"] & \Scr M_X 
\end{tikzcd}
\end{equation}
This diagram now has no dependence on $\ch\theta$ and is defined
entirely with respect to $\ch\eta$. 
Proposition~\ref{prop:Yact-strat} implies that there is a
stratification
\[ Z_v^{?,\ch\lambda,\ch\Theta', \ch\eta} = \bigcup_{\ch\nu} Z_v^{\ch\lambda-\ch\nu,\ch\lambda,\ch\Theta', \ch\eta} \]
where $\ch\nu$ runs through the weights of $V^{\ch\eta}$, and  
$Z_v^{\ch\lambda-\ch\nu,\ch\lambda,\ch\Theta',\ch\eta}$ 
admits a map to $\Scr Y^{\ch\lambda-\ch\nu,\ch\Theta'}$. 
The open stratum is
\[ Z_v^{\ch\lambda-\ch\eta,\ch\lambda,\ch\Theta',\ch\eta} \cong 
\Scr Y^{\ch\lambda-\ch\eta,\ch\Theta'} \times \{ t^{\ch\eta}\}. \]

We will make special use of the case $\ch\Theta'=0$ and $\ch\eta = \ch\theta$ 
because 
Lemma~\ref{lem:actimage} implies that $\Scr Y^{\ch\lambda,\ch\theta}$
is contained in the images of $Z_v^{?,\ch\lambda,0,\ch\theta} \to \Scr Y^{\ch\lambda}$ 
over all $v \in \abs C$. 
Note that $Z_v^{\ch\lambda-\ch\theta,\ch\lambda,0,\ch\theta} = \sY^{\ch\lambda-\ch\theta,0}$.

\smallskip 

Recall that $\mathfrak c_X^{\Cal D}$ denotes the monoid
generated by $\ch\nu_D$ for $D\in \Cal D$. 
\begin{cor} \label{cor:Ytheta}
Let $\ch\theta \in \mathfrak c_X^- \sm 0$. 
\begin{enumerate}
\item
The scheme $\Scr Y^{\ch\lambda,\ch\theta}$ is nonempty only if 
$\ch\lambda \succeq \ch\theta$. 

\item
We have an isomorphism $(\Scr Y^{\ch\theta,\ch\theta})_\red \cong C$ over
the diagonal $C \into \Scr A^{\ch\theta}$.

\item
The single point in the central fiber 
$\msf Y^{\ch\theta,\ch\theta}$ 
corresponds to $t^{\ch\theta} \in \msf S^{\ch\theta}\cap \Gr^{\ch\theta}_G$. 
\end{enumerate}
\end{cor}
\begin{proof}
As we remarked above, every point of $\sY^{\ch\lambda,\ch\theta}$ is contained
in the image of $Z_v^{?,\ch\lambda,0,\ch\theta}$ for some $v\in \abs C$. 
By the description of the stratification
of $Z_v^{?,\ch\lambda,0,\ch\theta}$, the latter is nonempty only
if $\Scr Y^{\ch\lambda-\ch\nu,0}$ is nonempty for some $\ch\nu$. 
This implies that $\ch\lambda \succeq \ch\nu$ by 
Lemma~\ref{lem:Ynonempty}. Since 
$\ch\nu \ge \ch\theta$, we deduce (i). 

If $\ch\lambda=\ch\theta$, then $\Scr Y^{\ch\theta-\ch\nu,0}$ 
is only nonempty when $\ch\nu=\ch\theta$, in which case $\Scr Y^0 = \pt$ and $Z_v^{0,\ch\theta,0,\ch\theta} = \{t^{\ch\theta}\}$. 
By moving the point $v$ around, we get 
a surjection $C \to (\Scr Y^{\ch\theta,\ch\theta})_\red$ and
it must be an isomorphism (on reduced schemes) since the composition 
with $\pi :\Scr Y^{\ch\theta,\ch\theta} \to \Scr A^{\ch\theta}$ gives 
the diagonal embedding.
\end{proof}

Let $\Scr Y^{\ch\lambda-\ch\theta,0} \oo\xt v$ denote the disjoint locus 
$\Scr Y^{\ch\lambda-\ch\theta,0} \oo\xt \msf Y^{\ch\theta,\ch\theta}_v$. 
Observe that the composition
\[ \sY^{\ch\lambda-\ch\theta} \oo\xt v \into 
\sY^{\ch\lambda-\ch\theta} \xt v = 
Z^{\ch\lambda-\ch\theta,\ch\lambda}_v \overset{\act_\sY}\longrightarrow
\sY^{\ch\lambda} \]
coincides with the composition 
$\sY^{\ch\lambda-\ch\theta} \oo\xt v \into 
\sY^{\ch\lambda-\ch\theta} \oo\xt \sY^{\ch\theta,\ch\theta} \to 
\sY^{\ch\lambda}$ 
where the second map is given by the graded factorization property.

\begin{lem} \label{lem:comp-cover}
Let $\ch\theta\in \mathfrak c_X^- \sm 0$. 
Then:
\begin{enumerate}
\item The open embedding 
$ \sY^{\ch\lambda-\ch\theta,0} \oo\xt \sY^{\ch\theta,\ch\theta} \into
\sY^{\ch\lambda,\ch\theta}$
given by the graded factorization property is dense. 
\item For $\ch\lambda$ large enough, 
the open stratum $\sY^{\ch\lambda-\ch\theta,0} \xt v
\into Z^{?,\ch\lambda,0,\ch\theta}_v$ is dense. 
\end{enumerate}
\end{lem}
\begin{proof}
First observe that by Lemma~\ref{lem:localglobalyoga} we may
always assume that $\ch\lambda$ is large enough so that
the conditions of Corollary~\ref{cor:Mcover} hold: namely, 
$\sY^{\ch\lambda}\to \sM_X$ is smooth with connected fibers. 

By definition, $\sY^{\ch\lambda,\ch\theta}$ is the preimage 
of $\sM^{\ch\theta}_X$. Now let $V$ be a connected component of 
$\sM^{\ch\theta}_X$ that intersects the image of $\sY^{\ch\lambda}$. 
Then 
$\sY^{\ch\lambda,\ch\theta}_V := \sY^{\ch\lambda} \xt_{\sM_X} V$ is
a nonempty connected component of the smooth scheme
$\sY^{\ch\lambda,\ch\theta}$.
Lemma~\ref{lem:actimage} implies that $V$ lies in  
$\act_\sM(U \ttimes \Gr^{\ch\theta}_{G,C})$, where $U$ is a connected
component of $\Bun_H$. 
Consider the fiber product diagram 
\begin{equation}  \label{e:Z0f-diagram}
\begin{tikzcd} 
Z^{?,\ch\lambda,0,\ch\theta} \ar[r]\ar[d] & \sY^{\ch\lambda} \ar[d] \\ 
\Bun_H \ttimes \Gr^{\ch\theta}_{G,C} \ar[r] & \sM_X 
\end{tikzcd}
\end{equation}
which is the analog of \eqref{e:Z0act-diagram} where we allow $v$
to vary. 
The fibers of $\sY^{\ch\lambda}\to \sM_X$ are irreducible, so
$Z^{?,\ch\lambda,0,\ch\theta}_U := Z^{?,\ch\lambda,0,\ch\theta} \xt_{\Bun_H} U$ is irreducible and the image of 
\[ Z^{?,\ch\lambda,0,\ch\theta}_U \to \sY^{\ch\lambda} \]
contains $\sY^{\ch\lambda,\ch\theta}_V$ as a dense open. 
Proposition~\ref{prop:Yact-strat} (twisted by $C$) implies that 
$Z^{?,\ch\lambda,0,\ch\theta}_U$ has a stratification 
$\bigcup_{\ch\nu} Z^{\ch\lambda-\ch\nu,\ch\lambda,0,\ch\theta}_U$ 
over the weights $\ch\nu$ of $V^{\ch\theta}$, and 
$Z_U^{\ch\lambda-\ch\nu,\ch\lambda,0,\ch\theta}$ maps to 
$\sY^{\ch\lambda-\ch\nu,0}\xt_{\Bun_H} U$. 
In particular, $\sY^{\ch\lambda-\ch\nu,0}_U := \sY^{\ch\lambda-\ch\nu,0}\xt_{\Bun_H} U$ is nonempty
for some $\ch\nu$. Since $\ch\nu \ge \ch\theta$, 
Corollary~\ref{cor:Mcover}(ii) implies that $\sY^{\ch\lambda-\ch\theta,0}_U
$ is also nonempty.
Therefore, the open stratum 
$Z_U^{\ch\lambda-\ch\theta,\ch\lambda,0,\ch\theta} \cong 
\sY^{\ch\lambda-\ch\theta,0}_U \times C$ is nonempty. 
By Corollary~\ref{cor:Ytheta}, we can identify 
$C \cong (\sY^{\ch\theta,\ch\theta})_\red$, and  
the map 
\[ \sY^{\ch\lambda-\ch\theta,0} \oo\xt C \into 
Z^{?,\ch\lambda,0,\ch\theta} \to \sY^{\ch\lambda} \]
coincides with the open embedding 
$\sY^{\ch\lambda-\ch\theta,0} \oo\xt \sY^{\ch\theta,\ch\theta} \into \sY^{\ch\lambda,\ch\theta}$ 
given by graded factorization. 
Then $\sY^{\ch\lambda-\ch\theta}_U \oo\xt C$ is a nonempty 
open subscheme of the irreducible component $\sY^{\ch\lambda,\ch\theta}_V$,
so it must be dense. 

To show (ii), note that \emph{a priori}  
the preimage of 
$\sY^{\ch\lambda,\ch\theta}_V$ in $Z^{?,\ch\lambda,0,\ch\theta}$
is contained in a finite union of 
$Z^{?,\ch\lambda,0,\ch\theta}_U$
for connected components $U$ of $\Bun_H$.
However we saw above that the open stratum $\sY^{\ch\lambda-\ch\theta,0}_U \times C$ is nonempty, and therefore dense in $Z^{?,\ch\lambda,0,\ch\theta}_U$, 
for each such $U$. 
\end{proof}

In the proof above, we have essentially shown 
Theorem \ref{thm:comp}(ii) along the way:

\begin{proof}[Proof of Theorem \ref{thm:comp}(ii)] 
To simplify notation, we only show the case $\ch\Theta = [\ch\theta]$. 
The multi-point version is proved in exactly the same way. 
We continue to use the notation from the previous proof
of Lemma~\ref{lem:comp-cover}. 
By Corollary~\ref{cor:Mcover} and the base change diagram \eqref{e:Z0f-diagram}, 
it is enough to show, by restricting to open dense subsets, that 
\[ Z^{?,\ch\lambda,0,\ch\theta} \to \sY^{\ch\lambda} \xt_{\sM_X} \barM_X^{\ch\theta} \] 
is birational for $\ch\lambda$ large enough. 
It follows from Lemma~\ref{lem:comp-cover} that 
$\sY^{\ch\lambda-\ch\theta,0} \oo\xt C$ is a dense open 
in both the target and source. 
\end{proof}

\begin{proof}[Proof of Theorem \ref{thm:comp}(iii)]
Again to simplify notation we only show the case $\ch\Theta = [\ch\theta]$.
By Zariski's Main Theorem, it suffices to show that 
the restriction of $\act_\sM$ to 
\begin{equation} \label{e:act-bijection}
\act_\sM^{-1}(\sM^{\ch\theta}_X) \cap (\Bun_H \ttimes \Gr_{G,C}^{\ch\theta}) \to \sM_X^{\ch\theta}
\end{equation}
is a bijection on $k$-points. 
By Lemma~\ref{lem:act-equivariant}, this map is equivariant with respect to
generic-Hecke modifications away from the marked point in $C$. 
Fix a marked point $v\in C$ and let $\sM^{\ch\theta}_{X,v}$ denote 
the preimage of $v$ under the projection $\sM^{\ch\theta}_X \to C$.
By Proposition~\ref{prop:genHecke-transitive}, all the $k$-points 
of $\sM^{\ch\theta}_{X,v}$ are equivalent under 
the equivalence relation generated by generic-Hecke correspondences.  
Thus it suffices to determine the fiber of \eqref{e:act-bijection}
at the single point $\{t_v^{\ch\theta}\} = \msf Y^{\ch\theta,\ch\theta}$.  
By Proposition \ref{prop:Yact-strat}, 
the fiber of $t_v^{\ch\theta}$ has a stratification by 
$\msf Y^{\ch\theta-\ch\nu,0} \ttimes (\msf S^{\ch\nu} \cap \ol\Gr^{\ch\theta}_G)$
for $\ch\nu$ a weight of $V^{\ch\theta}$. 
On the other hand, Corollary \ref{cor:Ytheta} implies that 
$\msf Y^{\ch\theta-\ch\nu,0}$ is nonempty only if $\ch\nu=\ch\theta$. 
In this case $\msf Y^{0,0} \times (\msf S^{\ch\theta}\cap \ol\Gr^{\ch\theta}_G)$
is a point. 
\end{proof}

\subsubsection{}
Recall that for an arbitrary algebraic group $H$, the algebraic fundamental
group $\pi_1(H)$ is defined as the quotient of the coweight lattice by the coroot lattice
of the reductive group $H/H_{\mathrm u}$, where $H_{\mathrm u}$ is the unipotent
radical of $H$.
For $\ch\mu\in \pi_1(H)$, let $\Bun_H^{\ch\mu}$ denote the
corresponding connected component of $\Bun_H$ and let $^{\ch\mu}\barM_X^0$ 
be its closure in $\sM_X$.

\begin{cor} \label{cor:strata-comp} 
Let $\ch\Theta \in \Sym^\infty( \mathfrak c_X^- \sm 0)$. 
Then the set of irreducible components of $\barM^{\ch\Theta}_X$ is in bijection 
with $\pi_1(H)$, where $\ch\mu\in \pi_1(H)$ corresponds to 
\[ ^{\ch\mu}\barM_X^{\ch\Theta} := 
\act_\sM(^{\ch\mu}\barM^0_X \ttimes \ol\Gr^{\ch\Theta}_{G,C^{\ch\Theta}}). 
\] 
\end{cor}

\begin{proof}
 It follows from Theorem \ref{thm:comp} that the irreducible
components of $\barM^{\ch\Theta}_X$ are in bijection with the
irreducible components of $\Bun_H \ttimes \ol\Gr^{\ch\Theta}_{G,C^{\ch\Theta}}$. 
Since $\ol\Gr^{\ch\Theta}_{G,C^{\ch\Theta}}$ is irreducible, the
latter are in bijection with $\pi_0(\Bun_H) = \pi_1(H)$.
\end{proof}

We will let $^{\ch\mu}\sM_X^{\ch\Theta} := {^{\ch\mu}}\barM_X^{\ch\Theta} \cap \sM_X^{\ch\Theta}$
denote the corresponding connected component of $\sM_X^{\ch\Theta}$ (which is smooth). 

\subsubsection{Mirkovi\'c--Vilonen cycles} 

We finish this subsection by proving a result that will be used in the following sections. The goal here is to show that the Mirkovi\'c--Vilonen cycles in the $\ch\theta$-stratum of the affine Grassmannian map, generically, to the $\ch\theta$-stratum of the global model of $X$, under the action map.

Fix $v\in \abs C$ and $\ch\theta \in \mf c_X^- \sm 0$. 
Consider the restriction of $\act_{\sM,v}$ to 
\begin{equation} \label{e:heckemapGr}
    \pt \xt \ol\Gr^{\ch\theta}_{G,v} \to \Bun_H \ttimes \ol\Gr^{\ch\theta}_{G,v} 
\to \barM_X^{\ch\theta}, 
\end{equation}
where $\pt \to \Bun_H$ corresponds to the trivial $H$-bundle. 
Note that the map \eqref{e:heckemapGr} above can be 
extended to a map 
\begin{equation}\label{e:actHG}
    \act_{v} : \Gr_G \xt_{\msf L X/\msf L^+G} (\msf L^+X/\msf L^+G) \to \sM_X
\end{equation}
using Beauville--Laszlo's theorem: a point of the left hand side 
consists of a $G$-bundle $\Scr P_G$ and a trivialization 
$\tau : \sP_G|_{C\sm v} \cong \sP_G^0|_{C\sm v}$ such that 
$\tau^{-1} \circ x_0 : C\sm v \to X \xt^G \sP_G$ is regular when localized at $\mf o_v$. 
Here $x_0 : C \to X \xt^G \Scr P_G^0 = X \xt C$ denotes the section corresponding to
the base point $x_0\in X^\circ$.  
Thus, $\act_{v}(\sP_G,\tau) := (\sP_G, \tau^{-1}\circ x_0) \in \sM_X^{\ch\theta}$
is well-defined. 

Define $\overset\circ\Gr{}^{\ch\theta}_G \subset \Gr^{\ch\theta}_G$ to be the 
open subscheme equal to the preimage 
of the stratum $\sM_X^{\ch\theta}$ under \eqref{e:heckemapGr}. 
We can also identify
$\displaystyle \overset\circ\Gr{}_G^{\ch\theta} = \Gr^{\ch\theta}_G \xt_{\msf L X/\msf L^+G} (\msf L^{\ch\theta}X / \msf L^+G)$.

Taking central fibers with respect to $v$, 
the restriction of $\act_v$ to a semi-infinite orbit
factors through 
\begin{equation}
 \label{e:actvY}
 \msf S^{\ch\lambda} \cap \ol\Gr^{\ch\theta}_G \into \msf Y^{\ch\lambda}\xt_{\sM_X} \barM_X^{\ch\theta}, 
\end{equation}
where $\ch\lambda$ is a weight of $V^{\ch\theta}$
and we consider $(\msf Y^{\ch\lambda})_\red$ as a subscheme of $\msf S^{\ch\lambda}$ via
Lemma~\ref{lem:Y=ScapGr}.
Note that $\msf S^{\ch\lambda}\cap\ol\Gr_G^{\ch\theta}$ is isomorphic to 
the \emph{closed} stratum $Z_v^{0,\ch\lambda}$ from Proposition~\ref{prop:Yact-strat}.
Under this identification \eqref{e:actvY} coincides with the restriction of the map $\act_\sY:Z_v^{?,\ch\lambda}\to \sY^{\ch\lambda}$ from \eqref{e:Zact-diagram}.

Observe that $\msf S^{\ch\theta} \cap \ol\Gr^{\ch\theta}_G = \{ t^{\ch\theta} \}$ 
is contained in $\overset\circ{\Gr}{}^{\ch\theta}_G$. 
We will use this to deduce: 

\begin{lem} \label{lem:MVcentralfiber}
Let $\ch\lambda,\ch\theta$ as above. Then 
$\msf S^{\ch\lambda} \cap \overset\circ\Gr{}^{\ch\theta}_G$
intersects every irreducible component 
of $\msf S^{\ch\lambda} \cap \ol\Gr^{\ch\theta}_G$. 
\end{lem}
\begin{proof}
Let $Z$ denote an irreducible component of $\msf S^{\ch\lambda}\cap \ol\Gr^{\ch\theta}_G$, which must be 
of dimension $d=\brac{\rho_G, \ch\lambda-\ch\theta}$.
Let $\ol Z$ denote its closure in $\olsf S^{\ch\lambda} \cap \ol\Gr^{\ch\theta}_G$. 
Then the proof of \cite[Theorem 3.2]{MV}, which we briefly recall in the next paragraph,
shows that $\ol Z$ contains $t^{\ch\theta} \in \overset\circ\Gr{}^{\ch\theta}_G$.
Thus, $\ol Z \cap \overset\circ\Gr{}^{\ch\theta}_G$ is open and nonempty, hence dense 
in $\ol Z$. 

Since the boundary of $\msf S^{\ch\lambda}$ in $\olsf S^{\ch\lambda}$ 
is a hyperplane section (Proposition~\ref{prop:Shyperplane}), 
$\ol Z\sm Z$ contains an irreducible component of dimension $d-1$
inside $\msf S^{\ch\lambda_1}\cap \ol\Gr^{\ch\theta}_G$ for $\ch\lambda_1 < \ch\lambda$. 
In this way we produce a sequence $\ch\lambda=\ch\lambda_0, \dotsc, \ch\lambda_d$
and irreducible components of $\msf S^{\ch\lambda_i} \cap \ol Z$ 
of dimension $d-i$. 
The only weight $\ch\lambda_d$ of $V^{\ch\theta}$ such that 
$\brac{\rho_G,\ch\lambda-\ch\lambda_d} \ge d$ 
is $\ch\lambda_d=\ch\theta$, so $\ol Z$ must contain $\msf S^{\ch\theta}\cap \ol\Gr^{\ch\theta}_G
= \{ t^{\ch\theta} \}$. 
\end{proof}

\subsection{Closure relations and components in the global model}
\label{sect:closures}

Let $\ch\Theta, \ch\Theta' \in \Sym^\infty(\mathfrak c_X^- \sm 0)$. Consider $\ch\Theta'-\ch\Theta$ as a formal sum 
in $\bigoplus_{\ch\theta\in \mathfrak c_X^-\sm 0} \mbb Z [\ch\theta]$. We say that 
\[ \ch\Theta' \succeq \ch\Theta \] 
if $\ch\Theta'-\ch\Theta$ can be written as a sum of formal differences 
$[\ch\theta']-[\ch\theta]$ where $\ch\theta,\ch\theta'\in \mathfrak c_X^-$ and $\ch\theta' \succeq \ch\theta$. 
Note that we allow $\ch\theta=0$ (in which case $[0]=0$), 
and for general $\ch\theta,\ch\theta'$ as
above, it is \emph{not} necessarily the case that $\ch\theta-\ch\theta' \in \Cal V$. 

\begin{prop} \label{prop:closure-rel}
Let $\ch\Theta,\ch\Theta' \in \Sym^\infty(\mathfrak c_X^- \sm 0)$. 
We have that $\sM^{\ch\Theta'}_X$ lies in the closure of $\sM^{\ch\Theta}_X$
if and only if there exists $\ch\Theta'' \in \Sym^\infty(\mathfrak c_X^-\sm 0)$ such that $\ch\Theta$ refines $\ch\Theta''$ and $\ch\Theta' \succeq \ch\Theta''$. 
\end{prop}

\begin{rem}
We warn that our notation in Corollary~\ref{cor:strata-comp} is unfortunately 
\emph{not} compatible
with the closure relations in the following sense: 
it is possible that $^{\ch\mu} \barM^{\ch\Theta}_X \cap \sM^{\ch\Theta'}_X 
\ne {^{\ch\mu}\sM^{\ch\Theta'}_X}$ for $\ch\Theta' \succ \ch\Theta$ 
and $\ch\mu \in \pi_1(H)$.
\end{rem}

The reader may want to skip the proof of this proposition at first reading, and focus on the corollaries that follow. Note that the proof will use Zastava models, despite the fact that the statement is about the global model. The proof of the proposition starts with the following special case:

\begin{lem} \label{lem:M0closure}
The closure of $\sM_X^0$ in $\sM_X$ intersects 
$\sM_X^{\ch\Theta}$ only if $\ch\Theta \succeq 0$. 
\end{lem}

\begin{proof}
Let $\ch\Theta$ correspond to a stratum such that 
$\sM_X^{\ch\Theta}$ intersects $\barM_X^0$. Then by
Corollary~\ref{cor:Mcover} there exists $\ch\lambda$ 
such that $\sY^{\ch\lambda,\ch\Theta}$ intersects the closure
of $\sY^{\ch\lambda,0}$ in $\sY^{\ch\lambda}$. 
Now consider the torsor $H'\bs G' \to X^\bullet$ from 
Lemma~\ref{lem:Xcover}. 
Let $X' = \Spec k[H'\bs G']$ considered as an affine $G'$-variety.
The corresponding compactified Zastava model 
$\barY^{D}_{X'} \to \sA^{D}_{X'}$ is indexed
by $D\in \mf c_{X'}=\mbb N^{\Cal D}$. 
Since $X'^\bullet/B' \cong X^\bullet/B$ as stacks, 
$\sY^{\ch\lambda,0}_X$ is a disjoint union of $\sY^D_{X^\bullet}=\sY^{D,0}_{X'}$
ranging over all $D\in \mbb N^{\Cal D}$ such that $\varrho_X(D)=\ch\lambda$
(see \S\ref{sect:Y0cc}). 
Choose $D\in \mbb N^{\Cal D}$ such that the closure of 
$\sY^D_{X^\bullet}$ in $\sY^{\ch\lambda}_X$ intersects 
the stratum $\sY^{\ch\lambda,\ch\Theta}_X$. 
The map 
\begin{equation} \label{e:mapbarYupdown}
 \barY^{D}_{X'} \to \barY^{\ch\lambda}_X  
\end{equation}
is proper because $\barY^{D}_{X'}, \barY^{\ch\lambda}_X$
are proper over $\sA^{D}_{X'},\sA^{\ch\lambda}_X$, respectively,
and the natural map $\sA^{D}_{X'}\to \sA^{\ch\lambda}_X$ 
is proper. Therefore, the closure of $\sY^{D}_{X^\bullet}$
in $\sY^{\ch\lambda}_X$ is contained in the image
of \eqref{e:mapbarYupdown}.
Note that the stratification of $\sM_{X'}$ is indexed by 
$\ch\Theta' \in \Sym^\infty(\mathfrak c_{X'}^-\sm 0)$. 
In particular, $\ch\Theta' \succeq 0$ since $\mf c_{X'}=\mbb N^{\Cal D}$. 
Therefore, $\barY^{D}_{X'}$ is a union of 
$\barY^{D}_{X'} \xt_{\sM_{X'}} \sM_{X'}^{\ch\Theta'}$
for $\ch\Theta' \succeq 0$, which implies its image in 
$\barY^{\ch\lambda}_X$ is contained in the union 
of $\barY^{\ch\lambda}_X \xt_{\sM_X} \sM_X^{\ch\Theta}$
for $\ch\Theta \succeq 0$.
\end{proof}

\begin{proof}[{Proof of Proposition~\ref{prop:closure-rel}}] 
By Theorem \ref{thm:comp}(i), the closure of $\sM^{\ch\Theta}_X$ is equal to 
the image of $\barM^0_X \ttimes \ol\Gr^{\ch\Theta}_{G,C^{\ch\Theta}}$,
so we will consider the latter. 
Note that $C^{\ch\Theta}$ is stratified by disjoint loci $\oo C^{\ch\Theta''}$ 
for all partitions $\ch\Theta''$ such that $\ch\Theta$ refines 
$\ch\Theta''$. By the description of the fibers of 
$\Gr^{\ch\Theta}_{G, C^{\ch\Theta}} \to C^{\ch\Theta}$ in 
\eqref{e:GrTheta-fiber}, we have an identification 
\[
\ol\Gr^{\ch\Theta}_{G,C^{\ch\Theta}} \xt_{C^{\ch\Theta}} \oo C^{\ch\Theta''} 
= \ol\Gr^{\ch\Theta''}_{G, \oo C^{\ch\Theta''} } 
\]
at the level of reduced schemes. Therefore, replacing $\ch\Theta$ by $\ch\Theta''$, 
it suffices to show that 
the image of $\barM_X^0 \ttimes \ol\Gr^{\ch\Theta}_{G, \oo C^{\ch\Theta}}$
contains $\sM_X^{\ch\Theta'}$ if and only if $\ch\Theta' \succeq \ch\Theta$.

The ``only if'' direction follows from the description of
$\act_\sM$ on $k$-points and Lemmas~\ref{lem:act-kpoints} 
and \ref{lem:M0closure}. 

\medskip

We will show the ``if'' direction only 
in the case when $\ch\Theta=[\ch\theta],\, \ch\Theta'=[\ch\theta']$ 
are singleton (so $\ch\theta'\succeq \ch\theta$) to lessen notation 
(allowing $\ch\theta=0$), but the multi-point version is proved in 
exactly the same way. 
Fix $v\in \abs C$. We have a distinguished point in $\sM_X^{\ch\theta'}$ degenerate
at $v$: the image under $\act_v$ of $t^{\ch\theta'} \in \msf Y^{\ch\theta',\ch\theta'}$. 
We will show that $\barM_X^{\ch\theta}$ contains this point. 
Then, stability of $\barM_X^{\ch\theta}$ under generic-Hecke correspondences 
(Theorem~\ref{thm:comp}(i) and Lemma~\ref{lem:act-equivariant}) 
and Proposition~\ref{prop:genHecke-transitive} imply that 
$\barM_X^{\ch\theta}$ contains all of $\sM_{X,v}^{\ch\theta'}$. 

\smallskip

Since $\ch\theta' \succeq \ch\theta$, we can decompose 
$\ch\theta' - \ch\theta = \sum_{j=1}^d \ch\nu_j$ for (not necessarily distinct) 
$\ch\nu_j$ equal to valuations of colors $D_j \in \Cal D$ 
(in case $D_j$ is not uniquely determined by its valuation, the choice of $D_j$ is arbitrary). 
The graded factorization property gives a map 
$\oo{\Scr C} := \sY^{\ch\theta,\ch\theta} \oo\xt \oo\prod_j \sY^{D_j}_{X^\bullet} 
\to \sY^{\ch\theta'}$. 
We claim that the image of 
$\oo{\Scr C}$ contains $C=\sY^{\ch\theta',\ch\theta'}_\red$ in its closure. 
(Note that $\oo{\Scr C}_\red = \oo C^{d+1}$.) 
This will produce an irreducible variety whose generic point maps to $\sM_X^{\ch\theta}$ 
while a special point maps to $\act_v(t^{\ch\theta'}) \in \sM_X^{\ch\theta'}$.

\smallskip

Consider the Beilinson--Drinfeld affine Grassmannian $\Gr_{B,C^{d+1}}$, 
whose fiber over $d+1$ pairwise distinct points $(v_0,\dotsc, v_d)\in \oo C^{d+1}$ 
is $\prod_{j=0}^d \Gr_B$ while the fiber over $v_0=\dotsc=v_d$ is $\Gr_B$. 
By Lemma~\ref{lem:Y=ScapGr}, we have an isomorphism $\msf Y^{\ch\lambda}\cong \Gr_B^{\ch\lambda}
\xt_{\msf L X/\msf L^+B}(\msf L^+X/\msf L^+B)$,
which depends on a fixed base point $x_0 \in X$. 
Now a reduction to $\mbb G_m \bs \GL_2$ (see proof of Lemma~\ref{lem:Ycolor}
and Example~\ref{eg:Yfiber})
shows that the central fiber $\msf Y^{D_j}_{X^\bullet} = \pt$ is contained in 
$\msf L^+ N \cdot t^{\ch\nu_j} \subset \Gr_B^{\nu_j}$.
Therefore, $\oo{\Scr C}$ is contained in 
the orbit of the multi-point jet space $(\Scr L^+N)_{C^{d+1}}$ acting on 
the closed subscheme $C^{d+1}\subset \Gr_{T, C^{d+1}}\subset \Gr_{B,C^{d+1}}$
given by $(v_0,\dotsc, v_d) \mapsto t_{v_0}^{\ch\theta} \prod_j t_{v_j}^{\ch\nu_j}$, where
the $v_j$'s are allowed to collide. 
The orbit of $(v_0,\dotsc, v_d) \in \oo C^{d+1}$ is 
\[ \textstyle
\{t_{v_0}^{\ch\theta}\} \xt \prod_{j=1}^d (\msf L^+ N \cdot t^{\ch\nu_j}_{v_j}) 
\subset \Gr_{B,v_0}^{\ch\theta} \xt \prod_{j=1}^d \Gr_{B,v_j}^{\ch\nu_j},
\]
while the orbit of the diagonal $v=v_0=\dotsc=v_d$ 
is $\msf L^+ N \cdot t^{\ch\theta'}_v = \{ t^{\ch\theta'}_v \} \subset \Gr_{B,v}^{\ch\theta'}$
since $\ch\theta'$ is antidominant. 

Now assume that $C=\mbb A^1$, which is justified by Proposition~\ref{prop:changecurve}.
Then we can identify $(\Scr L^+ N)_C = \msf L^+ N \times C$. 
Let $\msf Y^{D_j}_{X^\bullet}$ correspond to the point $n_j t^{\ch\nu_j}\in \Gr_B$
for $n_j \in \msf L^+N(k)$. 
For any pairwise distinct $v_0,\dotsc, v_d \in \mbb A^1$ we have
a line $\mbb A^1 \to C^{d+1} : a \mapsto (a v_0, \dotsc, a v_d)$ contracting all points to $0$. 
Multiplication defines a map 
$m: (\msf L^+ N \times C)^d \to (\Scr L^+ N)_{C^d}$. 
Letting $m(n_1,\dotsc,n_d)$ act 
on the point $t_{a v_0}^{\ch\theta} \prod_j t_{a v_j}^{\ch\nu_j} 
\in \Gr_{T,C^{d+1}}$ as $a\to 0$, 
we get a curve connecting $\{ t^{\ch\theta}_{v_0} \} \xt \prod \msf Y^{D_j}_{X^\bullet,v_j}$ to 
$t^{\ch\theta'}_0$. 
Hence the closure of $\oo{\Scr C}$ 
in $\sY^{\ch\theta'} \xt_{\sA^{\ch\theta'}} C^{d+1} \subset \Gr_{B,C^{d+1}}$ contains $\sY^{\ch\theta',\ch\theta'}_\red$.
Since the map $C^{d+1} \to \sA^{\ch\theta'}$ is finite and $\oo{\Scr C}$ is irreducible,
we have proved the claim.
\end{proof}

Now we draw some corollaries from Proposition \ref{prop:closure-rel}.

\begin{cor}  \label{cor:can-comp}
The open substack $\sM_X^0 = \Bun_H$ is dense in $\Scr M_{X^\can}$ iff $\mathfrak c_X^{\Cal D} \cap \Cal V = \mf c_{X^\can}^-$. 
\end{cor}

This is an analog of \cite[Proposition 1.2.3]{BG}, which says that 
$\Bun_B$ is dense in $\ol\Bun_B$ if $[G,G]$ is simply connected. 

\begin{proof}
Indeed, $\mathfrak c_X^{\Cal D} \cap \Cal V = \mf c_{X^\can}^-$ is precisely the condition that every $\ch\Theta \in \Sym^\infty(\mathfrak c_X^- \sm 0)$ is $\succeq 0$.
\end{proof}

Define $\Cal D^G_{\mathrm{sat}}(X)$ to be 
the set of primitive elements in $\Prim(\mathfrak c_X^-)$ that
cannot be decomposed as a sum $\ch\theta+\ch\lambda$ where 
$\ch\theta,\ch\lambda$ are both nonzero, 
$\ch\theta \in \mathfrak c_X^-$
and $\ch\lambda \succeq 0$ (see \S\ref{def:partition} for 
the definition of \emph{primitive}). 
Note that $\Cal D^G_{\mathrm{sat}}(X)$ contains $\varrho_X(\Cal D(X)\sm \Cal D)$, 
the valuations of the $G$-stable prime divisors,
but the containment may be strict.

\begin{cor} \label{cor:comp}
 There is a bijection between the set of irreducible
components of $\Scr M_X$ and 
\[ \pi_1(H) \xt \Sym^\infty(\Cal D^G_{\mathrm{sat}}(X)), \] 
such that $\ch\mu \in \pi_1(H),\, \ch\Theta \in  \Sym^\infty(\Cal D^G_{\mathrm{sat}}(X))$ corresponds to 
$^{\ch\mu}\barM_X^{\ch\Theta}$.  
\end{cor}

\begin{proof}
 For any $\ch\Theta' \in \Sym^\infty(\mathfrak c_X^- \sm 0)$, 
let $\ch\Theta'' \in \Sym^\infty(\mathfrak c_X^-\sm 0)$ 
be a minimal element with respect to the ordering $\preceq$
such that $\ch\Theta'' \preceq \ch\Theta'$. 
Then $\ch\Theta''$ can be refined to an element $\ch\Theta\in \Sym^\infty(\Cal D^G_{\mathrm{sat}}(X))$. Therefore, the closure relations
from Proposition~\ref{prop:closure-rel} tell us that any stratum 
is contained in the closure of $\sM^{\ch\Theta}_X$ for a partition 
$\ch\Theta$ as above. 
By definition if $\ch\theta_1,\ch\theta_2\in \Cal D^G_{\mathrm{sat}}(X)$
satisfy $\ch\theta_1\succeq \ch\theta_2$ then they must be equal. 
From this one deduces that $\sM^{\ch\Theta}_X$ is not contained
in the closure of any other stratum. 

Thus, the closure of each
$\sM^{\ch\Theta}_X,\, \ch\Theta\in \Sym^\infty(\Cal D^G_{\mathrm{sat}}(X))$, 
is a union of irreducible components, and no irreducible component 
is contained in two different such closures. The Corollary now follows from the description of irreducible components of $\sM^{\ch\Theta}_X$ by Corollary \ref{cor:strata-comp}.
\end{proof}

\begin{lem}\label{lem:Membedding}
Let $X_1 \to X_2$ be a $G$-equivariant map of affine spherical varieties with $X_1^\bullet = X_2^\bullet = H\bs G$. Then the induced map $\Scr M_{X_1} \to \Scr M_{X_2}$ is a closed embedding.
\end{lem}
\begin{proof}
Let $S$ be a test scheme and let $(\Scr P_G,\sigma_2) \in \Scr M_{X_2}(S)$, where
$\sigma_2 : C\xt S \to X_2 \xt^G \Scr P_G$ is a section. 
By Corollary~\ref{cor:Mcover} there exists a $B$-structure $\Scr P_B$ 
on $\Scr P_G$ after a suitable surjective \'etale base change $S'\to S$
such that $(\Scr P_B, \sigma_2)\in \Scr Y_{X_2}(S')$. Then in particular
there exists a relative effective divisor $D \subset C\xt S'$ such that 
$\sigma_2(C\xt S' \sm D) \subset X_2^\circ \xt^B \Scr P_B$, cf.~\S\ref{sss:YGr}. 
Since $X_1^\bullet=X_2^\bullet$, we have $X_1^\circ=X_2^\circ$. 
By Lemma~\ref{lem:extendsection} the condition that $\sigma_2|_{C\xt S'\sm D}$
extends to a section $C \xt S' \to X_1 \xt^B \Scr P_B = X_1 \xt^G \Scr P_G$ is
closed in $S'$. Therefore, $S' \xt_{\Scr M_{X_2}} \Scr M_{X_1} \to S'$ is a closed
embedding.
\end{proof}

Lemma~\ref{lem:Membedding} implies that 
$\sM_{X^\can}$ is a closed substack of $\sM_X$ containing $\barM_X^0$, and when the condition
of Corollary~\ref{cor:can-comp} is satisfied, $\sM_{X^\can} = \barM_X^0$.

\subsection{Closure relations and components in the Zastava model}

In order to extend the results above to strata of $\sY$, we will need the following result:

\begin{lem} \label{lem:Yraise}
Let $y \in \sY^{\ch\lambda}(k)$ for $\ch\lambda \in \mf c_X$.
For any simple root $\alpha$ with $\Cal D(\alpha) = \{ D_\alpha^+, D_\alpha^- \}$ and any $N \gg 0$, there exists
a $k$-point 
\[ y' \in \sY^{\ch\lambda} \oo\xt C^{(N)} \oo\xt C^{(N)} \subset 
\sY^{\ch\lambda} \oo\xt \sY^{N\ch\nu_\alpha^+} \oo\xt \sY^{N\ch\nu_\alpha^-} 
\] 
such that the first coordinate is $y$ and the image of $y'$ under 
the composition
\[ 
    \sY^{\ch\lambda} \oo\xt \sY^{N\ch\nu_\alpha^+} \oo \xt
    \sY^{N\ch\nu_\alpha^-} \to 
    \sY^{\ch\lambda+N\ch\alpha} \to \sM_X 
\]
coincides with the image of $y$.
\end{lem}

Above we are using the fact, from Lemma~\ref{lem:Ycolor}, that 
$\sY^{N\ch\nu_\alpha^\pm}$ contains $C^{(N)}$ if $\ch\nu_\alpha^+ \ne \ch\nu_\alpha^-$ 
or $C^{(N)}\sqcup C^{(N)}$ if $\ch\nu_\alpha^+ = \ch\nu_\alpha^-$. 
In the latter case, Lemma~\ref{lem:Yraise} is picking out one of these components
(depending on the point $y$). 

\begin{proof}
The point $y$ is equivalent to a datum $(\Scr P_B, \sigma)$
where $\Scr P_B$ is a $B$-bundle on $C$ and $\sigma : C \to X \xt^B \Scr P_B$ is a section. 
Let $\Scr P_G = G \xt^B \Scr P_B$ denote the induced $G$-bundle.
The image of $y$ in $\sM_X$ corresponds to $(\Scr P_G, \sigma)$.
Set $\Bbbk := k(C)$. Then $\sigma|_{\Spec \Bbbk}$ defines a trivialization 
of $\Scr P_G|_{\Spec \Bbbk}$, which we fix. 
With respect to this trivialization, $B$-structures on $\Scr P_G$
are in bijection with sections $\Spec \Bbbk \to G/B$,
and $\Scr P_B$ corresponds to $1B \in (G/B)(\Bbbk)$. 
The preimage of $(\Scr P_G,\sigma)$ in $\sY$ identifies with the
orbit 
$H(\Bbbk)\cdot 1B \subset (G/B)(\Bbbk)$. 
Since $X^\circ P_\alpha/\mf R(P_\alpha)=\mbb G_m\bs \PGL_2$, the orbit of 
$H\cap P_\alpha$ on $1\in P_\alpha/B = \mbb P^1$ is 
$\mbb G_m$. Let $r \in \mbb G_m(\Bbbk) = k(C)^\times$ 
be a rational function on the curve $C$. 
Then the principal divisor defined by $r$ is of the form 
$\underline v^+ - \underline v^-$ where $\underline v^\pm$ 
are effective divisors with disjoint supports.
By the Riemann--Roch theorem, for any $N \gg 0$  
there exists $r \in k(C)^\times$ such that $\deg(\underline v^+)=\deg(\underline v^-)=N$ and the supports of $\underline v^+$ and $\underline v^-$
are both contained in $\sigma^{-1}(X^\circ \xt^B \Scr P_B)$.
Under the isomorphism $X^\circ P_\alpha/B \cong \mbb G_m \bs \mbb P^1$,
let us identify $D_\alpha^+$ with $0\in \mbb P^1$ and
$D_\alpha^-$ with $\infty \in \mbb P^1$. 
Then the preimage of $(\Scr P_G,\sigma)$ in $\sY$ corresponding to 
$r\in \mbb P^1(\Bbbk)$ has the desired property.
\end{proof}

Recall from Corollary \ref{cor:strata-comp} that the irreducible components of $\sM_X^{\ch\Theta}$ are denoted by ${^{\ch\mu}\sM_X^{\ch\Theta}}$, with $\ch\mu\in \pi_1(H)$. 
For any $\ch\lambda\in \mathfrak c_X$, define 
$^{\ch\mu}\sY^{\ch\lambda,\ch\Theta} = \sY^{\ch\lambda} \xt_{\sM_X} {^{\ch\mu}\sM_X^{\ch\Theta}}$
and 
$^{\ch\mu}\sY^{\ch\lambda,\succeq \ch\Theta} := \sY^{\ch\lambda} \xt_{\sM_X} 
{^{\ch\mu}}\barM_X^{\ch\Theta}$.

\begin{cor} \label{cor:Ycomp}
For any $\ch\lambda \in \mathfrak c_X,\, \ch\mu\in \pi_1(H),\,
\ch\Theta \in \Sym^\infty(\mathfrak c_X^- \sm 0)$, the scheme 
$^{\ch\mu}\sY^{\ch\lambda,\succeq\ch\Theta}$ is irreducible (if nonempty), and $^{\ch\mu}\sY^{\ch\lambda,\ch\Theta} $ is open dense in it.
\end{cor}
\begin{proof}
For $\ch\lambda$ large enough, the claim immediately follows from
Corollaries~\ref{cor:Mcover} and \ref{cor:comp}. 
Now, for arbitrary $\ch\lambda$, consider 
$\ch\lambda'=\ch\lambda+\sum_{\alpha\in \Delta_G} n_\alpha \ch\alpha$
large enough. 
By the graded factorization property and Lemma~\ref{lem:Yraise}, 
there exists an \'etale map 
\begin{equation} \label{e:Ycomp1}
    ^{\ch\mu}\sY^{\ch\lambda,\succeq\ch\Theta} \oo\xt \Scr C \to {^{\ch\mu}}\sY^{\ch\lambda',\succeq\ch\Theta},
 \end{equation}
where $\Scr C = \oo\prod_{\Delta_G} C^{(n_\alpha)} \oo\xt C^{(n_\alpha)}$.

By the validity of the proposition for large $\ch\lambda'$,  
we know that $^{\ch\mu}\sY^{\ch\lambda',\ch\Theta}$ is 
connected and dense in $^{\ch\mu}\sY^{\ch\lambda', \succeq \ch\Theta}$.
Therefore, if $^{\ch\mu}\sY^{\ch\lambda,\succeq\ch\Theta}$ is not irreducible, 
there exist $y_1,y_2$ in disjoint connected components of 
$^{\ch\mu}\sY^{\ch\lambda,\ch\Theta}$ 
and $c_1,c_2\in \Scr C$ such that $(y_1,c_1)$ and $(y_2,c_2)$ have the same
image in $^{\ch\mu}\sY^{\ch\lambda',\ch\Theta}$.
Let $\abs{c_1}, \abs{c_2} \subset \abs C$ denote the support of $c_1,c_2$
as divisors. 
From the definition of the factorization map \eqref{e:Ycomp1}, 
if $(y_1,c_1)$ and $(y_2,c_2)$ have the same image, 
we have 
\[ y' = y_1 |_{C \sm (\abs{c_1}\cup \abs{c_2})} = y_2 |_{C\sm (\abs{c_1}\cup \abs{c_2})} : C \sm (\abs{c_1}\cup \abs{c_2}) \to X/B. \] 
We can extend $y'$ to an element of $\sY^{\ch\lambda-\ch\nu, \ch\Theta}$
with $\abs{c_1}\cup \abs{c_2}$ in its $B$-nondegenerate locus, 
for some $\ch\nu\succ 0$. 
By replacing $\Scr C$ by a possibly different irreducible scheme $\Scr C'$,
we may assume that $\abs{c_1}$ and $\abs{c_2}$ are disjoint. 
Now we must have $(y', c_2) \mapsto y_1,\, (y',c_1)\mapsto y_2$ under
the factorization map 
$\sY^{\ch\lambda-\ch\nu,\ch\Theta} \oo\xt \Scr C' \to \sY^{\ch\lambda,\ch\Theta}$.
This implies that $y_1,y_2$ were originally in the same connected component.
\end{proof}

\begin{cor}\label{cor:compY}
For every $\ch\lambda\in \mathfrak c_X$, there is an injection from the set of irreducible components of $\sY^{\ch\lambda}$ to $\pi_1(H) \xt \Sym^\infty(\Cal D^G_{\mathrm{sat}}(X))$.
\end{cor}

\begin{proof}
By Corollaries \ref{cor:Ycomp} and \ref{cor:comp}, the irreducible components $\sY^{\ch\lambda}$ are those $^{\ch\mu}\sY^{\ch\lambda,\succeq\ch\Theta}$ that are nonempty, for $\ch\Theta \in \Sym^\infty(\Cal D^G_{\mathrm{sat}}(X))$.
\end{proof}

In \S \ref{sect:connected-open}, and in particular Corollary \ref{cor:Ycomp2}, we will see a different description of the irreducible components of $\sY_X$, based on the partition of $\sY_{X^\bullet} = \sY_X^{?,0}$ into the subschemes $\sY^D_{X^\bullet}$ of \S \ref{sect:Y0cc}.

\section{Stratified semi-smallness}\label{sect:centralfiber}

We keep the assumptions of \S\ref{sect:Heckeact}, i.e.,  
$B$ acts simply transitively on $X^\circ$ and every simple root of $G$ is
a spherical root of type $T$.

\smallskip

The main result of this section is the following, which 
will be proved in \S\ref{sect:semismall}. 

\begin{thm}  \label{thm:barpiperverse}
Under the assumptions above, 
$\bar\pi_!(\IC_{\barY^{\ch\lambda}})$ is perverse
and constructible with respect to the stratification 
$\sA^{\ch\lambda} = \bigcup_{\deg(\mf P)=\ch\lambda} \oo C^{\mf P}$
of Proposition~\ref{prop:Astrata}.
\end{thm}

\subsection{Upper bounds on dimension}

Define 
\begin{align*}
    \barY^{\ch\lambda,\succeq \ch\theta} &= \barY^{\ch\lambda} \xt_{\sM_X} \barM_X^{\ch\theta} \subset \barY^{\ch\lambda} \\
    \olsf Y^{\ch\lambda, \succeq \ch\theta} &= \barY^{\ch\lambda, \succeq\ch\theta} \cap \olsf Y^{\ch\lambda} \subset \olsf S^{\ch\lambda}
\end{align*} 
and let $\sY^{\ch\lambda,\succeq \ch\theta}, \msf Y^{\ch\lambda,\succeq \ch\theta}$ denote the corresponding intersections with $\sY^{\ch\lambda}$
(the notation is  
justified by Proposition~\ref{prop:closure-rel}). 
From the stratification $\barY^{\ch\lambda} = \bigcup {_{\ch\nu}}\barY^{\ch\lambda}$ we deduce that 
\[\olsf Y^{\ch\lambda,\succeq\ch\theta} = \bigcup_{\ch\lambda' \le \ch\lambda} \msf Y^{\ch\lambda',\succeq\ch\theta}.\]

\begin{prop} \label{prop:centralstrata} 
Let $\ch\lambda \in \mathfrak c_X \sm 0$ and 
$\ch\theta \in \mathfrak c_X^-$. 
For any connected component $\mathscr Y$ of 
$\sY^{\ch\lambda,\ch\theta}$, 
we have
\begin{equation} \label{e:semismall-ineq}
    \dim ( \msf Y^{\ch\lambda} \cap \mathscr Y ) 
    \le \frac 1 2 ( \dim (\mathscr Y) - 1 ). 
\end{equation}

Moreover, whenever the inequality above is an equality for an irreducible component $\msf Y$ of $\msf Y^{\ch\lambda} \cap \mathscr Y$, the same holds for the inequality of Proposition \ref{prop:intersections}; that is, the closure $\olsf Y$ in the affine Grassmannian $\Gr_G$ meets a semi-infinite orbit $\msf S^{\ch\lambda'}$ with $\left< \rho_G, \ch\lambda-\ch\lambda'\right> = \dim \olsf Y$.
\end{prop}

\begin{proof}
Let $\ol{\mathscr Y}$ denote the closure of $\mathscr Y$ in $\barY^{\ch\lambda}$.
Recall from Corollary~\ref{cor:Ycomp} that $\mathscr Y = 
\sY^{\ch\lambda} \times_{\sM_X} \mathscr M$ 
where $\mathscr M$ is a connected component of 
$\sM_X^{\ch\theta}$. 
Then we have $\ol{\mathscr Y} \subset \barY^{\ch\lambda} \xt_{\sM_X} \ol{\mathscr M}$
where $\ol{\mathscr M}$ is the closure of $\mathscr M$ in $\barM_X^{\ch\theta}$.

Fix an irreducible component $\msf Y$ of 
$\msf Y^{\ch\lambda}\cap \mathscr Y$ and let $\olsf Y$ be its closure in 
$\olsf Y^{\ch\lambda} \cap \ol{\mathscr Y}$. 
From Proposition~\ref{prop:intersections} we know, by intersecting $\olsf Y$ with semi-infinite orbits in the affine Grassmannian, that there is a $\ch\lambda'\le \ch\lambda$ with $\msf Y':= \olsf Y \cap \msf S^{\ch\lambda'}$ nonempty of dimension zero, and $d:=\dim \msf Y \le \brac{ \rho_G, \ch\lambda-\ch\lambda'}$. 
Note that $\msf Y' \subset \sY^{\ch\lambda'} \cap \ol{\mathscr Y}
\subset \sY^{\ch\lambda'} \xt_{\sM_X} \ol{\mathscr M}$.

Since $\ch\lambda - \ch\lambda' \in \ch\Lambda^\pos_G$, we can decompose 
$\ch\lambda - \ch\lambda' = \sum_{\alpha\in \Delta_G} n_\alpha (\ch\nu_{D^+_\alpha}+\ch\nu_{D^-_\alpha})$
where the sum is over simple roots and 
$\sum n_\alpha = \brac{\rho_G,\ch\lambda-\ch\lambda'}$. 
By the graded factorization property and Lemma~\ref{lem:Yraise}, there is an \'etale map 
\begin{equation} \label{e:YaddC2}
    (\sY^{\ch\lambda'} \xt_{\sM_X} \ol{\mathscr M}) \oo\xt 
    \underset{\alpha \in \Delta_G}{\oo\prod} (C^{(n_\alpha)}\oo\xt C^{(n_\alpha)}) 
    \to
    \Scr Y^{\ch\lambda} \xt_{\sM_X} \ol{\mathscr M}
\end{equation}
over $\ol{\mathscr M}$. 

If $\ch\lambda' \ne 0$, then $\sY^{\ch\lambda'} \xt_{\sM_X} \ol{\mathscr M}$
contains $\msf Y' \tilde\times C$ and hence 
\eqref{e:YaddC2} implies that 
$\dim (\mathscr Y) \ge 1+2\brac{\rho_G,\ch\lambda-\ch\lambda'}$.
Therefore, 
\begin{equation} \label{e:MVineq}
d\le     \brac{\rho_G,\ch\lambda-\ch\lambda'} \le \frac 1 2 (\dim (\mathscr Y) - 1)
\end{equation}
if $\ch\lambda'\ne 0$. 
This proves the
claim in the case $\ch\lambda' \ne 0$. 

There remains to consider when $d= \brac{\rho_G,\ch\lambda-\ch\lambda'}$ and $\ch\lambda'=0$. 
In particular, $\msf Y^0 \xt_{\sM_X} \barM^{\ch\theta}_X$ is nonempty, which can only
be the case if $\ch\theta = 0$ since $\msf Y^0 = \msf Y^{0,0}=\pt$. 
When considering $\sY^{\ch\lambda,0}$, we may
use \S\ref{sect:freemonoid} to reduce to assuming that $\mf c_X = \mbb N^{\Cal D}$. 
In this case again by Proposition \ref{prop:intersections}, there is a simple root $\alpha$ 
and an irreducible component $\msf Y_{d-1}$ of $\olsf Y \cap \msf S^{\ch\alpha}$
of dimension $1$ such that $\ol{\msf Y_{d-1}}$ contains $\msf Y^{0,0} = \pt$. 
This implies that $\msf Y_{d-1}$ is an irreducible component of 
$\msf Y^{\ch\alpha,0}$ of dimension $1$. 
Since $\mf c_X$ is now the free monoid, $\ch\alpha = \ch\nu_{D_\alpha^+} + \ch\nu_{D_\alpha^-}$
for $\Cal D(\alpha) = \{ D_\alpha^+, D_\alpha^- \}$. 
Lemma~\ref{lem:Ycolor} implies that $\msf Y^{\ch\alpha,0}$ is empty, a contradiction.
\end{proof}

\subsection{Connected components of open Zastava} \label{sect:connected-open}
Recall from \S\ref{sect:Y0cc} that $\sY_X^{?,0}=\sY_{X^\bullet}$ is a disjoint union of subschemes
$\sY_{X^\bullet}^{D}$, indexed by $D \in \mbb N^{\Cal D}$, i.e., multisets of colors. 
Therefore, we have a map $\pi_0(\sY_{X^\bullet}) \to \mbb N^{\Cal D}$. 
On the other hand, Corollary~\ref{cor:compY} gives an injection 
$\pi_0(\sY_{X^\bullet}) \into \pi_1(H) \times \mf c_X$ 
where $^{\ch\mu}\sY^{\ch\lambda,0} \mapsto (\ch\mu,\ch\lambda)$ when the former
is nonempty. 

Observe that $\pi_1(H) \ot \bbQ \cong \mathcal X(H) \ot \bbQ$. 
Then tensoring \eqref{e:color-ses} by $\bbQ$, we get an injection 
\begin{equation} \label{e:color-inj-piH}
    \mathbb Z^{\Cal D} \into \mbb Q^{\Cal D} \into (\pi_1(H) \xt \ch\Lambda_X) \ot \bbQ. 
\end{equation}
One can check that the maps above fit into a commutative diagram 
\begin{equation} \label{e:Y0toBunH}
\begin{tikzcd}
\pi_0(\sY_{X^\bullet}) \ar[r, hook] \ar[d] & \pi_1(H)\times \mf c_X \ar[d] \\ 
\mbb N^{\Cal D} \ar[r, hook, "{\eqref{e:color-inj-piH}}"] & 
(\pi_1(H)\times \ch\Lambda_X) \ot \bbQ 
\end{tikzcd}
\end{equation}
We now show that the left vertical arrow is a bijection.

\begin{lem} \label{lem:oYconnected}
For $D \in \mathbb N^{\Cal D}$, 
the smooth scheme $\sY^{D}_{X^\bullet}$ is connected of dimension $\len(D)$.
\end{lem}
\begin{proof}
As discussed in \S\ref{sect:Y0cc}, we may assume that $X=X^\can$ and 
$\mathfrak c_X = \mathbb N^{\Cal D}$. 
Now the claim is equivalent to showing that, under those assumptions, 
$\sY^{\ch\lambda,0}$ is connected for $\ch\lambda = \sum n_{D'} \ch\nu_{D'}$. 
One deduces from the graded factorization property of $\sY$ that 
if $\sY^{\ch\lambda,0}$ is not connected, there must exist a possibly different $\ch\lambda$
with a connected component $\mathscr Y$ contained entirely in the preimage of the diagonal 
$\sY^{\ch\lambda,0} \xt_{\sA^{\ch\lambda}, \delta^{\ch\lambda}} C$. 
Then $\dim (\msf Y^{\ch\lambda} \cap \mathscr Y) = \dim \mathscr Y - 1$,
and the dimension inequality of Proposition~\ref{prop:centralstrata} can only hold if
$\dim \mathscr Y = 1$.
By Corollary~\ref{cor:Ycomp}, the component $\mathscr Y$ must be of the form 
\[ \mathscr Y = {^{\ch\mu}\sY^{\ch\lambda,0}} = 
\Bun_H^{\ch\mu} \xt_{\Bun_G} \Bun_B^{-\ch\lambda} \]
for some $\ch\mu\in \pi_1(H)$. 
Now by the graded factorization property and Lemma~\ref{lem:Yraise}, 
we have an \'etale map 
\[ ^{\ch\mu}\sY^{\ch\lambda,0} \oo\xt 
\prod^\circ_{\alpha\in\Delta_G} C^{(N_{\alpha})} \oo\xt C^{(N_{\alpha})} \to 
{^{\ch\mu}}\sY^{\ch\lambda',0} \]
for $\ch\lambda' = \ch\lambda + \sum N_\alpha\ch\alpha$ and any $N_\alpha$ large enough. 
Thus, $\dim {^{\ch\mu}}\sY^{\ch\lambda',0} = 1+2\sum N_\alpha$. 
By Lemma~\ref{lem:BunBtoGsm} we may assume that 
$\Bun_B^{-\ch\lambda'} \to \Bun_G$ is smooth. 
Note that $\dim \Bun_H^{\ch\mu}$ only depends on the image of $\ch\mu$ in $\pi_1(H)\ot \bbQ$. 
On the other hand, the commutative diagram \eqref{e:Y0toBunH} 
says that this image is determined by $\ch\lambda'$. 
These observations imply that $\ds \sY^{\ch\lambda',0}
= \Bun_H \xt_{\Bun_G} \Bun_B^{-\ch\lambda'}$ is equidimensional. 
However we know that $\sY^{\ch\lambda',0}$ has a connected component 
birational to $\sA^{\ch\lambda'}$, which is of dimension $\sum n_{D'} + 2\sum N_\alpha$.
The equality 
\[ \textstyle \dim {^{\ch\mu}}\sY^{\ch\lambda',0} = 1+2\sum N_\alpha = \sum n_{D'} + 2\sum N_\alpha \]
forces $\ch\lambda = \ch\nu_{D'}$ for some color $D'\in \Cal D$, hence $D=D'$.
Lemma~\ref{lem:Ycolor} now implies that $\sY^{\ch\nu_D,0}$ is connected.
\end{proof}

\begin{cor} \label{cor:Ycomp2}
For every $\ch\lambda\in\mf c_X$, $\ch\theta\in\mf c_X^-$, the connected components of $\sY^{\ch\lambda,\ch\theta}$ are
in bijection with the closures of 
\[ \sY^{D}_{X^\bullet} \oo\xt \sY^{\ch\theta,\ch\theta} \into \sY^{\ch\lambda,\ch\theta} \]
for $D\in \mbb N^{\Cal D}$ such that $\varrho_X(D) = \ch\lambda-\ch\theta$. 
\end{cor}
\begin{proof}
Immediate from Lemmas~\ref{lem:comp-cover}(i) and \ref{lem:oYconnected}.
\end{proof}

\subsection{Stratified semi-smallness} \label{sect:semismall}

Following \cite[\S 4]{MV}, we will use  
the notion of a stratified semi-small map, which we now review.
Let $f : Y \to A$ be a proper map between two stratified spaces 
$(Y,\Scr S)$ and $(A,\Scr T)$. Suppose all strata 
are smooth and connected and each $f(S), S\in \Scr S$ is a union of strata 
$S' \in \Scr T$. We say $f$ is \'etale-\emph{locally trivial} 
(in the stratified sense) if whenever $S' \subset f(S)$, the restriction of 
$f$ to $S\cap f^{-1}(S') \to S'$ is \'etale-locally
a trivial fibration. We say that $f$ is \emph{stratified semi-small} 
if it is \'etale-locally trivial and 
for any $S\in \Scr S$ and any $S'\in \Scr T$ such that
$S' \subset f(S)$ we have 
\begin{equation} \label{e:ss-ineq}
    \dim ( f^{-1}(a) \cap S) \le \frac 1 2 (\dim S - \dim S') 
\end{equation}
for any (and thus all) $a\in S'$. 

The notion of stratified semi-smallness is relevant due to 
the observation below, which follows from dimension counting and the 
definition of the perverse t-structure:

\begin{lem}[{\cite[Lemma 4.3]{MV}}]  \label{lem:semismallperverse}
If $f$ is a stratified semi-small map then $f_*(\Scr F) \in \rmP_{\Scr T}(A)$ for all 
$\Scr F \in \rmP_{\Scr S}(Y)$. 
\end{lem}

Note that the lemma holds even if the stratifications are not Whitney. 
In this case we simply define $\rmP_{\Scr S}(Y) := \rmP(Y) \cap \rmD^b_{\Scr S}(Y)$
to be the subcategory of perverse sheaves that are $\Scr S$-constructible, i.e., 
the $\Scr F \in \rmP(Y)$ such that $\rmH^i(\Scr F)|_S$ is a local system
of finite rank for all $i\in \bbZ$ and $S\in \Scr S$. 

\subsubsection{}
Let us return to our situation: consider the proper map 
$\bar\pi : \barY^{\ch\lambda} \to \sA^{\ch\lambda}$.

We have the smooth stratification defined in Proposition~\ref{prop:barYstrata}, 
\[ \barY^{\ch\lambda} = \bigcup_{\ch\nu,\ch\Theta} 
{}_{\ch\nu}\barY^{\ch\lambda,\ch\Theta}, \qquad 
{_{\ch\nu}}\barY^{\ch\lambda,\ch\Theta} \cong 
C_{\ch\nu} \xt \sY^{\ch\lambda-\ch\nu,\ch\Theta} \] 
for $\ch\nu \in \ch\Lambda^\pos_G,\, 
\ch\Theta \in \Sym^\infty(\mathfrak c_X^-\sm 0)$.
Let $(\barY^{\ch\lambda},\Scr S)$ denote the stratification by
 the connected components of
$_{\ch\nu}\sY^{\ch\lambda, \ch\Theta}$.
We will not show that $\Scr S$ is a Whitney stratification; for our purposes 
the following suffices: 

\begin{lem} \label{lem:barICwhitney}
For any $\ch\lambda \in \mathfrak c_X$, 
the IC complex of $\barY^{\ch\lambda}$ is $\Scr S$-constructible, 
i.e., 
\[ \IC_{\barY^{\ch\lambda}} \in \rmP_{\Scr S}(\barY^{\ch\lambda}). \]
\end{lem}
\begin{proof} 
Let $\ch\nu \in \ch\Lambda^\pos_G$ be such that 
$\barY^{\ch\lambda} = {_{\le\ch\nu}}\barY^{\ch\lambda}$, which exists 
since $\barY^{\ch\lambda}$ is of finite type. 
For any $\ch\mu \in \ch\Lambda^\pos_G$ large enough, we have
a smooth correspondence preserving stratifications
\[ \barY^{\ch\lambda} \leftarrow \barY^{\ch\lambda}\oo\xt \sY^{\ch\mu,0} \to 
{_{\le\ch\nu}}\barY^{\ch\lambda+\ch\mu} =: Y\]
where the right arrow comes from the graded factorization property of $\barY$.
Thus, it suffices to check that 
$\IC_Y$ is $\Scr S$-constructible. 
By \eqref{e:ICbar=bt}, we have an isomorphism 
\[ \IC_Y \cong \Bigl.\Bigl(\IC_{\sM_X} \bt_{\Bun_G} \IC_{\ol\Bun_B^{-\ch\lambda-\ch\mu}}\Bigr)\Bigr|_Y. \] 
Now Proposition~\ref{prop:globwhitney} implies that 
$\IC_{\sM_X}$ is constructible with respect to the fine stratification on 
$\sM_X$, and Theorem~\ref{thm:BFGM} implies that $\IC_{\ol\Bun_B}$ is constructible
with respect to the stratification by defect. 
Thus, it follows that $\IC_Y$ is $\Scr S$-constructible.
\end{proof}

Let $(\Scr A^{\ch\lambda}, \Scr T)$ denote the stratification 
from Proposition~\ref{prop:Astrata}, 
\[ \Scr A^{\ch\lambda} = \bigcup_{\mf P} \oo C^{\mf P} \]
for $\mf P\in \Sym^\infty(\mathfrak c_X\sm 0)$ such that 
$\deg(\mf P)=\ch\lambda$.

Now Theorem~\ref{thm:barpiperverse} follows 
from Lemmas~\ref{lem:semismallperverse}, \ref{lem:barICwhitney} 
and the following theorem:

\begin{thm} \label{thm:semismall}
The map $\bar\pi : (\barY^{\ch\lambda}, \Scr S) \to (\Scr A^{\ch\lambda}, \Scr T)$ is stratified semi-small. 
\end{thm}
\begin{proof}
Fix $\mf P \in \Sym^\infty(\mathfrak c_X\sm 0)$ such that 
$\ch\lambda=\deg(\mf P)$. 
Let 
$I = \{ 1,\dotsc, \abs{\mf P}\}$ be the finite set of cardinality 
$\abs{\mf P}$. We fix an ordering 
$\mf P = \sum_{i\in I} [\ch\lambda_i]$ on our partition $\mf P$,
so each $\ch\lambda_i \in \mathfrak c_X$. 
Let $\oo C^I$ denote the $I$-fold product $C^I$ with the diagonal
divisor removed. Then the map $\oo C^I \to \oo C^{\mf P}$
corresponds to 
choosing an ordering on an unordered multiset of points in $C$.

We leave it to the reader to check that $\bar\pi(S)$ is a union of
strata for $S \in \Scr S$. 
It is also a consequence of the graded factorization property
that for 
$\ch\nu\in \ch\Lambda^\pos_G, \ch\Theta \in \Sym^\infty(\mathfrak c_X \sm 0)$,
the fiber product 
\begin{equation} \label{e:preimagefactorize}
    _{\ch\nu}\barY^{\ch\lambda,\ch\Theta} 
    \xt_{\sA^{\ch\lambda}} \oo C^I  
    =   \bigsqcup_{i \mapsto \ch\nu_i,\ch\lambda'_i, \ch\theta_i}
    \underset{i\in I}{\oo\prod} (\msf Y^{\ch\lambda'_i, \ch\theta_i} \ttimes C ) 
\end{equation}
is equal to a disjoint union of open and closed subschemes,
running over all assignments $i\in I \mapsto \ch\nu_i \in \ch\Lambda^\pos_G, 
\ch\lambda'_i \in \mf c_X, \ch\theta_i \in \mf c_X^-$ (including zero)
such that $\ch\nu_i + \ch\lambda'_i = \ch\lambda_i$
and $\sum_I [\ch\theta_i] = \ch\Theta + (\abs I - \abs{\ch\Theta})[0] \in 
\Sym^\infty(\mf c_X)$ as a multiset \emph{with zero}.

Since $\oo C^I \to \oo C^{\mf P}$ is finite \'etale,
(and everything is of finite type) we deduce that the restriction of $\bar\pi$ to 
$_{\ch\nu}\barY^{\ch\lambda,\ch\Theta} \cap \bar\pi^{-1}( \oo C^{\mf P} )
    \to \oo C^{\mf P}$
is \'etale-locally a trivial fibration.

It remains to check the dimension inequality \eqref{e:ss-ineq}.
Let $\mf P, \ch\nu, \ch\Theta$ be as before, and  
fix a connected component $S$ of $_{\ch\nu}\barY^{\ch\lambda,\ch\Theta}$.
We deduce from 
\eqref{e:preimagefactorize} that for any $a \in \oo C^{\mf P}$,
the restricted fiber 
$\bar\pi^{-1}(a) \cap S$
is contained in a union of 
$\oo\prod_I (\msf Y^{\ch\lambda'_i,\ch\theta_i}\cap S_i)$ for 
$\ch\nu_i,\ch\lambda'_i,\ch\theta_i$ as above
and $S_i$ some connected component of $\sY^{\ch\lambda'_i,\ch\theta_i}$. 
By Proposition~\ref{prop:centralstrata}, we have 
\begin{equation}  \label{e:dimYile}
\dim( \msf Y^{\ch\lambda'_i, \ch\theta_i} \cap S_i ) \le 
\frac 1 2 ( \dim(S_i) - 1 ).
\end{equation}
The image of the composition
\[ \underset{i\in I}{\oo\prod} (C_{\nu_i} \xt S_i) \to \underset{i\in I}{\oo\prod}
{_{\ch\nu_i} \barY^{\ch\lambda_i,\ch\theta_i} } \to {_{\ch\nu}\barY^{\ch\lambda,\ch\Theta}} \]
is connected, so it must be contained in $S$. 
Thus, $\dim(S) \ge \brac{\rho_G,\ch\nu} + \sum_I \dim(S_i)$.
Summing \eqref{e:dimYile} over $I$, we get 
\[ \dim( \bar\pi^{-1}(a) \cap S ) \le 
\frac 1 2 ( \dim(S) - \dim( \oo C^{\mf P} )- \brac{\rho_G,\ch\nu}  ), \]
which establishes the inequality \eqref{e:ss-ineq}.
\end{proof}

\subsection{Euler product}
Let us explain how we can combine the graded 
factorization property of $\barY$ with 
Theorem~\ref{thm:semismall} to deduce that 
$\bar\pi_!(\IC_{\barY})$ ``looks like an Euler product''. 

In the expression that we are about to obtain, a special role will be played by those strata $\sY^{\ch\lambda,\ch\theta}$ of $\sY^{\ch\lambda}$ which correspond to elements $\ch\theta\in \Cal D^G_{\mathrm{sat}}(X) \cup \{0\}$. Recall, by Corollary \ref{cor:Ycomp}, that the closures of those strata are unions of irreducible components of $\sY$. Let 
\begin{equation}\label{e:crystallambda} \mB_{X,\ch\lambda} = \bigcup_{\ch\theta\in \Cal D^G_{\mathrm{sat}}(X) \cup \{0\}}\bigcup_{\mathscr Y} {\mB_{\mathscr Y}},
\end{equation}
where $\mathscr Y$ runs over all irreducible components of $\sY^{\ch\lambda,\ch\theta}$, and ${\mB_{\mathscr Y}}$ is the set of those irreducible components $\msf b$ of the central fiber $\msf Y^{\ch\lambda}\cap \mathscr Y$ for which the inequality of \eqref{e:semismall-ineq} is an equality, that is, 
\begin{equation}
 \dim\msf b = \frac{1}{2} (\dim \mathscr Y-1).
\end{equation}
Such components will be said to be of \emph{critical dimension}; we will explicate this dimension in Proposition \ref{prop:Vbasis}. The sets $\mB_{X,\ch\lambda}$ will define the \emph{crystal of $X$} in Section \ref{sect:crystal}. For now, we treat them as a black box.

We let $V_{X,\ch\lambda}$ denote the free vector space on $\mB_{X,\ch\lambda}$, that is, 
\begin{equation}\label{e:defVX}
V_{X,\ch\lambda} = \bigoplus_{\mB_{X,\ch\lambda}} \ol\bbQ_\ell.
\end{equation}

Note that, when $X$ is defined over a finite field $\mbb F$, and $k$ is its algebraic closure, the (geometric) Frobenius morphism induces a dimension-preserving bijection between the sets $\mB_{X,\ch\lambda}$ and $\mB_{X,\Fr\ch\lambda}$, for every $\ch\lambda\in\ch\Lambda$. Hence, $\Fr$ acts naturally on the sum of vector spaces $\bigoplus_{\ch\lambda\in\mathfrak c_X} V_{X,\ch\lambda}$. 

For a partition $\mf R = \sum_{\ch\mu \in \mathfrak c_X \sm 0} N_{\ch\mu} [\ch\mu] \in \Sym^\infty(\mathfrak c_X\sm 0)$, 
let $\iota^{\mf R} : \oo C^{\mf R} := \oo\prod \oo C^{(N_{\ch\mu})}
\into \sA^{\ch\lambda}$ denote the locally closed embedding, with $\ch\lambda = \deg(\mf R)$.
This extends to a finite map $\bar\iota^{\mf R} : C^{\mf R}:=\prod C^{(N_{\ch\mu})} \to \sA^{\ch\lambda}$ which is the normalization of the closure of $\oo C^{\mf R}$ in 
$\sA^{\ch\lambda}$.

\begin{prop} \label{prop:format}
For $\ch\lambda \in \mathfrak c_X$, there
exists a canonical isomorphism 
\begin{equation} \label{e:format1}
 \bar\pi_!(\IC_{\barY^{\ch\lambda}}) \cong 
\bigoplus_{\deg(\mf R)=\ch\lambda} 
\Bigl( \bigotimes_{\ch\mu} 
\Sym^{N_{\ch\mu}}(V_{X,\ch\mu}) \Bigr) \ot 
\bar\iota^{\mf R}_!( \IC_{C^{\mf R}} )
\end{equation}
where $\mf R = \sum_{\ch\mu \in \mathfrak c_X\sm 0} 
N_{\ch\mu} [\ch\mu]$ and the spaces $V_{X,\ch\mu}$ are defined by \eqref{e:defVX}. 

When $X$ is defined over a finite field $\mbb F$, and $k$ is its algebraic closure, this isomorphism is Galois-equivariant. 
\end{prop}

The implied Galois action on the right hand side  of \eqref{e:format1} is the one obtained by the action of Frobenius on the sum of spaces $V_{X,\ch\mu}$ (by permuting their basis elements), and the standard Weil structure $\ol\bbQ_\ell(\frac{|\mf R|}{2})[|\mf R|]$ on $\IC_{C^{\mf R}}$.

Note that if $\mf R = [\ch\lambda]$ is the singleton partition, then
$C^{[\ch\lambda]}=C$ and 
$\iota^{[\ch\lambda]}=\delta^{\ch\lambda} : C \into \sA^{\ch\lambda}$
is the diagonal embedding. 
The corresponding summand of $\bar\pi_!(\IC_{\barY^{\ch\lambda}})$
above is $V_{X,\ch\lambda} \ot \IC_C$. 
We call this the \emph{diagonal contribution} of  
$\bar\pi_!(\IC_{\barY^{\ch\lambda}})$. 

\begin{proof}
The proof follows the same logic as \cite[\S 5.4, 5.11]{BFGM}. 
Theorem~\ref{thm:barpiperverse} implies that 
$\bar\pi_!(\IC_{\barY^{\ch\lambda}})$ is perverse, and 
the decomposition theorem (\cite[Th\'eor\`eme 6.2.5]{BBDG}) 
implies that it is semisimple. 
Since $\bar\pi_!(\IC_{\barY^{\ch\lambda}})$ is constructible
with respect to the stratification by $\iota^{\mf R} : \oo C^{\mf R} \into \sA^{\ch\lambda}$ for $\mf R \in \Sym^\infty(\mathfrak c_X\sm 0),\,
\deg(\mf R)=\ch\lambda$, we deduce that there exists a \emph{canonical} 
decomposition
\begin{equation}\label{e:format2} 
\bar\pi_!(\IC_{\barY^{\ch\lambda}}) \cong 
\bigoplus_{\deg(\mf R)=\ch\lambda} \iota^{\mf R}_{!*}(\Scr L^{\mf R}) [\abs{\mf R}] 
\end{equation}
where $\iota^{\mf R}_{!*}$ denotes the middle extension functor
along the locally closed embedding, and $\Scr L^{\mf R}$ 
is a local system on $\oo C^{\mf R}$.

Now consider the singleton partition $[\ch\lambda]$. 
For $v \in \abs C$ let $\delta^{\ch\lambda}_v : v \to\sA^{\ch\lambda}$
denote the composition of $v\to C$ with $\delta^{\ch\lambda}:C\to \sA^{\ch\lambda}$.
Recall from \S\ref{sect:YttimesC} 
that $\barY^{\ch\lambda} \xt_{\sA^{\ch\lambda}, \delta^{\ch\lambda}} C
\cong \olsf Y^{\ch\lambda} \xt^{\Aut k\tbrac t} C^\wedge$. 
Taking the $*$-pullback along $\delta_v^{\ch\lambda}$ of \eqref{e:format2}, we 
have 
\[ R\Gamma( \olsf Y^{\ch\lambda}, \IC_{\barY^{\ch\lambda}}|^*_{\olsf Y^{\ch\lambda}})
\cong \bigoplus_{\deg(\mf R)=\ch\lambda} (\delta^{\ch\lambda}_v)^* \iota^{\mf R}_{!*}(\Scr L^{\mf R})[\abs{\mf R}]. \]
For $\mf R = [\ch\lambda]$, we have $\Scr L^{[\ch\lambda]}$ is a local system 
on $C$, so $(\delta_v^{\ch\lambda})^* \iota^{[\ch\lambda]}_{!*}(\Scr L^{[\ch\lambda]})[1]$
lives in cohomological degree $-1$. 
For $\mf R \ne [\ch\lambda]$, we have $\dim C^{\mf R} > 1$
so the uniqueness property of middle extension implies that
$(\delta^{\ch\lambda})^* \iota^{\mf R}_{!*}(\Scr L^{\mf R})[\abs{\mf R}]$ 
has lisse cohomology sheaves on $C$ and lives in usual cohomological degrees $< -1$. 
Therefore, we deduce that 
\[ 
    \Scr L^{[\ch\lambda]} |^*_{v\to C} = H_c^{-1}(\olsf Y^{\ch\lambda}, \IC_{\barY^{\ch\lambda}}|^*_{\olsf Y^{\ch\lambda}}),
\]
which is the top cohomological degree. 

Recall from Corollary \ref{cor:compY} that the irreducible components of $\sY^{\ch\lambda}$ are naturally parametrized by a subset of $\pi_1(H) \xt \Sym^\infty(\Cal D^G_{\mathrm{sat}}(X))$. 
Let $\sY^{\ch\lambda,\circ}$ denote the union of $\sY^{\ch\lambda,\ch\Theta}$ 
for all $\ch\Theta \in \Sym^\infty(\Cal D^G_{\mathrm{sat}}(X))$. 
Then $\sY^{\ch\lambda,\circ}$ is a disjoint union of smooth connected components
in bijection with the irreducible components of $\sY^{\ch\lambda}$, by Corollary \ref{cor:Ycomp}.   
Again, the uniqueness property of the IC complex implies that
$\IC_{\barY^{\ch\lambda}}|^*_{\barY^{\ch\lambda}\sm \sY^{\ch\lambda,\circ}}$ 
lives in strictly negative perverse cohomological degrees. 
Since $\bar\pi_!$ is perverse $t$-exact by Theorem~\ref{thm:semismall}, we
deduce that 
$\bar\pi_!( \IC_{\barY^{\ch\lambda}}|^*_{\barY^{\ch\lambda}\sm \sY^{\ch\lambda,\circ}} )$ 
is constructible and lives in strictly negative perverse degrees. 
This in turn implies that $(\delta^{\ch\lambda}_v)^* \bar\pi_!( \IC_{\barY^{\ch\lambda}}|^*_{\barY^{\ch\lambda}\sm \sY^{\ch\lambda,\circ}} )$ lives in (usual=perverse) degrees $<-1$. 
We conclude that 
\[
    \Scr L^{[\ch\lambda]}|^*_{v\to C} = 
    H^{-1}_c(\msf Y^{\ch\lambda}, (\IC_{\barY^{\ch\lambda}}|_{\sY^{\ch\lambda,\circ}})|^*_{\msf Y^{\ch\lambda}} ) = 
\bigoplus_{\mathscr Y} H_c^{\dim \mathscr Y - 1}(\msf Y^{\ch\lambda} \cap \mathscr Y, \ol\bbQ_\ell({\textstyle \frac{\dim \mathscr Y}{2}})),\]
with the sum running over all irreducible components of $\sY^{\ch\lambda,\circ}$.
Note that $\msf Y^{\ch\lambda}\cap \sY^{\ch\lambda,\ch\Theta}$ is empty unless 
$\ch\Theta = [\ch\theta]$ is singleton, for $\ch\theta \in \Cal D^G_{\mathrm{sat}}(X) \cup \{0\}$. 
Moreover, the right hand side consists of only \emph{top} cohomological degrees, by Proposition~\ref{prop:centralstrata}, so 
$H_c^{\dim \mathscr Y - 1}(\msf Y^{\ch\lambda} \cap \mathscr Y, \ol\bbQ_\ell (\frac{\dim \mathscr Y}{2}))$ 
is equal to the sum 
\[\bigoplus_{\mf B_{\mathscr Y}} \ol\bbQ_\ell(\smallfrac{1}{2}),\]
where $\mf B_{\mathscr Y}$ is the set of irreducible components of $\msf Y^{\ch\lambda}\cap \mathscr Y$ of dimension $\frac{\dim \mathscr Y - 1}{2}$.  
In particular, $\Scr L^{[\ch\lambda]}|^*_{v\to C}$ has trivial monodromy under $\Aut k\tbrac t$, 
so we  deduce that $\Scr L^{[\ch\lambda]} \cong V_{X,\ch\lambda} \otimes 
\IC_C$ where $V_{X,\ch\lambda}$ is defined by \eqref{e:defVX} (with the $(\frac{1}{2})$-twist absorbed by $\IC_C$). 

Next, consider an arbitrary partition 
$\mf R = \sum_{\ch\mu\in \mathfrak c_X\sm 0} N_{\ch\mu} [\ch\mu]$. 
We defined $\oo C^{\mf R} = \oo\prod_{\ch\mu} \oo C^{(N_{\ch\mu})}$.
By the graded factorization property, we have a diagram
with Cartesian squares
\[ 
\begin{tikzcd}
    \oo\prod_{\ch\mu} (\olsf Y^{\ch\mu}\ttimes C)^{\oo\xt N_{\ch\mu}} \ar[r, hook] \ar[d] &
    \oo\prod_{\ch\mu} (\barY^{\ch\mu})^{\oo\xt N_{\ch\mu}}
    \ar[r, "\text{\'etale}"] \ar[d] & \barY^{\ch\lambda}  
    \ar[d] \\ 
    \oo\prod \oo C^{N_{\ch\mu}} \ar[r, hook] &
    \oo\prod_{\ch\mu} (\sA^{\ch\mu})^{\oo\xt N_{\ch\mu}}
    \ar[r, "\text{\'etale}"] &     \sA^{\ch\lambda}
\end{tikzcd}
\]
and the composition of the bottom arrows 
factors through the $(\prod_{\ch\mu} {\mf S}_{N_{\ch\mu}})$-torsor
\[ \textstyle \oo \prod \oo C^{N_{\ch\mu}} \to \oo \prod \oo C^{(N_{\ch\mu})} = \oo C^{\mf R},
\] 
where $\mf S_N$ denotes the symmetric group on $N$ elements.
By induction, we deduce that 
\begin{equation} \label{e:monodromyS}
     \Scr L^{\mf R}[ \abs{\mf R} ]|^*_{\oo\prod \oo C^{N_{\ch\mu}} } 
\cong ( \bt_{\ch\mu} (V_{X,\ch\mu} \ot \IC_C)^{\bt N_{\ch\mu}} )|^*_{\oo \prod \oo C^{N_{\ch\mu}} }. 
\end{equation}
There is a natural $\mf S_{N_{\ch\mu}}$-equivariant structure 
on $(V_{X,\ch\mu} \ot \IC_C)^{\bt N_{\ch\mu}}$ 
compatible with the $\mf S_{N_{\ch\mu}}$-action on $C^{N_{\ch\mu}}$.
On the other hand, we have the map
\[ (\olsf Y^{\ch\mu} \ttimes C)^{\oo\xt N_{\ch\mu}} \to 
\sY^{N_{\ch\mu}\cdot\ch\mu} \xt_{\sA^{N_{\ch\mu}\cdot\ch\mu}} \oo C^{(N_{\ch\mu})} 
\] 
which is a $\mf S_{N_{\ch\mu}}$-torsor, where $\mf S_{N_{\ch\mu}}$
acts on the left hand side in the natural way. 
Now from the definition of $V_{X,\ch\mu}$ we deduce
that the isomorphism \eqref{e:monodromyS} must 
intertwine the $\prod \mf S_{N_{\ch\mu}}$-structures.
By Galois descent we conclude that
$\Scr L^{\mf R} [ \abs{\mf R} ]\cong (\bigotimes_{\ch\mu} 
\Sym^{N_{\ch\mu}}(V_{X,\ch\mu})) \ot \IC_{\oo C^{\mf R}}$. 
Since $\bar\iota^{\mf R} : C^{\mf R} \to \sA^{\ch\lambda}$ is
the normalization of the closure of the stratum $\oo C^{\mf R}$ in $\sA^{\ch\lambda}$,
the middle extension of $\IC_{\oo C^{\mf R}}$ is
isomorphic to $\bar\iota^{\mf R}_!(\IC_{C^{\mf R}})$.

When $X$ is defined over a finite field $\mbb F$ and $k$ is the algebraic closure, the Galois group of $k$ over $\mbb F$ acts naturally on the set of components $\mf B^+_X := \bigcup_{\ch\mu\in \mf c_X} \mB_{X,\ch\mu}$, and the isomorphisms above are clearly equivariant, taking into account that $\IC_{C^{\mf R}}$ is the constant sheaf $ \ol\bbQ_\ell(\frac{|\mf R|}{2})[|\mf R|]$.
\end{proof}

\subsection{Critical dimension}

We can now give a more precise description of the diagonal contribution
$V_{X,\ch\lambda}$ defined in \eqref{e:defVX}, that is, the free vector space on the set $\mB_{X,\ch\lambda}$ of components of critical dimension in the ``open strata'' of the central fiber $\msf Y^{\ch\lambda}$. 

\begin{prop} \label{prop:Vbasis} 
For $\ch\lambda\in \mathfrak c_X \sm 0$, the set $\mB_{X,\ch\lambda}$ of components of critical dimension on the central fiber $\msf Y^{\ch\lambda}$ consists of
\begin{enumerate}
\item the irreducible components of $\sY^D_{X^\bullet}\cap \msf Y^{\ch\lambda}$ of dimension $\frac 1 2(\len(D)-1)$, 
for $D\in \mathbb N^{\mathcal D}$ with $\varrho_X(D) =\ch\lambda$;
\item the irreducible components of $\msf S^{\ch\lambda} \cap \Gr_G^{\ch\theta}$, 
for $\ch\theta \in \Cal D^G_{\mathrm{sat}}(X)$, embedded in $\msf Y^{\ch\lambda}$ via \eqref{e:actvY}.
\end{enumerate}
\end{prop}

\begin{rem} \label{rem:othertheta}
We have MV cycles for every $\ch\theta \in \mf c_X^-$, but only those belonging to $\Cal D^G_{\mathrm{sat}}(X)$ 
contribute to $V_{X,\ch\lambda}$ since those correspond to $\sY^{\ch\lambda,\ch\theta}$ which are
connected components of $\sY^{\ch\lambda}$. We will reserve the term \emph{critical dimension} of 
$\msf Y^{\ch\lambda}$ for the maximal dimensions in the two cases above. 
\end{rem}

\begin{proof}
By definition, an element of $\mB_{X,\ch\lambda}$ is a component $\msf b$ of the central fiber $\msf Y^{\ch\lambda}\cap \mathscr Y$, where $\mathscr Y$ is an irreducible component of the smooth stratum $\sY^{\ch\lambda,\ch\theta}$, for some $\ch\theta \in \Cal D^G_{\mathrm{sat}}(X)\cup\{0\}$, such that 
\begin{equation} \label{e:ss-eq}
    \dim (\msf b) = \frac 1 2 ( \dim (\mathscr Y) -1).
\end{equation}

If $\ch\theta=0$, then $\sY^{\ch\lambda,0}$ is the disjoint union of 
connected components $\sY^D_{X^\bullet}$ 
for $D\in \mathbb N^{\mathcal D}$ with $\varrho_X(D) =\ch\lambda$, by Lemma~\ref{lem:oYconnected}. 
Then 
\eqref{e:ss-eq} becomes $\dim(\msf b)
= \frac 1 2 (\len(D)-1) $. 

If $\ch\theta \ne 0$, by Corollary \ref{cor:Ycomp2} the connected components of $\sY^{\ch\lambda,\ch\theta}$ are
in bijection with the closures of 
\[ \sY^{D}_{X^\bullet} \oo\xt \sY^{\ch\theta,\ch\theta} \into \sY^{\ch\lambda,\ch\theta} \]
for $D\in \mbb N^{\Cal D}$ such that $\varrho_X(D) = \ch\lambda-\ch\theta$. The statement now follows from the following lemma, which we write separately, for later use, because it applies to arbitrary $\ch\theta \ne 0$.

\end{proof}

\begin{lem} \label{lem:centralfibertheta}
For any $\ch\theta \in \mf c_X^- \sm 0$, $D\in \mbb N^{\Cal D}$
and $\ch\lambda = \varrho_X(D) +\ch\theta$, if we denote by $\mathscr Y$ the closure of the image of 
\[\sY^{D}_{X^\bullet} \oo\xt \sY^{\ch\theta,\ch\theta} \into \sY^{\ch\lambda,\ch\theta},
\]
then we have $\dim (\msf Y^{\ch\lambda}\cap \mathscr Y) \le \frac 1 2\len(D) = \frac 1 2(\dim \mathscr Y -1)$. The irreducible components of $\msf Y^{\ch\lambda,\ch\theta}$ for which this is an equality are precisely 
the MV cycles in $\msf S^{\ch\lambda}\cap \ol\Gr^{\ch\theta}_G$, 
embedded in $\msf Y^{\ch\lambda}$ via \eqref{e:actvY}.
\end{lem}
\begin{proof}
Proposition~\ref{prop:centralstrata} implies that $\dim(\msf Y^{\ch\lambda} \cap \mathscr Y)
\le \frac 1 2(\dim \mathscr Y -1)$. 
Since $\ch\theta\ne 0$, we have $\sY^{\ch\theta,\ch\theta}_\red=C$,  and by Lemma \ref{lem:oYconnected} this inequality translates to 
$\dim(\msf Y^{\ch\lambda} \cap \mathscr Y) \le \frac 1 2 \len(D)$. 

If $v\in \abs C$ is the point we are taking central fibers with respect to,
then $\msf Y^{\ch\lambda}$ maps to 
the substack $\sM_{X,v}\subset \sM_{X}$ of maps that are only $G$-degenerate at $v$. 
Recall from Theorem \ref{thm:comp} that $\sM_{X,v}^{\ch\theta}$ 
is contained in the image of 
\[ 
    \act_{\sM,v} : \Bun_{H} \ttimes \Gr_{G}^{\ch\theta} \to \sM_{X,v} 
\]
and the map is birational onto its image. 
We deduce from \eqref{e:Z0act-diagram} and 
Proposition~\ref{prop:Yact-strat} that the fiber product 
$(\Bun_{H} \ttimes \Gr^{\ch\theta}_G) \xt_{\sM_{X}} \msf Y^{\ch\lambda}$ 
has a stratification by 
\[ \bigcup_{\ch\nu} \msf Y^{\ch\lambda-\ch\nu,0} \ttimes (\msf S^{\ch\nu} \cap \Gr^{\ch\theta}_G) 
\] 
where $\ch\nu$ ranges over the weights of $V^{\ch\theta}$.  

Observe that we have an embedding $\ch\Lambda^\pos_G \into \mathbb N^{\Cal D}$
by sending a simple coroot $\ch\alpha \mapsto D_\alpha^+ + D_\alpha^-$. 
By restricting to $X^\circ P_\alpha$ we can deduce that the image of
$\ch\alpha$ in $\pi_1(H)\ot\bbQ$ under \eqref{e:color-inj-piH} is zero. 
Thus, the commutativity of diagram \eqref{e:Y0toBunH} ensures that if
we restrict to the connected component of $\Bun_H$ corresponding to 
$\sY^{D}_{X^\bullet}$, then the stratification above becomes 
\[ \bigcup_{\ch\nu} (\msf Y^{\ch\lambda-\ch\nu}\cap \sY^{D -(\ch\nu-\ch\theta)}_{X^\bullet}) \ttimes (\msf S^{\ch\nu} \cap\Gr^{\ch\theta}_G ),
\]
where $\ch\nu-\ch\theta\in \ch\Lambda^\pos_G \subset {\mbb N}^{\Cal D}$.
In particular, the dimension of the stratum corresponding to $\ch\nu$ is 
\[ \le \frac 1 2 \Bigl(\len(D-(\ch\nu-\ch\theta))-1 \Bigr)+ \brac{\rho_G, \ch\nu-\ch\theta}
\le \frac 1 2 (\len(D)-1) \]
by Proposition~\ref{prop:centralstrata} \emph{unless} $D= \ch\nu-\ch\theta$. 
Thus, in order for $\dim \msf Y^{\ch\lambda,\ch\theta}=\frac 1 2 \len(\ch\lambda-\ch\theta)$, we must have
$\ch\lambda=\ch\nu$ is a weight of $V^{\ch\theta}$ 
and $\msf Y^{\ch\lambda,\ch\theta}$ is birational to an irreducible
component of $\pt \ttimes (\msf S^{\ch\lambda}\cap \Gr^{\ch\theta}_G)$, i.e., 
a Mirkovi\'c--Vilonen cycle. By Lemma~\ref{lem:MVcentralfiber}, 
this latter case always occurs. 
\end{proof}

\begin{rem}
 By Proposition \ref{prop:Vbasis}, the irreducible components of central Zastava fibers of critical dimension, which give rise to the ``new'' contributions $V_{X,\ch\lambda}$ to the pushforward of the IC sheaf by Proposition \ref{prop:format}, are of two different kinds: those associated to the Zastava space of the open $G$-orbit $X^\bullet$, and those associated to certain strata of the affine Grassmannian. On the other hand, Theorem \ref{thm:heckeIC} gives a similar description of the intersection complex of the global model in terms of the Hecke action on the intersection complex of the global model for $X^\bullet$. 
These two descriptions ``match'' under the nearby cycles functor of Theorem~\ref{thm:nearby}
and the Hecke action on Drinfeld's compactification $\ol\Bun_{N^-}$ (cf.~\cite{BG}). 
\end{rem}

\section{The crystal of a spherical variety} \label{sect:crystal}

We keep the assumptions of \S\ref{sect:Heckeact}--\ref{sect:centralfiber}. In this section, we study the irreducible components of central Zastava fibers of critical dimension (Proposition \ref{prop:Vbasis}) which give rise to the ``new'' contributions $V_{X,\ch\lambda}$ to the pushforward of the IC sheaf by Proposition \ref{prop:format}. Our main result is that these components give rise to a crystal, in the sense of Kashiwara, if we formally attach to them their ``negatives''. These components are, by Proposition \ref{prop:Vbasis}, of two different kinds, namely those associated to the Zastava space of the open $G$-orbit $X^\bullet$ and those associated to certain strata of the affine Grassmannian. Since the relation of the latter to crystals is well-known by 
\cite{BGcrys, BFG}, the problem quickly reduces to the study of the crystal associated to $X^\bullet$.

\subsection{The crystal $\mB_{X}$} 

\subsubsection{Definition of crystal} 
We review the definition of crystal, in the sense of Kashiwara \cite{Kas93}, 
over the Langlands dual Lie algebra $\ch{\mf g}$. We refer the reader to 
\cite{Kas94, Kas95, BS, HK} for further details on crystals, which can be associated 
to any Kac--Moody algebra. 

Let $I$ denote the set of vertices of the Dynkin diagram associated to $G$, 
so $\{\alpha_i\}_{i\in I} = \Delta_G$ is the set of simple roots of $G$. 

A crystal $\mB$ over $\ch{\mf g}$ is a set with the following data: 
\begin{align*}
\mathrm{wt} &: \mB \to \ch\Lambda_G \\
\varepsilon_i, \varphi_i &: \mB \to \mbb Z \sqcup \{ -\infty \} &\text{for } i\in I, \\
\tilde e_i, \tilde f_i &: \mB \to \mB \sqcup \{0\} &\text{for } i\in I,
\end{align*}
satisfying the following axioms:
\begin{enumerate}[label={(\arabic*)}] 
\item $\varphi_i(\msf b) = \varepsilon_i(\msf b) + \brac{ \alpha_i, \mathrm{wt}(\msf b)}$
for $\msf b \in \mB,\, i\in I$.
\item If $\msf b \in \mB$ and $\tilde e_i \msf b \ne 0$, then 
\[ \mathrm{wt}(\tilde e_i \msf b) = \mathrm{wt}(\msf b) + \ch\alpha_i,\; 
\varepsilon_i( \tilde e_i \msf b) = \varepsilon_i(\msf b) -1,\; 
\varphi_i( \tilde e_i \msf b) = \varphi_i( \msf b ) + 1. \]
\item If $\msf b \in \mB$ and $\tilde f_i \ne 0$, then 
\[ \mathrm{wt}(\tilde f_i \msf b) = \mathrm{wt}(\msf b) - \ch\alpha_i,\;
\varepsilon(\tilde f_i \msf b) + 1,\; 
\varphi_i(\tilde f_i \msf b) = \varphi_i(\msf b) -1. 
\]
\item For $\msf b_1, \msf b_2 \in  \mB$, $\msf b_2 = \tilde f_i \msf b_1$ if and only if 
$\msf b_1 = \tilde e_i \msf b_2$.
\item If $\varphi_i(\msf b)=-\infty$, then $\tilde e_i \msf b = \tilde f_i \msf b = 0$. 
\end{enumerate}

A crystal $\mB$ is called \emph{seminormal}\footnote{This is the terminology of \cite{Kas94,Kas95}. In \cite{Kas93} the term \emph{normal} was used for what we call \emph{seminormal}.} 
if 
\[ \varepsilon_i(\msf b) = \on{max} \{ n \ge 0 \mid \tilde e_i^n \msf b \in \mB \} \in \mathbb N, 
\quad \varphi_i(\msf b) = \on{max} \{ n \ge 0 \mid \tilde f_i^n \msf b \in \mB \} \in \mathbb N 
\]
for all $\msf b\in \mB, i\in I$.
From now on we will only consider seminormal crystals, so the maps $\varepsilon_i, \varphi_i$
are uniquely determined by $\mathrm{wt}, \tilde e_i, \tilde f_i$. 

Kashiwara showed the existence and uniqueness of crystal bases for any integrable module of the
quantized enveloping algebra $U_q(\ch{\mf g})$. 
The crystal basis of an integrable $U_q(\ch{\mf g})$-module 
is the limit at $q=0$ of Lusztig's canonical basis 
(\cite{Lusztig90, Grojnowski-Lusztig}).
A crystal $\mB$ is called \emph{normal} if it is isomorphic to the crystal basis 
of an integrable $U_q(\ch{\mf g})$-module. 

For any subset $J\subset I$, let $\ch{\mf g}_J$ denote the corresponding Levi subalgebra. 
For a crystal $\mB$ of $\ch{\mf g}$, let $\Phi_J(\mB)$ denote $\mB$ regarded as a crystal 
over $\ch{\mf g}_J$. 
Then saying that $\mB$ is seminormal is equivalent to saying that 
$\Phi_{\{i\}}(\mB)$ is isomorphic to the crystal basis of an integrable 
$U_q(\ch{\mf g}_{\{i\}})$-module. 
One can check the normality of a crystal by restricting to every pair
of vertices in the Dynkin diagram:

\begin{prop}[{\cite[Proposition 2.4.4]{KKMMNN}, \cite[Theorem 5.21]{BS}}] 
\label{prop:normal2}
Let $\mB$ be a finite crystal over $\ch{\mf g}$ such that for every 
subset $\{i,j\} \subset I$, the crystal $\Phi_{\{i,j\}}(\mB)$ is isomorphic to the
crystal basis of a finite-dimensional $U_q(\ch{\mf g}_{\{i,j\}})$-module. 
Then $\mB$ is normal. 
\end{prop}

For a crystal $\mB$ one can construct an oriented \emph{crystal graph} with vertex set $\mB$ and 
edges given by the $\tilde f_i$. We can decompose $\mB$ into a disjoint union 
of crystals corresponding to the connected components of the crystal graph. We will 
call these the connected components of $\mB$.

For $\ch\lambda \in \ch\Lambda^+_G$, there is a unique crystal basis 
$\mB^{\ch\lambda}_{\ch{\mf g}}$ for the irreducible highest weight module $V^{\ch\lambda}$
of $U_q(\ch{\mf g})$. (We will abuse notation and use $V^{\ch\lambda}$ to denote 
both the representation of the quantized enveloping algebra and its classical limit at
$q=1$, which is the corresponding irreducible $\ch{\mf g}$-module.)
In other words, there is a unique normal connected crystal with highest weight vector of weight 
$\ch\lambda$. However, we warn that in general there may be many seminormal connected crystals
with the same property.

Given a crystal $\mB$, we can define a crystal $\mB^\vee$ by ``reversing the arrows'':
the set $\mB^\vee = \{ \msf b^\vee \mid \msf b \in \mB\}$ is formally the same as $\mB$, and 
$\mathrm{wt}(\msf b^\vee) = -\mathrm{wt}(\msf b)$. The roles of $\tilde e_i,\tilde f_i$ 
are swapped. 
The crystal $(\mB^{\ch\lambda}_{\ch{\mf g}})^\vee$ is isomorphic to the crystal 
basis of the irreducible $U_q(\ch{\mf g})$-module of lowest weight $-\ch\lambda$, which
we also denote by $V^{-\ch\lambda}$.

\subsubsection{} \label{sect:crysW}
Let us mention an important consequence of the structure of a seminormal crystal $\mB$.
Let $\wt W$ be the free group generated by $\{ s_i \mid i \in I \}$ with
the relation $s_i^2= 1$. 
The Weyl group $W$ is the quotient of $\wt W$ by the braid
relations. 

It follows from the classification of integrable $U_q(\mf{sl}_2)$-modules that
we have a natural action of $\wt W$ on $\mB$ defined by 
\[
s_i(\msf b) = \begin{cases} 
    \tilde f_i^{\brac{\alpha_i, \mathrm{wt}(\msf b)}}(\msf b) &\text{if } 
    \brac{\alpha_i,\mathrm{wt}(\msf b)} \ge 0 \\
    \tilde e_i^{-\brac{\alpha_i, \mathrm{wt}(\msf b)}}(\msf b) &\text{if } 
    \brac{\alpha_i,\mathrm{wt}(\msf b)} \le 0 
\end{cases}
\]
for $\msf b\in \mB_{X^\bullet}$. 
For $\tilde w\in \wt W$ we have $\mathrm{wt}(\tilde w \msf b) = \tilde w (\mathrm{wt}(\msf b))$, 
where $\wt W$ acts on $\ch\Lambda_G=\ch\Lambda_X$ through $W$. 
In other words, we have isomorphisms 
\begin{equation} \label{e:crysW}
 \tilde w: \mB_{\ch\lambda} \overset\sim\to \mB_{\tilde w \ch\lambda} 
\end{equation}
\emph{a priori depending on} $\tilde w \in \wt W$ for all $\ch\lambda\in\ch\Lambda_G$.

If $\mB$ is normal, the $\wt W$-action on $\mB$ factors through $W$.

\subsubsection{Definition of $\mB_X$}\label{sect:defcrys}

For $\ch\lambda \in \mathfrak c_X$, we have defined the set 
$\mB_{X,\ch\lambda}$ to consist of 
the irreducible components of $\msf Y^{\ch\lambda}$ of critical dimension (Proposition \ref{prop:Vbasis}), that is:
\begin{itemize}
 \item if $\ch\lambda\in\mathfrak c^{\Cal D}_X$, the irreducible components of $\msf Y^{\ch\lambda}$ (or equivalently, of $\msf Y_{X^\bullet}^{\ch\lambda} = \msf Y^{\ch\lambda,0}$) of dimension 
$\frac 1 2 (\len(\ch\lambda)-1)$;
 \item the irreducible components of $\msf S^{\ch\lambda} \cap \Gr_G^{\ch\theta}$
of dimension $\brac{\rho_G,\ch\lambda-\ch\theta}$, for $\ch\theta \in \Cal D^G_{\mathrm{sat}}(X)$, identified with their image in $\msf Y^{\ch\lambda}$ through the action map \eqref{e:actvY}.
\end{itemize}
Note that $\ch\lambda=0$ never satisfies the conditions above.

Define $\mB_{X,-\ch\lambda} := \mB_{X,\ch\lambda}$, which
is well-defined since $\Cal C_0(X)$ is strictly convex. 
Let 
\[ \mB^+_{X} = \bigcup_{\ch\lambda\in \mf c_X} \mB_{X,\ch\lambda}, \qquad 
\mB^-_{X} = \bigcup_{\ch\lambda\in \mf c_X} \mB_{X,-\ch\lambda}. \]
In other words $\mB^+_{X}$ is the set of all irreducible
components of the central fiber of $\sY$ of the maximal dimensions 
satisfying the semi-smallness equality.

We (rather artificially) define $\mB_{X} = 
\mB_{X}^+ \sqcup \mB_{X}^-$. Let $\mathrm{wt}:\mB_{X} \to \mf c_X$ be the map sending $\mB_{X,\ch\lambda}$ to $\ch\lambda$.

\begin{thm} \label{thm:crystal} 
The set $\mB_{X}$ has the structure of a semi-normal crystal over 
$\ch{\mf g}$  such that the defining bijection $\mB_{X}^+ \leftrightarrow \mB_{X}^-$ is an isomorphism of crystals $\mB_{X} \cong \mB_{X}^\vee$. 
\end{thm}

The statement of the theorem above is not optimal, of course, as it does not specify all the data that give rise to the structure of a crystal, such as the operations $\tilde e_i,\tilde f_i$. To do so, we will need to introduce a process of ``reduction to a Levi'', in particular, a Levi of semisimple rank one, that will provide these operators. We will define these operators, giving the structure of a self-dual semi-normal crystal to $\mB_X$, in \S\ref{sect:crystal-ef}. 

\begin{conj} \label{conj:crystal}
The crystal $\mB_{X}$ is isomorphic to the unique crystal basis 
of a finite-dimensional $\ch G$-module $\rho_X$. 
\end{conj} 

In Remark~\ref{rem:conj2} we explain that it suffices to prove Conjecture~\ref{conj:crystal}
when $G$ has semisimple rank $2$, where there are finitely many cases (corresponding
to the wonderful varieties in \cite{Wasserman} with only spherical roots of type $T$).

\subsubsection{Reduction to $X^\bullet$}

Theorem \ref{thm:crystal} and Conjecture \ref{conj:crystal} immediately reduce to the study of the irreducible components of critical dimension in the central fiber of $\sY_{X^\bullet}$:

\begin{lem}\label{lem:reduction-Xbullet}
 Let $\mB_{X^\bullet}^+ = \mB_X \times_{\ch\Lambda_X} \mf c_X^{\mathcal D}$, hence $\mB_{X^\bullet}^+$ is the set of irreducible components of central fibers of $\sY_{X^\bullet}$ of critical dimension. Define $\mB_{X^\bullet}^-$ and $\mB_{X^\bullet} = \mB_{X^\bullet}^+ \sqcup \mB_{X^\bullet}^-$ as before. Then, Theorem \ref{thm:crystal} and Conjecture \ref{conj:crystal} hold if they hold for $\mB_{X^\bullet}$, with a decomposition of the crystal $\mB_X$ into a disjoint union of crystals:
\begin{equation}
\label{eq:crystaldecomp}
\mB_X =  \mB_{X^\bullet} \sqcup \bigsqcup_{\ch\theta \in \pm \Cal D^G_{\mathrm{sat}}(X)} \mB^{\ch\theta}_{\ch{\mf g}},
\end{equation}
where $\mB^{\ch\theta}_{\ch{\mf g}}$ is the crystal 
associated to the irreducible
$\ch G$-module $V^{\ch\theta}$ of lowest weight $\ch\theta$, if $\ch\theta \in \Cal D^G_{\mathrm{sat}}(X) \subset \ch\Lambda_G^-$, or highest weight $\ch\theta$ if $-\ch\theta \in \Cal D^G_{\mathrm{sat}}(X)$.
\end{lem}

\begin{proof}
 Indeed, by Proposition \ref{prop:Vbasis}, the elements of $\mB_X^+$ consist of elements of $\mB_{X^\bullet}^+$ and irreducible components of $\msf S^{\ch\lambda} \cap \Gr_G^{\ch\theta}$
of dimension $\brac{\rho_G,\ch\lambda-\ch\theta}$, for $\ch\theta \in \Cal D^G_{\mathrm{sat}}(X)$.
As explained in \cite{BGcrys}, the latter can be identified with the elements of
the crystal basis in the $\ch\lambda$-eigenspace of the irreducible $\ch G$-module $V^{\ch\theta}$ of lowest weight $\ch\theta$.
\end{proof}

While the methods of this paper are insufficient to prove Conjecture~\ref{conj:crystal} 
for $\mB_{X^\bullet}$, we do show that it must satisfy the following properties
in \S\ref{sect:crystal-hw}.

\begin{thm} \label{thm:crystal-properties}
The crystal $\mB_{X^\bullet}$ has the following properties: 
\begin{enumerate}
\item The set $\mathrm{wt}(\mB_{X^\bullet})$ is equal\footnote{Here we only describe an equality of sets counted without multiplicities.} to the set of weights of 
$\bigoplus_{\ch\lambda \in \ch\Lambda^+_G \cap W\varrho_X(\Cal D)} V^{\ch\lambda}$, 
where $\ch\Lambda^+_G \cap W\varrho_X(\Cal D)$ denotes the dominant Weyl translates of valuations
of colors.  
\item If $\msf b \in \mB_{X^\bullet}^+$, then 
there is a sequence of lowering operators $\tilde f_{i_j}$ sending $\msf b$ to an element of 
$\mB_{X^\bullet,\ch\nu_D}$ for some color $D\in \Cal D$.
\item For $\ch\lambda \in W\varrho_X(\Cal D)$,
the cardinality of $\mf B_{X^\bullet,\ch\lambda}$ is equal to $1$, unless $\ch\lambda =\frac{\ch\gamma}{2}$ for some (not necessarily simple) coroot $\ch\gamma$, in which case the cardinality is $2$. 
\end{enumerate}
\end{thm}

\begin{rem}
 The ``multiplicity 2'' case appears when two colors have the same valuation, e.g., $X = \mathbb G_m\backslash\PGL_2$; see \S \ref{sect:spherical}.
\end{rem}

\begin{rem} \label{rem:conj}
Note that if Conjecture~\ref{conj:crystal} is true, then properties (i)--(iii) of 
Theorem~\ref{thm:crystal-properties} uniquely determine the $\ch G$-module $\rho_X$: it must 
be isomorphic to 
\begin{equation} \label{e:conjVX}
 \bigoplus_{\ch\lambda \in \ch\Lambda^+_G \cap W\varrho_X(\Cal D)} 
(V^{\ch\lambda})^{\oplus \abs{\mf B_{X^\bullet,\ch\lambda}}} \oplus 
\bigoplus_{\ch\theta \in \pm\Cal D^G_{\mathrm{sat}}(X)} V^{\ch\theta},
\end{equation}
where the cardinality of $\mf B_{X^\bullet,\ch\lambda}$ is specified by property (iii).
\end{rem}

\begin{cor} \label{cor:minuscule}
If all coweights in $\ch\Lambda^+_G \cap W\varrho_X(\Cal D)$ are 
minuscule, then Conjecture~\ref{conj:crystal} holds, i.e., $\mf B_{X}$ is the crystal basis of 
the $\ch G$-module given by \eqref{e:conjVX}. 
\end{cor}
\begin{proof}
This is immediate from Theorems~\ref{thm:crystal}, \ref{thm:crystal-properties} and 
\S\ref{sect:crysW} after we make the assumption $\mf c_{X^\bullet} = \mbb N^{\Cal D}$, which is
allowed by \eqref{e:YXprime}.  
\end{proof}

We also show in Corollary~\ref{cor:Hreductive} that if $X$ is affine homogeneous (equivalently, $H$ is reductive), then 
all coweights in $\ch\Lambda^+_G\cap W\varrho_X(\Cal D)$ must be minuscule. 

\subsection{Reduction to Levi}
\label{sect:crystal-reduction}

From now on, having reduced the problem to giving a crystal structure to the set $\mB_{X^\bullet}$, we may (by \eqref{e:YXprime}) and will assume, unless
otherwise specified, that $X = X^\can$ and $\mathfrak c_X \cong 
\mathbb N^{\Cal D}$. 
Under this assumption, $\Scr Y^{\ch\lambda,0}$ is dense in $\Scr Y^{\ch\lambda}$ by 
Corollary~\ref{cor:Ycomp}, and $\mB_X = \mB_{X^\bullet}$. Moreover, any $\ch\lambda\succeq 0$ is an element of $\mathbb N^{\Cal D}$, so the length function $\len$ is a function of $\ch\lambda$.

\medskip

Let $P$ be a standard parabolic subgroup of $G$, i.e., $P \supset B$. 
Let $N_P$ denote its unipotent radical and $M=P/N_P$ the Levi quotient. 
Observe that the map $X \to X/\!\!/N$ factors through 
$X \to X \sslash N_P \to X/\!\!/N$. 
Set 
\[ X_M := X \sslash N_P = \Spec k[X]^{N_P}. \]
Then $X_M$ is an affine spherical $M$-variety
and the map $X \to X_M$ is $M$-equivariant. 
However, note that even if $X = X^\can$, it will not in general 
be true that $X_M$ is the canonical embedding of $X_M^\bullet$. We will use this
to our advantage later, using the crystals of Lemma \ref{lem:reduction-Xbullet} to produce the $\tilde e_i,\tilde f_i$ operations, when $M$ is taken to have semisimple rank one.

For now, we work with a general parabolic $P$. Let $B_M$ denote the image of the Borel subgroup $B$ in $M$.
We have $k[X]^{(B)} = k[X_M]^{(B_M)}$, therefore $\mathfrak c_{X_M} = \mathfrak c_X$. On the other hand, $\mathfrak c_{X_M}^- =\mathfrak c_{X_M} \cap \ch\Lambda_M^-$ is, in general, larger than $\mathfrak c_X^-$. The open $P$-orbit $X^\circ P$ maps to the open $M$-orbit $X_M^\bullet$, and we have
\begin{lem}
 \label{lem:openPorbit}
The preimage of $X_M^\bullet$ under the quotient map $X\to X_M$ coincides with the open $P$-orbit $X^\circ P$, and the quotient stacks $X^\circ P/P$ and $X_M^\bullet/M$ are isomorphic.
\end{lem}

\begin{proof}
A color $D\in \mathcal D$ belongs to the open $P$-orbit $X^\circ P$ if and only if $D\in \mathcal D(\alpha)$ for some $\alpha \in \Delta_M$; otherwise, it is $P$-stable, and induces an $M$-stable valuation on $k(X_M)$, which is the function field of $k[X]^{N_P}$. This valuation is nontrivial (because it is nontrivial on $k[X]^{(B)} = k[X_M]^{(B_M)}$), therefore the image of $D$ cannot belong to $X_M^\bullet$.   

Since $N$ acts freely on $X^\circ$, the subgroup $N_P$ acts freely on $X^\circ P$, and therefore $X^\circ P/P = (X^\circ P/N_P)/M = X_M^\bullet/M$. 
\end{proof}

Define the parabolic Zastava model 
\[ \sY_{X,P} := \Maps_\gen( C, X / P \supset X^\circ P / P ) \subset \sM_X \xt_{\Bun_G} \Bun_P, \]
which naturally maps to $\Bun_P$. 
The Cartesian diagram 
\[\begin{tikzcd}
X/B \ar[r] \ar[d] & 
X/P \ar[d] \\ 
X_M/B_M  \ar[r] & X_M/M  
\end{tikzcd}\]
gives rise to a diagram 
\begin{equation} \label{e:YXPM1}
\begin{tikzcd}
\msf Y^{\ch\lambda}_X \ar[r] \ar[d] & 
\sY_X \ar[r] \ar[d, swap, "q"] & \sY_{X,P} \ar[d, "\pi_{X,P}"]  \\ 
\msf Y^{\ch\lambda}_{X_M} \ar[r] & \sY_{X_M} \ar[r] & \sM_{X_M}  
\end{tikzcd}
\end{equation}
with all squares Cartesian.
Central fibers are taken with respect to 
a fixed point $v \in \abs C$. 

Our goal is to study the components of critical dimension of $\msf Y^{\ch\lambda}_X$ in terms of $\msf Y^{\ch\lambda}_{X_M}$ and the fibers of the map $\pi_{X,P}$. At this point, it will be critical to distinguish the stratum $\msf Y_{X_M}^{\ch\lambda,\ch\theta}$ where the image of the generic point of a component $\msf b \in \mathfrak B_{X,\ch\lambda}$ lies, that is, the $M(\mathfrak o_v)$-orbit of its image in $X_M(\mathfrak o_v)$ (after trivialization in a formal neighborhood of $v$). The reason is, as we are about to see, that this stratum will completely determine the fiber of the map $\pi_{X,P}$ over the image.

Indeed, recall from \S \ref{def:centralfiber} that lifting a point from $\sM_{X_M}$ to $\msf Y^{\ch\lambda}_{X_M}$ induces a trivialization of the corresponding $G$-bundle away from $v$ (depending on a fixed choice of base point $x_0\in X^\circ$), which identifies the central fiber $\msf Y^{\ch\lambda}_{X_M}$ with a subscheme of the affine Grassmannian $\Gr_M$. Moreover, the map $\msf Y^{\ch\lambda}_{X_M} \to \sM_{X_M}$ factors through the map
\[ \Gr_M \xt_{\msf L X_M / \msf L^+M} (\msf L^+ X_M / \msf L^+M) \to \sM_{X_M}.\]
as defined in \eqref{e:actHG}.
Similarly, the map $\msf Y^{\ch\lambda}_X \to \sY_{X,P}$ factors through 
\[\Gr_P \xt_{\msf L X / \msf L^+P} (\msf L^+ X / \msf L^+P) \to \sY_{X,P}.\]
Hence we have a commutative diagram 
\begin{equation} \label{e:YXPM2}
\begin{tikzcd}
\msf Y^{\ch\lambda}_X \ar[r,"\iota_{X,P}" ] \ar[d,"q"] &  
\Gr_P \underset{\msf L X / \msf L^+P}\times (\msf L^+ X / \msf L^+P) \ar[r, "\act_v"]\ar[d] & \sY_{X,P} \ar[d, "\pi_{X,P}"]  \\ 
\msf Y^{\ch\lambda}_{X_M} \ar[r,"\iota_{X_M}"] & \Gr_M \underset{\msf L X_M / \msf L^+M}\times (\msf L^+ X_M / \msf L^+M)  \ar[r, "\act_v"] & \sM_{X_M}  
\end{tikzcd}
\end{equation}
with all squares Cartesian.

Let $H_M$ be the stabilizer in $P$ of the base point $x_0$. By Lemma \ref{lem:openPorbit}, it is isomorphic to the stabilizer in $M$ of the image of $x_0$ in $X_M$. We first note:

\begin{lem}  \label{lem:XPfiber}
For any $\ch\theta \in \mathfrak c_{X_M}^-$, the fibers of the map of ind-schemes  
\begin{equation}\label{e:XPfiber}
   \Gr_P \xt_{\msf L X / \msf L^+P} (\msf L^+ X / \msf L^+P) \to \Gr_M \xt_{\msf L X_M / \msf L^+M} (\msf L^+ X_M / \msf L^+M)
\end{equation}
over the stratum $\msf L^{\ch\theta} X_M / \msf L^+M$ are isomorphic under the $\msf LH_M$-action, and this action gives rise to a canonical bijection between the irreducible components of any two fibers. 

More precisely, all fibers are isomorphic to $\{t^{\ch\theta}\} \xt_{\sM_{X_M}} \sY_{X,P}$ and of dimension $\le \frac 1 2 (\len(\ch\theta)-1)$, unless $\ch\theta = 0$, in which case the restriction of \eqref{e:XPfiber} to the $\ch\theta$-stratum is an isomorphism.
\end{lem}

Notice that, under our assumption that $\mathfrak c_X \cong 
\mathbb N^{\Cal D}$ since the beginning of this subsection, $\len(\ch\theta)$ makes sense.
\begin{proof}
Since $\msf LH_M$ acts transitively on 
$\Gr_M \xt_{\msf L X_M / \msf L^+M} (\msf L^{\ch\theta} X_M / \msf L^+M)$,
the fibers of \eqref{e:XPfiber} over the $\ch\theta$-stratum are all isomorphic. 

We may now choose the point $t^{\ch\theta}\in \msf Y_{X_M}^{\ch\theta}$, whose image in $\msf L^+ X_M / \msf L^+M$ lies in the $\ch\theta$-stratum --- in fact, by Corollary~\ref{cor:Ytheta}, $\msf Y_{X_M}^{\ch\theta,\ch\theta} = \{ t^{\ch\theta}\}$. By Corollary~\ref{cor:stab-conn}, the stabilizer in $\msf L H_M$ of its image 
$t^{\ch\theta}\in \Gr_M$
is connected. Therefore, the action of $\msf L H_M$ induces a \emph{canonical} bijection between irreducible components of the fibers.

Finally, if $\ch\theta\ne 0$, the dimension of $\msf Y^{\ch\theta}_X$ is $\le \frac 1 2 (\len(\ch\theta)-1)$, as explained in Proposition \ref{prop:Vbasis}, and therefore so is, \emph{a fortiori}, the dimension of the fiber over $\msf Y_{X_M}^{\ch\theta,\ch\theta} = \{ t^{\ch\theta}\}$. For $\ch\theta =0$, we observe that 
\[ \sY_{X,P} \xt_{\sM_{X_M}}     \sM_{X_M}^0 = \Maps(C, X^\circ P/P) = \Maps(C, X_M^\bullet /M) =  \sM_{X_M}^0,\]
by Lemma \ref{lem:openPorbit}, so the fibers are singletons.
\end{proof}

\begin{rem}
At this point, we would like to emphasize a fine point in the arguments that follow: Consider the decomposition of $\mB_{X_M}^+$ according to \eqref{eq:crystaldecomp} (restricted to the $+$-part):
\[\mB_{X_M}^+ =  \mB_{X_M^\bullet}^+ \sqcup \bigsqcup_{\ch\theta \in \Cal D^M_{\mathrm{sat}}(X_M)} \mB^{\ch\theta}_{\ch{\mf m}}\]
(where we have denoted $\mB^{\ch\theta}$ by $\mB^{\ch\theta}_{\ch{\mf m}}$, to emphasize that it corresponds to the $\ch\theta$-lowest weight crystal of a $\ch{\mf m}$-module). We \emph{will not} claim that the map $q$ of \eqref{e:YXPM1} induces a map from $\mB_X^+$ to $\mB_{X_M}^+$. Indeed, the generic fiber of a $\msf b\in \mB_X^+$ may map to the image of $\msf S_M^{\ch\lambda} \cap \Gr_M^{\ch\theta} \to \msf Y_{X_M}^{\ch\lambda}$ (where $\msf S_M^{\ch\lambda}$ denotes the semi-infinite orbit corresponding to $\ch\lambda$ in $\Gr_M$), for some $\ch\theta \in \mathfrak c_{X_M}^-$ that is \emph{not} an element of $\Cal D^M_{\mathrm{sat}}(X_M)$. These are the MV cycles that were discussed in Remark \ref{rem:othertheta}, which are not ``of critical dimension'' in terms of $X_M$. 
Representation-theoretically, if we believe that $\mB_X$ corresponds to a representation of $\check G$ (as predicted by Conjecture \ref{conj:crystal}), this just says that the $\check M$-lowest weights of the spans of some vectors do not need to be extremal in $\mathfrak c_{X_M}^-$; however, in \S \ref{sect:crystal-hw} we will see that there are weight-lowering operators $\tilde f_i$, possibly corresponding to roots not in $\ch M$, which eventually lower such weights to the weight of a color.
\end{rem}

For that reason, for the following proposition, which is the main technical result of this subsection, we denote by $\mB^{\ch\theta}_{\ch{\mf m}}$ the crystal corresponding to the representation of $\ch{\mf m}$ of lowest weight $\ch\theta$, that is, the set of irreducible components of $\msf S_M^{\ch\lambda} \cap \Gr_M^{\ch\theta} $, for \emph{any} $\ch\theta\in \mathfrak c_{X_M}^-$.

\begin{prop}\label{prop:crys-decomp}
For any $\ch\lambda\in \mathfrak c_X^{\Cal D}$, the diagram \eqref{e:YXPM1} induces a canonical decomposition
\begin{equation} \label{e:crys-decomp}
    \mB_{X^\bullet,\ch\lambda} = \mB_{X_M^\bullet,\ch\lambda} \sqcup 
        \bigsqcup_{\ch\theta \in \mathfrak c_{X_M}^-\sm 0} 
        \mB^P_{X,\ch\theta} \xt \mB_{\ch{\mf m}, \ch\lambda}^{\ch\theta},
\end{equation}
where $\mB^P_{X,\ch\theta}$ denotes the set of irreducible components
of $\{t^{\ch\theta}\} \xt_{\sM_{X_M}} \sY_{X,P}$ of dimension $\frac 1 2 (\len(\ch\theta)-1)$.

Taking the union over all such $\ch\lambda$, we get
\begin{equation}
    \mB_{X^\bullet}^+ = \mB_{X_M^\bullet}^+ \sqcup \bigsqcup_{\ch\theta\in \mathfrak c_{X_M}^-\sm 0} 
    \mB^P_{X,\ch\theta} \xt \mB^{\ch\theta}_{\ch{\mf m}}.
\end{equation}
\end{prop}

The set $\mB^P_{X,\ch\theta}$ should be thought of as the multiplicity space for the irreducible representation with basis $\mB^{\ch\theta}_{\ch{\mf m}}$, and is something of a ``black box'' to us. 

\begin{proof}
The dimension yoga here goes as follows: 
Let $\msf b \in \mB_{X^\bullet, \ch\lambda}$, and suppose that its generic point lands in the stratum $\msf Y^{\ch\lambda,\ch\theta}_{X_M}$ under the map $q$ of \eqref{e:YXPM2}.

If $\ch\theta =0$, then $\ch\lambda \succeq_{X_M^\bullet} 0$, i.e., it belongs to the positive span of colors in $X_M$, hence $\len(\ch\lambda)$ is the same, whether we define it with respect to $X$ or with respect to $X_M$. By Lemma \ref{lem:XPfiber} the irreducible components of critical dimension $\frac{1}{2} (\len(\ch\lambda) - 1)$ of $\msf Y_X$ and $\msf Y_{X_M^\bullet}$ are in bijection, hence the set of $\msf b \in \mB_{X,\ch\lambda}$ which map generically to $\msf Y^{\ch\lambda,0}_{X_M}$ is identified with $\mB_{X_M^\bullet,\ch\lambda}$.

If $\ch\theta \ne 0$, then $q$ sends $\msf b$ to $\msf Y^{\ch\lambda, \succeq \ch\theta}_{X_M}$, 
which has dimension $\le \frac 1 2 \len(\ch\lambda-\ch\theta)$ by Lemma~\ref{lem:centralfibertheta}. 
On the other hand, Lemma \ref{lem:XPfiber} states that the dimension of the corresponding fibers of $\pi_{X,P}$ is $\le \frac{1}{2} (\len(\ch\theta)-1)$. Thus, the only way that $\msf b$ is of critical dimension $\frac{1}{2} (\len(\ch\lambda)-1)$ is if both inequalities are equalities. 
In this case Lemma~\ref{lem:centralfibertheta} implies that $\ch\lambda \ge \ch\theta$ 
and the generic point of $\msf b$ is sent under the map $(\act_v\circ \iota_X,\iota_{X_M}\circ q)$ to an element of $\mB^P_{X,\ch\theta} \xt \mB_{\ch{\mf m}, \ch\lambda}^{\ch\theta}$. Vice versa, for any irreducible component (MV cycle) of  $\msf S^{\ch\lambda}_M\cap \Gr_M^{\ch\theta}$ (i.e., every element of $\mB_{\ch{\mf m},\ch\lambda}^{\ch\theta}$), 
Lemma~\ref{lem:MVcentralfiber} guarantees that it corresponds to a component $\msf b'$ of $\msf Y_{X_M}^{\ch\lambda,\ch\theta}$ of the same dimension, and Lemma \ref{lem:XPfiber} ensures that the components of $\msf Y^{\ch\lambda}_X$ of critical dimension in the preimage of $\msf b'$ are in canonical bijection with $\mB^P_{X,\ch\theta}$.
\end{proof}

\subsubsection{Kashiwara operations $\tilde e_i, \tilde f_i$} \label{sect:crystal-ef}
For $i\in I$ let $P_i = P_{\alpha_i}$ denote the corresponding parabolic subgroup of semisimple rank one.
 Let $M_i$ denote the Levi factor. Then the Langlands 
dual Lie algebra $\ch{\mf m}_i$ equals $\ch{\mf g}_{\{i\}}$ in our previous notation.
Applying Proposition~\ref{prop:crys-decomp} to $M_i$ we get 
the disjoint union
\[ \mB^+_{X^\bullet} = \mB^+_{X_{M_i}^\bullet} \sqcup \bigsqcup_{\ch\theta \in \mathfrak c_i^-\sm 0} \mB^{P_i}_{X,\ch\theta} \xt \mB^{\ch\theta}_{\ch{\mf m}_i},
\]
where $\mathfrak c_i^- = \mathfrak c_{X_{M_i}}^-$.

Now, by our ``type $T$'' assumption (see \S \ref{subsection:typeT}), $X_{M_i}^\bullet / B_{M_i} = \mbb G_m \bs \mbb P^1$ as stacks. 
Therefore, $\sY_{X_{M_i}^\bullet} = \Sym C \oo\xt \Sym C$ (see Example~\ref{eg:YHecke}) 
and $\mB_{X_{M_i}^\bullet}^+$ consists of two elements, which can be identified with their images $\ch\nu_i^{\pm}$ in $\ch\Lambda_X$, 
$\ch\nu_i^\pm = \varrho_X( D_{\alpha_i}^\pm )$ are the valuations of the two colors in $\mathcal D(\alpha_i)$.
Therefore, we have a bijection of sets 
\[ \mB_{X^\bullet} = \mB^+_{X^\bullet} \cup \mB^-_{X^\bullet} = \{ \ch\nu_i^+, \ch\nu_i^-, -\ch\nu_i^+, -\ch\nu_i^- \} \sqcup \bigsqcup_{\ch\theta \in \mathfrak c_i^-\sm 0} 
\mB^P_{X,\ch\theta} \xt \bigl( 
\mB^{\ch\theta}_{\ch{\mf m}_i} \sqcup (\mB^{\ch\theta}_{\ch{\mf m}_i})^\vee \bigr).
\]
Observe that $\{\ch\nu_i^+, -\ch\nu_i^-\}$ is in bijection with 
the normal crystal $\mB_{\ch{\mf m}_i}^{\ch\nu_i^+}$ since $\brac{\alpha_i,\ch\nu_i^+}=1$
and $\ch\nu_i^+ - \ch\alpha_i = -\ch\nu_i^-$. 
We also observe that $\{ \ch\nu_i^-, -\ch\nu_i^+\} = \mB_{\ch{\mf m}_i}^{\ch\nu_i^-} = 
(\mB_{\ch{\mf m}_i}^{\ch\nu_i^+})^\vee$ as sets. 

Now we simply define the operations $\tilde e_i, \tilde f_i$ such that 
\begin{equation}\label{e:iselfdual}
    \Phi_{\{i\}}(\mB_{X^\bullet}) = \mB_{\ch{\mf m}_i}^{\ch\nu_i^+} \sqcup 
    \mB_{\ch{\mf m}_i}^{\ch\nu_i^-} \sqcup 
    \bigsqcup_{\ch\theta \in \mathfrak c_i^-\sm 0} \mB^P_{X,\ch\theta} 
    \xt \bigl( 
    \mB^{\ch\theta}_{\ch{\mf m}_i} \sqcup (\mB^{\ch\theta}_{\ch{\mf m}_i})^\vee \bigr)
\end{equation}
as normal crystals over $\ch{\mf m}_i$, where $\mB^P_{X,\ch\theta}$ is treated as
an abstract set. 
This gives the structure of a seminormal 
crystal over $\ch{\mf g}$ to $\mB_X$, such that the bijection $\mB_X^+\leftrightarrow\mB_X^-$ identifies it with its dual. This completes the proof of Theorem \ref{thm:crystal}.

\begin{rem} \label{rem:conj2}
The decomposition \eqref{e:crys-decomp} gives a decomposition into crystals
over $\ch{\mf m}$. Therefore, if we consider Proposition~\ref{prop:crys-decomp}
for all standard parabolics corresponding to $\{i,j\} \subset I$, then 
Proposition~\ref{prop:normal2} implies that $\mf B_X$ is normal (i.e., Conjecture~\ref{conj:crystal} holds)  
if $\mf B_{X_M^\bullet}$ is normal for all $M$ of semisimple rank $2$. 
\end{rem}

The discussion of \S\ref{sect:crystal-ef} also leads to the following observation:

\begin{lem} \label{lem:plusWorbits}
The $W$-orbit of $\varrho_X(\Cal D)$ is contained in $\mathfrak c_X^{\Cal D} \sqcup -\mathfrak c_X^{\Cal D}$. 
If $\ch\lambda\in \mathrm{wt}(\mB_{X^\bullet}^+)$ is not in $W\varrho_X(\Cal D)$, the entire 
$W$-orbit $W\ch\lambda$ is contained in the monoid $\mathfrak c_X^{\Cal D}$. 
\end{lem}
\begin{proof} Let $\msf b \in \mB_{X^\bullet}^+$ with $\ch\lambda=\mathrm{wt}(\msf b)$. 
The decomposition of $\Phi_{\{i\}}(\mB_{X^\bullet}^+)$ from \S\ref{sect:crystal-ef}
shows that if $\ch\lambda \notin \{\ch\nu_i^\pm\}$, 
we have $s_i\msf b \in \mB^+_{X^\bullet}$, so $s_i\ch\lambda\in \mathfrak c_X^{\Cal D}$.
Since $s_i \ch\nu_i^\pm = -\ch\nu_i^\mp$, we can iteratively 
apply simple reflections to deduce the claims.
\end{proof}

\subsection{Lowering operators via hyperplane intersections} \label{sect:crystal-hw}

In this subsection we prove Theorem~\ref{thm:crystal-properties}. 
Property (iii) follows from Lemma~\ref{lem:Ycolor} and the $\wt W$-action on $\mB_X$ 
given by semi-normality of the crystal (\S\ref{sect:crysW}).

To prove properties (i)--(ii) we will need a geometric interpretation of the weight-lowering operators $\tilde f_i$. 
This interpretation is already hiding behind the crystal structure of $\mB^{\ch\theta}_{\ch{\mf m}}$ (in the notation of Proposition \ref{prop:crys-decomp}), and has to do with closure relations of semi-infinite orbits in the affine Grassmannian. To bring such closure relations into our discussion, we need to extend the considerations of \S \ref{sect:crystal-reduction} to the compactified Zastava models. 

\begin{prop} \label{prop:f-hyper}
For $\ch\lambda\in \mathfrak c_X^{\Cal D}$, let $\msf b \in \mB_{X^\bullet,\ch\lambda}$ be an irreducible component of critical dimension, and let $\ol{\msf b}$
be its closure in $\olsf Y^{\ch\lambda}$. For $i\in I$, consider the
intersection 
\[ \ol{\msf b} \cap \msf Y^{\ch\lambda-\ch\alpha_i} \subset \olsf Y^{\ch\lambda}. \]

\begin{enumerate}
\item
If the intersection above is non-empty, then $\tilde f_i \msf b \ne 0$
and it corresponds to an irreducible component of dimension $\dim(\msf b) -1$
of $\ol{\msf b} \cap \msf Y^{\ch\lambda-\ch\alpha_i}$. Vice versa, if $\tilde f_i \msf b \ne 0$ then the intersection above is non-empty, unless $\ch\lambda=\ch\nu_i^\pm$ is a color, in which case $\msf b \subset \msf Y^{\ch\lambda}$ is a point. 

\item The intersection $\olsf b \cap \msf Y^{\ch\lambda-\ch\alpha_i}$ is empty
only if either $\ch\lambda=\ch\nu_i^\pm$ or $\brac{\alpha_i, \ch\lambda} \le 0$. 
\end{enumerate}
\end{prop}

We remark that it may be possible for $\olsf b \cap \msf Y^{\ch\lambda-\ch\alpha_i}_X$
to be reducible (cf.~\cite[Proposition 19.2]{BFG}, which replaces the erroneous 
Proposition 15.2 of \emph{loc.~cit.}). 

The proof of this proposition will be given at the end of this section. We first use 
the proposition to prove Theorem~\ref{thm:crystal-properties}. 
Both properties (i)--(ii) of the theorem rely on the following observation:

\begin{lem}
\label{lem:lowering}
  For $\ch\lambda \in \mathfrak c^{\Cal D}_X$ and $\msf b \in \mB_{X^\bullet,\ch\lambda}$
there is a sequence $\alpha_1, \dots,\alpha_d$ of simple roots (possibly with repetitions), where $d = \dim \msf b$, such that 
we have 
\begin{itemize}
\item $\msf b_j := \tilde f_{\alpha_{j}} \dotsb \tilde f_{\alpha_1}(\msf b) \ne 0$ 
for $0\le j \le d$,
\item 
the intersection $\ol{\msf b_{j-1}} \cap \msf S^{\ch\lambda - \ch\alpha_1 - \cdots -\ch\alpha_j}$ is nonempty of dimension $d-j$ for $1\le j \le d$, 
\item $\ch\lambda - \sum_{j=1}^d \ch\alpha_j = \ch\nu_D$ for some color $D \in \Cal D$.
\end{itemize} 
In particular, $\ch\lambda \ge \ch\nu_D$. 
\end{lem}

\begin{proof}
By definition, $\msf b$ is of critical dimension $d= \frac{1}{2}(\len(\ch\lambda)-1)$.
Proposition \ref{prop:intersections} shows that there exists
a $\ch\lambda' \le \ch\lambda$ 
such that $\brac{\rho_G, \ch\lambda-\ch\lambda'} \ge \frac 1 2(\len(\ch\lambda)-1)$, and $\bar{\msf b} \cap \msf S^{\ch\lambda'}$ is nonempty of dimension zero. Proposition \ref{prop:centralstrata} states that the dimension inequality should be an equality, in which case Proposition \ref{prop:intersections} again provides the sequence of simple roots as in the statement. In that case, $\len(\ch\lambda')=1$, hence $\ch\lambda' = \ch\nu_D$ for some color $D\in \Cal D$. 
Then Proposition \ref{prop:f-hyper}(i) applied inductively shows that
$\tilde f_{\alpha_j} \dotsb \tilde f_{\alpha_1}(\msf b) \ne 0$
satisfies the claim.
\end{proof}

\begin{proof}[Proof of Theorem~\ref{thm:crystal-properties}(ii)]
Immediate from Lemma \ref{lem:lowering}.
\end{proof}

\begin{proof}[Proof of Theorem \ref{thm:crystal-properties}(i)]
We assume as in \S\ref{sect:crystal-reduction} that $\mf c_X = \mbb N^{\Cal D}$. 
First we show that the weights of $\mB_{X^\bullet}$ are contained in the weights of 
$V^{\ch\lambda}$ for $\ch\lambda \in \ch\Lambda^+_G \cap W\varrho_X(\Cal D)$. 
Let $\ch\theta \in \mathrm{wt}(\mB_{X^\bullet}^+)$. 
By \eqref{e:crysW} and Lemma~\ref{lem:plusWorbits}, 
we may assume that $\ch\theta \in \ch\Lambda^-_G$. 
Now Lemma~\ref{lem:lowering} gives some color $D\in \Cal D$ such that 
$\ch\nu_D \le \ch\theta$. Since $\ch\theta$ is antidominant, 
it must be a weight of $V^{\ch\lambda}$ where $\ch\lambda$ is the unique dominant coweight in 
the $W$-orbit of $\ch\nu_D$.
If $\alpha$ is a simple root such that $\Cal D(\alpha) = \{ D, D' \}$, then 
$s_\alpha(\ch\nu_{D'}) = -\ch\nu_D$. Thus, we see that 
$W\varrho_X(\Cal D) = -W\varrho_X(\Cal D)$. Hence all of $\mathrm{wt}(\mB_{X^\bullet}^-)$ 
is also contained in the weights of the claimed representations. 

Next suppose that $\ch\mu$ is a weight of $V^{\ch\lambda}$ for $\ch\lambda \in \ch\Lambda^+_G \cap W\varrho_X(\Cal D) \subset \mathfrak c_X^{\Cal D}$. 
We will show that $\mf B_{X^\bullet,\ch\mu}$ is nonempty. 
By Theorem~\ref{thm:crystal-properties}(iii), there exists
an element $\msf b \in \mB_{X,\ch\lambda}$, and  
by \eqref{e:crysW} we may assume that 
$\ch\mu \in \ch\Lambda^+_G$ is also dominant. 
By Lemma~\ref{lem:walls} below, 
we can find a sequence of simple coroots $\ch\alpha_1,\dotsc, \ch\alpha_d$
(possibly with repetitions), where $d = \brac{\rho_G,\ch\lambda-\ch\mu}$,
such that $\ch\lambda_j := \ch\lambda-\ch\alpha_1-\dotsb-\ch\alpha_j$ 
satisfies $\ch\lambda_j + \ch\rho_G \in \ch\Lambda^+_G$ 
for $j=1,\dotsc,d$ and $\ch\lambda_d = \ch\mu$. 
In particular, this means that $\brac{\alpha_{j+1},\ch\lambda_j-\ch\alpha_{j+1}}
= \brac{\alpha_{j+1},\ch\lambda_{j+1}} \ge -1$ for all $0\le j<d$. 
Equivalently, $\brac{\alpha_{j+1},\ch\lambda_j} > 0$ for all $0\le j < d$.
Now Proposition~\ref{prop:f-hyper}(ii) 
implies that $\tilde f_{\alpha_d} \dotsb \tilde f_{\alpha_1} (\msf b) \in \mB_{X^\bullet,\ch\mu}$ 
is nonzero. 
\end{proof}

\begin{lem}  \label{lem:walls}
Let $\ch\mu \in \ch\Lambda_G^+$ and $\ch\lambda \in \ch\Lambda_G$
such that $\ch\lambda \ge \ch\mu$ and $\ch\lambda+\ch\rho_G \in \ch\Lambda^+_G$. 
Then $\ch\lambda-\ch\mu = \ch\alpha_1+\dotsb + \ch\alpha_d$ for a
sequence of simple coroots $\ch\alpha_j$ (possibly with repetitions)
where $d = \brac{\rho_G, \ch\lambda-\ch\mu}$
such that $\ch\lambda - \ch\alpha_1 - \dotsb - \ch\alpha_j + \ch\rho_G
\in \ch\Lambda^+_G$ for $j=1,\dotsc,d$. 
\end{lem}
\begin{proof}
Since $\ch\lambda\ge\ch\mu$ we can decompose
$\ch\lambda-\ch\mu = \ch\alpha_1+\dotsb+\ch\alpha_d$ into a sum of simple coroots (possibly with repetitions),
where $d = \brac{\rho_G,\ch\lambda-\ch\mu}$. 
To prove the lemma, it suffices, by induction on $\ch\lambda$, to 
show that there exists an $\ch\alpha_j$ such that
$\ch\lambda-\ch\alpha_j + \ch\rho_G \in \ch\Lambda^+_G$ for some 
$1\le j\le d$. 
We claim that if $\ch\lambda \ne \ch\mu$, then there exists some $\ch\alpha_j$ 
with $\brac{\alpha_j, \ch\lambda} \ge 1$. 
If not, then $\brac{\alpha_j,\ch\lambda-\ch\mu} \le \brac{\alpha_j, \ch\lambda} \le 0$ for all $j$. This implies that $\ch\lambda-\ch\mu$ 
has non-positive norm with respect to an appropriate inner product on
$\mf t$, and hence $\ch\lambda=\ch\mu$. 
Therefore if $d>0$, then we have an $\ch\alpha_j$ 
such that $\brac{\alpha_j, \ch\lambda} \ge 1$. 
Equivalently, $\brac{\alpha_j, \ch\lambda+\ch\rho_G} \ge 2$ 
and $\brac{\alpha_j, \ch\lambda-\ch\alpha_j+\ch\rho_G} \ge 0$. 
This implies that $\ch\lambda-\ch\alpha_j +\ch\rho_G \in \ch\Lambda^+_G$.
\end{proof}

\begin{cor} \label{cor:Hreductive}
If $X^\bullet = H\bs G$ is affine, then all dominant Weyl translates 
in $\ch\Lambda^+_G\cap W\varrho_X(\Cal D)$ are minuscule 
and $\mB_X$ is a normal crystal.
\end{cor}
\begin{proof}
Assume $X^\bullet$ is affine, so $\mf c_{X^\bullet}^- = 0$. 
If there exists a non-minuscule coweight in $\ch\Lambda^+_G \cap W\varrho_X(\Cal D)$, then
Theorem~\ref{thm:crystal-properties}(i) implies that $\mB_{X^\bullet, \ch\theta}$ is non-empty
for some $\ch\theta \in \mf c_{X^\bullet}\sm 0$ not in $W\varrho_X(\Cal D)$. 
Lemma~\ref{lem:plusWorbits} and \eqref{e:crysW} allow us to assume $\ch\theta \in \mf c_{X^\bullet}^- = 0$,
which gives a contradiction.
\end{proof}

\subsubsection{} The rest of this section is devoted to the proof of Proposition~\ref{prop:f-hyper}.
We use the notation from \S\ref{sect:crystal-ef}. 
For brevity we write $X_i = X_{M_i},\, H_i=H_{M_i},\, B_i = B_{M_i}$ and 
$N_i=N_{B_i}$. 
We say $\ch\mu \le_i \ch\lambda$ if $\ch\lambda-\ch\mu \in \mathbb N\ch\alpha_i$.

We would like to embed the left Cartesian square of \eqref{e:YXPM2} into a Cartesian square involving compactified Zastava spaces. For that purpose, consider the extension of the map $\iota_{X_M} = \iota_{X_i}$ of that diagram to $\olsf Y^{\ch\lambda}_{X_i}$, and define $\msf Y^{\le_i \ch\lambda}_{X}$ by the Cartesian diagram
\begin{equation} \label{e:YXPM3}
\begin{tikzcd}
\msf Y^{\le_i \ch\lambda}_{X} \ar[r,"\iota_{X,P_i}" ] \ar[d,"q"] & 
\Gr_{P_i} \underset{\msf L X / \msf L^+P_i}\times (\msf L^+ X / \msf L^+P_i) \ar[r, "\act_v"]\ar[d] & \sY_{X,P_i} \ar[d, "\pi_{X,P_i}"]  \\ 
\olsf Y^{\ch\lambda}_{X_i} \ar[r,"\iota_{X_i}"] & \Gr_{M_i} \underset{\msf L X_i / \msf L^+M_i}\times (\msf L^+ X_i / \msf L^+M_i)  \ar[r, "\act_v"] & \sM_{X_i}  
\end{tikzcd}
\end{equation}

\begin{lem}
\label{lem:CartXi}
The scheme $(\msf Y^{\le_i \ch\lambda}_{X})_\red$ is naturally a locally closed subscheme of $\olsf Y^{\ch\lambda}_X$, equal to the union of strata
\[ 
\msf Y^{\le_i \ch\lambda}_{X} = \bigcup_{n\ge 0} \msf Y^{\ch\lambda- n\ch\alpha_i}_X.
\]
\end{lem}

\begin{proof}
From the Cartesian diagram \eqref{e:YXPM1}, we see that 
the stratification 
$\olsf Y^{\ch\lambda}_{X_i} = \bigcup_{n\ge 0} \msf Y^{\ch\lambda-n\ch\alpha_i}_{X_i}$ 
induces a stratification 
of $\msf Y^{\le_i \ch\lambda}_{X}$ by 
$\bigcup_{n\ge 0}\msf Y^{\ch\lambda-n\ch\alpha_i}_X$.

A point of $\msf Y^{\le_i\ch\lambda}_{X}$ is a map of stacks
\[
    y: C \to X \xt^{P_i} \ol{M_i/N_i} / T 
\]
such that $C\sm v$ is sent to the open substack $\pt = X^\circ/B$. 
In particular, $y$ defines a $P_i \xt T$-bundle together with a trivialization on $C\sm v$, i.e.,
a point of $\Gr_{P_i} \xt \Gr_T$. 
This gives a map 
\begin{equation} \label{e:YPBintoGr}
    \msf Y^{\le_i,\ch\lambda}_{X} \to \Gr_{P_i} \xt \Gr_T^{\ch\lambda}. 
\end{equation}
Since $X \xt \ol{M_i/N_i}$ is affine and $X \xt \ol{M_i/N_i} / (P_i \xt T) \supset 
(X \xt M_i)/ (P_i\xt N_i T) = X/B$ 
contains $\pt$ as an open substack, Lemma~\ref{lem:extendsection} 
implies that \eqref{e:YPBintoGr} is a closed embedding (the argument is
the same as the proof of Lemma~\ref{lem:barYGr}).
Therefore, at the level of reduced schemes we have 
a closed embedding 
\[ (\msf Y^{\le_i\ch\lambda}_{X})_\red \into \Gr_{P_i}. \]
We deduce that $(\msf Y^{\le_i \ch\lambda}_{X})_\red$ is a locally closed 
subscheme of $(\olsf Y^{\ch\lambda}_X)_\red$, since the (reduced) connected components of 
$\Gr_{P_i}$ are locally closed subschemes of $\Gr_G$.
\end{proof}

We are now ready to prove Proposition \ref{prop:f-hyper}:

\begin{proof}[Proof of Proposition \ref{prop:f-hyper}]
Let $\msf b \in \mB_{X^\bullet, \ch\lambda}$.
The generic point of $\msf b \subset \msf Y^{\ch\lambda}_X$ maps to 
$\msf Y^{\ch\lambda,\ch\theta}_{X_i}$ 
for some $\ch\theta \in \mathfrak c_{X_i}^-$.

If $\ch\theta=0$, then in the decomposition \eqref{e:crys-decomp} we have that $\msf b\in \mB_{X_i^\bullet,\ch\lambda}$ and 
$\ch\lambda = \ch\nu_i^\pm$ is the valuation attached to a color of $X_i^\bullet$. In that case, $\ch\nu_i^\pm-\ch\alpha_i \notin \mathfrak c_X$,
so  $\msf Y^{\ch\lambda-\ch\alpha_i,0}_{X_i} = \emptyset$. At the same time, $\msf b$ is a point, as recalled in \S \ref{sect:crystal-ef}.

Now assume $\ch\theta \ne 0$. As in the proof of Proposition \ref{prop:crys-decomp}, the composition $\iota_{X_i}\circ q$  sends $\msf b$ into $\msf S_{M_i}^{\ch\lambda}\cap \ol\Gr_{M_i}^{\ch\theta}$; let $\msf b_i$ be the image. By construction, $\tilde f_i \msf b$ either 
\begin{itemize}
 \item is zero iff $\ch\lambda = \ch\theta$, which happens precisely when $\msf S_{M_i}^{\ch\lambda}\cap \ol\Gr_{M_i}^{\ch\theta} = \{t^{\ch\theta}\}$ is closed in $\ol\Gr_{M_i}^{\ch\theta}$;
 \item or has image (under $\iota_{X_i}\circ q$) equal to an irreducible component $\tilde f_i \msf b_i$ of $\olsf b_i \cap \msf S_{M_i}^{\ch\lambda-\ch\alpha_i}$. Indeed, this is a property of the crystal structure on MV cycles by \cite{BGcrys}.
\end{itemize}

In either case, the closure relations ``downstairs'' lift to closure relations ``upstairs'' under \eqref{e:YXPM3}, as $\msf b$ and $\tilde f_i\msf b$ get identified with the lifts of $\msf b_i$ and $\tilde f_i\msf b_i$ corresponding to \emph{the same} element of $\mB^{P_i}_{X,\ch\theta}$ under the identification of fibers afforded by Lemma \ref{lem:XPfiber}. Therefore, in the first case the closure of $\msf b$ in $\msf Y^{\le_i \ch\lambda}_{X}$ is entirely contained in $\msf Y^{\ch\lambda}$, which by Lemma \ref{lem:CartXi} means that $\olsf b \cap \msf Y^{\ch\lambda-\ch\alpha_i}$ is empty; while in the second case $\tilde f_i\msf b$ corresponds to an irreducible component of $\olsf b \cap \msf Y^{\ch\lambda-\ch\alpha_i}$. 
\end{proof}

\section{Nearby cycles} \label{sect:nearby}

\subsection{Principal degeneration} \label{sect:degeneration}

In this section we will consider the principal degeneration $\mf X \to \mbb A^1$ 
degenerating $X$ to a horospherical variety. This is the base change of
the affine family $\mathscr X \to \ol{T_{X,\mss}}$ from \S\ref{sect:affine-degeneration}
along a certain line in the base. The principal degeneration 
was studied by \cite{Pop86} in characteristic $0$ and \cite{Grosshans92} in positive
characteristic.

We fix a choice of 
a regular dominant coweight $\ch\varrho \in \ch\Lambda^+_G \cap \ch\Lambda^\pos_G$
once and for all; none of our results depend on this choice. 
The image of $\ch\varrho$ lies in $-\Cal V$
and thus induces a map $\mbb A^1 \to \ol{T_{X,\mss}}$. 
We define the \emph{principal degeneration} 
\[ \mf X := \mathscr X \xt_{\ol{T_{X,\mss}}, \ch\varrho} \mbb A^1. \]  
We have a $G \xt \mbb G_m$-morphism 
$\mf X \to \mbb A^1$, where $G$ acts trivially on $\mbb A^1$. 
This forms an affine flat family (\cite[Proposition 9]{Pop86}). 
The fiber over $1 \in \mbb A^1$ is canonically isomorphic to 
$X$ and the fiber over $0\in \mbb A^1$ is $X_\emptyset := \Spec(\on{gr} k[X])$. 
The $\mbb G_m$-action on $\mf X$ induces an isomorphism 
$X \xt \mbb G_m \cong \mf X \xt_{\mbb A^1} \mbb G_m$. 
We also have a canonical isomorphism 
$\mf X /\!\!/ N = (X/\!\!/N) \times \mbb A^1$.

For $k$ of arbitrary characteristic, Grosshans showed that 
there is an injection 
\begin{equation}\label{e:Popov}
    k[X_\emptyset] \to (k[X\sslash N] \ot k[N^-\bs G])^T 
\end{equation}
of $G$-algebras, which is an isomorphism 
if and only if $k[X]$ admits a $G$-module filtration with subquotients
isomorphic to dual Weyl modules of $G$ (\cite[Theorems 8, 16]{Grosshans92}). 
Of course this always holds in characteristic $0$; in positive characteristic 
we will assume that \eqref{e:Popov} is an isomorphism in what follows.

Let $\mf X^\circ$ denote the preimage of $T_X \xt \mbb A^1 \subset X/\!\!/N \xt \mbb A^1$
under the map $\mf X \to \mf X/\!\!/N$. This is the 
open subvariety which specializes 
over each fiber above $\mbb A^1$ to the dense open $B$-orbit of 
that fiber. 
Let $\mf X^\bullet \subset \mf X$ denote the open subvariety which specializes
over each fiber to the open $G$-orbit of that fiber.

\medskip

The isomorphism \eqref{e:Popov} implies that $X_\emptyset$ has a natural left $T_X$-action, 
and the orbit of the coset $N^- 1$ gives an embedding 
$T_X \into X_\emptyset$. 
We will temporarily denote its closure by $\ol{T_X}$.
\begin{lem}  \label{lem:sectionX}
The composition $\ol{T_X} \into X_\emptyset \to X_\emptyset/\!\!/N = X/\!\!/N$ is an 
isomorphism. In other words, we have a section
$X/\!\!/N \into X_\emptyset$. 
\end{lem}
\begin{proof}
This is essentially a special case of \cite[Proposition 7]{AT}.
For $\lambda \in \Lambda_G^+$, the
isotypic component $k[N^-\bs G]^{(\lambda)}$ is the dual Weyl module
of highest weight $\lambda$. The embedding $T \into N^-\bs G$ 
corresponds to the algebra map 
$k[N^-\bs G] \to k[T]$ that sends $k[N^-\bs G]^{(\lambda)} \to k e^\lambda$ (explicitly it sends all $T$-eigenvectors not of highest weight to zero). 
Thus, using \eqref{e:Popov}, the map $k[X_\emptyset]\to k[T_X]$ 
sends $k[X_\emptyset]^{(\lambda)} \to k e^\lambda$ for $\lambda \in \mf c_X^\vee$.
\end{proof}

\subsubsection{Contracting action} 
Let $s : X/\!\!/N\into \mf X$
denote the composition of the section given by Lemma~\ref{lem:sectionX} and the embedding
$X_\emptyset \into \mf X$ as the zero fiber.
We will construct a $\mbb G_m$-action on $\mf X$ that 
contracts $\mf X$ to the section $s$.

Recall the coweight $\ch\varrho: \mbb G_m \to T$ used to define $\mf X$. 
Let $\mbb G_m$ act on $\mf X$ via the group homomorphism 
$\mbb G_m \to G \xt \mbb G_m : a \mapsto (\ch\varrho(a^{-1}), a)$ and the
natural $G\xt \mbb G_m$-action on $\mf X$. 

\begin{lem} \label{lem:contraction}
The action map $\mbb G_m \times \mf X \to \mf X$ extends to a regular map 
$\mbb A^1 \times \mf X \to \mf X$ such that the composition $0 \times \mf X \to \mbb A^1 \times \mf X \to \mf X$ coincides with the composition $\mf X \to X/\!\!/N \overset{s}\to \mf X$.
\end{lem}
\begin{proof}
The action map can be described as the map of rings 
\[ k[\mf X] \to k[\mbb G_m] \ot k[\mf X] : f_\mu e^n \mapsto 
e^{n-\brac{\mu,\ch\varrho}} \ot f_\mu e^n  \]
for a $T$-eigenvector $f_\mu \in k[X]_{(\lambda)}$ of weight $\mu$. 
We have $\brac{\mu,\ch\varrho} \le \brac{\lambda,\ch\varrho} \le n$,
so the image of the co-action lies in the subalgebra 
$k[\mbb A^1] \ot k[\mf X]$. 

Moreover, since $\ch\varrho$ is regular dominant, observe that $n-\brac{\mu,\ch\varrho}=0$ 
if and only if $f_\mu$ is a highest weight vector 
and $n=\brac{\lambda,\ch\varrho}$, 
which implies that the composition $0\times \mf X \to \mbb A^1 \times \mf X \to \mf X$ 
factors through $\mf X \to X/\!\!/N \overset s \to \mf X$.
\end{proof}

\subsection{Grinberg--Kazhdan theorem in families}
The map $\mf X \to \mbb A^1$ 
induces a map of formal arc spaces $\msf L^+ \mf X \to \msf L^+ \mbb A^1$. Define
$\msf L^+_{\mbb A^1} \mf X = \msf L^+ \mf X \underset{\msf L^+ \mbb A^1}\times \mbb A^1$
where $\mbb A^1 \into \msf L^+ \mbb A^1$ is the map of constant
arcs. Then $\msf L^+_{\mbb A^1} \mf X$ is an affine scheme over $\mbb A^1$
with fiber over $\mbb G_m$ isomorphic to $\msf L^+X$ and zero fiber isomorphic to 
$\msf L^+ X_\emptyset$. 
While there does not currently exist a theory of nearby cycles for infinite type
schemes, we explain below how the nearby cycles of the ``IC complex'' of 
$\msf L^+ X$ can be modeled by nearby cycles on the global or Zastava model.

\subsubsection{}
From now on we return to assuming that $B$ acts simply transitively on $X^\circ$, 
so $T_X=T$. Note that now \eqref{e:Popov} implies that $X_\emptyset$ is an 
affine embedding of $N^-\bs G$. 

\smallskip

First, we define the analogous models in families: let 
\[ \sM_{\mf X} = \Maps_\gen(C, \mf X/G \supset \mf X^\bullet/G), \quad 
\sY_{\mf X} = \Maps_\gen(C, \mf X/B \supset \mf X^\circ/B), \]
where $\mf X^\circ/B = \mbb A^1$.  
Since $C$ is proper, $\sM_{\mf X}$ maps to $\mbb A^1$ 
with generic fiber isomorphic to $\sM_X$ and special fiber $\sM_{X_\emptyset}$.
The arguments of \S\ref{sect:models} easily generalize to
show that $\sM_{\mf X}$ is an algebraic stack locally of finite type over $\mbb A^1$

The $B$-equivariant map $\mf X \to X\sslash N \xt \mbb A^1$ induces
a map $\sY_{\mf X} \to \sA \xt \mbb A^1$. We think of $\sY_{\mf X}\to \mbb A^1$
as a family degenerating $\sY=\sY_X$ to $\sY_\emptyset := \sY_{X_\emptyset}$. 
For $\ch\lambda\in \mf c_X$, let $\sY_{\mf X}^{\ch\lambda}$
denote the preimage of $\sA^{\ch\lambda}$ under $\pi_{\mf X} : \sY_{\mf X} \to \sA$.
We summarize the properties of $\sY_{\mf X}$ below; the proofs
are easy generalizations of those for $\sY$. 
\begin{itemize}
\item $\sY_{\mf X}^{\ch\lambda}$ is representable by a finite type scheme,
\item $\sY_{\mf X}$ satisfies the graded factorization property \emph{in families},
in the sense that there is a natural isomorphism 
\[ 
\Scr Y_{\mf X}^{\ch\lambda} \underset{\Scr A^{\ch\lambda}}\times 
(\Scr A^{\ch\lambda_1} \oo\xt \Scr A^{\ch\lambda_2}) 
\cong 
(\Scr Y_{\mf X}^{\ch\lambda_1} \underset{\mbb A^1}\times \Scr Y_{\mf X}^{\ch\lambda_2})|_{\Scr A^{\ch\lambda_1}\oo\times \Scr A^{\ch\lambda_2}} 
\]
when $\ch\lambda_1+\ch\lambda_2=\ch\lambda$.
\item There exists a closed embedding $\sY_{\mf X} \into \Gr_{B,\Sym C} \xt \mbb A^1$. 
\end{itemize}

The models $\sM_{\mf X}, \sY_{\mf X}$ are smooth-locally isomorphic 
as families, by the following generalization of Lemma~\ref{lem:localglobalyoga}: 
Let $\sY_{\mf X^\bullet} = \Maps_\gen(C, \mf X^\bullet/B \supset \mf X^\circ/B)$. 

\begin{lem}\label{lem:yoga-fam}
For fixed $\ch\lambda\in \mf c_X$ and any $\ch\mu\in \ch\Lambda^\pos_G$ large enough,
there is a correspondence 
\begin{equation}\label{e:Ycorrespondence}
    \Scr Y^{\ch\lambda}_{\mf X} \leftarrow \Scr Y^{\ch\lambda}_{\mf X} \mathbin{\underset{\mbb A^1}{\oo\times}} \Scr Y^{\ch\mu}_{\mf X^\bullet} \to \Scr M_{\mf X} 
\end{equation}
over $\mbb A^1$, where the left arrow is smooth surjective and the right arrow is smooth. 
\end{lem}
The proof is the same as in \emph{loc cit.}, together with the observation that
$\Maps(C, \mf X^\bullet/G)$ is smooth because $\mf X^\bullet/G$ is the classifying stack
of a smooth group scheme over $\mbb A^1$. 

\subsubsection{}
Fix an arc $\gamma_0 \in X_\emptyset(k\tbrac t) \cap X_\emptyset^\bullet(k\lbrac t)$
and consider $\gamma_0$ as a point in $\msf L^+_{\mbb A^1} \mf X (k)$.

\begin{thm}[Grinberg--Kazhdan, Drinfeld]  \label{thm:DGKfamily}
There exists a point $y \in \Scr Y_\emptyset(k) \subset \Scr Y_{\mf X}(k)$
such that the formal neighborhood $(\wh{\msf L^+_{\mbb A^1} \mf X})_{\gamma_0}$ 
is isomorphic to $\wh{\mbb A}^\infty \times \wh{\Scr Y}_{\mf X, y}$. 
\end{thm}

\begin{proof}
Fix a point $v \in \abs C$ and an identification $\mf o_v \cong k\tbrac t$. 
Note that each orbit in $X_\emptyset^\bullet(F_v)/G(\mf o_v)$
has a representative in $X_\emptyset^\circ(F_v)$. 
Thus, by $G(\mf o_v)$-translation we may assume that $\gamma_0 \in X_\emptyset(\mf o_v) \cap X_\emptyset^\circ(F_v)$.
We may consider $\gamma_0$ as a section $\Spec \mf o_v \to X_\emptyset \xt^B \hat\sP_B^0$
where $\hat\sP_B^0$ is the trivial $B$-bundle on $\Spec \mf o_v$. 
By Lemma~\ref{lem:B-L}, this is equivalent to a map
$y: C \to X_\emptyset/B$ with $y(C\sm v) = \pt$. 
This is the point $y\in \sY_\emptyset(k)$ we will use. 

\smallskip

Next, we define a map of formal schemes
\begin{equation}  \label{e:LXtoY}
(\wh{\msf L^+_{\mbb A^1} \mf X})_{\gamma_0} \to \wh{\Scr Y}_{\mf X,y} 
\end{equation}
as follows: 
Formal schemes are determined by
their $R$-points where $R$ is a local commutative
$k$-algebra with residue field $k$ whose maximal ideal $\mf m$ is nilpotent. 
Let $\gamma : \Spec R\tbrac t \to \mf X$ be an $R$-point 
of $(\wh{\msf L^+_{\mbb A^1} \mf X})_{\gamma_0}$. 
The reduction modulo $\mf m$ of $\gamma|_{\Spec R\lbrac t}$ 
equals $\gamma_0|_{\Spec k\lbrac t} \in \mf X^\circ(k\lbrac t)$. 
Since $k\lbrac t$ is the unique closed point of $R\lbrac t$, 
we deduce that $\gamma|_{\Spec R\lbrac t}$ has image contained in $\mf X^\circ$. 
Consider $\gamma$ as a section $\Spec(R \hot \mf o_v) \to \mf X \xt^B \hat\sP_B^0$
where $\hat\sP_B^0$ is the trivial $B$-bundle. 
By an easy generalization of Lemma~\ref{lem:B-L}, 
the pair $(\hat\sP^0_B, \gamma)$ is equivalent to a map 
$\tilde y : C\times \Spec(R) \to \mf X/B$
such that $\tilde y|_{(C\sm v)\times \Spec R}$
factors through $\Spec(R) \to \mbb A^1 = \mf X^\circ/B$.
We define \eqref{e:LXtoY} on $R$-points by sending 
$\gamma \mapsto \tilde y$.

Since $(\wh{\msf L^+ B})_1$ is non-canonically isomorphic to 
$\wh{\mbb A}^\infty$, the theorem follows from 
the proposition below. 
\end{proof}

\begin{prop}\label{prop:DGKtorsor}
Let $y : C\to X_\emptyset/B$ satisfy $y(C\sm v)=\pt$ and $y|_{\Spec \mf o_v}$ corresponds
to $(\hat\sP_B^0, \gamma_0)$. 
Then the map \eqref{e:LXtoY} is a $(\wh{\msf L^+ B})_1$-torsor. 
\end{prop}

\begin{proof}
Let $(R,\mf m)$ as above.  
By a generalization of Lemma~\ref{lem:B-L}, an $R$-point of $\wh{\Scr Y}_{\mf X,y}$ is
equivalent to a map $\bar\gamma : \Spec R \tbrac t \to \mf X/B$ such that 
$\bar\gamma|_{\Spec R\lbrac t}$ factors through $\Spec R \to \mbb A^1$. 
The fiber of \eqref{e:LXtoY} over $\bar\gamma$ parametrizes 
maps $\gamma : \Spec R \tbrac t \to \mf X$ that induce $\bar\gamma$ 
and whose reduction modulo $\mf m$ equals $\gamma_0$.
Since any $B$-bundle on $\Spec R\tbrac t$ can be trivialized, 
we see that $(\wh{\msf L^+B})_1(R)$ acts simply transitively on this fiber, since 
$\gamma_0 \in X^\circ_\emptyset(k\lbrac t)\cong B(k\lbrac t)$.
\end{proof}

\subsection{Results on nearby cycles} \label{sect:results-nearby}
We will now consider nearby cycles on the global and Zastava models. 

\subsubsection{}
Fix $\ch\lambda \in \mathfrak c_X$ and consider the family 
$f:\Scr Y^{\ch\lambda}_{\mf X} \to \mbb A^1$. 
We have complementary embeddings 
\[ \Scr Y^{\ch\lambda}_\emptyset \overset i\into \Scr Y^{\ch\lambda}_{\mf X} \overset j\hookleftarrow \Scr Y^{\ch\lambda} \times \mbb G_m \]
where $i,j$ correspond to $f^{-1}(0), f^{-1}(\mbb G_m)$, respectively.
We let 
\[ \Psi_\sY : \rmD^b_c(\sY^{\ch\lambda} \times \mbb G_m) \to \rmD^b_c(\sY_\emptyset^{\ch\lambda})
\] 
denote the nearby cycles functor 
defined in \cite[Expos\'e XIII]{SGA7-2}, shifted by $1$ so that
$\Psi_\sY$ is perverse $t$-exact.  

Let $\Psi_\sY^u$ denote the direct summand where the monodromy operator acts unipotently. 
\begin{rem}
Since $\mf X$ has a $\mbb G_m$-action making $f$ equivariant
and $\sF \bt \IC_{\mbb G_m}$ is $\mbb G_m$-equivariant for any sheaf $\sF$ on $\sY^{\ch\lambda}$,
a standard argument (\cite[Remark 14]{AB}) shows that $\Psi_\sY(\sF \bt \IC_{\mbb G_m}) = \Psi_\sY^u(\sF \bt \IC_{\mbb G_m})$, i.e., the monodromy is unipotent. 
\end{rem}

By $t$-exactness, $\Psi_\sY(\IC_{\sY^{\ch\lambda}\times \mbb G_m})$ is also perverse. 
In this section we will compute the image 
of $\Psi_\sY(\IC_{\sY^{\ch\lambda} \times \mbb G_m})$ in the Grothendieck group
of $\rmD^b_c(\sY^{\ch\lambda}_\emptyset)$ (equivalently, that of $\rmP(\sY^{\ch\lambda}_\emptyset)$); 
this determines the semisimplification of $\Psi_\sY(\IC_{\sY^{\ch\lambda} \times \mbb G_m})$
in the Artinian category $\rmP(\sY^{\ch\lambda}_\emptyset)$
over the algebraically closed field $k$. 

\smallskip

There is an analogous global family $\Scr M_{\mf X} \to \mbb A^1$ 
and complementary embeddings 
\[ \Scr M_{X_\emptyset} \into \Scr M_{\mf X} \hookleftarrow \Scr M_X \times \mbb G_m.  \]
Denote the associated unipotent nearby cycles functor by 
$\Psi_{\Scr M} : \rmD^b_c(\Scr M_X\times \mbb G_m) \to \rmD^b_c(\Scr M_{X_\emptyset})$.
We will simultaneously compute 
$[\Psi_{\Scr M}(\IC_{\Scr M_X\times \mbb G_m})]$.

\subsubsection{Stratifications in the horospherical case} \label{sect:horo-strat}
Before proceeding, we give a more concrete description of the stratifications
on $\sY_\emptyset, \sM_{X_\emptyset}$ using \eqref{e:Popov}. 

Since $T_X=T$, we have $X_\emptyset^\bullet = N^-\bs G$ and 
$\Scr M_{X_\emptyset^\bullet} = \Bun_{N^-}$. 
The isomorphism \eqref{e:Popov} induces a map of affine schemes
$X/\!\!/N \times^T \ol{N^-\bs G} \to X_\emptyset$, which in turn induces
a map of stacks 
\[ \bar\iota_{\Scr M} : \Scr A \underset{\Bun_T} \times \ol\Bun_{B^-} \to \Scr M_{X_\emptyset}. \]
Let $\iota_{\Scr M} : \Scr A \underset{\Bun_T}\times \Bun_{B^-} \into \Scr M_{X_\emptyset}$ 
denote the restriction. 
For $\ch\lambda \in \mf c_X$, let $\iota^{\ch\lambda}_{\Scr M}, \bar\iota^{\ch\lambda}_{\Scr M}$ 
denote the maps corresponding to $\sA^{\ch\lambda}$.
The following is a variant of \cite[Proposition 1.2.7]{BG}, whose proof we leave to the reader.
\begin{prop}
The map $\bar\iota_{\Scr M}^{\ch\lambda}$ is finite, and its restriction $\iota_{\Scr M}^{\ch\lambda}$ is 
a locally closed embedding. 
\end{prop}

The subschemes
$\Scr M_{X_\emptyset}^{(\ch\lambda)} = \iota_\sM^{\ch\lambda}(\Scr A^{\ch\lambda} \times_{\Bun_T^{-\ch\lambda}} \Bun_{B^-}^{-\ch\lambda})$ for $\ch\lambda\in \mf c_X$  
form a (possibly non-smooth) stratification
of $\Scr M_{X_\emptyset}$. 
This is a coarser stratification than the one defined in \S\ref{sect:globstrat}
in the following sense: 
Note that $\Cal V(X_\emptyset) = \mf t_X$ so $\mf c_{X_\emptyset}^- = \mf c_X$. 
For a partition $\mf P \in \Sym^\infty(\mf c_X\sm 0)$, 
there is a locally closed embedding $\oo C^{\mf P} \into \Scr A^{\deg(\mf P)}$, and the collection of these
subschemes ranging over all partitions $\mf P$ with $\deg(\mf P)=\ch\lambda$ forms a smooth
stratification of $\Scr A^{\ch\lambda}$. It follows from the constructions
that the strata from \S\ref{sect:globstrat} are given by 
\[ \Scr M^{\mf P}_{X_\emptyset} \cong 
\oo C^{\mf P} \underset{\Bun_T}\times \Bun_{B^-} \cong \oo C^{\mf P} \underset{\Scr A^{\deg(\mf P)}}\times \Scr M_{X_\emptyset}^{(\deg(\mf P))} \]
indexed over all $\mf P \in \Sym^\infty(\mf c_X\sm 0)$.

\medskip

Next if we consider the corresponding Zastava model, we have 
\[ \Scr Y_{X^\bullet_\emptyset} = \Scr Z^{?,0}  = \Maps(C, N^- \bs G / B \supset \pt) \]
is the \emph{open Zastava space} of Finkelberg--Mirkovi\'c. 
The \emph{Zastava space}\footnote{The Finkelberg--Mirkovi\'c Zastava space 
is the Zastava model for $\ol{G/N}$. In this paper we made a slight distinction in 
semantics between `model' and `space' to avoid confusion, but the two terms are 
interchangeable.} is in turn defined by 
$\Scr Z = \Maps_\gen(C, N^- \bs \ol{G/N} / T \supset \pt)$.
The geometry of the Zastava space has been extensively studied
in \cite{FM, FFKM, BFGM}. 
The components of $\Scr Z$ are indexed by $\ch\Lambda^\pos_G$. 

We can also define the relative open Zastava space 
${\Scr Z}^{?,0}_{\Bun_T} = \Maps_\gen(C, B^- \bs G/ B \supset \pt/T)$. 
The spaces $\Scr Z^{?,0}$ and $\Scr Z^{?,0}_{\Bun_T}$ are smooth locally isomorphic (cf.~\cite[\S 3.1]{BFGM}). 
For $\ch\lambda \in \ch\Lambda^\pos_G$, let $\Scr Z_{\Bun_T}^{\ch\lambda,0}$ denote
the preimages of $\Bun_{B^-}^{\ch\mu} \times \Bun_B^{\ch\mu-\ch\lambda}$ 
running over all $\ch\mu\in \ch\Lambda_G$. 

Since $\Scr Y_\emptyset$ is open in $\Scr M_{X_\emptyset}\times_{\Bun_G} \Bun_B$, we deduce by base change that for any $\ch\lambda\in \mf c_X, \ch\mu\in \ch\Lambda^\pos_G$, 
there is a locally closed embedding
\[ \iota_{\Scr Y_\emptyset}^{\ch\lambda,\ch\mu} : \Scr A^{\ch\lambda}\underset{\Bun_T}\times {\Scr Z}_{\Bun_T}^{\ch\mu,0} \into \Scr Y_\emptyset^{\ch\lambda+\ch\mu} \]
where we are mapping ${\Scr Z}^{?,0}_{\Bun_T} \to \Bun_{B^-} \to \Bun_T$. 
Note that ${\Scr Z}^{0,0}_{\Bun_T} = \Bun_T$, so  
$\iota^{\ch\lambda,0}_{\Scr Y_\emptyset}$ defines a map 
$\Scr A^{\ch\lambda} \into \Scr Y^{\ch\lambda}_\emptyset$. 
One can check that this map corresponds to applying $\Maps(C,T\bs ?)$ to 
the section $s_\emptyset : X/\!\!/N \into X_\emptyset$
from Lemma~\ref{lem:sectionX}. 
Therefore, 
\[ \mf s_\emptyset^{\ch\lambda} := \iota^{\ch\lambda,0}_{\Scr Y_\emptyset} : \Scr A^{\ch\lambda} \into \Scr Y^{\ch\lambda}_{\emptyset} \]
is a section of the projection $\pi_\emptyset : \Scr Y^{\ch\lambda}_\emptyset \to \Scr A^{\ch\lambda}$.

\subsubsection{}
We compute the $*$-restriction of nearby cycles to the strata above, which
suffices to determine nearby cycles in the Grothendieck group. 
Let $\Omega(\ch\mfn_C)^{-\ch\nu} = \mbb D (\Upsilon(\ch\mfn_C)^{\ch\nu})$ 
denote the Verdier dual of the factorization algebra defined in \S\ref{sect:compareIC}. 
Recall that we defined a convolution product $\star$ on $\rmD^b_c(\sA)$ in \S\ref{sect:convolution-product}. 
The statement of our main result is: 

\begin{thm} \label{thm:nearby}
We have equalities in the Grothendieck group
\begin{align*}
[\Psi_\sM(\IC_{\sM_X \times \mbb G_m})] &= 
\sum_{\ch\lambda\in \mf c_X} \sum_{\ch\nu\ge 0} 
\Bigl[ \iota^{\ch\lambda}_{\sM,!}\Bigl( \bigl(\mf i_{\Scr A,\ch\nu,!}(\Omega(\ch\mfn_C)^{-\ch\nu}) \star \bar\pi_!(\IC_{\barY^{\ch\lambda-\ch\nu}}) \bigr) \bt_{\Bun_T} \IC_{\Bun_{B^-}} \Bigr) \Bigr] \\ 
[\Psi_\sY(\IC_{\sY^{\ch\mu} \times \mbb G_m})] &= 
\sum_{\substack{\ch\lambda\in \mf c_X\\ \ch\lambda \le \ch\mu}} \sum_{\ch\nu\ge 0} 
\Bigl[ \iota^{\ch\lambda,\ch\mu-\ch\lambda}_{\sY,!}\Bigl(\bigl(\mf i_{\Scr A,\ch\nu,!}(\Omega(\ch\mfn_C)^{-\ch\nu}) \star \bar\pi_!(\IC_{\barY^{\ch\lambda-\ch\nu}}) \bigr) \bt_{\Bun_T} \IC_{\Scr Z^{\ch\mu-\ch\lambda,0}_{\Bun_T}} \Bigr) \Bigr] 
\end{align*}
for any $\ch\mu\in \mf c_X$ in the second equality.
\end{thm}

We point out that the description of $\bar\pi_!\IC_{\sY^{\ch\lambda}}$ is 
the main content of the previous sections of this paper. In particular it has the format \eqref{e:format1}. 

Recall that $\Scr F\underset{\Bun_T}\boxtimes \Scr G$ denotes the $*$-restriction of 
$\Scr F\boxtimes \Scr G$ to the corresponding fiber product over $\Bun_T$,
shifted by $[-\dim \Bun_T]$. Note that $\Bun_{B^-}$ and ${\Scr Z}^{?,0}_{\Bun_T}$ are smooth stacks, so the respective IC complexes are shifted constant sheaves.

\begin{proof}
Theorem~\ref{thm:nearby} follows by combining Corollary~\ref{cor:pibarY} 
with Lemma~\ref{lem:Psireduction} and Theorem~\ref{thm:PsiRadon} below.
\end{proof}

First, a well-known argument using some Zastava-to-global yoga 
allows us to reduce from computing restrictions to all strata to only
computing  
$\mf s^{\ch\lambda,*}_\emptyset \Psi_{\Scr Y}(\IC_{\Scr Y^{\ch\lambda}\times\mbb G_m})$
for all $\ch\lambda\in \mf c_X$. 

\begin{lem} \label{lem:Psireduction}
We have equalities in the Grothendieck group\footnote{In fact the isomorphisms hold in the derived category, but we omit the proof as it uses generic-Hecke equivariance to show no
twist exists on the $\Bun_{B^-}$ factor.}
\begin{align}
\label{e:YimpliesM}
    [\iota^{\ch\lambda,*}_{\Scr M} \Psi_{\Scr M}(\IC_{\Scr M_X\times \mbb G_m})] &= [\mf s^{\ch\lambda,*}_\emptyset \bigl(\Psi_{\Scr Y}(\IC_{\Scr Y^{\ch\lambda}\times\mbb G_m})\bigr) \underset{\Bun_T}{\boxtimes} \IC_{\Bun_{B^-}}]  \\ 
\label{e:MimpliesY}
    [\iota^{\ch\lambda,\ch\nu,*}_{\Scr Y_\emptyset} \Psi_{\Scr Y}(\IC_{\Scr Y^{\ch\lambda+\ch\nu} \times \mbb G_m})] &= [\mf s^{\ch\lambda,*}_\emptyset \bigl(\Psi_{\Scr Y}(\IC_{\Scr Y^{\ch\lambda} \times \mbb G_m}) \bigr) \underset{\Bun_T}\boxtimes \IC_{\Scr Z^{\ch\nu,0}_{\Bun_T}}],
\end{align}
where $\ch\lambda\in \mathfrak c_X, \,\ch\nu\in \ch\Lambda^\pos_G$. 
\end{lem}

\begin{proof}
The argument is the same as \cite[\S 3.1]{BFGM}, \cite[\S 8(1)]{BFGM-erratum} 
and \cite[Proof of Proposition 4.4]{BG2}. 
The strategy is that we first show \eqref{e:YimpliesM} and then use it to show \eqref{e:MimpliesY}. 

Fix $\ch\lambda'\in \mf c_X$ and $\ch\mu\in \ch\Lambda^\pos_G$ large enough as in Lemma~\ref{lem:yoga-fam} 
and consider the correspondence \eqref{e:Ycorrespondence}. 
The fiber of \eqref{e:Ycorrespondence} over $0\in \mbb A^1$
gives the correspondence
\[ \Scr Y^{\ch\lambda'}_\emptyset \leftarrow \Scr Y^{\ch\lambda'}_\emptyset \mathbin{\oo\times} {\Scr Z}^{\ch\mu,0} \to \Scr M_{X_\emptyset}. \]
Since nearby cycles commutes with smooth base change of the family over $\mbb A^1$, 
we deduce that there is an isomorphism 
\begin{equation} \label{e:PsiY=PsiM}
 \Psi_{\Scr Y}(\IC_{\Scr Y^{\ch\lambda'}\times\mbb G_m})|^{!*}_{\Scr Y^{\ch\lambda'}_\emptyset 
\oo\times {\Scr Z}^{\ch\mu,0}} \cong \Psi_{\Scr M}(\IC_{\Scr M_X\times \mbb G_m})|^{!*}_{\Scr Y^{\ch\lambda'}_\emptyset \oo\times {\Scr Z}^{\ch\mu,0}},
\end{equation}
We now restrict to strata: for $\ch\lambda\in \mf c_X$ observe that there is a commutative diagram 
where both squares are Cartesian
\begin{equation} \label{e:cdstratacorresp}
\begin{tikzcd}
\Scr A^{\ch\lambda} \underset{\Bun_T}\times {\Scr Z}^{\ch\lambda'-\ch\lambda,0}_{\Bun_T} 
\ar[d,hook, swap,"\iota_{\Scr Y_\emptyset}^{\ch\lambda,\ch\lambda'-\ch\lambda}"] & 
(\Scr A^{\ch\lambda} \underset{\Bun_T}\times {\Scr Z}^{\ch\lambda'-\ch\lambda,0}_{\Bun_T}) 
\mathbin{\oo\times} {\Scr Z}^{\ch\mu,0} \ar[r] \ar[d,hook] \ar[l]  &
\Scr A^{\ch\lambda} \underset{\Bun_T}\times \Bun_{B^-} \ar[d,hook, "\iota^{\ch\lambda}_{\Scr M}"] \\ 
\Scr Y^{\ch\lambda'}_\emptyset &
\Scr Y^{\ch\lambda'}_\emptyset \mathbin{\oo\times} {\Scr Z}^{\ch\mu,0} \ar[l] \ar[r] &
\Scr M_{X_\emptyset}
\end{tikzcd}
\end{equation}
By the argument of \cite[\S 8(1)]{BFGM}, every point of 
$\Scr A^{\ch\lambda} \times_{\Bun_T}\Bun_{B^-}$ is in 
the image of $\Scr A^{\ch\lambda} \mathbin{\oo\times} {\Scr Z}^{\ch\mu,0}$ 
for some $\ch\mu$ large enough, i.e., we only need to consider the diagram
\eqref{e:cdstratacorresp} when $\ch\lambda'=\ch\lambda$.
Note that $\sA^{\ch\lambda}\xt_{\Bun_T} \Bun_{B^-}$ is of finite type, so 
there exists a single $\ch\mu$ such that the map 
$\sY^{\ch\lambda}_\emptyset \oo\xt \Scr Z^{\ch\mu,0} \to \sM_{X_\emptyset}$ 
has geometrically irreducible fibers and 
contains $\Scr A^{\ch\lambda} \times_{\Bun_T}\Bun_{B^-}$ in its image (Corollary~\ref{cor:Mcover}). 
In particular, pullback along this map is fully faithful on perverse sheaves. 

By restricting the isomorphism \eqref{e:PsiY=PsiM} to the stratum
$\Scr A^{\ch\lambda} \mathbin{\oo\times} {\Scr Z}^{\ch\mu,0}$, we get 
\[
    \mf s_\emptyset^{\ch\lambda,*}\Psi_{\Scr Y}(\IC_{\Scr Y^{\ch\lambda}\times\mbb G_m})|^{!*}_{\Scr A^{\ch\lambda} \oo\times {\Scr Z}^{\ch\mu,0}} \cong \iota^{\ch\lambda,*}_{\Scr M}\Psi_{\Scr M}(\IC_{\Scr M_X\times \mbb G_m})|^{!*}_{\Scr A^{\ch\lambda} \oo\times {\Scr Z}^{\ch\mu,0}}. 
\]
This establishes the equality \eqref{e:YimpliesM} in the Grothendieck group
by fully faithfulness of the pullback.

The equality \eqref{e:MimpliesY} follows from \eqref{e:YimpliesM} and \eqref{e:PsiY=PsiM}
in the same fashion by considering the diagram \eqref{e:cdstratacorresp} with 
$\ch\lambda'=\ch\lambda+\ch\nu$.
\end{proof}

Define the functor $\Psi : \rmD^b_c(\sY^{\ch\lambda}) \to \rmD^b_c(\sY^{\ch\lambda}_\emptyset)$
by $\Psi(\sF) := \Psi_\sY(\sF \boxtimes \IC_{\mbb G_m}) = \Psi_\sY^u(\sF \bt \IC_{\mbb G_m})$. 
The crucial fact that will allow us to do our computations is the following: 

\begin{thm} \label{thm:PsiRadon}
There are natural isomorphisms of functors $\rmD^b_c(\sY^{\ch\lambda}) \to \rmD^b_c(\sA^{\ch\lambda})$:  
\begin{align*} 
\mf s^{\ch\lambda,*}_\emptyset \Psi 
&\cong \pi_{\emptyset,*}\Psi
\cong \pi_*  \\
\mf s^{\ch\lambda,!}_\emptyset \Psi 
&\cong \pi_{\emptyset,!}\Psi \cong \pi_!. 
\end{align*}
\end{thm}

The theorem will be proved using a standard argument
involving the contraction principle.

\subsection{Contraction principle}

Set $\mf s^{\ch\lambda} = i \circ \mf s^{\ch\lambda}_\emptyset : \Scr A^{\ch\lambda} \into \Scr Y_{\mf X}^{\ch\lambda}$. 
We drop the superscript to denote the section $\mf s : \Scr A \into \Scr Y_{\mf X}$
on all components.
Recall that $\mf s$ corresponds to the map induced by 
the embedding $s : X/\!\!/N \to \mf X$,
and Lemma~\ref{lem:contraction} defines an action of $\mbb A^1$ on $\mf X$
that contracts to $s$. 
The action of $\mbb A^1$ commutes with that of $G$, so we get an action 
$\mbb A^1 \times \Scr Y_{\mf X} \to \Scr Y_{\mf X}$
such that $0 \times \Scr Y_{\mf X} \to\Scr Y_{\mf X}$ coincides with 
$\mf s \circ \pi_{\mf X}$. 
In this situation, the \emph{contraction principle} (\cite[Lemma 5.3]{BFGM}, 
\cite[Lemme 2.2]{VLaff}, which is closely related to Braden's theorem \cite{Braden}) says that 
there is a natural isomorphism of functors $\pi_{\mf X,*} \cong \mf s^* : \rmD^b_c(\Scr Y_{\mf X}) \to \rmD^b_c(\Scr A)$.

\begin{proof}[Proof of Theorem \ref{thm:PsiRadon}]
We will prove the first line of isomorphisms; the second line follows from the first 
by Verdier duality. 
If we apply the contraction principle to 
$\mf s^{\ch\lambda,*}_\emptyset\Psi
= \mf s^{\ch\lambda,*} i_* \Psi$, 
we immediately get the first isomorphism 
$\mf s^{\ch\lambda,*}_\emptyset \Psi \cong \pi_{\emptyset,*}\Psi$.

Next, we will show the isomorphism $\mf s^{\ch\lambda,*}_\emptyset \Psi \cong \pi_*$. 
Recall that \cite{Beilinson2} gives an equivalence 
$D^b \rmP(\sY^{\ch\lambda}) \cong \rmD^b_c(\sY^{\ch\lambda})$, 
so we only need to define the isomorphism on perverse sheaves. 
Let $\sF \in \rmP(\sY^{\ch\lambda})$. 
For any $a \ge 1$, let $\Scr L_a$ denote the local system on $\mbb G_m$
whose monodromy is a unipotent Jordan block of rank $a$. There are canonical
injections $\Scr L_a \to \Scr L_{a+1}$.
Beilinson's construction of the unipotent nearby cycles functor 
(cf.~\cite[2.3]{Beilinson}) gives an isomorphism 
\[ \Psi(\Scr F) \cong \colim_{a\ge 1} i^*j_*(\Scr F \bt \Scr L_a). \]
We can further apply $\mf s^{\ch\lambda,*}_\emptyset$ to get
an isomorphism $\mf s^{\ch\lambda,*}_\emptyset \Psi(\Scr F) \cong \colim \mf s^{\ch\lambda,*} j_*(\Scr F \bt \Scr L_a)$.
Applying the contraction principle, we get an isomorphism 
\[ \mf s^{\ch\lambda,*}_\emptyset \Psi(\Scr F) \cong \colim_{a\ge 1}
(\pi_{\mf X} \circ j)_*(\Scr F \bt \Scr L_a). \]
Note that $\pi_{\mf X}\circ j : \Scr Y \times \mbb G_m \to \Scr A$ 
is equal to the composition of the first projection $\Scr Y\times \mbb G_m\to \Scr Y$
and $\pi:\Scr Y \to \Scr A$. Therefore, 
\[ (\pi_{\mf X}\circ j)_* (\sF\bt \Scr L_a) 
= \pi_*(\sF) \ot H^*(\mbb G_m, \Scr L_a). \]
Since $\colim_{a\ge 1} H^*(\mbb G_m, \Scr L_a) = \ol\bbQ_\ell$, 
we conclude that $\mf s^{\ch\lambda,*}_\emptyset \Psi(\sF) \cong \pi_*(\sF)$.
\end{proof}

\section{Function-theoretic corollaries}

\subsection{Pushforward of the basic function} \label{sect:finitefields} 

When $k$ is the algebraic closure of a finite field $\mathbb F$, and $X$ is defined over $\mathbb F$, satisfying the assumptions of \S \ref{assumptions-finitefield}, the action of the geometric Frobenius $\Fr$ morphism on $ \bar\pi_!(\IC_{\barY^{\ch\lambda}})$ is described, up to a yet unknown permutation action on the set $\mB^+_X = \bigcup_{\ch\lambda} \mB_{X,\ch\lambda}$ of central components of critical dimension, by Proposition \ref{prop:format}. We use this to prove Theorems \ref{theorem-affine-closure} and \ref{theorem-general} from the introduction.

\begin{proof}[Proof of Theorems \ref{theorem-affine-closure} and \ref{theorem-general}]
Recall that $\mf o$, in the context of these theorems, denotes the ring $\mathbb F\tbrac t$, where $\mathbb F$ is the finite field of definition of $X$, and $F$ is its fraction field. 

We need to recall the definition of the IC function $\Phi_0$  from \cite{BNS}: It is a function on $(X(\mf o)\cap X^\bullet(F))/G(\mf o)$ which, in our case, is parametrized by the set $(\mf c_X^-)^{\Fr}$ of elements of $\mf c_X^-$ that are fixed under the Galois group. To define it, choose an arc $\gamma$ in the coset of such an element $\ch\theta$, and consider a finite-dimensional formal model $\wh{Y}_y$ of the formal neighborhood of $\gamma$ in the arc space $\msf L^+X$ (Definition \ref{def:formalmodel}).  
In our case, we can take $Y= \sY^{\ch\theta}$ and $y=$ the point $t^{\ch\theta}$ on the central fiber $\msf Y^{\ch\theta,\ch\theta}$, by Theorem \ref{thm:DGK}. Then, the value of $\Phi_0$ on $\ch\theta$ is equal to the trace of geometric Frobenius on the stalk of the intersection complex $\IC_Y$ at $y$, \emph{where the intersection complex is normalized to be constant} (without Tate or cohomological twists) \emph{on the smooth locus of $Y$}. 

In our setting, this means that for every component $\mathscr Y$ as in Proposition \ref{prop:format}, the Tate and cohomological twist on $\IC_{\mathscr Y}$ should be modified from $\ol\bbQ_\ell(\frac{\dim\mathscr Y}{2})[\dim \mathscr Y]$ to $\ol\bbQ_\ell$, and \eqref{e:defVX} should be replaced by the space 
\begin{equation}\label{e:defVX2}
\bigoplus_{\mathscr Y} \bigoplus_{\mB_{\mathscr Y}} \ol\bbQ_\ell(-\smallfrac{\dim\mathscr Y}{2})[-\dim \mathscr Y] = \bigoplus_{\msf b\in \mB_{X,\ch\lambda}} \ol\bbQ_\ell(-\dim \msf b - \smallfrac{1}{2})[-2\dim \msf b - 1].
\end{equation}

To calculate the value at $\ch\lambda(t)$ of the integral of the basic function that was denoted by $\pi_! \Phi_0$ in the introduction, for $\ch\lambda$ fixed by Frobenius, we need to calculate the (alternating) trace of Frobenius on the fiber of $ \bar\pi_!(\IC_{\barY^{\ch\lambda}})$ over an $\mathbb F$-point of the diagonal $C\hookrightarrow \sA^{\ch\lambda}$, and then divide by the factor $\tr_{\check T}(\Fr, \Sym^\bullet (\check{\mathfrak n}(1)))$ in order to account for the difference between $\IC_{\sY}$ and $\IC_{\barY}$ (Corollary \ref{cor:pibarY}).

Taking into account the twists in the intersection complexes of the $C^{\mf R}$'s in \eqref{e:format1}, we deduce that, with this normalization of the IC sheaves, Frobenius acts on that fiber as on
\[
\bigoplus_{\deg(\mf R)=\ch\lambda} 
\Bigl( \bigotimes_{\ch\mu} 
\Sym^{N_{\ch\mu}}(\bigoplus_{\msf b\in \mB_{X,\ch\mu}} \ol\bbQ_\ell(-\dim \msf b)[2\dim \msf b] )\Bigr).
\]

Thus, in the notation of Theorems \ref{theorem-affine-closure} and \ref{theorem-general}, we have
\[\pi_! \Phi_0 = \tr_{\check T}(\Fr, \Sym^\bullet (\check{\mathfrak n}(1)))^{-1} \cdot \tr_{\check T}(\Fr, \Sym^\bullet (\bigoplus_{\msf b\in \mB_X^+} \ol\bbQ_\ell(-\dim \msf b))),\]
where we remind that $\mB_X^+ = \bigcup_{\ch\mu\in \mf c_X} \mB_{X,\ch\mu}$. 

The dimensions $\dim \msf b$ are given by Proposition \ref{prop:Vbasis}:
they are either equal to $\frac 1 2(\len(D)-1)$, if $\msf b$ meets $\sY_{X^\bullet}^{D}$, or $\brac{\rho_G,\ch\lambda-\ch\theta}$, if $\msf b$ meets $\sY^{\ch\lambda,\ch\theta}$ for $\ch\theta \in \Cal D^G_{\mathrm{sat}}(X)$. 

Both theorems assume the existence of a $G$-eigen-volume form on $X^\bullet$, which we fix, with eigencharacter $\mathfrak h$. We denote the absolute value of this eigencharacter by $\eta$. By Remark \ref{rem:length}, for any $D\in \mbb N^{\Cal D}$ with $\varrho_X(D)=\ch\lambda$, we have $ \len (D) = \left< \mathfrak h + 2\rho_G, \ch\lambda\right>$. The effect of multiplying by $  (\eta\delta)^\frac{1}{2}(t)$ will be to replace $\ol\bbQ_\ell(-\dim \msf b)$, in the expression above, by $\ol\bbQ_\ell(\frac{1}{2})$, for those $\msf b$ in $\sY_{X^\bullet}^{D}$; this proves Theorem \ref{theorem-affine-closure}(i),(iv),(v).  

For the rest of the components $\msf b$, taking into account that $\ch\lambda\ge \ch\theta$ and $\mf h$ is trivial on the roots, therefore $\left<\mathfrak h ,\ch\lambda \right> = \left<\mathfrak h, \ch\theta\right>$ and $\left< 2\rho_G, \ch\lambda\right>$ has the same parity as $\left<  2\rho_G, \ch\theta\right>$, the effect of multiplying by $  (\eta\delta)^\frac{1}{2}(t)$ will be to replace $\ol\bbQ_\ell(-\dim \msf b)$ by $\ol\bbQ_\ell(\frac{\left<\mathfrak h + 2\rho_G,\ch\theta\right>}{2})$, proving Theorem \ref{theorem-general}. 

The remaining parts, (ii), (iii) of Theorem \ref{theorem-affine-closure} follow directly from Theorem \ref{thm:crystal-properties}.
\end{proof}

\subsection{Asymptotics and the basic function} \label{sect:asymptotics}

We explain the proof of Corollary \ref{cor:asymptotics}, computing the basic function and its asymptotics.

\begin{proof}[Proof of Corollary \ref{cor:asymptotics}]
The (function-theoretic) Radon transform $\pi_{\emptyset !}$ on $X_\emptyset^\bullet$ is well-understood; in particular, the function $1_{X_\emptyset^\bullet(\mathfrak o)}$ maps to 
\[ \delta^{-\frac{1}{2}} \tr_{\check T}(\Fr, \Sym^\bullet (\check{\mathfrak n})) \cdot \tr_{\check T}(\Fr, \Sym^\bullet (\check{\mathfrak n}(1)))^{-1},\]
in the notation of \eqref{equation-pushforward-IC}. Moreover, $\pi_{\emptyset !}$ is equivariant for the action of the torus $T(F)$. This implies the formula \eqref{equation-asymptotics-IC} for $e_\emptyset^*\Phi_0$.

This, in turn, implies the formula \eqref{equation-basicfunction} for $\Phi_0$, by \cite[Corollary 5.5]{SaSatake}, where it is proven that, for an appropriate parametrization of the $G(\mf o)$-orbits, any $G(\mf o)$-invariant function $\Phi$ on $X^\bullet(F)$ is equal to its asymptotic $e_\emptyset^* \Phi$, restricted to the antidominant lattice $\ch\Lambda^{-,\Fr}$. Note that that paper was written under the assumption of $G$ being split (and $X$ being ``wavefront'', which is automatic in our setting), but the proof of the result that we are quoting is valid without this assumption. (For example, ``quasi-split'' is enough for the Casselman--Shalika method, which is the basis of that argument, and for the asymptotics theory developed in \cite{SV}.)
\end{proof}

\subsection{Euler factors of global integrals} \label{sect:nfold-explanation}

Finally, to demonstrate how our results apply directly to compute Euler factors of global integrals of automorphic forms, let us return to the setting of Examples \ref{example-nfold}, \ref{example-nfold-global}, in order to discuss the Euler factorization of the pertinent global period integral. This will not directly invoke the nearby cycles functor (hence, is based on results of previous sections only), but it makes use of the asymptotics map $e_\emptyset^*$ that we just recalled.

To recall, the group is $G=(\mathbb G_m\times \SL_2^n)/\mu_2$, and $X$ is the affine closure of $H\backslash G$, where $H$ is the product of the group $H_0$ of \eqref{example-nfold} with a copy of $\mathbb G_m$ embedded diagonally into $G$ as 
\begin{equation}\label{Gmembedding} a \mod \mu_2 \mapsto \left( a, \begin{pmatrix} a \\ & a^{-1}\end{pmatrix}^n \right ) \mod \mu_2.
\end{equation}
We let $\Phi = \prod_v \Phi_v$ be a function on $X^\bullet(\mathbb A)$ as in Example \ref{example-nfold-global}.

Let $\phi\in \pi$ be a cusp form, where the restriction of $\pi$ to $\mathbb G_m$ is a character of the form $\chi_0 |\bullet|^s$ for $\chi_0$ unitary and some $s \gg 0$. Fix a nontrivial character $\psi$ of $\mbb A/\Bbbk$, identified as a character of the quotient $N^-/H_0$, where $N^-$ is the lower triangular unipotent subgroup (mapping to the additive group through summation of the entries). The standard method of writing $\phi$ in terms of its Whittaker function $W_\phi$ with respect to $(N^-,\psi)$ shows that the integral   
\[ \int_{G(\Bbbk)\backslash G(\mbb A)} \phi(g) \sum_{\gamma\in X^\bullet(\Bbbk)} \Phi(\gamma g) dg = \int_{X^\bullet(\mbb A)} \Phi(g) \int_{H(\Bbbk)\backslash H(\mbb A)} \phi(hg) dh dg\]
``unfolds'' to the Eulerian integral
 \[ \int_{H_0\backslash G (\mbb A)} W_\phi(g) \Phi(g) dg.\]
 
We can write the Whittaker function as $W_\Phi(g) = \prod_v W_{\phi, v}(g_v)$, so that $W_{\phi,v}(1) = 1$ for almost all $v$, and we can factorize the invariant measure on $H_0\backslash G$ so that for almost all $v$, $\mu_v(H_0\backslash G(\mf o_v)) = 1$. 

\begin{prop}\label{prop:localzeta}
At places $v$ where $W_{\phi,v}(1)$ is $G(\mf o_v)$-invariant with $W_{\phi,v}(1) = 1$ (in particular, $\pi_v$ is unramified), $\Phi_v=\Phi_{0,v}$ (the IC function), $\mu_v(H_0\backslash G(\mf o_v)) = 1$, and the conductor of $\psi|_{\Bbbk_v}$ is the ring of integers, we have
\begin{equation}\label{eq:localzetafactor}
   \int_{H_0\backslash G (\Bbbk_v)} W_{\phi,v}(g_v) \Phi_v(g_v) dg_v = L(\pi_v, \otimes, 1 - \frac{n}{2}).
\end{equation}
\end{prop}

\begin{proof}
The proof can be obtained by direct computation from the explicit formula \eqref{equation-basicfunction} for the basic function. However, we would like to sketch a ``pure thought'' argument, which applies to other cases as well, without providing all details. 
We drop the index $v$, denoting the place under consideration, and write $F$ for $\Bbbk_v$.

First of all, the \emph{square of the absolute value} of \eqref{eq:localzetafactor} follows immediately from Plancherel-theoretic considerations. Indeed, we can write the local integral as 
 \[    Z(\Phi, \pi)=  \int_{N^-\backslash G (F)} W_{\phi}(g) W_{\Phi}(g)  dg,\]
 where 
 \begin{equation}\label{unfolding} W_{\Phi}(g) = \int_{H_0\backslash N^-(F)} \Phi(ng)\psi(n) dn.
 \end{equation}
 We identify the abelianization $G^{\text{ab}}$ with $\mathbb G_m$, so that the composite $\mathbb G_m\to G\to\mathbb G_m$ is the square character. 
 Let $\eta$ be the character $a\mapsto |a|^{1-n}$ on $G$ --- it is the character by which it acts on the $\SL_2^n$-invariant measure on $X^\bullet$. Fix that measure giving volume $1$ to $X^\bullet(\mf o)$.
 It is known from \cite[Theorem 9.5.9]{SV} that the ``unfolding'' map $\Phi \mapsto W_{\Phi}$ extends to an $L^2$-isometry
 \[ L^2(X^\bullet(F)) \xrightarrow\sim L^2(N^-\backslash G(F), \psi^{-1}),\]
 where the Haar measure on $N^-\backslash G(F)$ is fixed so that the volume of $N^-\backslash G(\mf o)$ is $1$.
 
 The local zeta integrals appear in the Plancherel decomposition of Whittaker functions,  \begin{equation}\label{eq:PlancherelWhittaker} \Vert \Phi\Vert^2_{L^2(X^\bullet(F))} = \Vert W_{\Phi}\Vert^2_{L^2(N(F)\backslash G(F), \psi^{-1})} = \int |Z(\Phi, \pi)|^2 d\pi, 
 \end{equation}
  where $\pi$ ranges over the unitary unramified dual of $G(F)$, which can be identified with the set of semisimple conjugacy classes in the compact form of the complex dual group $\check G(\mbb C)$, and the Plancherel measure $d\pi$ is the Weyl measure $\left|(1-e^{\ch\alpha})\right|^2 d\chi$ (when we identify the unramified dual with the quotient $\ch T^1/W$, where $\ch T^1$ is the group of unitary unramified characters --- the maximal compact subgroup of the complex points of the dual Cartan $\ch T$).
 
 On the other hand, the Plancherel decomposition of $\Phi$ can also be expressed in terms of its asymptotics:
 \[ \Vert \Phi\Vert^2_{L^2(X^\bullet(F))} = \frac{1}{|W|} \Vert e_\emptyset^* \Phi\Vert^2_{L^2(X_\emptyset^\bullet(F))},
 \]
 where we have applied \cite[Theorem 7.3.1]{SV}, together with the following observations: 1) our calculation of $e_\emptyset^* \Phi$ shows that it already lies in $L^2(X_\emptyset^\bullet(F))$, so there is no difference between that and what is denoted by $\iota_\emptyset^* \Phi$ in \emph{loc.cit.}; 2) there are no unramified functions $\Phi$ on $X^\bullet(F)$ with $e_\emptyset^*\Phi = 0$; this is essentially \cite[Theorem 6.2.1]{SaSpc}, and it implies, together with the previous point, that the summands of \cite[Theorem 7.3.1]{SV} with $\Theta\ne \emptyset$ do not contribute to the Plancherel decomposition. (This is not an essential point for this proof; the argument would go through even if there were contributions outside of the most continuous spectrum.)

By our formula \eqref{equation-pushforward-IC-expl}, as specialized in Example \ref{example-nfold}, and the commutation of asymptotics with Radon transform, \eqref{Radon-asymptotics}, we can write the Plancherel decomposition of the basic function as 
 \begin{equation}\label{eq:Plancherelnfold} \Vert \Phi\Vert^2_{L^2(X^\bullet(F))} = \frac{1}{|W|} \int_{\check T^1} 
\left|\frac{\prod_{\check\alpha\in\check\Phi^+} (1-e^{\check\alpha}(\chi))}{\prod_{\check\lambda \in \mB^+} (1-q^{-\frac{1}{2}} e^{\check\lambda}(\chi))}\right|^2 d\chi =   
 \int L(\pi, \otimes, \frac{1}{2}) L(\tilde \pi, \otimes, \frac{1}{2})
 d\pi.
 \end{equation}
 
Comparing \eqref{eq:PlancherelWhittaker} and \eqref{eq:Plancherelnfold}, we get that $|Z(\Phi, \pi)|^2 = L(\pi, \otimes, \frac{1}{2}) L(\tilde\pi, \otimes, \frac{1}{2})$ for $\pi$ unitary. This is what we get ``for free'' from the $L^2$-decomposition.
 
To upgrade it to a formula on $Z(\Phi, \pi)$, we utilize the following information about the image $W_{\Phi}$ under the unfolding map. Note that this is a $G(\mf o)$-invariant Whittaker function, hence can be identified as a function on the \emph{antidominant} coweights, by evaluating it on the elements $t^{\ch\theta}$. The facts that we need are:

\begin{enumerate}
 \item The support of $W_{\Phi}$ lies in the intersection of the antidominant lattice with the cone generated by $\mf c_X^-$ and the coweight \[\ch\theta_0:= - \frac{\check\alpha_1+\check\alpha_2+\dots+\check\alpha_n - \check m}{2}.\]
 Its value at $1 = t^0$ is $1$.
 
 Indeed, the support of $\Phi$ lies in $H \xi t^{\mf c_X^-} G(\mf o)$, where $\xi$ is the diagonal image of $\begin{pmatrix} 1 & 0 \\ 1 & 1\end{pmatrix}$, and from the definition \eqref{unfolding} of the unfolding map, the support of $W_{\Phi}$ will lie in $H \xi N^- t^{\mf c_X^-}  G(\mf o) = N^- \mbb G_m t^{\mf c_X^-}  G(\mf o)$, where $\mbb G_m$ is embedded as in \eqref{Gmembedding}. But a $G(\mf o)$-invariant Whittaker function with respect to a character of $N^-$ whose conductor is $\mf o$ can only be supported on antidominant elements. 
 
 When $g=1$, the integrand of \eqref{unfolding} only lies in the support of $\Phi$ when $n\in H_0\backslash N^-(\mf o)$.
 
 \item The value of $W_{\Phi}$ at any $t^{\ch\theta}$ is a polynomial in $q^{-\frac{1}{2}}$. 
 
 This fact, which can be seen as expressing the ``motivic'' nature of this function, requires some explanation. First of all, the statement is true if $W_{\Phi}$ is replaced by the pushforward $\pi_! \Phi$, by Theorem \ref{theorem-affine-closure} (see also \S \ref{sect:finitefields}). Secondly, the unfolding integral \eqref{unfolding} of the characteristic function of each $G(\mf o)$-orbit has this property; this can be seen by direct calculation, or by a similar geometric argument, and we omit the details. 
 
 \item If $\varphi_{\ch\theta}$ denotes the $G(\mf o)$-invariant Whittaker function supported on the coset $N^-t^{\ch\theta} G(\mf o)$ (with $\ch\theta$ antidominant) and equal to $1$ on $\ch\theta$, its asymptotic $e_\emptyset^* \varphi_{\ch\theta}$ is supported on the union of a finite number of cosets $N^-t^{\ch\theta'} G(\mf o)$, with $\ch\theta'-\ch\theta$ in the positive root monoid, and is a polynomial in $q^{-\frac{1}{2}}$. 
 
 Here, the asymptotics map is for the Whittaker model, but we denote it by the same symbol. It takes values in $N^-\backslash G(F)$, without a character on $N^-$. This fact is a corollary of the Casselman--Shalika formula. (See \cite[Theorem 6.8 and Example 6.4]{SaSatake} for an intepretation of the Casseman--Shalika formula in terms of asymptotics.)

\end{enumerate}

Granted, now, the facts above, we can represent $e_\emptyset^* W_{\Phi}$ as a formal series in the elements $e^{\ch\theta}$, which stand for the characteristic functions of the sets $N^- t^{\ch\theta} G(\mf o)$, multiplied by $\delta^{-\frac{1}{2}}(t^{\ch\theta}) = q^{\left<\rho_G,\ch\theta\right>}$, with coefficients in $\mathbb C[q^{-\frac{1}{2}}]$:
\[ e_\emptyset^* W_{\Phi} = \sum_{i,\ch\theta} c_{i,\ch\theta} q^{-\frac{i}{2}} e^{\ch\theta}.\]

According to the first and the third facts above, the support of this function lies in the set of $\ch\theta$'s which are translates, by the positive coroot monoid, of the intersection
\[ (\mf c_X^- + \text{span}(\ch\theta_0)) \cap \ch\Lambda^-.\]
The description of colors of $X$ in Example \ref{example-nfold} allows us to conclude that this monoid only includes coweights for which the ``determinant'' character 
\[\frac{\kappa \ch\mu + \sum_i \kappa_i \ch\alpha_i}{2} \mapsto \kappa\]
is \emph{nonnegative}; moreover, the restriction to the kernel of the determinant
cocharacter is simply equal to the asymptotics of the ``basic Whittaker function'' $\varphi_0$, which by \cite[Theorem 6.8 and Example 6.4]{SaSatake} is equal to 
\[\prod_{\ch\alpha \in \ch\Phi^+} (1-e^{\ch\alpha}).\]

Thus, if we ``divide'' the function $e_\emptyset^* W_{\Phi}$ by the product above (this corresponds to acting on it by a series in the Hecke algebra of the torus $T$), we obtain another series
\[\prod_{\ch\alpha \in \ch\Phi^+} (1-e^{\ch\alpha})^{-1} e_\emptyset^* W_{\Phi} = 
1\cdot e^0 + \sum_{j=1}^\infty \sum_{i=0}^\infty q^{-\frac{i}{2}} e^{j\frac{\ch\mu}{2}} C_{i,j}
,\]
where the $C_{i,j}$'s are finite linear combinations of the elements $e^{\ch\lambda}$, where $\ch\lambda$ now ranges over half-multiples of elements in the coroot lattice. (The finiteness of the linear combination also follows from \cite[Theorem 6.8 and Example 6.4]{SaSatake}: the asymptotics of \emph{every} $\varphi_{\ch\theta}$ is a ``multiple'' of the factor $\prod_{\ch\alpha \in \ch\Phi^+} (1-e^{\ch\alpha})$.)
Note that $\frac{\ch\mu}{2}$ is not, by itself, a coweight of $G$, so one has to expand this series to make sense of it as a linear combination of the elements $e^{\ch\theta}$ with $\ch\theta \in \ch\Lambda$.

Now we invoke the Plancherel formula: the fact that the unfolding map is an isometry tells us that the Plancherel density of $W_{\Phi}$ is also given by \eqref{eq:Plancherelnfold}. On the other hand, this Plancherel density can be expressed in terms of the asymptotics
$e_\emptyset^* W_{\Phi}$ as 
\[ \Vert W_{\Phi}\Vert^2 = \frac{1}{|W|} \Vert e_\emptyset^* W_{\Phi}\Vert^2 = \frac{1}{|W|} \int_{\check T^1} 
\left|1\cdot e^0 + \sum_{j=1}^\infty \sum_{i=0}^\infty q^{-\frac{i}{2}} e^{j\frac{\ch\mu}{2}} C_{i,j}\right|^2 \left| \prod_{\ch\alpha \in \ch\Phi^+} (1-e^{\ch\alpha})\right|^2 d\chi = \]
\[ = \frac{1}{|W|} \int_{\check T^1} 
\left(1\cdot e^0 + \sum_{j=1}^\infty \sum_{i=0}^\infty q^{-\frac{i}{2}} e^{j\frac{\ch\mu}{2}} C_{i,j}\right) \left(1\cdot e^0 + \sum_{j=1}^\infty \sum_{i=0}^\infty q^{-\frac{i}{2}} e^{-j\frac{\ch\mu}{2}} C^*_{i,j}\right) \prod_{\ch\alpha \in \ch\Phi^+} \left|(1-e^{\ch\alpha})\right|^2 d\chi,\]
where we now identify the elements $e^{\ch\theta}$ with characters of the dual torus $\ch T$, use the fact that for a unitary character $\overline{e^{\ch\theta}(\chi)} = e^{-\ch\theta}(\chi)$, and set $C^*_{i,j} = \sum_{\ch\lambda} \overline{c_{\ch\lambda}} e^{-\ch\lambda}$ if $C_{i,j} = \sum_{\ch\lambda} c_{\ch\lambda} e^{\ch\lambda} $. 

In other words, we have expressed the Plancherel density 
\[ 
 L(\pi, \otimes, \frac{1}{2}) L(\tilde \pi, \otimes, \frac{1}{2})
\]
as the product 
\[
 \left(1\cdot e^0 + \sum_{j=1}^\infty \sum_{i=0}^\infty q^{-\frac{i}{2}} e^{j\frac{\ch\mu}{2}} C_{i,j}\right) \left(1\cdot e^0 + \sum_{j=1}^\infty \sum_{i=0}^\infty q^{-\frac{i}{2}} e^{-j\frac{\ch\mu}{2}} C^*_{i,j}\right). 
\]
We will be done if we can identify the first factor of the former with the first factor of the latter. But, viewed as series in the elements $q^{-\frac{i}{2}} e^{-j\frac{\ch\mu}{2}}$ (with coefficients in polynomials in the group ring of the half-root lattice), the first factor of the former is the restriction of the series to the elements of the form $q^{-\frac{i}{2}} e^{i\frac{\ch\mu}{2}}$, while the product 
$ L(\pi, \otimes, \frac{1}{2}) L(\tilde \pi, \otimes, \frac{1}{2})$ is supported on elements of the form $q^{-\frac{i}{2}} e^{j\frac{\ch\mu}{2}}$ with $j\le i$. It follows that the series $ 1\cdot e^0 + \sum_{j=1}^\infty \sum_{i=0}^\infty q^{-\frac{i}{2}} e^{j\frac{\ch\mu}{2}} C_{i,j}$ is also supported on elements of the form $q^{-\frac{i}{2}} e^{j\frac{\ch\mu}{2}}$ with $j \le i$, coincides with $L(\pi, \otimes, \frac{1}{2})$ on elements with $i=j$, and a simple inductive argument in the variable $i-j$ shows that it coincides with it everywhere.
\end{proof}

\appendix 

\section{Properties of the global stratification} \label{appendix:A}

\subsection{The factorizable space of formal loops} \label{sect:multijet}

We briskly review the definitions of multi-point versions of the 
spaces of formal arcs and formal loops. We refer the reader to \cite{KV}, \cite[\S 3.1]{Xinwen} for a more complete account.  

Let $C$ be a smooth curve over $k$. 
For any $N \in \mbb N$, we have the $N$th symmetric product
$C^{(N)}$ of $C$, which identifies with the Hilbert scheme $\on{Hilb}^N(C)$
parametrizing relative effective divisors in $C$ of degree $N$.

Recall that if $S$ is an affine scheme and $D \subset C\times S$ 
is a closed affine subscheme, we denote by $\wh C'_D$ the 
spectrum of the ring of regular functions
on the formal completion of $C\times S$ along $D$ (so $\wh C'_D$ is a true scheme, not merely a formal scheme). 
Let $\wh C^\circ_D := \wh C'_D \sm D$ denote
the open subscheme. 

For any $k$-scheme $X$, we define the global space of formal arcs
by the functor
\begin{align*} 
(\Scr L^+ X)_{C^{(N)}}(S) & = \{ D \in C^{(N)}(S), \gamma \in X(\wh C'_D) \}. 
\end{align*}
for affine test schemes $S$. 
By \cite[Proposition 2.4.1]{KV}, the functor $(\Scr L^+ X)_{C^{(N)}}$ is representable by a scheme of infinite type over $C^{(N)}$.
If $X$ is affine, then $(\Scr L^+ X)_{C^{(N)}}$ is affine
over $C^{(N)}$, cf.~\cite[2.4.3]{KV}. 
More specifically, 
define the space of $n$-jets $(\Scr L^+_n X)_{C^{(N)}}$ by 
\[ (\Scr L^+_n X)_{C^{(N)}}(S) = \{ D\in C^{(N)}(S), \gamma \in X(\wh C^n_D) \}, \]
where $\wh C^n_D$ denotes the $n$th infinitesimal neighborhood of $D$
in $C\times S$. Then $(\Scr L^+_n X)_{C^{(N)}}$ is representable
by a scheme over $C^{(N)}$, which is affine (resp.~of finite type) if $X$ is.  As $n$ varies 
the schemes $(\Scr L^+_n X)_{C^{(N)}}$ form a projective system of schemes 
with affine transition maps, and 
$(\Scr L^+_n X)_{C^{(N)}}$ is equal to the projective limit of this system.
If $X$ is smooth, the schemes $(\Scr L^+_n X)_{C^{(N)}}$ are
smooth over $C^{(N)}$ with smooth surjective transition maps (cf.~\cite[Lemma 2.5.1]{cpsii}).

We can also define the functor for the global loop space by 
\begin{align*} 
(\Scr L X)_{C^{(N)}}(S) & = \{ D \in C^{(N)}(S), \gamma \in X(\wh C^\circ_D) \}. 
\end{align*}
\emph{If $X$ is affine}, then $(\Scr L X)_{C^{(N)}}$ 
is representable by an ind-scheme ind-affine over $C^{(N)}$,
cf.~\cite[Proposition 2.5.2]{KV}. 
We have a closed embedding $(\Scr L^+ X)_{C^{(N)}} \into (\Scr L X)_{C^{(N)}}$.

\subsubsection{}\label{sect:Llocalization}
Let $\Aut^0(k\tbrac t)$ denote the functor sending a $k$-algebra $R$
to the group of $R$-algebra automorphisms of $R\tbrac t$ that reduce to
the identity map mod $t$. Then $\Aut^0(k\tbrac t)$ is representable 
by the group scheme $\Spec k[a_1^{\pm 1}, a_2, a_3, \dotsc]$, cf.~\cite[(1.3.13)]{Xinwen}. 

\smallskip

There is an $\Aut^0(k\tbrac t)$-torsor $\on{Coord}^0(C) \to C$
classifying $v \in C$ together with an isomorphism $k\tbrac t \cong \mf o_v$
that sends $t$ to a uniformizer, cf.~\cite[(3.1.10)]{Xinwen}. 
We can think of $(\Scr L^+ X)_C, (\Scr L X)_C$ as twisted products 
\[
    (\Scr L^+ X)_C = C \ttimes \msf L^+ X, \qquad (\Scr L X)_C = C\ttimes \msf L X, 
\]
where $C \ttimes \msf L^+ X := \on{Coord}^0(C) \times^{\Aut^0(k\tbrac t)} \msf L^+ X$.

\begin{rem}
The space $\Scr L X$ really lives over
the Ran space of $C$. Essentially this just means $\Scr L X$
only cares about the support of the divisor $D$ and not its multiplicities. 
More specifically for any finite set $I$ we have a map 
$C^I \to C^{(\abs I)}$ where $\abs I$ denotes the cardinality of $I$. 
Then the spaces $(\Scr L X)_{C^I} := C^I \times_{C^{(\abs I)}} (\Scr L X)_{C^{(\abs I)}}$ 
have a factorization monoid structure as
defined in \cite[Definition 2.2.1]{KV}, and we can think 
of the collection of these spaces as $(\Scr L X)_{\on{Ran}_C}$. 
This is certainly the more philosophically correct approach to 
considering the \emph{loop} space, but for technical simplicity 
it will suffice for our study of \emph{arc} spaces to work with $\Scr L^+ X$
over $\Sym C$.
\end{rem}

\subsubsection{} \label{sect:localGr}
We can apply the constructions above to the algebraic group $G$. 
Since $G$ is smooth, $\Scr L^+ G$ is a group scheme formally smooth over 
$\Sym C$.

Consider the (factorizable) Beilinson--Drinfeld affine Grassmannian $\Gr_{G,\Sym C}$ defined in \S\ref{sect:YGr}. 
By Beauville--Laszlo's theorem (see \cite[Remark 2.3.7]{BD}, \cite[Proposition 3.1.9]{Xinwen}), we have an isomorphism 
$ \Gr_{G,\Sym C} \cong \Scr L G / \Scr L^+ G$.

\subsubsection{} Let $X$ be an affine spherical 
$G$-variety. Define 
\[
    \Scr L^\bullet X := \Scr L X \sm \Scr L(X\sm X^\bullet),
\]
which admits an open embedding into $\Scr L X$. 
The $G$-action on $X$ induces a natural action of $\Scr L G$ on 
$\Scr L X$ and the subspace $\Scr L^\bullet X$ (resp.~$\Scr L^+ X$) is stable under the action of 
$\Scr L G$ (resp.~$\Scr L^+ G$). 

We will primarily be concerned with the global space of 
non-degenerate arcs
\[ (\Scr L^+X)^\bullet := \Scr L^+X \times_{\Scr L X} \Scr L^\bullet X.\] 
The study of the loop space $\Scr L^\bullet X$ is beyond the scope of 
this paper.

\subsection{Multi-point orbits} \label{sect:defstrata}
We now consider the $\Scr L^+G$-orbits on $(\Scr L^+X)^\bullet$, and
we will prove a multi-point version of Proposition~\ref{prop:smoothfiberstrata}.

\subsubsection{What is going on at the level of $k$-points} 
A $k$-point of $\Ran_C$ is a nonempty finite subset 
$\{ v_i \in \abs C \}_{i\in I}$ of points on $C$. 
Then a $k$-point of $(\Scr L^+X)^\bullet_{\Ran_C}$ over $\{v_i\}$ consists
of points $x_i \in X(\mf o_{v_i}) \cap X^\bullet(F_{v_i})$. 
Each $x_i$ belongs to an orbit $X^\bullet(F_{v_i})_{G:\ch\theta_i}$
for a unique $\ch\theta_i \in \mathfrak c_X^-$ by Theorem~\ref{thm:Gorbits}(ii). 
The collection $\ch\theta_{i},\, i\in I$ forms a multiset
in $\mathfrak c_X^-$. 
Note that $\ch\theta_{i}$ may be zero. 
Therefore, the tuple $(x_i)_{i\in I}$ is a $k$-point in the product of orbits 
$\underset{i\in I}\prod \msf L^{\ch\theta_i}_{v_i} X$.

The idea for what follows is that this product of orbits only depends on
the unordered multiset $\{\ch\theta_i\}$, counted with multiplicity. 
Moreover since $\Cal C_0$ is a strictly convex cone, 
the set of formal sums $\sum_I \ch\theta_i \cdot v_i$ admits a positive grading.

\subsubsection{Construction}\label{sect:constructstrata}
Let $\ch\Theta_0$ denote an (unordered) multiset in $\mathfrak c_X^-$, by 
which we mean a formal sum 
\[ \sum_{\ch\theta\in \mathfrak c_X^- } N_{\ch\theta} [\ch\theta] \in \Sym^\infty(\mathfrak c_X^-) \]
where all but finitely many $N_{\ch\theta}=0$. 
Note that we include the case $\ch\theta=0$ and $N_0$ may be nonzero. 
See \S\ref{def:partition} for notation.

We have an \'etale map 
\[ \oo C^{\abs{\ch\Theta_0}} = \underset{\ch\theta \in \mf c_X^-}{\oo\prod} \oo C^{N_{\ch\theta}} 
\to \underset{\ch\theta \in \mf c_X^-}{\oo\prod} \oo C^{(N_{\ch\theta})} = \oo C^{\ch\Theta_0}. 
\]
By the factorization property, the base change 
$\oo C^{\abs{\ch\Theta_0}} \xt_{\Sym C} \Scr L^+ X$ identifies with the $\abs{\ch\Theta_0}$-fold 
disjoint product of $C \ttimes \msf L^+ X$. 
By Proposition~\ref{prop:smoothfiberstrata}, we have the locally closed 
subscheme 
\begin{equation} \label{e:productstrata}
\underset{\ch\theta \in \mf c_X^-}{\oo\prod} (C \ttimes \msf L^{\ch\theta} X)^{\oo\xt N_{\ch\theta}} 
\into \oo C^{\abs{\ch\Theta_0}} \xt_{\Sym C} \Scr L^+ X.
\end{equation}
This is stable under the action of $\prod_{\ch\theta} {\mf S}_{N_{\ch\theta}}$,
and therefore \eqref{e:productstrata} descends to 
a locally closed subscheme 
\[ \Scr L^{\ch\Theta_0} X \into 
\oo C^{\ch\Theta_0} \underset{\Sym C}\times \Scr L^+ X =: (\Scr L^+ X)_{\oo C^{\ch\Theta_0}}. \]

\begin{prop} \label{prop:smoothstrata}
The scheme $\Scr L^{\ch\Theta_0} X$ is formally smooth over $\oo C^{\ch\Theta_0}$,
and second projection induces a locally closed embedding $\Scr L^{\ch\Theta_0}X \into 
\Scr L^+ X$. 
\end{prop}
\begin{proof}
Formal smoothness follows from Proposition \ref{prop:smoothfiberstrata}. 
It remains to show that the second projection 
$\on{pr}_2 : \Scr L^{\ch\Theta_0} X \to \Scr L^+ X$
is a locally closed embedding. 
Note that $\on{pr}_2$ is the composition of the finite \'etale map 
\[ \on{pr}'_2: (\Scr L^+ X)_{\oo C^{\ch\Theta_0}} \to 
(\Scr L^+ X)_{\oo C^{( \abs{\ch\Theta_0})}} := \oo C^{(\abs{\ch\Theta_0})} \underset{\Sym C}\times \Scr L^+ X \]
and the open embedding $(\Scr L^+ X)_{\oo C^{(\abs{\ch\Theta_0})}} \into \Scr L^+ X$. 
Hence it suffices to show that the restriction 
$\Scr L^{\ch\Theta_0} X  \subset (\Scr L^+ X)_{\oo C^{\ch\Theta_0}} \to 
(\Scr L^+ X)_{\oo C^{(\abs{\ch\Theta_0})}}$
is locally closed.  
Let $Y$ be the closure of $\Scr L^{\ch\Theta_0} X$ in $(\Scr L^+ X)_{\oo C^{\ch\Theta_0}}$, 
so $Y'=\on{pr}'_2(Y)$ is 
a closed subscheme of $(\Scr L^+ X)_{\oo C^{(\abs{\ch\Theta_0})}}$
and $\on{pr}'_2(\Scr L^{\ch\Theta_0} X)$ is open in $Y'$. 
Observe that the induced map 
\begin{equation} \label{e:productstratapr2} 
\Scr L^{\ch\Theta_0} X \to \on{pr}'_2(\Scr L^{\ch\Theta_0} X)
\end{equation}
is a bijection on geometric points: 
a point in the right hand side consists of an unordered set of points $\{ v_i \} \subset C$
and $x_i \in \msf L^+_{v_i} X$ such that if $\ch\theta_i$ denotes the $G$-valuation
of $x_i$, then $\sum_i [\ch\theta_i] = \ch\Theta_0$. 
There is a unique way to partition the $v_i$'s according to distinct values of $\ch\theta_i$'s
to get a point in $\Scr L^{\ch\Theta_0} X$. 
Therefore \eqref{e:productstratapr2} is \'etale and a bijection on geometric points, hence
an isomorphism. 
\end{proof}

\subsection{Proof of Lemma {\ref{lem:globstrata}}}
\label{proof:globstrata}

Assume that $C$ is complete. 
Let $(\sM_X \xt \Sym C)^\bullet$ denote the substack of $\sM_X \xt\Sym C$
consisting of those pairs $(f,D)$ where $f(C \sm D) \subset X^\bullet/G$, i.e., 
$C\sm D$ is contained in the non-degenerate locus.
This is an open substack since $C$ is complete. 
Define the map
\begin{equation} \label{e:MtoLX}
    (\Scr M_X \times \Sym C)^\bullet \to \Scr L^+ X / \Scr L^+G
\end{equation}
over $\Sym C$ by sending $(f,D)$ to $(f|_{\wh C'_D}, D)$. 
Here we are using the fact that an $\Scr L^+G$-torsor on an affine scheme $S$
is the same as a $G$-torsor on $\wh C'_D$ by formal lifting.

Recall that we defined a partition $\ch\Theta$, or \emph{unordered multiset without zero}, in $\mathfrak c_X^-$, 
to mean an element of $\Sym^\infty(\mf c_X^- \sm 0)$. 
To such a partition $\ch\Theta$, we have a corresponding 
locally closed subscheme $\Scr L^{\ch\Theta} X \into \Scr L^+ X$
by Proposition~\ref{prop:smoothstrata}. 
Then we can \emph{define} $\Scr M^{\ch\Theta}_X$ 
to be the preimage of $\Scr L^{\ch\Theta} X / \Scr L^+G$ under the map 
\eqref{e:MtoLX}. 
One can check from the construction that 
this definition gives the description on $k$-points 
from \S\ref{sect:globstrat}.
By base change $\Scr M^{\ch\Theta}_X$ is a locally closed
subscheme of $\Scr M_X \times C^{(\abs{\ch\Theta})}$. 
In particular, $\Scr M^{\ch\Theta}_X$ is an algebraic 
stack locally of finite type over $k$. 

\begin{lem} \label{lem:MtoLsm}
The natural map $\Scr M^{\ch\Theta}_X \to \Scr L^{\ch\Theta} X / \Scr L^+G$ is formally 
smooth. 
\end{lem}
\begin{proof}
Let $S \into S'$ be a nilpotent thickening of affine schemes. 
Let $(f,D) \in \Scr M_X^{\ch\Theta}(S)$. 
This maps to the point $(f|_{\wh C'_D}, D) \in \Scr L^{\ch\Theta} X(S)$. 
Suppose that we have a lift of this point to 
$(\hat f, D') \in \Scr L^{\ch\Theta} X(S')$, where 
$D'\subset C\times S'$ is a relative effective Cartier divisor
and $\hat f: \wh C'_{D'} \to X/G$ is equivalent to 
the datum of a $G$-bundle $\hat{\Scr P}'_G$ on $\wh C'_{D'}$
and a section $\hat \sigma' : \wh C'_{D'} \to X \times^G \hat{\Scr P}'_G$.

We would like to lift $(f,D)$ to an $S'$-point of $\Scr M_X^{\ch\Theta}$
that maps to $(\hat f,D')$.
The map $f : C\times S \to X/G$
consists of the datum of a $G$-bundle $\Scr P_G$ on $C\times S$
and a section $\sigma : C\times S \to X \times^G \Scr P_G$
satisfying the condition that $\sigma(C\times S \sm D) \subset X^\bullet \times^G \Scr P_G = (H\bs G)\times^G \Scr P_G$.
The restriction $\sigma|_{C \times S \sm D}$ 
gives a reduction of $\Scr P_G|_{C\times S\sm D} \cong G\times^H \Scr P_H$ to an
$H$-bundle $\Scr P_H$ on $C\times S \sm D$ such that
$\sigma|_{C\times S\sm D}$ identifies with
the canonical section 
\[ C\times S \sm D \cong H\bs \Scr P_H \into (H\bs G)\overset H\times \Scr P_H \cong (H\bs G) \overset G\times \Scr P_G|_{C\times S \sm D}\]
corresponding to $H1\in H\bs G$.

The obstruction to lifting $\Scr P_H$ to an $H$-bundle $\Scr P'_H$ on $C\times S' \sm D'$ is an element in $H^2(C\times S\sm D, \mf h_{\Scr P_H} \otimes_{\Scr O_S} I)$ where $I$
is the zero ideal of $S\into S'$ and $\mf h_{\Scr P_H}$ denotes the quasicoherent sheaf on $C\times S \sm D$ obtained
by twisting the adjoint representation of $H$ by $\Scr P_H$.
This obstruction vanishes since $C\times S \sm D$ has
relative dimension $1$ over $S$ and we can compute
cohomology over the Zariski site. 
Thus, we obtained an $H$-bundle $\Scr P'_H$ over $C\times S' \sm D'$.
Let $\sigma' : C\times S'\sm D' \to X \times^G (G\times^H\Scr P'_H)$ denote the corresponding section.

We know that after base change to $S$, there exists an isomorphism
\[ \tau : \hat{\Scr P}'_G|_{\wh C^\circ_D} \cong \Scr P_G|_{\wh C^\circ_D} \cong (G\overset H\times \Scr P'_H)|_{\wh C^\circ_D} \]
such that $\tau \circ \hat\sigma'|_{\wh C^\circ_D} = \sigma'|_{\wh C^\circ_D}$. This is equivalent to a section 
$\beta: \wh C^\circ_D \to 
\Scr P'_H\times^H G\times^G \hat{\Scr P}'_G$
such that $\beta$ is sent under 
\begin{equation} \label{e:twistedGtoX} 
    \Scr P'_H \overset H\times G \overset G\times \hat{\Scr P}'_G \to (C\times S')\times X \overset G\times \hat{\Scr P}'_G
\end{equation}
 to the restriction $\hat \sigma'|_{\wh C^\circ_{D}}$.
The map \eqref{e:twistedGtoX} is smooth since $G\to X^\bullet=H\bs G$ is smooth.
The scheme $\wh C^\circ_{D'}$ is affine since $S'$ is affine.
The zero ideal of $\wh C^\circ_D\into \wh C^\circ_{D'}$ is still
nilpotent, so by formal smoothness of \eqref{e:twistedGtoX}, we
can lift $\beta$ to a section $\beta'$ that maps to $\hat \sigma'|_{\wh C^\circ_{D'}}$.
Such a section $\beta'$ is equivalent to an isomorphism 
$\tau' : \hat{\Scr P}'_G|_{\wh C^\circ_{D'}} \cong (G\times^H \Scr P'_H)|_{\wh C^\circ_{D'}}$ such that $\tau' \circ \hat\sigma' |_{\wh C^\circ_{D'}} = \sigma'|_{\wh C^\circ_{D'}}$.
By Beauville--Laszlo's theorem (Lemma~\ref{lem:B-L}),
the data $((\hat{\Scr P}'_G,\hat\sigma'), (\Scr P'_H,\sigma'), \tau')$ descends to a map $f' : C\times S' \to X/G$. 
By construction, $(f',D')$ is an $S'$-point of $\Scr M^{\ch\Theta}_X$
lifting $f$.
\end{proof}

\begin{proof}[{Proof of Lemma~\ref{lem:globstrata}}] 
Lemma~\ref{lem:MtoLsm} and Proposition~\ref{prop:smoothstrata} together
imply that $\Scr M^{\ch\Theta}_X$ is formally smooth over $k$. 
Since $\Scr M^{\ch\Theta}_X$ is locally of finite type,
it is therefore smooth over $k$.

We claim that the first projection $\on{pr}_1 : \Scr M_X\times C^{(\abs{\ch\Theta})} \to \Scr M_X$ induces a locally closed
embedding $\Scr M^{\ch\Theta}_X \into \Scr M_X$. 
Let $Z \subset \Scr M_X \times C^{(\abs{\ch\Theta})}$ denote the substack
with $S$-points consisting of those $(f,D)$ such that 
$f^{-1}(X^\bullet/G)\cap D = \emptyset$. 
Since $D$ is faithfully flat over $S$, the image of $f^{-1}(X^\bullet/G)\cap D$ in $S$ is open. Therefore, $Z$
is a closed substack of $\Scr M_X \times C^{(\abs{\ch\Theta})}$.
Since $\ch\Theta$ is a multiset without zero, 
observe that $\Scr M^{\ch\Theta}_X$ embeds into 
$Z$. Now $(f,D)\in Z(k)$ satisfies the property that the support 
of $D\in C^{(\abs{\ch\Theta})}$ is contained in the support of $C \sm f^{-1}(X^\bullet/G)$. We deduce that ${\on{pr}_1}|_Z : Z \to \Scr M_X$ is 
proper and quasifinite, hence finite. 
On the other hand, $\on{pr}_1|_{\sM^{\ch\Theta}_X}$ is injective
on $k$-points, and 
$({\on{pr}_1}|_Z)^{-1}(\on{pr}_1(\Scr M^{\ch\Theta}_X)) = \Scr M^{\ch\Theta}_X$. 
From this we deduce that $\on{pr}_1|_{\sM^{\ch\Theta}_X}$ 
is a locally closed embedding.
\end{proof}

\subsection{Generic-Hecke modifications} 
We review the notion of \emph{generic-Hecke modifications} between quasimaps
introduced in \cite[\S 2.2]{GN}, applied to our situation. 
We assume that $B$ acts simply transitively on $X^\circ$ and that $H$ is connected. 

\subsubsection{Function-theoretic analog} 
We explain the idea behind the generic-Hecke modifications at the level of sets; 
this construction appears in the geometric Langlands
program in the construction of Whittaker models, cf.~\cite[\S 5.3.1]{G:outline}.

We use the notation of \S\ref{sect:adelic}.
For any finite subset $\underline v \subset \abs C$, 
let $\bbA^{\ul v} = \prod'_{v'\notin \ul v} F_{v'},\, 
\bbA_{\ul v} = \prod_{v'\in \ul v} F_{v'}$ and similarly for 
$\mbb O^{\ul v}, \mbb O_{\ul v}$. Then we can consider 
the set  
\begin{equation} \label{e:genericHeckeset} 
 G(\bbA^{\ul v})/ G(\mbb O^{\ul v}) \xt H(\bbA_{\ul v})/ H(\mbb O_{\ul v})   
\end{equation}
which maps to $H(\Bbbk) \bs G(\bbA) / G(\mbb O)$ in two ways: 
(i) by projecting along the first factor to $( H(\Bbbk) \xt_{H(\bbA_{\ul v})} H(\mbb O_{\ul v})) \bs G(\bbA^{\ul v}) / G(\mbb O^{\ul v})
\subset H(\Bbbk) \bs G(\bbA) / G(\mbb O)$
and (ii) by projecting to 
\[ H(\Bbbk) \bs \Bigl( G(\bbA^{\ul v}) / G(\mbb O^{\ul v}) \xt H(\bbA_{\ul v})/H(\mbb O_{\ul v}) \Bigr) 
\subset H(\Bbbk) \bs G(\bbA)/ G(\mbb O).
\]
Every meromorphic quasimap (element of $H(\Bbbk) \bs G(\mbb A) / G(\mbb O)$) 
belongs to the image of the second projection for some 
$\underline v$. If $H$ is connected, then by \emph{weak approximation} 
$H(\bbA_{\underline v}) = H(\Bbbk) H(\mbb O_{\underline v})$ so 
every quasimap also belongs to the image of the first projection.
Thus, the union of \eqref{e:genericHeckeset} over all finite subsets $\ul v$ 
defines a groupoid acting on the set of quasimaps.

\subsubsection{}
We define the ind-stack $\Scr H_{H,\sM_X}$ of \emph{generic-Hecke modifications} 
to be the stack classifying data 
\[ (\Scr P^1_G, \Scr P^2_G, \sigma_1,\sigma_2; \ul v, \tau) \]
where $(\Scr P^i_G, \sigma_i) \in \sM_X$, 
$\ul v \in \Sym C$ is a divisor with support $\ul v$ contained in the non-degenerate locus
$\sigma_i^{-1}(X^\bullet \xt^G \sP^i_G)$ for both $i=1,2$
and $\tau$ is an isomorphism of $G$-bundles 
\[ \tau: \Scr P^1_G|_{C\sm {\ul v}} \cong \Scr P^2_G|_{C\sm {\ul v}} \]
such that the following diagram commutes
\[ 
\begin{tikzcd} 
C \sm {\ul v} \ar[r, "\sigma_1"] \ar[rd, swap,"\sigma_2"] & \Scr P^1_G \overset G\times X |_{C\sm {\ul v}} \ar[d, "\tau" ] \\
& \Scr P^2_G \overset G \times X |_{C\sm {\ul v}}. 
\end{tikzcd} 
\]
Note that the definition only depends on the support of $\ul v$ and not its multiplicities. 

We call a generic-Hecke modification \emph{trivial} if the isomorphism $\tau$
extends to an isomorphism over all of $C$. 
We have the natural projections 
\[ \Scr M_{X} \overset{h^\leftarrow}\leftarrow \Scr H_{H,\sM_X} \overset{h^\to}\to \Scr M_{X}, \] 
and $\Scr H_{H,\sM_X} \to \Sym C$. 
By definition, the generic-Hecke correspondence preserves the strata 
$\sM_X^{\ch\Theta}$. 

Define a \emph{smooth generic-Hecke correspondence} to be any stack $U$ equipped
with smooth maps 
\[ \Scr M_{X} \overset{h^\leftarrow_U}\leftarrow U \overset{h^\to_U}\to \Scr M_{X} \] 
such that there exists a map $U\to \Scr H_{H,\sM_X}$ such that
the following diagram commutes
\[ 
\begin{tikzcd} 
& U \ar[ld, swap, "h^\leftarrow_U" near start] \ar[d] \ar[rd, "h^\to_U" near start] \\
\Scr M_{X} & \Scr H_{H,\sM_X} \ar[l,swap,"h^\leftarrow" near start] \ar[r,"h^\to" near start] &
\Scr M_{X} 
\end{tikzcd}
\]
Call a smooth generic-Hecke correspondence $U$ trivial if the image of 
$U \to \Scr H_{H,\sM_X}$ consists of trivial generic-Hecke modifications. 

Define a morphism of smooth generic-Hecke correspondences to be a map 
$p : U_1 \to U_2$
such that $h^\leftarrow_{U_2}\circ p = h^\leftarrow_{U_1}$ and $h^\to_{U_2}\circ p = h^\to_{U_1}$.

\subsubsection{Generic-Hecke equivariant sheaves}
We define a \emph{generic-Hecke equivariant} perverse sheaf on $\sM_X$ 
to be an object $\sF \in \rmP(\sM_X)$ of the category of perverse sheaves on $\sM_X$
equipped with isomorphisms 
\[ \phi_U : h^{\leftarrow *}_U(\sF) \overset\sim\to h_U^{\rightarrow *}(\sF) 
\] 
for every smooth generic-Hecke correspondence $U$, satisfying some natural 
conditions (see \cite[\S 2.3]{GN} for details).

For any $\ch\Theta \in \Sym^\infty(\mf c_X^- \sm 0)$, 
the uniqueness of the IC complex endows $\IC_{\barM^{\ch\Theta}_X}$ 
with the structure of a generic-Hecke equivariant perverse sheaf. 
(In general, when $H$ is connected, the condition of generic-Hecke equivariance
is a \emph{property}, not additional structure, of a perverse sheaf on $\sM_X$
by \cite[Proposition 3.5.2]{GN}.)

\subsubsection{}
Fix $\ch\theta \in \mathfrak c_X^- \sm 0$. 
Recall that by definition $\sM_X^{\ch\theta}$ is a substack of $\sM_X \xt C$. 
For a fixed $v_0 \in \abs C$, define $\sM_{X,v_0}^{\ch\theta} := \sM_X^{\ch\theta} \xt_C v_0$
to be the \emph{based} stratum consisting of maps $C \to X/G$ such that $C\sm v_0$
maps to $H\bs \pt$ and the $G$-valuation at $v_0$ is $\ch\theta$.
Observe that the generic-Hecke modifications preserve the substack
$\sM_{X,v_0}^{\ch\theta}$. 

\begin{prop}[{\cite[Proposition 3.5.1]{GN}}] 
\label{prop:genHecke-transitive}
If $H$ is connected, then all geometric points of $\sM_{X,v_0}^{\ch\theta} \subset \sM_X$
are equivalent under the equivalence relation generated by 
the generic-Hecke correspondences. 
\end{prop}

The preimage of $C\subset \Sym C$ in $\Scr H_{H,\sM_X}$ may be realized as the
twisted product of an open substack of $\sM_X \xt C$ with the affine Grassmannian
$\Gr_H$. From this it is not hard to deduce that all geometric points of 
$\sM_{X,v_0}^{\ch\theta}$ are equivalent under \emph{smooth} generic-Hecke correspondences. 

\subsubsection{} 
Recall the $G$-Hecke action defined in Proposition~\ref{prop-const:action}. 
By construction, this action commutes with the generic $H$-Hecke modifications (in 
our set-theoretic description above, we are considering the action of right
$G(F_{v'})$-translation at some point $v' \ne v_0$):

\begin{lem} \label{lem:act-equivariant}
The map $\act_C : \Scr M_X \ttimes \ol\Gr^{\ch\theta}_{G,C} \to \Scr M_X \xt C$ is equivariant with respect to the generic-Hecke modifications
away from the marked point in $C$. 
\end{lem}

\subsubsection{Proof of Proposition \ref{prop:globwhitney}} \label{proof:whitney}

Proposition~\ref{prop:genHecke-transitive} allows us to consider the based strata with 
generic-Hecke modifications
as an analog of the stratification by orbits of a group action. 
This will be the idea behind the proof of Proposition~\ref{prop:globwhitney}, together
with the following fact: 

\begin{thm}[{\cite[Theorem 2]{Kaloshin}}]  \label{thm:K}
Let $S$ be a smooth stratum in an algebraically stratified scheme of finite type over $k=\mbb C$. 
Then the subset of points in $S$ that do not satisfy Whitney's condition B 
form a constructible subset of dimension strictly lower than $\dim S$.
\end{thm}

Thus, if we show that any two points in a given stratum have neighborhoods in $\sM_X$ 
that are smooth locally isomorphic and compatible with the stratification, 
Theorem~\ref{thm:K} will imply that every point
must satisfy Whitney's condition B.

\smallskip

Fix a connected component of $\sM_X^{\ch\Theta}$ 
for some $\ch\Theta \in \Sym^\infty(\mathfrak c_X \sm 0)$. 
Since Whitney's condition B is local in the smooth topology, 
it suffices by Corollary~\ref{cor:Mcover} to show that every point in 
$\sY^{\ch\lambda,\ch\Theta}$
satisfies Whitney's condition for all $\ch\lambda \in \mathfrak c_X$.
Now the graded factorization property of $\sY$ allows us to 
to reduce to the case where $\ch\Theta = [\ch\theta]$ is singleton. 
By Proposition~\ref{prop:changecurve}, we may replace our curve $C$ 
with $\bbA^1 = \mbb P^1 \sm \infty$, i.e., $\sY^{\ch\lambda}=\sY^{\ch\lambda}(\bbA^1)$.
Now $\bbA^1$ acts on itself by translation, which induces
an $\bbA^1$-action by automorphisms on $\sY^{\ch\lambda}(\bbA^1)$. 
This action allows us to move the degenerate point of any 
$y \in \sY^{\ch\lambda,\ch\theta}(\bbA^1)$ to $v_0=0 \in \bbA^1$, i.e., 
the map $y:\bbA^1 \to X/B$ sends $\bbA^1 \sm 0 \to X^\bullet/B$ and 
the $G$-valuation at $v_0$ equals $\ch\theta$.
Thus, we are reduced to showing that any point in 
$\sY^{\ch\lambda,\ch\theta}$ with degenerate point at $v_0$ satisfies Whitney's condition.

Embed $\bbA^1 = \mbb P^1 \sm \infty \subset \mbb P^1$ so we 
also have an open embedding $\sY^{\ch\lambda}(\bbA^1) \subset \sY^{\ch\lambda}(\mbb P^1)$. 
Lemma~\ref{lem:localglobalyoga} shows that $\sY^{\ch\lambda}(\bbP^1)$ 
is smooth locally isomorphic to $\sM_X=\sM_X(\bbP^1)$ in a way that preserves strata 
and degenerate points. Therefore, we may reduce to checking
that any point in $\sM_{X,v_0}^{\ch\theta}$ satisfies Whitney's condition B.
Proposition~\ref{prop:genHecke-transitive} implies that
all such points are equivalent under smooth generic-Hecke correspondences 
(which preserve strata), so either they all satisfy or fail to satisfy
Whitney's condition. By Theorem~\ref{thm:K}, we deduce that every
point satisfies Whitney's condition B.

\medskip

The same argument as above also shows that the closure of any stratum in $\sM_X$
must equal a union of strata. 

\subsubsection{Arbitrary characteristic} 

Let the characteristic of $k$ be arbitrary. 

\begin{prop} \label{prop:Whitney-poschar}
For any $\ch\Theta \in \Sym^\infty(\mf c_X^-\sm 0)$,
the complex $\IC_{\barM^{\ch\Theta}_X}$ 
is constructible with respect to the fine stratification of $\sM_X$.
\end{prop}

\begin{proof}
The argument is exactly the same as in the proof of Proposition~\ref{prop:globwhitney} 
above, except that we replace the use of Theorem~\ref{thm:K} with
the fact that for any connected smooth stratum $S$ in $\sM_X$, 
there exists some open $U\subset S$ such that 
$\rmH^i(\IC_{\barM^{\ch\Theta}_X})|_U$ is a local system for all $i$.
We implicitly use the uniqueness of the IC complex, which in
particular implies that $\IC_{\barM^{\ch\Theta}_X}$ is generic-Hecke equivariant.
\end{proof}

\subsection{Proof of Theorem~\ref{thm:heckeIC}} \label{sect:pf:heckeIC}

For $\ch\Theta \in \Sym^\infty(\mathfrak c_X^-\sm 0)$, 
let $\rmP_{\Scr L^+ G}(\ol\Gr^{\ch\Theta}_{G,C^{\ch\Theta}})$ denote the
category of perverse sheaves on $\ol\Gr^{\ch\Theta}_{G,C^{\ch\Theta}}$ that
are equivariant with respect to the action of $(\Scr L^+ G)_{C^{(\abs{\ch\Theta})}}$, 
considered as a group scheme over $\Sym C$ 
(defined in \S\ref{sect:multijet}). 
For $\Scr F \in \rmP(\Scr M_X)$ and 
$\Scr G \in \rmP_{\Scr L^+G}(\ol\Gr^{\ch\Theta}_{G,C^{\ch\Theta}})$, we 
can form the twisted external product 
\[ \Scr F \mathbin{\tilde\bt} \Scr G \in \rmP(\Scr M_X \ttimes \ol\Gr^{\ch\Theta}_{G,C^{\ch\Theta}}) \]
with respect to the projections of 
$\widehat{\Scr M_X} \xt^{\Scr L^+G} \ol\Gr^{\ch\Theta}_{G,C^{\ch\Theta}}$ to 
$\Scr M_X$ and $\Scr L^+G \bs \ol\Gr^{\ch\Theta}_{G,C^{\ch\Theta}}$.
Then we define the external convolution product by 
\[ \Scr F \star \Scr G := \act_{\sM,!}( \Scr F \mathbin{\tilde\bt} \Scr G ) \in \rmD^b_c(\Scr M_X). \]

We have introduced all the ingredients in the statement of Theorem \ref{thm:heckeIC}. The remainder of this section will be devoted to its proof.

Observe that $\IC_{\barM_X^0} \mathbin{\tilde\bt} 
\IC_{\ol\Gr^{\ch\Theta}_{G,C}}
= \IC_{\barM^0_X \ttimes \ol\Gr^{\ch\Theta}_{G,C}}$, which we 
will simply denote $\Scr{IC}$ for brevity. 
We have a stratification 
\[ \barM_X^0 \ttimes\ol\Gr^{\ch\Theta}_{G,C^{\ch\Theta}} 
= \bigcup_{\ch\Theta',\ch\Theta''} 
 \barM_X^{\ch\Theta'} \ttimes \Gr^{\ch\Theta''}_{G,\oo C^{\ch\Theta''}} \] 
running over all $\ch\Theta' \succeq 0$ 
and $\ch\Theta'' \ge \ch\Theta_1$ where $\ch\Theta$ refines $\ch\Theta_1$ (the definition of $\ch\Theta'' \ge \ch\Theta_1$ is analogous to that of $\succeq$).  
Lemma~\ref{lem:actimage}(i) implies that $\act_\sM^{-1}(\sM^{\ch\Theta}_X)$
is contained in the dense open stratum corresponding to 
$\ch\Theta'=0$ and $\ch\Theta'' = \ch\Theta$.

\subsubsection{} \label{sect:ICglobineq}
By Proposition~\ref{prop:globwhitney} we know that 
$\IC_{\barM_X^0}$ is constructible with respect to the fine stratification of 
$\sM_X$. Therefore, $\Scr{IC}$ is constructible
with respect to the stratification above, so 
if we let $L_{\mathrm{glob}}^{\ch\Theta', \ch\Theta''}$ denote
the $*$-restriction of $\Scr{IC}$ to the stratum 
$\sM_X^{\ch\Theta'} \ttimes \Gr^{\ch\Theta''}_{G,\oo C^{\ch\Theta''}}$, we have that its cohomology sheaves are local systems. 
Since $L^{\ch\Theta', \ch\Theta''}_{\mathrm{glob}}$ is the 
restriction of an IC complex, it lives in perverse cohomological degrees 
$\le 0$, and the inequality is strict unless $\ch\Theta'=0$ and $\ch\Theta''=\ch\Theta$. 
Since its cohomology sheaves are local systems, this implies that 
$L^{\ch\Theta', \ch\Theta''}_{\mathrm{glob}}$ 
lives in usual cohomological degrees 
$\le -\dim( \sM_X^{\ch\Theta'} \ttimes \Gr^{\ch\Theta''}_{G,\oo C^{\ch\Theta''}})$, and the inequality is strict unless 
$\ch\Theta'=0$ and $\ch\Theta''=\ch\Theta$. 
We are abusing notation here: $\sM^{\ch\Theta'}_X$ may not be connected, in which case each connected component should be considered separately. 
 
\smallskip

We abuse notation and also use $L^{\ch\Theta',\ch\Theta''}_{\mathrm{glob}}$
to denote its $!$-extension to $\sM_X \ttimes \ol\Gr^{\ch\Theta}_{G, C^{\ch\Theta}}$. 
Then by the characterizing properties of the intermediate extension 
(and the fact that $\Scr{IC}$ is Verdier self-dual), 
Theorem~\ref{thm:heckeIC} is equivalent
to the following assertion: 

\begin{prop} \label{prop:heckeIC}
For $\ch\Theta',\ch\Theta''$ as above, 
consider the complex $\act_{\sM,!}(L^{\ch\Theta',\ch\Theta''}_{\mathrm{glob}})$.  
Then:
\begin{enumerate}
\item It lives in $^p\rmD^{\le -1}(\sM_X)$ unless $\ch\Theta'=0$ and $\ch\Theta''=\ch\Theta$.
\item The $*$-restriction of 
$\act_{\sM,!}(L^{0,\ch\Theta}_{\mathrm{glob}})$ to $\barM^{\ch\Theta}_X\sm \sM^{\ch\Theta}_X$ lives in perverse cohomological degrees $\le -1$.
\item There is a natural identification
$\act_{\sM,!}(L^{0,\ch\Theta}_{\mathrm{glob}})|^*_{\sM^{\ch\Theta}_X} \cong \IC_{\sM^{\ch\Theta}_X}$. 
\end{enumerate}
\end{prop}

Point (iii) follows immediately from Theorem \ref{thm:comp}(iii). 
Points (i)-(ii) concern the cohomological degrees 
of the $*$-restriction of $\act_{\sM,!}(L^{\ch\Theta',\ch\Theta''}_{\mathrm{glob}})$
to a stratum of $\barM^{\ch\Theta}_X \sm \sM_X^{\ch\Theta}$. 
To simplify notation we only consider the restriction
to a stratum of the form $\sM_X^{\ch\eta}$ where $\ch\eta \in \mf c_X^- \sm 0$. 
The general case is proved in the same way by considering multiple points on $C$ 
simultaneously. 

\smallskip

By Theorem \ref{thm:comp} and Proposition \ref{prop:closure-rel}, the preimage 
$\act_\sM^{-1}(\sM_X^{\ch\eta})$ intersects $\sM_X^{\ch\Theta'} \ttimes \Gr^{\ch\Theta''}_{G,\oo C^{\ch\Theta''}}$ only if 
$\ch\Theta' = [\ch\theta']$ and $\ch\Theta'' = [\ch\theta'']$ are singletons 
($\ch\theta'$ is allowed to be $0$). 
Moreover we must have $\ch\theta'' \ge \deg(\ch\Theta)$ and $\ch\eta \succeq \ch\theta' + \ch\theta''$. 

\subsubsection{Fiber dimension}
Fix a point $v \in \abs C$ and consider the subscheme $\sM_{X,v}$
of maps where there is only one $G$-degenerate point at $v$. 
Let $\sM_{X,v}^{\ch\eta}$ denote the substack of $\sM_{X,v}$ where the $G$-valuation at
this degenerate point is $\ch\eta$.
Then the preimage $\act_\sM^{-1}(\sM_{X,v}^{\ch\eta}) \cap (\sM_X^{\ch\theta'} \ttimes \Gr^{\ch\theta''}_{G,C})$ is contained in $\sM_X^{\ch\theta'} \ttimes \Gr^{\ch\theta''}_{G,v}$.
The restricted map 
\begin{equation} \label{e:actrestricted}
 \act_\sM : \sM_X^{\ch\theta'} \ttimes \Gr^{\ch\theta''}_{G,v} \to \sM_{X,v} 
\end{equation}
is equivariant with respect to generic-Hecke modifications away from $v$ by Lemma \ref{lem:act-equivariant}. Meanwhile these modifications act transitively on the stratum 
$\sM_{X,v}^{\ch\eta}$. 
Therefore we deduce that the fibers of \eqref{e:actrestricted} over 
all points in $\sM_{X,v}^{\ch\eta}$  are isomorphic to one another. 
Let this fiber dimension be denoted $d$. 

We have a special point $\{ t^{\ch\eta} \} = \msf Y^{\ch\eta,\ch\eta}_v \to \sM^{\ch\eta}_{X,v}$
by Corollary~\ref{cor:Ytheta}. 
We deduce from Proposition~\ref{prop:Yact-strat} that 
the preimage of $\msf Y^{\ch\eta}_v$ under the map \eqref{e:actrestricted} has a 
stratification by 
\[ \msf Y^{\ch\eta - \ch\nu, \ch\theta'} \ttimes (\msf S^{\ch\nu} \cap \Gr^{\ch\theta''}_G) \]
where $\ch\nu$ runs over the weights of $V^{\ch\theta''}$.
Therefore
\begin{equation} \label{e:actfiberdim}
d \le \underset{\ch\nu}{\on{max}} (\dim \msf Y^{\ch\eta-\ch\nu,\ch\theta'} + \brac{\rho_G,\ch\nu-\ch\theta''}). 
\end{equation}

\subsubsection{Passage to Zastava model}
By Corollary~\ref{cor:Mcover}, it suffices to prove 
Proposition~\ref{prop:heckeIC} after base change to the Zastava model.
Consider the fiber product diagram 
\[\begin{tikzcd}
Z^{?,\ch\lambda,\ch\theta',\ch\theta''} \ar[r] \ar[d] & \sY^{\ch\lambda} \ar[d] \\
\sM_X^{\ch\theta'} \ttimes \Gr^{\ch\theta''}_{G,C} \ar[r, "\act_\sM"] & 
\sM_X 
\end{tikzcd}
\]
which is the analog of \eqref{e:Z0act-diagram} where we allow $v\in \abs C$ to vary. 
We pull back along this diagram for all $\ch\lambda$ large enough. 
The corresponding diagram for strata is given by \eqref{e:Z0act-diagram}.
We $*$-pullback with a shift by the fiber dimension of 
$\sY^{\ch\lambda}\to \sM_X$. 
With this shift, the discussion of \S\ref{sect:ICglobineq} 
implies that $L^{\ch\Theta',\ch\eta}_\glob$ goes to 
a complex 
\[ L^{?,\ch\lambda,\ch\theta',\ch\theta''}_\flag \in \rmD^b_c(Z^{?,\ch\lambda,\ch\theta',\ch\theta''}), \]
which lives in usual cohomological degrees 
$\le -\dim( Z^{?,\ch\lambda,\ch\theta',\ch\theta''} )$, and
the inequality is strict unless $\ch\theta'=0$ and $\ch\theta''=\ch\theta$. 
The usual cohomology sheaves of $L^{?,\ch\lambda,\ch\theta',\ch\theta''}_{\mathrm{flag}}$ 
are local systems. 

To prove Proposition~\ref{prop:heckeIC} it is
enough, by the definition of the perverse t-structure,
to prove the following:

\begin{lem} \label{lem:heckeICflag}
Let $\ch\theta',\ch\theta'',\ch\eta,\ch\lambda,d$ be as above. 
Then we have 
\[
    -\dim( Z^{?,\ch\lambda,\ch\theta',\ch\theta''} ) + 2 d \le - \dim \sY^{\ch\lambda,\ch\eta} 
\]
and the inequality is strict if $\ch\theta' = 0$ and $\ch\theta'' \ne \ch\eta$. 
\end{lem}
Note that the statement of the lemma has nothing to do with $\ch\theta$.

\begin{proof}
By \S\ref{sect:freemonoid}, we may assume that $\mf c_{X^\bullet} = \mbb N^{\Cal D}$. 
Now it makes sense to talk about $\len(\ch\lambda)$ for $\ch\lambda \succeq 0$, cf.~\S\ref{sect:Y0cc}. 
By Lemma \ref{lem:oYconnected} and Corollary~\ref{cor:Ycomp2},
we have that $\dim \sY^{\ch\lambda,\ch\eta} = \len(\ch\lambda-\ch\eta)+1$. 

For $\ch\nu$ a weight of $V^{\ch\theta''}$, we have 
$\dim \msf Y^{\ch\eta-\ch\nu,\ch\theta'} \le \frac 1 2(\dim \sY^{\ch\eta-\ch\nu,\ch\theta'}-1)$
by Proposition~\ref{prop:centralstrata}, unless $\ch\eta=\ch\nu$.  
If $\ch\theta'=0$, then $\dim \sY^{\ch\eta-\ch\nu,0} = \len(\ch\eta-\ch\nu)$. 
Otherwise $\dim \sY^{\ch\eta-\ch\nu,\ch\theta'} = \len(\ch\eta-\ch\nu)+1$. 
We also have $\brac{\rho_G,\ch\nu-\ch\theta''} = \frac 1 2 \len(\ch\nu-\ch\theta'')$. 
Therefore \eqref{e:actfiberdim} implies that 
\begin{equation} \label{ineq:dfiber}
 d \le \frac 1 2 \len(\ch\eta-\ch\theta'-\ch\theta'')
\end{equation}
and the inequality is strict if $\ch\theta' = 0$ and $\ch\eta$ is not a weight of $V^{\ch\theta''}$. 

In the case $\ch\theta'=0$ and $\ch\eta$ is a weight of $V^{\ch\theta''}$ not equal to 
$\ch\theta''$, we claim the inequality \eqref{ineq:dfiber} above is still strict. 
Indeed, equality can only hold if $\ch\eta=\ch\nu$ and an open subvariety of 
$\msf Y^{0,0} \ttimes (\msf S^{\ch\eta}\cap \Gr^{\ch\theta''}_G)$ is sent to 
the special point $t^{\ch\eta}$ under \eqref{e:actrestricted}. 
However we know from Lemma~\ref{lem:MVcentralfiber} that every irreducible
component of $\msf S^{\ch\eta}\cap \Gr^{\ch\theta''}_G$ generically maps to 
the stratum $\sM^{\ch\theta''}_G$. Thus if $\ch\eta \ne \ch\theta''$, 
the inequality \eqref{ineq:dfiber} is strict.

Thus to prove the lemma it suffices to show that 
\[
 \len(\ch\lambda-\ch\theta'-\ch\theta'')+1 = \len(\ch\eta-\ch\theta'-\ch\theta'') + \len(\ch\lambda-\ch\eta) + 1 \le 
\dim (Z^{?,\ch\lambda,\ch\theta',\ch\theta''}).
\]
Proposition~\ref{prop:Yact-strat} (twisted by $C$) implies that 
$Z^{?,\ch\lambda,\ch\theta',\ch\theta''}$ has a stratification by 
\[ \sY^{\ch\lambda-\ch\nu, \ch\theta'} \ttimes (\msf S^{\ch\nu}\cap \Gr^{\ch\theta''}) \ttimes C,\]
where $\ch\nu$ ranges over weights of $V^{\ch\theta''}$ (in particular, $\ch\nu\ge \ch\theta''$). 
By Corollary~\ref{cor:Mcover}(ii), if
$\sY^{\ch\lambda-\ch\nu,\ch\theta'}$ maps to a connected component $M$ of 
$\sM^{\ch\theta'}_X$, then $\sY^{\ch\lambda-\ch\theta'', \ch\theta'}$ must map to the
same connected component. 
Therefore $Z^{?,\ch\lambda,\ch\theta',\ch\theta''}$ is irreducible with 
dense open stratum
$\sY^{\ch\lambda-\ch\theta'',\ch\theta'} \times C$. 
Therefore 
$\dim Z^{?,\ch\lambda,\ch\theta',\ch\theta''} = \dim(\sY^{\ch\lambda-\ch\theta'',\ch\theta'})+1
\ge \len(\ch\lambda-\ch\theta'-\ch\theta'') + 1$, as desired.
\end{proof}

This completes the proof of Theorem~\ref{thm:heckeIC}.

\bibliographystyle{amsalpha}
\bibliography{sphericalL}

\end{document}